\documentclass[a4paper,11pt,english,reqno]{amsart}
\usepackage{mystyle}

\usepackage{bm}
\usepackage{upgreek}
\usepackage{mathabx}

\numberwithin{equation}{section}
\usepackage[normalem]{ulem} 
\usepackage[a4paper, margin=2.4cm]{geometry}
\usepackage{graphicx,xcolor}

\newcommand\cH{\mathcal{H}}
\newcommand\cJ{\mathcal{J}}
\newcommand\cE{\mathcal{E}}
\newcommand\cF{\mathcal{F}}
\newcommand\cR{\mathcal{R}}

\newcommand\cP{\mathcal{P}}

\newcommand\cU{\mathcal{U}}
\newcommand\B{\mathbb{B}}

\newcommand\T{\mathbb{T}}
\newcommand\I{\mathbb{I}}
\newcommand\J{\mathbb{J}}
\newcommand\sD{\mathsf{D}}
\newcommand\sE{\mathsf{E}}
\newcommand\gf{\mathfrak{f}}
\newcommand\gP{\mathfrak{P}}
\newcommand\gG{\mathfrak{G}}
\newcommand\gI{\mathfrak{I}}
\newcommand\gJ{\mathfrak{J}}

\DeclareMathOperator{\dBer}{Ber}
\newcommand{\wt}{\widetilde}

\newcommand{\wh}{\widehat}
\newcommand{\vep}{\varepsilon}
\renewcommand{\Im}{{\rm Im} \,}
\renewcommand{\Re}{{\rm Re} \,}

\def\corEpty{}

\title[Large deviations of the largest eigenvalue]{Large deviations of the largest eigenvalue of supercritical sparse Wigner matrices}  
\author{Fanny Augeri} 
\address[Fanny Augeri]{Laboratoire de Probabili\'es, Statistique et Mod\'elisation (LPSM), Universit\'e de Paris, 75205 Paris Cedex 13, France.}
\email{augeri@lpsm.paris}
\author{Anirban Basak}
\address[Anirban Basak]{International Centre for Theoretical Sciences, Tata Institute of Fundamental Research, Bangalore 560089, India.}
\email{anirban.basak@icts.res.in}
\begin{document}
\maketitle

\begin{abstract}
Consider a random symmetric matrix with i.i.d.~entries on and above its diagonal that are products of Bernoulli random variables and random variables with sub-Gaussian tails. Such a matrix will be called a sparse Wigner matrix and can be viewed as the adjacency matrix of a random network with sub-Gaussian weights on its edges.  In the regime where the mean degree is at least logarithmic in dimension, the edge eigenvalues of an appropriately scaled sparse Wigner matrix stick to the edges of the support of the semicircle law. We show that in this sparsity regime, the large deviations upper tail event of the largest eigenvalue of a sparse Wigner matrix with sub-Gaussian entries is generated by either the emergence of a high degree vertex with a large vertex weight or that of a clique with large edge weights. Interestingly, the rate function obtained is discontinuous at the typical value of the largest eigenvalue, which accounts for the fact that its large deviation behaviour is generated by finite rank perturbations. This complements the results of Ganguly and Nam \cite{GN22}, and Ganguly, Hiesmayr, and Nam \cite{GHN22} which considered the case where the mean degree is constant. 
\end{abstract}

%\tableofcontents

\section{Introduction and main results}
\subsection{Setup and main results} The atypical behaviour of the spectrum of Wigner matrices has been challenging the standard theory of large deviations for decades. Owing to the intricate relationship between the entries of a matrix and its spectrum, most of the results for dense random matrices were known either for integrable models (see \cite{Dembo, BenArous}) or Wigner matrices with heavy tails (see \cite{LDPei, LDPtr, Bordenave}). Regarding the edge eigenvalues, a recent breakthrough was achieved by Guionnet and Husson \cite{GH20}, showing a universal large deviation behaviour for Wigner matrices with `sharp sub-Gaussian tails', and later it was extended for general sub-Gaussian entries in \cite{AGH21, CDG23}.   On the other hand, the large deviations of the extreme eigenvalues of sparse Wigner matrix models  are only well understood \cite{GHN22, GN22} when the typical degrees, i.e.~squared $\ell_2$ norms of its columns, are $O(1)$. The goal of the present paper is to close this gap in the literature between this very sparse regime and the dense regime where the typical degree is $O({n})$, $n$ being the dimension of the matrix, and compute the large deviation upper tail of the top eigenvalue in the optimal regime where the typical degrees grow much faster than $\log n$.

\begin{Def} (Sparse Wigner matrix)\label{defmodel}
Let $G$ be a symmetric random matrix of size $n \times n$ with entries $\{G_{i,j}\}_{i,j \in [n]}$ (where $[n]:=\{1,2,\ldots, n\}$) such that $\{G_{i,j}\}_{i<j}$ are i.i.d.~random variables with zero mean and unit variance, and $\{G_{i,i}\}_{i\leq n}$ are also i.i.d.~random variables with zero mean and possibly some other common law. We then set $\widehat X:= G \circ \Xi$, the Hadamard product of $G$ and $ \Xi := (\xi_{i,j})_{i,j \in [n]}$, where $ \Xi$ is a symmetric random matrix of size $n \times n$ with i.i.d.~$\dBer(p)$ entries on and above the diagonal. Finally, let $X:= \widehat X/\sqrt{np}$. 
\end{Def}

We impose the following assumption on the entries of $G$. %\corAQ{add sub-Gaussian concentration}

\begin{Hypo}[Sub-Gaussian entries]\label{hypo:subg-entry}
For any $i,j\in[n]$ denote $\Lambda_{i,j}$ to be the log-Laplace transform of $G_{i,j}$. That is for any $i,j\in[n]$ and $\theta\in\RR$, $\Lambda_{i,j}(\theta) := \log \EE \exp(\theta G_{i,j})$. Assume that 
\begin{equation} \label{subgauss}\sup_{\theta \neq 0} \theta^{-2} \max\{\Lambda_{1,1}(\theta), \Lambda_{1,2}(\theta)\} < \infty,\end{equation} and
further
\begin{equation} \label{subgaussian}
{\alpha} :=\lim_{\theta\to +\infty} \frac{2\Lambda_{1,1}(\theta)}{\theta^2}  \quad \text{ and } \quad {\beta}:=\lim_{\theta\to +\infty} \frac{2\Lambda_{1,2}(\theta)}{\theta^2} \ge \limsup_{\theta\to -\infty} \frac{2\Lambda_{1,2}(\theta)}{\theta^2}. \end{equation}
\end{Hypo}

Assumption \eqref{subgauss} merely says that the entries of $G$ are sub-Gaussian, whereas \eqref{subgaussian} imposes that the left tail of $G_{1,2}$ are asymptotically heavier than its right tail. %, \corAQ{\sout{which is an assumption also made in \cite{GHN22}}}. 
We refer the reader to Section \ref{background} for a discussion on this assumption.

\begin{Def}[Rate function]\label{def:rate-fn}
To define the relevant rate function we set
\[
L(\theta):= \EE [\exp({\theta G_{1,2}^2})], \, \theta \in \RR, \qquad \text{ and } \qquad h_L(x) :=\sup_{\theta \in \RR}\{ \theta x - L(\theta)+1\}, \, \text{ for } x \in \RR.
\]
Further let $m: \CC \setminus [-2,2] \mapsto \CC$ be the Stieltjes transform of the standard semicircle distribution, given by
\begin{equation}\label{eq:m-def}
m(z):= \frac{1}{2\pi}\int_{-2}^2 \frac{\sqrt{4-x^2}}{z -x}  dx = \frac{z - \sqrt{z^2 -4}}{2}, \quad z \in \CC\setminus [-2,2],
\end{equation}
where the rightmost equality is well known (e.g.~see \cite[Eqn.~(2.4.7)]{AGZ}). Set
\begin{equation*}\label{eq:rate-fn}
\widehat I(\lambda):= \inf \Big\{\frac{r^2}{2\alpha} + h_L(s) : r + m(\lambda)s=\lambda, r \ge0, s \ge 1 \Big\} \quad \text{ and } \quad I(\lambda):=\min\Big\{ \frac{1}{4\beta m(\lambda)^2}, \widehat I(\lambda) \Big\},
\end{equation*}
for $\lambda  \in (2,+\infty)$, with the convention that $\widehat I(\lambda)=h_L(\lambda/m(\lambda))$ in the case where $\alpha=0$. 
\end{Def}

%Let $X = \xi \odot G/\sqrt{np}$, where $\xi$ is a random matrix independent of $G$ such that $(\xi_{i,j})_{i\leq j}$ are i.i.d.~Bernoulli $p$ random variables.

For any symmetric matrix $W$ of size $n \times n$ we write $\lambda_n(W) \le \cdots \le \lambda_2(W) \le \lambda_1(W)=: \lambda_W$ to be its eigenvalues arranged in a non increasing order. Below is the main result of this paper. 

\begin{The}[Sub-Gaussian case] \label{theo-main-weight}
Let $p$ be such that $\log n/n \ll p \ll 1$. Under Assumption \ref{hypo:subg-entry} and for $\lambda>2$,
\[ 
\lim_{n\to +\infty} \frac{1}{np} \log \PP\left( \lambda_X>\lambda \right) = - I(\lambda).
\]
\end{The}

The matrix $X$ can be viewed as a random network with edges determined by the random symmetric Bernoulli matrix $\Xi$, and edge and vertex weights determined by the off diagonal and the diagonal entries of $G/\sqrt{np}$, respectively. Hereafter, the degree of vertex $i \in [n]$ in the random network $X$ will mean the $\ell^2$ norm of the $i$-th column $X$ after zeroing out the $i$-th diagonal entry.

\begin{Rem}\label{rem:I-prop}
The rate function $I$ does not admit a closed form expression for arbitrary $\alpha, \beta$, and general sub-Gaussian distributions. Here we list some key properties of $I$.\begin{enumerate}

\item[(i)]
The rate function $I$ accounts for the phenomenon that the upper tail large deviation event is generated by the presence of either a vertex with high degree and large weight in the random network induced by $X$ (reflected through the definition $\wh I$) or a clique with large edge weights and of diverging but not too large size (determined by the speed) which accounts for the term $1/(4\beta m(t)^2)$ in the definition of $I(t)$. This will be more evident from the outline of the proof of the large deviations lower bound. See Section \ref{sec:lbd-outline}. 

\item[(ii)] If $\alpha \geq 2\beta$, then $I(t) = \widehat I(t)$ for any $t>2$. Indeed, note that the %functional
equation %of the Stieltjes transform of the semicircle law 
$m(t) + 1/m(t) = t$, implies that we can take $s=1$ and $r=1/m(t)$ in the variational problem defining $\widehat I(t)$ in \eqref{eq:rate-fn}. This yields $\widehat I(t) \leq 1/(2\alpha m(t)^2)$, and as a consequence $I(t) = \widehat I(t)$ if $\alpha\geq 2\beta$. Thus in this case the large deviation event is generated only by the presence of a high degree vertex with a large vertex weight.

\item[(iii)] Assume that the entries of $G$ are bounded or have lighter tails than a sub-Gaussian random variable in the sense that $\alpha=\beta=0$. Then the rate function simplifies to
\[
I(t) = h_L\Big( \frac{t}{m(t)}\Big), \qquad t >2.
\]
%\corAB{So in this case the rate function admit a simpler looking expression and the upper tail event of $\lambda_X$ is generated by the presence of a large degree vertex in the random network induced by $X$. A similar phenomenon was observed in the context of the spectral radius of a non centred adjacency matrix of a sparse Erd\H{o}s-R\'enyi graph for the sparsity regimes covered in \cite{BBG21, B21}.} 

\item[(iv)] When $\alpha \vee \beta > 0$, it can be checked from the definition that the asymptotic behaviour of $I$ near infinity is quadratic and is given by
\[ 
\lim_{t \to +\infty} t^{-2}I(t)= \frac{1}{4\max(\beta,\frac{\alpha}{2})}.
\]

%\item[(iii)] Let $\{G_{i,j}\}_{i < j}$ be i.i.d.~Rademacher random variables. Then 
%\[
%h_L(x)= x\log x - x+1, \qquad x >0. 
%\]
%
%\item[(iv)]  In the case $G_{i,j}$ are standard Gaussians, $L(\theta) = (1-2\theta)^{-1/2}$ for any $\theta <1/2$ and is equal to $+\infty$ otherwise, so that
%\[ \ h_L(x) = \frac{1}{2} x - \frac{3}{2} x^{1/3} +1, \quad \corAB{x > 0}. 
%\]
\item[(v)] Let $\alpha=0$. Using properties of $h_L$, $L$,  and the Stieltjes transform $m$ it follows that if $h_L(2) < (4\beta)^{-1}$ then there exists some $t_* \in (2, +\infty)$ such that $I(t)=h_L(t/m(t))$ for $t \in (2, t_*]$ and $I(t)=(4\beta m(t)^2)^{-1}$ otherwise. Thus there is exactly one {\em phase transition} in this case. On the other hand, if $h_L(2) \ge (4\beta)^{-1}$ it can be further argued that $I(t)= (4\beta m(t)^2)^{-1}$ for all $t \in (2, +\infty]$ showing that there is {\em no phase transition} in this case.

Let $\{G_{i,j}\}_{i < j}$ be i.i.d.~standard Gaussian and $\{G_{i,i}\}_{i \in [n]}$ be are either bounded or have lighter tail than a sub-Gaussian random variable. Then $h_L(x) = \frac{1}{2} x - \frac{3}{2} x^{1/3} +1$ for $x >0$ and hence we deduce from above that in this setting there is exactly one phase transition in the behaviour of the rate function $I$. 
\item[(vi)] As $h_L$ vanishes only at $1$ and $m(2)=1$, it follows that the rate function is always discontinuous at $2$, meaning that $\lim_{t \downarrow 2} I(t) >0$. Such a discontinuity of the rate function was noted in the large deviation of the top eigenvalue of Wigner matrices with heavy tails \cite{LDPei}. This discontinuity is related to the fact that the large deviation upper tail of the top eigenvalue is generated by finite rank perturbations, whose extreme eigenvalues are known to undergo the so called Baik-Ben Arous-P\'ech\'e (BBP) transition phenomenon (see \cite[Theorem 2.1]{P14}) in their typical behaviour.

\item[(vii)] Using properties of $h_L$ (e.g.~see Lemma \ref{lem:hL-prop}) one can check that the map $s\mapsto (t-m(t)s)^2/(2\alpha) + h_L(s)$ is decreasing and increasing around neighbourhoods of $s=1$ and $s=t/m(t)$, respectively.  This makes it difficult to obtain any closed form expression for $\wh I(t)$ for $\alpha \ne 0$, and shows that a genuine compromise between the degree and the weight should be met to create the optimal large deviation strategy.

\end{enumerate}
\end{Rem}

\begin{Rem}[Lower tail large deviation]
The speed of the lower tail large deviation of $\lambda_X$ is at least  $n^2p$. %, the speed of large deviations of the empirical spectral distribution (ESD). 
Indeed, the large deviation lower tail event of $\lambda_X$, in the regime and in the setting of Theorem \ref{theo-main-weight}, corresponds to a deviation of the empirical spectral distribution (ESD) in a closed set for the weak topology which does not contain the semicircle distribution. By concentration arguments one can show that the large deviation speed of the ESD is at least $n^2 p$, which gives a lower bound on the speed of the exponential decay of the lower tail large deviation event of $\lambda_X$. 
\end{Rem}

Next, under the additional assumption that the entries of $G$ are bounded we strengthen Theorem \ref{theo-main-weight} to show that conditioned on the upper tail large deviation event of $\lambda_X$, the maximum degree must be large and the corresponding eigenvector must possess a non negligible localised component. To state the result, we let $X_i$, $i \in [n]$, to be the $i$-th column of $X$. Write ${\bm u}$ for an eigenvector of unit $\ell_2$ norm corresponding to $\lambda_X$. For $\vep >0$ we set ${\bm v}^\vep$ to be the sub vector of ${\bm u}$ consisting of entries greater than $\vep$ in absolute value.

\begin{The}[Structural result:~bounded case]\label{thm:bdd-case}
Consider the same setup of Theorem \ref{theo-main-weight} and additionally assume that the entries of $G$ are uniformly bounded. Then, for any $\lambda>2$ there exist some $\eta_\lambda, \vep_\lambda >0$ such that, for any $\delta >0$, 
\[
\limsup_{n\to +\infty} \frac{1}{np}  \log \PP\big( \max_{i\in [n]} \|X_i\|^2 \leq  \lambda/m(\lambda)-\delta, \lambda_X\geq \lambda \big) <-I(\lambda)
\]
and
\[
\limsup_{n \to +\infty} \frac{1}{np} \log \PP(\|{\bm v}^{\vep_\lambda}\| \le \eta_\lambda, \lambda_X >\lambda ) < - I(\lambda). 
\]  
Consequently,
\[
\lim_{n \to +\infty} \PP( \max_{i\in [n]} \|X_i\|^2 >  \lambda/m(\lambda)-\delta \text{ and } \|{\bm v}^{\vep_\lambda}\| > \eta_\lambda \mid \lambda_X \geq  \lambda) =1. 
\]
\end{The}

\begin{Rem}The above theorem shows that in the bounded case, at the large deviation scale, the top eigenvalue is approximately equal to a certain function of the maximum degree in the network $X$. 
At a heuristic level a similar result was obtained in \cite{ADK21, TY21} for the typical value of $\lambda_X$  in the bounded case, for $p$ such that $np \asymp \log n$ as $n \to +\infty$,
(same was also shown for $\lambda_{{\rm Adj}_o}$ and $\lambda_2({\rm Adj})$). 
In particular, they showed that the typical behavior of $\lambda_X$ is well approximated by that of $\Lambda(\max(2, \max_{i}\|X_i\|^2))$, where $\Lambda(d): = d/\sqrt{d-1}$ for any $d\geq 2$, which is the inverse of the map $\lambda \in [2,+\infty) \mapsto \lambda/m(\lambda)$.
\end{Rem}

A naturally occurring sparse random matrix related to sparse Wigner matrices is the adjacency matrix of an Erd\H{o}s-Renyi graph with $n$ vertices and parameter $p$, denoted hereafter by ${\rm Adj}$, an $n \times n$ symmetric random matrix with zero diagonal and i.i.d.~$\dBer(p)$ entries above the diagonal. 
Since the entries of $\mathrm{Adj}$ do not possess a {\em product} structure, Theorem \ref{theo-main-weight} is not readily applicable to derive the upper tail large deviations of $\lambda_{\mathrm{Adj}_o}$, where  $\mathrm{Adj}_o:= \mathrm{Adj} - \EE \mathrm{Adj}$.  Nevertheless, the methods employed to prove Theorem \ref{theo-main-weight} can be applied to obtain the upper tail large deviation of $\lambda_{\mathrm{Adj}_o}$. As it will be seen below, in this case the rate function will involve the following {\em Poissonian tail} function:
\begin{equation}\label{eq:def-h}
h(x):= x \log x - x +1, \qquad x>0.
\end{equation}
\begin{The}[Centered Adjacency matrix]\label{thm:cen-adj}
For $p$ such that $ \log n/n \ll p \ll 1$ and $t \in (2,+\infty)$ %we have that
\[
\lim_{n \to \infty}\frac{1}{np}\log\PP\left(\lambda_{\mathrm{Adj}_o} \ge t\sqrt{np} \right) = -  h(t/m(t)). 
\]
\end{The}
Let $\{G_{i,j}\}_{i < j}$ be i.i.d.~Rademacher random variables. Then \(
h_L(x)= x\log x - x+1\), for $x >0$. Therefore, by Remark \ref{rem:I-prop}(iii) and Theorem \ref{thm:cen-adj} we find that the large deviation (speeds and the) rate functions for the upper tail of the largest eigenvalue of a symmetric sparse Rademacher matrix and ${\rm Adj}_o/\sqrt{np}$ are same for all $p$ such that $\log n \ll np \ll n$.

In recent years there have been quite a few works studying the large deviations of $\lambda_{{\rm Adj}}$ (see \cite{CD20, BG20, BBG21, B21}). It is well known that $\lambda_{{\rm Adj}}$ has a different order of magnitude compared to the rest of spectrum. In contrast, there exists a specific constant $b_*$ such that for $np > b_* \log n$ one has $\lambda_2({\rm Adj})=2(1+o(1))\sqrt{np}$ with probability approaching one (see \cite{ADK21, TY21}).  Our next result provides the upper tail large deviations of $\lambda_2({\rm Adj})$ when $p$ decays polynomially in $n$.

\begin{The}[Second largest eigenvalue of  non-centered Adjacency matrix]\label{thm:ncen-adj}
For $p$ such that $\log (np) \gtrsim \log n$ and $t \in (2,+\infty)$ %we have that
\[
\lim_{n \to \infty}\frac{1}{np}\log\PP\left(\lambda_2(\mathrm{Adj}) \ge t \sqrt{np}\right) = -  h(t/m(t)).
\]
\end{The}

It will be explained later in Section \ref{sec:outline-pf} that the proofs of Theorems \ref{theo-main-weight}, \ref{thm:cen-adj}, and \ref{thm:ncen-adj} require isotropic local laws. In particular, we will need to estimate $\EE \langle u, \cR(z) v\rangle$ for some $z \in \CC_+$ close to the real line, and $u,v \in \mathbb{S}^{n-1}$, where $\cR(z)$ is the resolvent of the matrix in context. In the proof of Theorems \ref{theo-main-weight} and \ref{thm:cen-adj} we can take $u$ and $v$ to have a small $\ell^1$ norm. In contrast, the proof of Theorem \ref{thm:ncen-adj} requires us to take $v$ to be the constant vector. A polynomial in $n$ lower bound on $p$ is needed to accommodate such a $v$.

%Because of the non-centred nature of the entries of ${\rm Adj}$ the error bounds in the isotropic law are not strong enough to consider $p$ such that $\log(np) \ll \log n$. This is in contrast to Theorems \ref{theo-main-weight} and \ref{thm:cen-adj}.
Let us further add that by the interlacing inequality and the min-max theorem it follows that $\lambda_2(\mathrm{Adj}) \le \lambda_{\mathrm{Adj}_o}$. Therefore the large deviation upper bound for the upper tail of $\lambda_2(\mathrm{Adj})/\sqrt{np}$ is immediate from Theorem \ref{thm:cen-adj}. To prove a matching lower bound, we will use that there exists an eigenvector corresponding to $\lambda_{\mathrm{Adj}}$ close to %${\bf e}_n$
the constant vector of length $n$ normalized to have its norm one (cf.~Lemma \ref{lem:lam1u1}). See Section \ref{sec:ldp-lbd} for its proof.

\subsection{Background and related works} \label{background}
The first works on the large deviations of spectral observables of random matrices date back to \cite{BenArous} and \cite{Dembo} where the large deviations of the ESD and of the largest eigenvalue were obtained for Gaussian beta ensembles, respectively. A key input to these works was the joint explicit density of the eigenvalues, paving the way to the use of Varadhan's lemma to tackle the large deviation behaviour of the spectrum.

Outside the integrable models, the first breakthrough was obtained by Bordenave and Caputo in \cite{Bordenave} where they studied the large deviation of the ESD for symmetric random matrices with entries with tails heavier than that of the Gaussian. Later, in this same setting, large deviations of the normalized traces and that of the largest eigenvalue were studied in \cite{LDPtr} and \cite{LDPei}, respectively. 
At a very high level the philosophy behind the approach used (and the results obtained) in these works is the observation that in a random matrix with such heavy tailed entries there will be only a handful of entries that will be `large' and the large deviation events for spectral observables will be governed by those large entries. It should be added that at a heuristic level the entries of a sparse Wigner matrix {\em behave like heavy tailed entries}. This behaviour is reflected in particular by the presence of a clique with large edge weights via the rate function $I$.

On the other hand using a method based on tilting of measures by `spherical integrals' and upon analyzing the annealed spherical integrals, \cite{AGH21, CDG23, GH20} derive the upper tail large deviations of $\lambda_X$ when $p \asymp 1$. In particular, \cite{GH20} unearths a universality of the large deviations speed and rate function when the entries of $G$ have `sharp sub-Gaussian' tails. It seems that this technique is not amenable to treat the case $p \ll 1$.

The recent work \cite{GHN22} studied the large deviations of $\lambda_X$ for sparse Wigner matrices for which $p \asymp n^{-1}$ and the entries of $G$ have a Weibull distribution with shape parameter $\vartheta  >0$. 
The earlier paper \cite{GN22} focused on the case when the entries of $G$ were Gaussian (essentially corresponding to the case $\vartheta=2$ from \cite{GHN22}). Notice that $\vartheta \ge 2$ implies a sub-Gaussian tail (with $\vartheta =+\infty$ can be interpreted as the entries of $G$ being bounded) and for such $\vartheta$ \cite{GHN22, GN22} showed that the typical value of $\lambda_X$ depends on $\vartheta$, the large deviation speed is $\log n$, and the rate function does not depend on $\vartheta$. %if $\vartheta >2$. 
They decompose the weighted graph induced by $X$ into two subgraphs:~One is induced by those edges whose edge weights are not `too large' and the other is the complement of that. 
Using \cite{BBG21} they argue that the spectral radius of the first subgraph is negligible at the large deviations scale implying that the upper tail large deviations of $\lambda_X$ is given by that for the largest eigenvalue of the second subgraph.  The latter was treated by deriving several geometric properties of the second subgraph such as showing that its connected components are not too large and  that they have a small number of excess edges. In our setting neither the spectral radius of the first subgraph is negligible nor those geometric properties of the second subgraph hold at the large deviations scale.

\subsubsection*{Na\"ive mean-field approximation} Since the seminal work of \cite{CD16} which introduced the `non linear large deviation theory', there have been numerous works attempting to identify criteria in order to reduce a given large deviation question to the solution of the mean-field variational problem associated to the problem at hand. A couple of examples in this direction are the large deviations of the subgraph counts in Erd\H{o}s-R\'enyi graphs and that of the top eigenvalue (cf.~\cite{NLAu, B21, BB23, BG20, CD20, Eldan, HMS22}). Although the non linear large deviation techniques are not applicable here, we conjecture that in the setting of Theorem \ref{theo-main-weight} the probability of the upper tail event of $\lambda_X$ is asymptotically equal to the solution of the relevant mean-field variational problem.

\subsubsection*{Optimality of the assumptions in Theorem \ref{theo-main-weight}}
%\corAQ{sub-G conc?} 
There are a couple of directions in which one may seek to relax the assumptions in Theorem \ref{theo-main-weight}. One is the lower bound on $p$. It follows from \cite{ADK21, TY21} that if $\liminf_{n\to +\infty}np/\log n > b_*$, where $b_* := h_L(2)^{-1}$, then under the boundedness assumption on $G$  one has $\lambda_X \to 2$, in probability. We predict that if $\lim_{n \to +\infty} np/\log n = b > b_*$ then under the boundedness assumption of $G$ the large deviation speed for the upper tail of $\lambda_X$ remains the same and the rate function is given by $\wt I(t):= h_L(t/m(t)) - b^{-1}$ for $t \in (2, +\infty)$. Many of the machineries and the approaches developed in Sections \ref{sec:concflat} and \ref{sec:bdd-entry} can be adapted and be used to confirm this prediction {for large enough deviations}. For $p$ such that $\limsup_{n\to +\infty} np/\log(n) < b_*$,  $\lambda_X$ no longer sticks to the support of the semicircle law. One needs to define the upper tail event accordingly and possibly one may also need other ideas.

Another direction is the relaxation of the assumption \eqref{subgaussian} that the left tail of $G_{1,2}$ is heavier than its right tail. This assumption is critical to see the clique scenario  emerging as one of the optimal strategies. If on the other hand the right tail of $G_{1,2}$ is heavier than its left tail, we conjecture that bipartite subgraphs with large negative edge weights should replace the clique scenario.

Finally, one could try to relax the sub-Gaussian assumption of $G$ and allow its entries to have a Weibull-type distribution with shape parameter $\vartheta \in (0,2)$ (e.g.~entries with sub-exponential tail). In this case, the supercritical regime where the largest eigenvalue sticks to the support of the semicircle law would correspond to $np \gg (\log n)^{2/\vartheta}$, and we expect that in this sparsity regime the large deviations will be created that by a single large entry.

%There are two differences. First, the large deviation behavior of $\|X_i\|$ and the tail behavior of $G_{i,j}$ being different \corAB{it} would alter the rate function. \corAB{It is expected that the large deviations will be created that by a single large entry}. The second is that for $np \ll (\log n)^{2/\vartheta}$ easy computations show that the typical value of the largest eigenvalue and the large deviation speed would change implying that in this case the supercritical regime would be $np \gg (\log n)^{2/\vartheta}$. \corAB{In the subcritical regime, similar to the sub-Gaussian case, possibly some new ideas are required.}

\subsection{Notational conventions}\label{sec:notation}
We use the following set of notation throughout this paper. We begin with some standard asymptotic notation. For two sequences positive reals $\{a_n\}_{n \in \NN}$ and $\{b_n\}_{n \in \NN}$ we write $a_n \ll b_n$, $b_n \gg a_n$, and $a_n =o(b_n)$ to denote $\limsup_{n \to +\infty} a_n/b_n =0$. If $\limsup_{n \to +\infty} a_n/b_n < +\infty$ then we write $a_n \lesssim b_n$ and $a_n=O(b_n)$. The notation $a_n \gtrsim b_n$ will be used to denote $\liminf_{n \to +\infty} a_n/b_n >0$. We further write $a_n \asymp b_n$ when $a_n \lesssim b_n \lesssim a_n$. The notation $a=o_R(1)$ will be reserved to denote that $\lim_{R \to +\infty}a =0$, where $a=a(R)$ is some function of $R$. We use $\updelta_{x,y}$ to denote the standard Dirac delta function, i.e. $\updelta_{x,y}=1$ if $x=y$ and $0$ otherwise. For $a, b \in \RR$ we set $a \vee b:= \max\{a,b\}$, $a \wedge b:=\min\{a,b\}$, and $a_+:=a \vee 0$.

For $n \in \NN$ we write $\mathcal{H}_n$ to denote the set of symmetric matrices of size $n\times n$. We set $\cH^\lambda_n:=\{K \in \cH_n: \lambda_K <  \lambda\}$, for $\lambda \in \RR$. When the choice of $n$ is clear from the context we will drop the subscript. %s from these notation. 
For any $K \in \cH_n$ and $\cJ \subset [n]:= \{1,2,\ldots, n\}$ we write $K_\cJ$, $K^{(\cJ)}$, and $\widecheck{K}_\cJ$ to be the sub matrices of $K$ with rows and columns indexed by $\cJ$ and $\cJ$, $\cJ^c$ and $\cJ^c$, and $\cJ^c$ and $\cJ$, respectively. If $\cJ=\{i\}$ for some $i \in [n]$, to lighten the notation, we write $K^{(i)}$ and $\widecheck{K}_i$ instead of $K^{(\{i\})}$ and $\widecheck{K}_{\{i\}}$, respectively. We use $K_i$ to denote the $i$-th column of $K$, for any $i \in [n]$. Furthermore, for $\cJ \subset [n]$ we write $H_{\cJ}$ to denote the set of symmetric matrices with rows and columns indexed by $\cJ$ and set $\cH^\lambda_\cJ:=\{K \in \cH_\cJ: \lambda_K < \lambda\}$, for $\lambda \in \RR$. We will use the notation ${\rm Id}_n$ to denote the identity matrix of size $n$. When the choice of the dimension is clear from the context we will omit the subscript. For $k \in [n]$ we let ${\bm e}_k \in \RR^n$  to be the vector that is $1/\sqrt{k}$ in its first $k$ coordinates and zero otherwise.

For $p \in \NN \cup\{\infty\}$  and $y \in \RR^n$ we write $\|y\|_p$ to denote its $\ell^p$-norm. Most of the time we will only deal with the $\ell^2$-norm, and hence when there is no scope of ambiguity we will write $\|\cdot\|$ instead of $\| \cdot\|_2$. For a matrix $W$ the notation $\|W\|$ will be used to denote its operator norm, ${\rm Spec}(W)$ will be used to denote its spectrum, and $W^{\sf T}$ will denote its transpose. For $y_1, y_2 \in \RR^n$ we use $\langle y_1, y_2 \rangle$ to denote the standard $\ell^2$ inner product, while for $W_1, W_2 \in \cH$ we set $\langle W_1, W_2 \rangle:= {\rm Tr}(W_1 W_2)$, where ${\rm Tr}$ denotes the trace of a matrix.  For a random variable $Y$ we use $\|Y\|_p$ to denote its $L^p$-norm. If $Y_1$ and $Y_2$ are independent random variables (or vectors) then for a real valued function $f(Y_1, Y_2)$ we use the following shorthand:~$\EE_{Y_1}[f(Y_1,Y_2)]:= \EE[f(Y_1, Y_2)|Y_2]$. Similarly we define $\PP_{Y_1}$. 

  The notation $\B^n$ and $\mathbb{S}^{n-1}$ will be used to denote the unit $\ell^2$ ball and sphere, respectively. For $\cJ \subset [n]$ we further let $\B^\cJ \subset \B^n$  to be set of vectors $u \in \B^n$ such that $\mathrm{supp}(u) \subset \cJ$, where $\mathrm{supp}(v)$ denotes the support of the vector $v$. Similarly we define $\mathbb{S}^\cJ \subset \mathbb{S}^{n-1}$ and $\RR^\cJ \subset \RR^n$. %The notation ${\bf e}_k$ is used to denote the constant vector of dimension $k$ normalized to have $\|{\bf e}_k\|=1$. 
For a set $S$ we use both $\#S$ and $|S|$ to denote its cardinality.
  
%\corAB{To prove the isotropic law we will need to use the following set of notation: We let $\mathbb{H}:=\{z' \in \CC: \Im z'  >0\}$. Fixing $\kappa, \vep \in (0,1)$ we set 
%\[
%\mathbb{H}_\kappa:= \{z \in \mathbb{H}: |\Re z| \le \kappa^{-1} \text{ and } |\Re z \pm 2| \ge \kappa\}, \quad \mathbb{H}_{\kappa, \vep}:= \mathbb{H}_{\kappa, \vep}^n:= \{z \in \mathbb{H}_\kappa: \Im z \in [n^{-1+\vep}, \kappa^{-1}]\},
%\]
%and 
%\[
%\wt{\mathbb{H}}_\kappa:= \bigg\{ z \in \mathbb{H}_\kappa:  \sqrt[4]{\frac{1}{\log(1/p)}} \le  \Im z \le \kappa^{-1} \bigg\}. 
%\]
%For $W \in \cH_n$ and $z \in \mathbb{H} \cup (\RR \setminus {\rm Spec}(W))$ we define {$\cR(z,W):= (z -W)^{-1}$}, where ${\rm Spec}(\cdot)$ denotes the spectrum.  When the choice of $z$ and/or $W$ will be clear from the context we will suppress the dependencies of these variables in the notation $\cR(\cdot, \cdot)$. For $\T \subset [n]$, $W \in \cH_n$, and $z \in \mathbb{H} \cup (\RR \setminus {\rm Spec}(W))$ we further let $\cR^{(\T)}:= \cR^{(\T)}(z,W):= (z -W^{(\T)})^{-1}$.} 

\subsection*{Organisation of the rest of the paper}
Section \ref{sec:outline-pf} provides a high level description of the proof of Theorem \ref{theo-main-weight}. Section \ref{sec:ldp-lbd} contains the proof of the large deviation lower bound. In Section \ref{sec:concflat} we develop an upper tail concentration bound for convex Lipschitz functions with `flat' sub gradients of Wigner matrices with bounded entries. Sections  \ref{sec:bdd-entry} and \ref{sec:ld-ub-gen} provides the proof of the large deviation upper bound in the bounded and in the unbounded cases, respectively. Section \ref{sec:loc-law} proves an isotropic local law for supercritical Wigner matrices which is instrumental in the proofs of the main results. The appendices contain proofs of several auxiliary results.

\subsection*{Acknowledgements} The research of AB was partially supported by DAE Project no.~RTI4001 via ICTS, the Infosys Foundation via the Infosys-Chandrashekharan Virtual Centre for Random Geometry, an Infosys–ICTS Excellence Grant,  and a MATRICS grant (MTR/2019/001105) from Science and Engineering Research Board.

\section{Outline of the proof}\label{sec:outline-pf}

The proof of Theorem \ref{theo-main-weight} broadly splits into two parts:~the large deviation upper bound and the lower bound. The proof of the lower bound is much less involved and most of this article is devoted to the proof of the upper bound. To prove the upper bound we first consider the bounded case first, i.e.~when the entries of $G$ are assumed to be bounded. The unbounded case requires some additional arguments. 
The main ideas behind these steps are outlined below. 

\subsection{The lower bound}\label{sec:lbd-outline} 
We derive lower bounds on the probabilities of two possible optimal scenarios which consists of  planting a clique of sub entropic size with large edge weights on the one hand, and having a high degree vertex with a large weight on the other hand. For the first scenario, plant a clique on the first $k$ vertices, with $1\ll k \ll \sqrt{np/\log(1/p)}$ and prescribe the edge weights $G_{i,j}$ to be of order $\sqrt{np}/(km(t))$ for a given $t>2$. On this event, we show that we can approximate our matrix by a rank one perturbation of a certain matrix. Namely, we show that $X= \widetilde{X}+ P +\Delta^*$, where $\widetilde{X}$ is the matrix obtained from $X$ after zeroing out $X_{[k]}$, $P := (1/m(t)) {\bm e}_k {\bm e}_k^{\sf T}$, %, \corAB{${\bm e}_k \in \mathbb{S}^{n-1}$ is obtained from ${\bm e}_k$ by appending $(n-k)$ zeros}, 
and $\Delta^*$ is some matrix such that $\|\Delta^*\|=o(1)$ (recall our notation from Section \ref{sec:notation}). Since $1/m(t)>1$, the BBP transition entails that typically on this event, the top eigenvalue of $X$ is located around $t$. Now the cost of this event (to ensure that $\|\Delta^*\|=o(1)$) at the exponential scale $np$, can be shown to be
\[ \binom{k}{2} \log(1/p) +\binom{k}{2} \frac{np}{2\beta k^2 m(t)^2} = \frac{np}{4\beta m(t)^2} (1+o(1)),\]
using the fact that $1\ll k \ll \sqrt{np/\log(1/p)}$ and that the diagonal entries are typically small.  

For the second scenario, prescribe the first vertex to have degree $s$ in the network $X$, in other words $\|\widecheck{X}_1\|^2 =s+o(1)$, and a large weight $X_{1,1} = r+o(1)$, where $r\geq 0$ and $s\geq 1$ are such that $r + m(t) s = t$ for a given $t>2$ and are the optimisers of the variational problem defining $\widehat{I}(t)$. On this event, one can approximate $X$ by a rank two perturbation $Q$ of $W$, %$X\approx W + Q$, where
where
\[ W := \begin{pmatrix} r & 0 \\ 0 & X^{(1)} \end{pmatrix} \qquad \text{ and } \qquad Q := \begin{pmatrix} 0 & \widecheck{X}_1^{\sf T} \\ \widecheck{X}_1 & 0 \end{pmatrix}.\]
In contrast to the first scenario, the resolvent of $W$ at $z \in \CC$ is typically $(z-r)^{-1}{\bm e}_1{\bm e}_1^{\sf T} + m(z) {\rm Id}_{n-1}$. Although the BBP transition is not readily applicable, the same type of computation shows that typically the top eigenvalue of $W+Q$ is located at $t$. On the other hand, the cost of this event (i.e.~for the approximation mentioned above to have a negligible norm) can easily be checked to be at the exponential scale $np$, $r^2/(2\alpha)+ h_L(s) = \widehat{I}(t)$. Putting together the lower bounds obtained using these two scenarios we get the large deviation lower bound rate function $I(t)$.

\subsection{The upper bound:~the bounded case}\label{sec:outline-bdd} As in many instances of large deviation problems involving sparse models, the atypical behaviour is governed by the emergence of some localised structures. Following this ansatz we expect the top eigenvector to localise upon a large deviation of the top eigenvalue. Choosing  ${\bm u}$ a top eigenvector of $X$ of unit norm, a level $\veps>0$ of localisation,  and setting $J$ to be the subset of vertices where the entries of ${\bm u}$ are in absolute value greater than $\veps$, we decompose $X$  as the following block matrix: 
\begin{equation} \label{decomp-X} X = \begin{pmatrix}
X_{{J}} &  \widecheck{X}_{{J}}^{\sf T} \\
\widecheck{X}_{{J}}  &  X^{({J})}
\end{pmatrix},\end{equation}
and we write ${\bm u}^{\sf T}=({\bm v}^{\sf T},{\bm w}^{\sf T})$, where ${\bm v}\in \RR^{{J}}$ and ${\bm w} \in \RR^{{J}^c}$\footnote{The vectors ${\bm v}$ and ${\bm w}$ depend on $\vep$. In this section we suppress this dependency. In latter sections, the dependence in $\vep$ will be made explicit}. From the eigenvalue-eigenvector equation, we get
\begin{equation} \label{eq-top-ei-sec2} \begin{cases}
&X_{{J}} {\bm v} + \widecheck{X}_{{J}}^{\sf T} {\bm w} =\lambda_X {\bm v},\\
& \widecheck{X}_{{J}}{\bm v}  + X^{({J})} {\bm w} = \lambda_X {\bm w}.
\end{cases}
\end{equation}
If $\lambda_X \notin {\rm Spec}(X^{({J})})$ then from the second line of \eqref{eq-top-ei-sec2} we find that ${\bm w} = (\lambda_X -X^{({J})})^{-1} \widecheck{X}_{{J}}{\bm v}$. 
%\begin{equation} \label{eq-top-ei2} \begin{cases}
%&X_{J} v_{\text{top}}  + Y^{\sf T} w_{\text{top}} =\lambda_X v_{\text{top}} ,\\
%& w_{\text{top}}  = (\lambda_X -X^{(J)})^{-1}Yv_{\text{top}} .
%\end{cases}
%\end{equation}
Taking an inner product with ${\bm w}$ in the first line gives the following: 
\begin{equation}\label{eq-top-ei-sec2-1} \begin{cases}
&\langle {\bm v}, X_{{J}} {\bm v}\rangle  + \langle \widecheck{X}_{{J}} {\bm v}, (\lambda_X -X^{({J})})^{-1}\widecheck{X}_{{J}} {\bm v} \rangle =\lambda_X \|{\bm v}\|^2,\\
&{\bm w} = (\lambda_X -X^{({J})})^{-1} \widecheck{X}_{{J}}{\bm v}.
\end{cases}
\end{equation}
Note that as ${\bm u}$ is of unit norm the set $J$ is of bounded size. In particular, as we assumed $np \gg \log n$, the set $J$ carries no entropy at the exponential scale $np$, and therefore one can argue as if $J$ were a deterministic set. At the price of adding a certain set of vertices of size at most the multiplicity of $\lambda_X$ (which is bounded with high probability), one can actually assume that $\lambda_X\notin {\rm Spec}(X^{(J)})$ (see Lemmas \ref{top-ei-submatrix} and \ref{conc-multi}). Now we move forward by showing that upon a large deviation of the top eigenvalue, the top eigenvector indeed localises in the sense that ${\bm v}$ carries a non trivial $\ell^2$ norm, %of the top eigenvector, 
and that the upper tail deviation of $ \langle \widecheck{X}_{{J}} {\bm v}, (\lambda_X -X^{({J})})^{-1}\widecheck{X}_{{J}} {\bm v} \rangle $ is of much faster speed than $np$, which enables us to argue that
\begin{equation}\label{eq:pf-outline1}
\langle \widecheck{X}_{{J}}{\bm v}, (\lambda_X -X^{({J})})^{-1} \widecheck{X}_{{J}}{\bm v} \rangle \le  m(\lambda_X) \|\widecheck{X}_{{J}}{\bm v}\|^2 +o(1),
\end{equation}
at the large deviations scale (recall \eqref{eq:m-def}). Since the boundedness of the entries of $G$ imply that $\langle {\bm v}, X_{{J}} {\bm v}\rangle =o(1)$, given \eqref{eq:pf-outline1}, by the first line of \eqref{eq-top-ei-sec2-1}, the fact that $x \mapsto x/m(x)$ is increasing on $[2, +\infty)$, and that $\lambda_X$ is exponentially tight at the exponential scale $np$ we deduce that on the event where $\lambda_X\geq t>2$, 
\[
 t /m(t)  \le \lambda_X /m(\lambda_X) \le \|\widecheck X_{{J}} {\bm v}\|^2 /\|{\bm v}\|^2 +o(1)\le \|\widecheck X_{{J}}\|^2 +o(1),
\]
using that $\|{\bm v}\|\gtrsim 1$ at the large deviation scale. 
Next showing that $\|\widecheck X_{{J}}\|^2$ is well approximated by $\max_{i \in [n]} \|X_i\|^2$ at the exponential scale the proof for the bounded case completes.  

The proof of \eqref{eq:pf-outline1} requires considerable work. Denote by $f_{\lambda, w}(K):= \langle w, (\lambda - K)^{-1} w \rangle$ (whenever defined). Observe that formally, by our decomposition of ${\bm u}$, the subgradient of this map for $\lambda=\lambda_X$ and ${\bm w}=\widecheck{X}_{{J}} {\bm v}$ at $K=X^{({J})}$  is given by ${\bm w}{\bm w}^{\sf T}$. Note that ${\bm w}$ is delocalised because by our choice $\|{\bm w}\|_\infty \le \vep$.
%and has a low rank. 
This key observation motivates the following two step strategy:~First, develop a general concentration bound for convex {Lipschitz} functions of $X$ with `flat' subgradients, meaning that the subgradients are of the form $\sum_{\ell=1}^k \theta_\ell {w}_\ell {w}_\ell^{\sf T}$ where $|\theta_\ell|\leq 1$, $\|w_\ell\|\leq 1$ and $\|w_\ell \|_\infty \leq \veps$,  showing that the probability that it exceeds its expectation is negligible at the exponential scale $np$.  This program is carried out in detail in Section \ref{sec:concflat}.  Due to the complicated nature of dependence between ${\bm v}, \lambda_X, {J}$  in the entries of $X$, one cannot simply apply such a concentration result 
to $f_{\lambda_X, \widecheck{X}_{{J}} {\bm v}}(X^{({J})})$, as it does not allow one to use the joint independence of $\{X_{i,j}\}_{i \le j}$. To overcome this technical difficulty one additionally needs to regularise $f_{\lambda, w}$ and work with its smoothened version that equals $f_{\lambda,w}$ on a `good' set and retain the properties that it is Lipschitz and its subgradient is flat, and use net arguments. We refer the reader to Section \ref{section:concresolv} for more details. Let us remark in passing that the lower bound $\|{\bm v}\| \gtrsim 1$, at the large deviations scale, is obtained also as a consequence of the concentration bounds developed in Section \ref{sec:concflat}.

 In a second step, we need to prove a bound on the expectation of (a smoothed version of)  $f_{\lambda_X,\widecheck{X}_J{\bm v}}(X^{(J)})$, which amounts to show an isotropic local law for our model.  Namely for $z \in \CC$ with $|\Im z| \ll 1 $ and $|\Re z\pm 2| \gtrsim 1$ we derive in Section \ref{sec:loc-law} that $\EE \langle w, (z-X)^{-1} w \rangle \to m(z)$ as $n \to +\infty$, for any $w \in \mathbb{S}^{n-1}$ such that $\|w\|_1 \ll np$
\footnote{If $\log(np) \ll \log n$ we could not remove the upper bound $\|w\|_1 \ll np$. Nevertheless it is good enough for Theorem \ref{theo-main-weight}.}, in the entire sparsity regime $np \gg \log n$. 

Let us mention that  results on isotropic local laws in `sparse' settings are there in the literature. For example, see \cite{BHY17, BKY17, EKHY13}. However, as we need the isotropic local law for vectors $w$ that are not necessarily orthogonal to the constant vector and in the entire supercritical regime, none of these results are applicable. To prove the required local law, at a high level we use the general strategies developed in \cite{BHY17}, namely expressing the entries of the resolvent using the Schur complement formula and the resolvent identity, and then perform a more delicate counting argument (compared to  \cite{BHY17}). This ends the sketch of the proof of \eqref{eq:pf-outline1}.

\subsection{Upper bound:~the general case} In the unbounded case one would like to employ a similar strategy as described in Section \ref{sec:outline-bdd}. However, the concentration bounds developed in Section \ref{sec:concflat} need the entries of the matrix under consideration to be bounded. This prompts one to decompose $G$ into $\widetilde A$ and $\widetilde B:= G- \widetilde A$, where the entries of $\widetilde A$ are bounded and centered with the truncation threshold being some large, but of order $1$, parameter that will eventually be sent to infinity. Such a split introduces dependencies between $\widetilde{ A}$ and $\widetilde{B}$. 
Nevertheless, upon carefully choosing a threshold for truncation and using a decoupling argument, we argue in Section \ref{section:decoupl} that it suffices to find the upper tail large deviations of $\lambda_Z$, where $Z:=H+B$, $B:=\widetilde B \circ \Xi/\sqrt{np}$, and $H:=\widetilde{H} \circ \Xi'/\sqrt{np}$, with matrices $\widetilde{H}$, $\Xi'$ and $B$ being independent, $\Xi'$ with the same distribution as $\Xi$, and for any $i,j\in [n]$, $\widetilde{H}_{i,j}$ having the distribution of $\widetilde{A}_{i,j}$ conditioned to be non zero. 

As we expect a new scenario to emerge - the planted clique of sub entropic size with large edge weights - we consider a different localisation threshold compared to the bounded case. For a choice of top eigenvector ${\bm u}$ of unit norm of $Z$ we set $J$ to be the subset of vertices $i$ such that $|{\bm u}_i|\geq \veps^{-1} \sqrt{\log(1/p)/(np)}$, together with the vertices $i$ with non trivial degree in the network $B$, that is $\|B_i\|\geq \delta$, for some well chosen $\delta>0$. As it can be shown that there are only a bounded number of vertices $i$ such that $\|B_i\|$ is bounded away from zero with high probability, the possible number of choices for the random set ${J}$ is still negligible at the exponential scale $np$ and therefore one can again treat ${J}$ as a deterministic subset. 

Similarly as in the bounded case, we start by showing that the top eigenvector localises, in the sense that $\|{\bm v}\|\gtrsim 1$ upon a large deviation of the top eigenvalue. Next, decomposing again ${\bm u}$ and $Z$ along the set $J$, and proceeding as in \eqref{decomp-X}-\eqref{eq-top-ei-sec2-1} give the equations
\begin{equation}\label{eq-top-ei-sec2-2} \begin{cases}
&\langle {\bm v}, Z_{{J}} {\bm v}\rangle  + \langle {\bm x}, (\lambda_Z -H^{({J})})^{-1} {\bm x} \rangle -\langle {\bm w}, B^{({J})} {\bm w} \rangle =\lambda_Z \|{\bm v}\|^2,\\
&{\bm w} = (\lambda_{Z} -H^{({J})})^{-1} {\bm x}, \qquad {\bm x}:= \widecheck{Z}_{{J}}{\bm v} + B^{({J})}{\bm w}.
\end{cases}
\end{equation}
Arguing as in the bounded case and using that $\|w\|_\infty \leq \veps^{-1}\sqrt{\log(1/p)/(np)} =o(1)$ we get that the upper tail large deviation of the resolvent $\langle {\bm x}, (\lambda_Z-H^{(J)})^{-1}{\bm x}\rangle$ concentrates at a faster exponential scale. Therefore arguing as in \eqref{eq:pf-outline1} we obtain
\begin{equation}\label{eq:pf-outline2}
 \lambda_Z \|{\bm v}\|^2 \leq \langle {\bm v}, Z_{{J}} {\bm v}\rangle  + m(\lambda_Z) \|\widecheck{Z}_{{J}}{\bm v} + B^{({J})}{\bm w}\|^2 -\langle {\bm w}, B^{({J})} {\bm w} \rangle+o(1),
\end{equation}
at the exponential scale $np$. Recall that $\widecheck{Z}_J= \widecheck{H}_J+\widecheck{B}_J$. Using repetitively our concentration bounds from Section \ref{sec:concflat} we can further simply this inequality. First, we show that with high probability $\|B^{(J)}{\bm w}\| \ll 1$, by taking advantage of the fact that $J$ does not contain any vertex $i$ for which $\|B_i\| \gtrsim 1$.  %has a significant norm. 
Next, using the independence between $\widecheck{H}_J$ and $\widecheck{B}_J$, we obtain with high probability that $\widecheck{H}_J{\bm v}$ and $\widecheck{B}_J {\bm v}$ are almost orthogonal. As the entries of $H$ bounded, one can further identify the contributions to $\|\widecheck{H}_J {\bm v}\|^2$, and prove at the exponential scale that 
\[ \|\widecheck{H}_J {\bm v}\|^2 = \sum_{i\in \mathfrak{J}} {\bm v}_i^2 \|\widecheck{H}_i\|^2 +o(1), \qquad \text{where } \, \mathfrak{J}:=\{i\in J : |{\bm u}_i|\geq \veps\}.\]
Putting everything together yields at the large devitaion scale the inequality 
\[  \lambda_Z \|{\bm v}\|^2 \leq \langle {\bm v}, Z_{{J}} {\bm v}\rangle  + m(\lambda_Z) \big(  \sum_{i\in \mathfrak{J}} {\bm v}_i^2 \|\widecheck{H}_i\|^2 + \|\widecheck{B}^{(J)}{\bm v}\|^2\big)+o(1).\]
At this stage all the variables that are at play, upon fixing the set $J$, are independent and have a tractable large deviation behaviour since ${\bm v}$ is a short vector. As a result, we deduce that the $\log$-probability of the upper tail event of $\lambda_Z$ can be bounded from above by a certain variational problem (see Proposition \ref{varpbUT}). Finally analyzing the variational problem obtained we can conclude that there are only two optimal scenarios:~either all the mass of ${\bm v}$ is localised on one vertex, yielding the high degree vertex scenario, or it is delocalised on its support, implying the planted clique scenario.

\section{Large deviation lower bounds}\label{sec:ldp-lbd}
In this section we prove the large deviation lower bounds of Theorems \ref{theo-main-weight}, \ref{thm:cen-adj}, and \ref{thm:ncen-adj}. 
\subsection{Typical behaviour of the top eigenvalue and the resolvent}
In a first step we collect some information about the spectrum of sparse Wigner matrices when $np \gg \log n$. Most of our results in this section and in the latter sections  hold under the above assumption of sparsity. Therefore, for brevity we have chosen not to repeat the sparsity assumption in the statements of those results. The sparsity assumption will be mentioned whenever it is different from the above.
\begin{Lem}\label{cvedgetyp}
$\lambda_X$ and $\|X\|$ converge to $2$ in probability when $n \to +\infty$. Moreover, the same holds for ${\rm Adj}_o/\sqrt{np}$ instead of $X$. \end{Lem}
\begin{proof}We will only prove the statement for $X$ as the one for ${\rm Adj}_o$ is similar. %Note that the convergence in probability of $\|X\|$ follows readily from the convergence of $\lambda_X$ and $\lambda_{-X}$. 
Let $R>0$ and define the matrix $\widetilde{A}^R$ by $\widetilde{A}^R_{i,j} :=G_{i,j} \Car_{|G_{i,j}| \le R} - \EE [G_{i,j} \Car_{|G_{i,j}| \le R}]$ for any $i,j\in[n]$.
Letting $\widetilde{B}^R := G- \widetilde{A}^R$, set 
\[ 
A^R := (\widetilde{A}^R \circ \Xi)/\sqrt{np} \qquad \text{ and } \qquad  B^R := (\widetilde{B}^R \circ \Xi)/\sqrt{np},
\]
Since $\EE ((\widetilde{A}^R)_{i,j})^2=1+o_R(1)$ for any $i\neq j$, and $|\widetilde{A}^R_{i,j}|\leq 2R$, we get by  \cite[Theorem 2.7]{BeBoKn} that $\EE( \lambda_{A^R} )\leq 2+o_R(1)$ for $n$ large enough. Since the entries above the diagonal of the matrix $B^R$ are i.i.d.~applying Seginer's theorem (see \cite[Theorem 1.1]{Seginer}) and a standard symmetrisation argument (see \cite[Proof of Theorem 5]{HWX16}), it gives $\EE \|B^R\|\lesssim \EE \max_{i\in[n]} \|B_i^R\|$. Now, for any $\theta >0$ and $i,j \in [n]$, denoting $\widehat{B}_{i,j}^R= G_{i,j} \Car_{|G_{i,j}|>R}$,
\begin{equation} \label{loglaplaB2} \log \EE ( e^{\theta \widehat{B}_{i,j}^2 \xi_{i,j}}) = \log (1-p+p\EE (e^{\theta \widehat{B}_{i,j}^2})) \leq p \EE(e^{\theta \widehat{B}_{i,j}^2}-1)\leq  p\EE( \Car_{|G_{i,j}|>R)} e^{\theta G_{i,j}^2}).\end{equation}
As $G$ is sub-Gaussian we deduce that for $\theta$ small enough independent of $(i,j)$,  $\log \EE (e^{\theta \widehat{B}_{i,j}^2 \xi_{i,j}})=o_R(1)p$, which entails, as $\EE[G_{i,j} \Car_{|G_{i,j}|>R}]=o_R(1)$, that $\log \EE(e^{\theta np \|B_i^R\|^2}) = o_R(1)np$ for any $i\in[n]$. As $np \gg \log n$, it follows that $\EE \max_{i\in [n]} \|B_i^R\| =o_R(1)$, and as a result $\EE (\lambda_{B^R}) =o_R(1)$. Since $\lambda_X\leq \lambda_{A^R}+\lambda_{B^R}$, this shows that $\EE( \lambda_X )\leq 2+o_R(1)$ for $n$ large enough. Besides, as $K\mapsto \lambda_K$ is convex and $1$-Lipschitz with respect to the Hilbert-Schmidt norm on $\mathcal{H}_n$, and the entries of $X$ are sub-Gaussian by Assumption \ref{hypo:subg-entry}, we know from  \cite[Theorem 1.3]{HT21} that for any $t\geq 1/\sqrt{np}$, 
\[ \PP\big(\lambda_X- {\rm Med}(\lambda_X)>t\big) \leq \exp\Big(-\frac{t^2 np}{C\log(C^2n)}\Big),\]
where $C$ is a positive constant depending on the parameters $\alpha$ and $\beta$ of Assumption \ref{hypo:subg-entry}, and ${\rm Med}(\lambda_X)$ is a median of $\lambda_X$. In particular, integrating this inequality yields $\EE (\lambda_X - {\rm Med}(\lambda_X))_+ =o(1)$. Since $\EE(\lambda_X)\leq 2+o_R(1)$ for $n$ large enough, we conclude that $\PP(\lambda_X\geq 2+\delta)$ goes to $0$ as $n$ goes to $+\infty$ for any $\delta>0$. The same holds as well for the lower tail as a consequence of the convergence in probability of the ESD of $X$ to the semicircle law \cite{KhKhPS}. This shows that $\lambda_X \to 2$ in probability as $n \to +\infty$. Replacing $X$ by $-X$ we get the same for $\lambda_{-X}$ and thus we further derive that $\|X\| \to 2$ in probability as $n \to +\infty$. This completes the proof. 
\end{proof}

Next we state a concentration bound for the resolvent $\langle u, (\lambda-X)^{-1} u\rangle$, where $u$ is a unit vector away from the constant vector direction and $\lambda$ is above the spectrum of $X$. Its proof is deferred to Section \ref{sec:loc-law} where we state and prove the isotropic local law for sparse supercritical Wigner matrices. This latter result is required to identify the limit of $\EE\langle u, (\lambda-X)^{-1} u\rangle$.

\begin{Pro}\label{prop:conc-loc-law}
Fix any sequence $\{s_n\}_{n \in \NN}$ such that $\lim_{n \to \infty}s_n=0$. Define  ${\mathcal{U}}_{n}:= \{u \in \mathbb{S}^{n-1}: \|u\|_1 \le s_n np\}$. For any $t,\delta>0$, $n$ large enough, and any $u\in{\mathcal{U}}_n$,
\begin{equation}\label{eq:conc-loc-law}
\PP\Big(\sup_{\lambda\geq 2+2\delta} \big|  \langle u, (\lambda-X)^{-1} u\rangle - m(\lambda)  \big| \ge t, \lambda_X\leq 2+\delta \big) \leq \exp \Big(-\frac{\delta^4 t^2 np}{c\log n}\Big),
\end{equation}
where $c>0$ is a positive numerical constant. Moreover, the same result holds for ${\rm Adj}_o/\sqrt{np}$ instead of $X$.
\end{Pro}

\subsection{Large deviation estimates}The second piece needed in our argument of the large deviation lower bound is the large deviation tail estimate of individual entries of $G$ and of the degree of a vertex in the network $X$, understood as the square of the $\ell^2$-norm of a column of $X$ after zeroing out its diagonal entries. 
\begin{Lem}\label{lbcoeff}
For any $t>0$, 
\[ \liminf_{n\to +\infty} \frac{1}{np} \log \PP\big(G_{1,1}\geq t\sqrt{np}\big)\geq - \frac{t^2}{2\alpha}, \  \liminf_{n\to +\infty} \frac{1}{np} \log \PP\big(G_{1,2}\geq t\sqrt{np}\big)\geq - \frac{t^2}{2\beta}.\]
\end{Lem}
\begin{proof}We will only prove the large deviation lower bound for $G_{1,2}$, as the proof of $G_{1,1}$ is identical. In a first step we show that
\begin{equation}\label{eq:Lap+}
\lim_{n \to +\infty} \frac{1}{np} \log \EE \exp(\theta \sqrt{np} (G_{1,2})_+) = \left\{\begin{array}{ll}
\frac{\theta^2\beta}{2} & \mbox{ if }\theta >0,\\ 
0 & \mbox{ otherwise}.
\end{array} 
\right.
\end{equation}
To see \eqref{eq:Lap+} we fix $\theta >0$. As $\EE \exp(\zeta G_{1,2})\geq 1$ for any $\zeta\in \RR$ (since $G_{1,2}$ is centered), we have $\EE e^{\zeta (G_{1,2})_+}\leq 2 \EE e^{\zeta G_{1,2}}$ for any $\zeta \in \RR$. It follows that 
\begin{equation}\label{eq:Lap+1}
\limsup_{n \to +\infty} \frac{1}{np} \log \EE \exp(\theta \sqrt{np} (G_{1,2})_+) \le \limsup_{n \to +\infty} \frac{1}{np} \log \EE \exp(\theta \sqrt{np} G_{1,2}) \le \frac{\theta^2 \beta}{2}.
\end{equation}
Further 
\begin{multline}\label{eq:Lap+2}
\liminf_{n \to +\infty} \frac{1}{np}\log \EE \exp(\theta \sqrt{np} (G_{1,2})_+) \ge \liminf_{n \to +\infty} \frac{1}{np}\log \EE[ \exp(\theta \sqrt{np} (G_{1,2})){\bf 1}_{\{G_{1,2} >0\}}] \\
\ge  \liminf_{n \to +\infty} \frac{1}{np}\log \EE \exp(\theta \sqrt{np} G_{1,2}) = \frac{\theta^2 \beta}{2},
\end{multline}
where in the last inequality we use Harris' inequality and the fact that $\PP(G_{1,2} >0) >0$ as $\EE G_{1,2}=0$ and $\EE G_{1,2}^2>0$. Combining \eqref{eq:Lap+1} and \eqref{eq:Lap+2} we get \eqref{eq:Lap+} for $\theta >0$. 
Now consider $\theta \le 0$. In this case, using $\theta \sqrt{np}(G_{1,2})_+\le 0$ we get the claimed upper bound.
%\begin{equation}\label{eq:Lap+3}
%\limsup_{n \to +\infty}  \frac{1}{np}\log \EE \exp(\theta \sqrt{np} (G_{1,2})_+) \le 0.
%\end{equation}
To prove the matching lower bound we observe that $
\EE \exp(\theta \sqrt{np} (G_{1,2})_+) \ge \PP(G_{1,2} \le 0) >0$, which ends the proof of \eqref{eq:Lap+} for $\theta \le 0$. From \eqref{eq:Lap+} the proof is complete by an application of G\"artner-Ellis theorem \cite[Theorem 2.3.6(c)]{DZ}.
\end{proof}

\begin{Lem}\label{LDP-deg}
Under Assumption \ref{hypo:subg-entry}, for any $t>1$, 
\begin{equation}\label{eq:LDP-deg-ut}
\lim_{n \to +\infty} \frac{1}{np}\log \PP (\|\widecheck{X}_1\|^2 \geq t) = - h_L(t),
\end{equation}
where $h_L$ is defined in Definition \ref{def:rate-fn}. In particular, with $h$ as in \eqref{eq:def-h},
\begin{equation}\label{eq:LDP-deg-ut-R}
\lim_{n \to +\infty} \frac{1}{np}\log \PP \Big(\sum_{i=1}^n \xi_{i,1} \geq t np\Big) = -h (t).
\end{equation}
%where $h$ is defined in \eqref{eq:def-h}.
\end{Lem}
To prove Lemma \ref{LDP-deg} we will need the following properties of $L$ and $h_L$. The following notation is used below:~for a set $S \in \RR$, we let ${\rm Int}(S)$ to be the interior of $S$. 

%The proof will use G\"artner-Ellis Theorem \cite[Theorem 2.3.6]{DZ}. To apply that theorem we will need the following lemma. We refer the reader to \cite[Definition 2.3.5]{DZ} for a definition of essentially smooth and steep functions. 
\begin{lemma}\label{lem:hL-prop} 
Under Assumption \ref{hypo:subg-entry} the following properties hold:
\begin{enumerate}
\item[(a)] Let $\mathcal{D}_L:=\{\theta \in \RR: L(\theta) < +\infty\}$. Then ${\rm Int}(\mathcal{D}_L) = (-\infty, 1/(2\beta))$. 

\item[(b)] The map $\theta \mapsto L(\theta)$ is convex and infinitely differentiable on ${\rm Int}(\mathcal{D}_L)$. Furthermore, on ${\rm Int}(\mathcal{D}_L)$ the maps $\theta \mapsto L(\theta)$ and $\theta \mapsto L'(\theta)$ are strictly increasing. 

\item[(c)] If $x_\star:=\lim_{\theta \uparrow 1/(2\beta)} L'(\theta) < +\infty$ then $L_\star:=  \lim_{\theta \uparrow 1/(2\beta)} L(\theta) < +\infty$.

\item[(d)] Assume $x_\star < +\infty$. Then, for $x \ge x_\star$ we have that $h_L(x)=\frac{x}{2\beta} - L_\star +1$.

\item[(e)] The map $x \mapsto h_L(x)$ is differentiable on $(0, +\infty)$, with $h_L'(1)=0$, {increasing on $(1, +\infty)$, and  decreasing on $(0, 1)$.} The derivative $h_L'$ on $(0, +\infty)$ is differentiable, except possibly at $x_\star$, non decreasing, and concave with strict concavity  on $(0, x_\star)$ and $h_L'(x) = 1/(2\beta)$ for $x\geq x_\star$. 
\end{enumerate}
\end{lemma}

The proof of Lemma \ref{lem:hL-prop} is postponed to Appendix \ref{app:aux-res}.

\begin{proof}[Proof of Lemma \ref{LDP-deg}]
Observe that for any $\theta \in \RR$,
\[ 
\Lambda_n(\theta) := \frac{1}{np}  \log \EE \big(e^{np \theta \|\widecheck{X}_1\|^2} \big)=\frac{n-1}{np}\log \left( 1- p +pL(\theta)\right) \to L(\theta)-1, \text{ as } n \to +\infty.\]
%We have for any $\theta \in \RR$,
%\[ \Lambda_n(\theta)  = \frac{n-1}{np}\log \big( 1- p +pL(\theta)\big).\]
%Therefore, for any $\theta \in \RR$,
%\[ \lim_{n\to +\infty} \Lambda_n(\theta) = L(\theta)-1.\]
Thus the upper bound
\begin{equation}\label{eq:ut-ubd}
\limsup_{n \to +\infty} \frac{1}{np}\log \PP (\|\widecheck{X}_1\|^2 \ge t) \le - h_L(t),
\end{equation}
is immediate from Chernoff's inequality. To prove a matching lower bound we need to consider different cases separately. 

If $x_\star = +\infty$ (recall Lemma \ref{lem:hL-prop}(c)) or $\beta=0$ (recall Lemma \ref{lem:hL-prop}(a)), then the map $\theta \mapsto L(\theta)$ is essentially smooth (see \cite[Definition 2.3.5]{DZ}). Therefore, the claimed large deviations principle is immediate from G\"artner-Ellis theorem \cite[Theorem 2.3.6(c)]{DZ}.

Thus for the rest of the proof we will assume that $x_\star < +\infty$ and $\beta >0$. By Lemma \ref{lem:hL-prop}(b)-(c) for any $x \in (0, x_\star)$ there exists some $\theta \in {\rm Int}(\mathcal{D}_L)$ such that $L'(\theta)=x$. Therefore, applying \cite[Theorem 2.3.6(b)]{DZ}, \cite[Lemma 2.3.9(b)]{DZ}, and Lemma \ref{lem:hL-prop}(e)  we have that
\begin{equation}\label{eq:ut-lbd}
\liminf_{n \to +\infty} \frac{1}{np} \log \PP(\|\widecheck{X}_1\|^2 \ge x) \ge - h_L(x), \quad x \in(1, x_\star).
\end{equation}
To extend the lower bound \eqref{eq:ut-lbd} for $t \ge x_\star$ we set $\Theta:= \|\widecheck X_1\|^2 - X_{1,2}^2$ and note that the lower bound \eqref{eq:ut-lbd} continues to hold for $\Theta$. 
Using %\eqref{eq:Lap+}
Lemma \ref{lbcoeff}, G\"artner-Ellis theorem \cite[Theorem 2.3.6]{DZ} and the fact that $np \gg \log(1/p)$ in the regime $np \gg \log n$ we obtain
\begin{equation}\label{eq:Xentrylbd}
\liminf_{n \to +\infty} \frac{1}{np} \log \PP(X_{1,2} \ge y ) \ge \liminf_{n \to +\infty} \frac{1}{np} \log \PP(G_{1,2} \ge y \sqrt{np}) \ge -\frac{y^2}{2\beta},
\end{equation}
for any $y >0$. 
Now using the independence of $\Theta$ and $X_{1,2}$, by \eqref{eq:ut-lbd} and \eqref{eq:Xentrylbd}, we deduce that
\begin{multline*}
\liminf_{n \to +\infty} \frac{1}{np} \log \PP(\|\widecheck{X}_1\|^2 \ge t) \ge \liminf_{n \to +\infty} \frac{1}{np} \log \PP(\Theta \ge x) + \liminf_{n \to +\infty} \frac{1}{np} \log \PP(X_{1,2}^2 \ge t-x)  \\
\ge - h_L(x) - \frac{t-x}{2\beta},
\end{multline*}
for any $x \in (1,x_\star)$ and $t \ge x_\star$. Letting $x \uparrow x_\star$ and using Lemma \ref{lem:hL-prop}(d) we extend \eqref{eq:ut-lbd} for all $t >1$. This completes the proof of \eqref{eq:LDP-deg-ut}.  The limit in \eqref{eq:LDP-deg-ut-R} follows from \eqref{eq:LDP-deg-ut} by taking $G_{i,j}$ to be Rademacher variables.
\end{proof}

\subsection{Proofs of the large deviation lower bounds}

\begin{proof}[Proof of Theorem \ref{theo-main-weight} (lower bound)]
We split the proof into two parts. In the first part we will show that
\begin{equation}\label{eq:lbd-part1}
\liminf_{n \to +\infty} \frac{1}{np} \log \PP(\lambda_X \ge t) \ge - \frac{1}{4\beta m(t)^2}, \qquad t \in (2, +\infty), \alpha \le 2\beta.
\end{equation}
For this part of the proof assume $\beta >0$. Otherwise, there is nothing to be proved. Let $\lambda>2$ and $y>1$ to be chosen later. Recall that, for $k \in [n]$ the vector ${\bm e}_k \in \RR^n$ is the one with its %be the vector \corAB{in} $\RR^n$ where the 
first $k$ coordinates are equal to $1/\sqrt{k}$ and $0$ otherwise. Define the matrices 
\[ \widetilde{X} := X - \begin{pmatrix} X_{[k]} & 0 \\ 0 & 0 \end{pmatrix},   P := y{\bm e}_k {\bm e}_k^{\sf T},  \text{ and }  T_y := \begin{pmatrix} X_{[k]} & 0 \\ 0 & 0\end{pmatrix}-P.  \]
Let $\delta>0$ and choose $y=1/m(\lambda+3\delta)$.  As $X = \widetilde{X} + P+T_y$, using Weyl's inequality we get $\lambda_X\geq \lambda_{\widetilde{X}+P} - \|T_y\|$. Thus, using independence,
\begin{align*}
 \PP(\lambda_X\geq \lambda) \geq \PP(\lambda_X\geq \lambda, \|T_y\|\leq \delta) & \geq  \PP(\lambda_{\widetilde{X}+P}\geq \lambda+\delta, \|T_y\|\leq \delta)\\
&=  \PP(\lambda_{\widetilde{X}+P}\geq \lambda+\delta)\PP( \|T_y\|\leq \delta).
\end{align*}
Using Weyl's inequality again it yields
\[  \PP(\lambda_X\geq \lambda) \geq \PP(\lambda_{X+P}\geq \lambda+2\delta, \|X_{[k]}\|\leq \delta)\PP(\|T_y\|\leq \delta).\]
We will show that 
\begin{equation} \label{costlb1} \liminf_{\delta \to 0\atop k \to +\infty} \liminf_{n\to +\infty} \frac{1}{np} \log \PP\big(\|T_y\|\leq \delta)\geq - \frac{1}{4\beta m(\lambda)^2},\end{equation}
and 
\begin{equation} \label{typical1} \lim_{n\to +\infty} \PP(\lambda_{X+P}\geq \lambda+2\delta, \|X_{[k]}\|\leq \delta)=1.\end{equation}
Once these two estimates are proved, the lower bound \eqref{eq:lbd-part1} immediately follows. 

The first estimate is straightforward from Lemma \ref{lbcoeff}. Indeed, on the one hand note that using Lemma  \ref{lbcoeff}, independence and the fact that $np \gg \log n$,
\[ \liminf_{n\to +\infty} \frac{1}{np} \log \PP\big( \max_{i\neq j\in[k] } |G_{i,j} - y\sqrt{np}/k| \leq \delta \sqrt{np}/(2k)\big) \geq -\binom{k}{2}\frac{(y-\delta/2)^2}{2 \beta k^2}.\]
On the other hand, we have $\max_{i\in[k]} |P_{i,i}| =O(1/k)$, and trivially
\[ \PP\big(\max_{i\in[k]} |G_{i,i}|\leq (\delta/4) \sqrt{np}\big) \underset{n\to+\infty}{\longrightarrow} 1.\]
Since for any $A\in\mathcal{H}_{k}$, $\|A\|\leq \max_{i\in[k]} |A_{i,i}|+k \max_{i\neq j \in [k]} |A_{i,j}|$, and $\log \PP(\forall i,j \in [k], \xi_{i,j}=1) = o(np)$ as $np\gg \log(1/p)$, the estimate \eqref{costlb1} follows by independence between the diagonal and off-diagonal coefficients. 

We now turn our attention to prove \eqref{typical1}.
For any $w\notin \mathrm{Spec}(X)$, we have
\[ \det(z-(X+P))= \det(z-X)\det\big({\rm Id}_n-y(z-X)^{-1}{\bm e}_k {\bm e}_k^{\sf T}\big).\]
Now, we know that for any $m\in\NN$ and $A\in \mathcal{M}_{n,m}$\footnote{The notation $\mathcal{M}_{m',n'}$ is used to denote the set of all $m' \times n'$ matrices.}, $B\in\mathcal{M}_{m,n}$, $\det({\rm Id}_n- AB) = \det({\rm Id}_m-BA)$. Thus, for any $z\notin \mathrm{Spec}(X)$, 
\begin{equation} \label{eqtopei} \det(z-(X+P))= \det(z-X)\big(1-y\langle {\bm e}_k , (z-X)^{-1} {\bm e}_k\rangle \big).\end{equation}
Let $f_X(z) := 1- y\langle {\bm e}_k , (z-X)^{-1} {\bm e}_k\rangle $. From the equation \eqref{eqtopei}, we deduce that if $f_X$ vanishes at $\mu > \lambda_X$, then $\lambda_{X+P} \geq \mu$.  We will show that with our choice of $y$, $f_X$ typically has a zero in $[\lambda+2\delta,+\infty)$. To this end define $f(z) := 1- ym(z)$ for any $z>2$, and for any $\veps>0$, let the event $E_{\delta,\veps} :=\big\{\lambda_X\leq \lambda+\delta, \ \sup_{z\geq \lambda+2\delta} |f_X(z)-f(z)|\leq \delta\big\}$.

 We claim that there exists an $\veps=\veps(\delta)>0$ such that $E_{\delta,\veps} \subset \{\lambda_{X+P}\geq  \lambda+2\delta\}$. Indeed, $f$ is a continuous increasing function on $[2,+\infty)$ vanishing at $m^{-1}(1/y)=\lambda+3\delta$. Thus, using the mean-value theorem, there exists an $\veps>0$ depending on $\delta$ and $f$ such that for any continuous function $g$ on $[\lambda+2\delta,+\infty)$, if $\sup_{z\geq \lambda+2\delta}|f(z)-g(z)|\leq \veps$, then $g$ has a zero on $(\lambda+2\delta,+\infty)$.
Since $f_X$ is continuous on $(\lambda_X,+\infty)$, we get the claim. This enables us to write 
\[ \PP(\lambda_{X+P}\geq \lambda+2\delta, \|X_{[k]}\|\leq \delta) \geq \PP\big( E_{\delta,\veps} \cap \{\|X_{[k]} \|\leq \delta\}\big) \ge \PP( E_{\delta,\veps}) + \PP(\|X_{[k]} \|\leq \delta) -1,\]
for some $\veps>0$. Note that by Lemma \ref{cvedgetyp} and Proposition \ref{prop:conc-loc-law}, $\PP(E_{\delta,\veps})\to 1$ as $n\to +\infty$. Moreover, clearly $\|X_{[k]}\| \to 0$ in probability as $n \to+\infty$. This ends the proof of \eqref{typical1} and thus that of \eqref{eq:lbd-part1}.

We now move on to the second part of the proof. Let $r\geq  0$ and $s\geq 1$ such that $r+m(\lambda)s=\lambda$. To prove the lower bound it suffices to show that for this arbitrary choice of $r$ and $s$
\begin{equation} \label{lb2} \liminf_{n\to+\infty} \frac{1}{np}\log \PP\big(\lambda_X\geq \lambda) \geq -\frac{r^2}{2\alpha} - h_L(s).\end{equation}
One can find $s_\delta$ and $r_\delta$ functions such that $r_\delta \to r$ and $s_\delta \to s$ as $\delta\to 0$ satisfying for any $\delta>0$, $r_\delta+m(\lambda+3\delta)s_\delta=\lambda+3\delta$ and $r_\delta\geq 0$, $s_\delta\geq 1$. %Decompose $X$ along the top-left corner and 
Define %the matrices $W,P'$, and $\Delta$ as
\[ 
%X =\left( \begin{array}{cc}
%X_{1,1}&  \widecheck{X}_1^{\sf T} \\
%\widecheck{X}_1  &  X^{(1)} \end{array} \right), \ 
W :=  \begin{pmatrix} r_\delta & 0 \\ 0 & X^{(1)}\end{pmatrix}, \ P' := \begin{pmatrix} 0 & \widecheck{X}_1^{\sf T} \\ \widecheck{X}_1 & 0\end{pmatrix}, \text{ and } \Delta := (X_{1,1}-r_\delta){\bm e}_1 {\bm e}_1^{\sf T}.\]
where we refer the reader to Section \ref{sec:notation} for the notation $\widecheck{X}_1$ and $X^{(1)}$.
 %is a $(n-1)$-dimensional column vector, and $X^{(1)}$ is the submatrix spanned by the columns and rows in $\{2,\ldots,n\}$. 
 Observe that $P'$ is a rank two matrix with eigenvalues $\pm \|\widecheck{X}_1\|$ and corresponding eigenvectors $(1, \pm \widecheck{X}_1)^{\sf T}$. Therefore %we can write
\[ P' = Q + \Delta', \text{ where }  Q:= O D O^{\sf T}, \  O := \frac{1}{\sqrt{2}}\begin{pmatrix} 1 & 1 \\ V & -V \end{pmatrix}, \ D := \begin{pmatrix} \sqrt{s_\delta} & 0 \\ 0 & -\sqrt{s_\delta} \end{pmatrix},\]
$V := \widecheck{X}_1/\|\widecheck{X}_1\|$, and $\Delta'$ is some matrix such that $\|\Delta '\| \leq |s_\delta-\|\widecheck{X}_1\|^2|^{1/2}$. 
For any $\kappa>0$ set %$F_{\delta,\kappa}$ to be the large deviation event 
$F_{\delta,\kappa}:= \{|X_{1,1}-r_\delta|\leq \delta, \ s_\delta \leq \|\widecheck{X}_1\|^2 \leq s_\delta+\delta^2,  \sum_{j=2}^n \xi_{1,j} \leq \kappa np \}$. As $X = W+Q+\Delta+\Delta'$, by Weyl's inequality we have  $\lambda_X\geq \lambda_{W+Q} - \|\Delta \|-\|\Delta'\|$. Thus,
\[ \PP(\lambda_X\geq \lambda) \geq \PP\big(\{\lambda_{W+Q}\geq \lambda+2\delta\}  \cap F_{\delta,\kappa}\big).\]
Arguing as in the proof of \eqref{eq:lbd-part1}, we know that if $\lambda_{W+Q}>\lambda_W$, then $\lambda_{W+Q}$ is the largest zero of the function $\wt f_W(z) := \det({\rm Id}_n - (z-W)^{-1}OD O^{\sf T})$, defined for any $z>\lambda_W$. Using again the identity $\det({\rm Id}_n- AB) = \det({\rm Id}_2-BA)$ for any rectangular matrices $A \in \mathcal{M}_{n,2}$, $B\in\mathcal{M}_{2,n}$, we get for any $z>\lambda_W$,
\begin{equation} \label{deffdet} \wt f_W(z) =  \det({\rm Id}_2 - O^{\sf T}(z-W)^{-1}OD)= \det\begin{pmatrix} 
1 - s_\delta \langle v, (z-W)^{-1} v\rangle & s_\delta\langle v, (z-W)^{-1} w\rangle \\
-s_\delta \langle v, (z-W)^{-1} w\rangle & 1+s_\delta \langle w,(z-W)^{-1}w\rangle\end{pmatrix},\end{equation}
where $v :=  (1, V) /\sqrt{2}$, $w:= (1, -V)/\sqrt{2}$. Define for any $z>\max(r_\delta,2)$, 
\begin{equation} \label{deffWdet} \wt f(z) := \det\begin{pmatrix}
1 - \frac{s_\delta}{2}\big(\frac{1}{z-r}+ m(z)\big) & \frac{s_\delta}{2} \big(\frac{1}{z-r_\delta}-m(z)\big) \\
-\frac{s_\delta}{2} \big(\frac{1}{z-s} -m(z)\big) & 1+\frac{s_\delta}{2}\big(\frac{1}{z-r_\delta} +m(z)\big)
\end{pmatrix}=1 - \frac{s_\delta^2 m(z)}{z-r_\delta}.\end{equation}
Note that as $m$ positive and decreasing, $\wt f$ is increasing on $(r_\delta,+\infty)$ and vanishes at $\lambda+3\delta$, by definition of $r_\delta$ and $s_\delta$.
 Consider for any $\veps>0$ the event  $\wt E_{\delta,\veps} := \big\{\lambda_W\leq \lambda+\delta, \sup_{z\geq \lambda+2\delta}|\wt f_W(z) - \wt f(z)| \leq \veps\}$.
With the same argument as in the first part we deduce %from the mean value theorem 
that there exists $\veps>0$ such that $\wt E_{\delta,\veps} \subset \{\lambda_{W+Q}\geq \lambda+2\delta\}$. Thus, for this choice of $\veps=\veps(\delta)$, we have the lower bound
\begin{equation}\label{Ldlowerb2-1} 
\PP(\lambda_X\geq \lambda) \geq \PP(\wt E_{\delta,\veps} \cap F_{\delta,\kappa}).
\end{equation}
Now we show that
\begin{equation}\label{claimcondproba} \Car_{F_{\delta,\kappa}}\PP( \wt E_{\delta,\veps}^c \mid X_{1,1},\widecheck{X}_1)\underset{n\to+\infty}{\longrightarrow} 0 \quad \text{in}\  L^\infty,\end{equation} 
and for $\kappa$ large enough
\begin{equation}\label{Ldlowerb2} \lim_{\delta\to 0}\liminf_{n\to+\infty} \frac{1}{np} \log  \PP(F_{\delta,\kappa})\geq - \frac{r^2}{2\alpha} -h_L(s).\end{equation}
Clearly, as $F_{\delta,\kappa}$ is $(X_{1,1},\widecheck{X}_1)$-measurable, \eqref{claimcondproba} and \eqref{Ldlowerb2} imply the large deviation lower bound \eqref{lb2}.

So it remains to prove \eqref{claimcondproba} and \eqref{Ldlowerb2}. For the first claim \eqref{claimcondproba}, note that for any $z>\lambda_W$, 
\begin{equation} \label{conc-Stiel1} 2\langle v, (z-W)^{-1} v\rangle =  2\langle w, (z-W)^{-1} w\rangle  = (z-r)^{-1} +  \langle V, (z-X^{(1)})^{-1}V\rangle,\end{equation}
whereas $2\langle v, (z-W)^{-1} w\rangle=(z-r)^{-1} -  \langle V, (z-X^{(1)})^{-1}V\rangle$. On  $F_{\delta,\kappa}$ we have $|\supp(\widecheck{X}_1)| \le \kappa np$, and hence $\|V\|_1\leq \sqrt{\kappa np}$. Using Lemma \ref{cvedgetyp}, Proposition \ref{prop:conc-loc-law}, the independence between $X^{(1)}$, and $(X_{1,1},\widecheck{X}_1)$ and the continuity of the determinant yield the claim \eqref{claimcondproba}. 

We now turn to prove \eqref{Ldlowerb2}.
By independence, Lemmas \ref{lbcoeff} and \ref{LDP-deg}, and the continuity of $h_L$, \eqref{Ldlowerb2} holds for $F_\delta :=\{|X_{1,1}-r_\delta|\leq \delta, \  s_\delta\leq \|\widecheck{X}_1\|^2 \leq s_\delta+\delta^2\}$. From \eqref{eq:LDP-deg-ut-R} we find that $\sum_{j=2}^n \xi_{1,j}/np$ is exponentially tight. Thus, for $\kappa$ large enough \eqref{Ldlowerb2} holds for $F_{\delta,\kappa}=F_\delta \cap \{\sum_{j=2}^n \xi_{1,j}\leq \kappa np\}$.
This finally completes the proof of the large deviation lower bound \eqref{lb2}.  
\end{proof}

Next we prove the large deviation lower bounds for $\lambda_{{\rm Adj}_o}$ and for $\lambda_2({\rm Adj})$. While the proof for %large deviation lower bound of 
$\lambda_{{\rm Adj}_o}$ is a straightforward adaptation of the proof of the lower bound of Theorem  \ref{theo-main-weight}, the one for $\lambda_2({\rm Adj})$ requires some information about the structure of the top eigenvector of ${\rm Adj}$. In the following lemma we reproduce the part  of  \cite[Lemma 4.9]{BR21} that we will need for our argument.

\begin{Lem}\label{lem:lam1u1}
Assume that $np \to +\infty$ as $n \to +\infty$. Denote by ${\bm u}$ a unit eigenvector associated to the top eigenvalue of ${\rm Adj}$. For any $M\geq 1$ and any $n\in\NN$ large enough, the following bounds hold deterministically on the event $\{\|{\rm Adj}_o\| \le M\sqrt{np}\}$,
\[|\lambda_1({\rm Adj}) - np| \le 2M \sqrt{np} \qquad \text{ and } \qquad \|{\bm u} - {\bf e}_n\| \le \frac{16 M}{\sqrt{np}}. 
\]
\end{Lem}

We also need the following concentration bound. Notice its difference from Proposition \ref{prop:conc-loc-law}. %Recall that ${\bm e}_k \in \RR^n$, for $k \in [n]$, is the vector that is $k^{-1/2}$ in its first $k$ coordinates and zero otherwise.

\begin{Pro}\label{prop:conc-loc-law-o}
Let $\mathcal{U}_n$ be as in Proposition \ref{prop:conc-loc-law}. Assume $\log(np) \gtrsim \log n$. 
%Define  $\corAQ{\wh{\mathcal{U}}_{n}:= \{u \in \mathbb{S}^{n-1}: \|u\|_1 \le np/\sqrt[4]{\log(1/p)}\}}$. 
Fix $t,\delta>0$. Then, there exists a numerical constant $c>0$ such that for  any $u\in{\mathcal{U}}_n$ and $v \in \mathbb{S}^{n-1}$, and all $n$ large enough, 
\begin{equation}\label{eq:conc-loc-law}
\PP\Big(\sup_{\lambda\geq 2+2\delta} \big|  \langle u, (\lambda-{\rm Adj}_o/\sqrt{np})^{-1} v\rangle - m(\lambda) \langle u, v \rangle \big| \ge t, \lambda_{{\rm Adj}_o}\leq 2+\delta \big) \leq \exp \big(-{c\delta^4 t^2 np}\big).
\end{equation}
%where $c>0$ is a positive numerical constant. %Moreover, the same result holds for ${\rm Adj}_o/\sqrt{np}$ instead of $X$.
\end{Pro}

\begin{proof}[Proof of Theorems \ref{thm:cen-adj} and \ref{thm:ncen-adj} (lower bounds)]
Note that as ${\rm Adj}_o/\sqrt{np}$ satisfies Proposition \ref{prop:conc-loc-law} and $h$ is the rate function of the square of the $\ell^2$-norm of a column of ${\rm Adj}_o/\sqrt{np}$, the same proof leading to \eqref{lb2} (with $r=0$ and $s=\lambda/m(\lambda))$ can be used, yielding the large deviation lower bound of Theorem \ref{thm:cen-adj}.

We now turn ourselves to the proof of the lower bound of Theorem \ref{thm:ncen-adj}. The overall strategy is as follows: We  show that on a certain event (similar to the ones appearing in the proof of the lower bound \eqref{lb2}), 
%
%the upper tail large deviation event of use for ${\rm Adj }$ the same large deviation event as for the centered matrix ${\rm Adj}_o$, and show that 
with high probability,
%on that event 
the projection of a top eigenvector of ${\rm Adj}_o$ on the hyperplane orthogonal to a top eigenvector of ${\rm Adj}$ gives a direction where the quadratic form  induced by ${\rm Adj}$ is large, and thus entails a large deviation of $\lambda_2({\rm Adj})$. We now carry out the details.
%\corAQ{To this end, decompose ${\rm Adj}_o$ as 
%%\[ {\rm Adj}_o = \begin{pmatrix} 
%%0 & {Y}^{\sf T} \\ Y & {\rm Adj}_o^{(1)} \end{pmatrix},\]
%where $Y$ is a $(n-1)$-dimensional column vector and ${\rm Adj}_o^{(1)}$ is the submatrix spanned by the rows and columns in $\{2,\ldots,n\}$. Define} 

%The proof of \eqref{lb2} shows that if $\lambda >2$ and $\delta>0$, then there exists $\veps>0$ such that for $n$ large enough $E_{\delta,\veps} \cap F_{\delta} \subset \{\lambda_{{\rm Adj}_o}\geq (\lambda+\delta)\sqrt{np}\}$, where 

Set $\bar{F}_{\delta} := \{ s_\delta np \leq \|Y\|^2 \leq (s_\delta+\delta^2) np \}$, where $s_\delta=(\lambda+4\delta)/m(\lambda+4\delta)$ and $Y \in \RR^{n-1}$ is the first column of ${\rm Adj}_o$ after removing its diagonal entry. Notice that $(1-2p) \sum_{i=2}^n \xi_{i,1} \leq \|Y\|^2$ and therefore on $\bar F_\delta$ we also have that $\|Y\|_1/\|Y\|_2 \lesssim \sqrt{np}$ for all large $n$. Denote $\bar E_{\delta, \vep}:=\{\lambda_{W_o/\sqrt{np}} \le \lambda+\delta, 
\sup_{z \ge \lambda+2\delta}|\wt f_{W_o/\sqrt{np}}(z) - \bar f(z)| \le \vep\}$ (recall \eqref{deffdet}), where $W_o$ is the $n \times n$ matrix obtained from ${\rm Adj}_o$ after zeroing out the first row and column, and $\bar f(z):= 1 - s_\delta^2m(z)/z$. As the entries of ${\rm Adj}_o$ are bounded arguing as in the steps leading to \eqref{Ldlowerb2-1} we deduce that there exists $\veps>0$ such that $\bar E_{\delta,\veps} \cap \bar F_{\delta} \subset \{\lambda_{{\rm Adj}_o}\geq (\lambda+\delta)\sqrt{np}\}$ for all large $n$.  Using that $\|Y\|_1/\|Y\|_2 \lesssim \sqrt{np}$ and arguing as in the proofs of \eqref{claimcondproba} and \eqref{Ldlowerb2} we further derive that 
%and arguing as in the proof of \eqref{} we deduce that here exists $\veps>0$ such that $E_{\delta,\veps} \cap F_{\delta} \subset \{\lambda_{{\rm Adj}_o}\geq (\lambda+\delta)\sqrt{np}\}$ for all large $n$. 
%
% Letting $\bar E_{\delta,\veps}$ to be the event with $X$ replaced by ${\rm Adj}_o^{(1)}$ 
%
%, and $E_{\delta,\veps}$ is such that 
\begin{equation} \label{typicaleventE}\Car_{\bar F_{\delta}}\PP( \bar E_{\delta,\veps}^c \mid Y)\underset{n\to+\infty}{\longrightarrow} 0 \quad \text{in}\  L^\infty \quad \text{ and } \quad \lim_{\delta \to 0} \liminf_{n\to +\infty} \frac{1}{np} \log \PP( \bar F_\delta)\geq - h(\lambda/m(\lambda)). \end{equation}
%and 
%\begin{equation} \label{LDeventAdj}\lim_{\delta \to 0} \liminf_{n\to +\infty} \frac{1}{np} \log \PP( F_\delta)\geq - h(\lambda/m(\lambda)).\end{equation}
%Here, we implicitly used that $(1-2p) \sum_{i=2}^n \xi_{i,1} \leq \|Y\|^2$, so that on $F_\delta$ the degree of the first vertex is a most of order $np$, and we can simplify the event $F_{\delta,\kappa}$ into $F_\delta$. 
Next let ${\bm u}_o$, ${\bm u}$ be unit eigenvectors for the top eigenvalues of ${\rm Adj}_o$ and ${\rm Adj}$ respectively. Write ${\bm u}_o^{\sf T} =({\bm v_o},{\bm w_o})^{\sf T}$, where ${\bm v}_o \in \RR$ and ${\bm w}_o\in \RR^{n-1}$. Observe that the projection of the equality ${\rm Adj}_o {\bm u}_o = \lambda_{{\rm Adj}_o} {\bm u}_o$ on the last $(n-1)$-coordinates, on the event $\lambda_{{\rm Adj}_o} \notin {\rm Spec}({\rm Adj}_o^{(1)})$, yields that ${\bm w}_o = {\bm v}_o (\lambda_{{\rm Adj}_o}-{\rm Adj}_o^{(1)})^{-1} Y$. This leads us to define the event 
\[ G_{\delta,\veps} :=\big\{\sup_{z\geq 2+\delta} \big| \langle {\bm e}_{n-1}, (z-{\rm Adj}_o^{(1)}/\sqrt{np})^{-1}Y\rangle - m(z) \langle {\bm e}_{n-1}, Y\rangle \big| <\veps\|Y\|, \|{\rm Adj}_o^{(1)}\|\leq (2+ \delta/2)\sqrt{np}\big\}.\]
By Lemma \ref{cvedgetyp} and Proposition \ref{prop:conc-loc-law-o}  we have %that 
%\begin{equation} \label{typicalG} 
$\Car_{F_\delta} \PP(G_{\delta,\veps}^c\mid Y) \stackrel{L^\infty}{\to} 0$, as $n \to +\infty$ %inL^\infty$.%\end{equation}
This together with \eqref{typicaleventE} imply that 
%denoting $\mathcal{E}_{\veps,\delta}:= E_{\delta,\veps} \cap F_\delta \cap G_{\corAB{\delta, \vep}}$ and using \eqref{typicaleventE}, \eqref{LDeventAdj}, \eqref{typicalG}, we get
\begin{equation} \label{lbE} \lim_{\delta \to 0}\liminf_{n\to+\infty} \frac{1}{np} \log \PP( \mathcal{E}_{\veps,\delta}) \geq-h(\lambda/m(\lambda)), \text{ where } \mathcal{E}_{\veps,\delta}:= \bar E_{\delta,\veps} \cap \bar F_\delta \cap G_{\delta, \vep}. 
\end{equation}
Hence it remains to show that the event $\mathcal{E}_{\delta,\veps}$ entails a large deviation of $\lambda_2({\rm Adj})$. As $\lambda_{{\rm Adj}_o}\geq \lambda +\delta$ on the event $\mathcal{E}_{\delta,\veps}$, we have for $n$ large enough on that event
\[ |\langle {\bm e}_n,{\bm u}_o\rangle |\leq 1/\sqrt{n}+|\langle {\bm e}_{n-1} ,{\bm w}_o\rangle|\leq 1/\sqrt{n} +\|Y\|_1/(n\sqrt{p}) + \veps \|Y\|_2/\sqrt{np}\, \lesssim  \veps.\]
where we used the fact that $\|Y\|_1=O(np)$ and $\|Y\|_2=O(\sqrt{np})$ on $\mathcal{E}_{\delta,\veps}$.
%, and $c\geq 1$ is some constant depending on $s$. 
By Lemma \ref{lem:lam1u1} the above inequality implies that 
%${\bm u}_o$ is almost orthogonal to ${\bm u}$ on $ \mathcal{E}_{\delta,\veps}$. More precisely, 
on the same event $ \mathcal{E}_{\delta,\veps}$, $|\langle {\bm u}, {\bm u}_o\rangle|\leq  |\langle {\bm e}_n,{\bm u}_o\rangle|+ \|{\bm u} -{\bm e}_n\|\, \lesssim \veps$. Further, if $\widetilde{{\bm u}}_o$ is the projection of ${\bm u}_o$ on ${\bm u}^\perp$ (the hyperplane orthogonal to ${\bm u}$), then again on $\mathcal{E}_{\delta,\veps}$, $|\langle \widetilde{\bm u}_o,{\bm e}_n\rangle|\leq \|{\bm e}_n-{\bm u}\|\leq \veps$ and $\|\widetilde{\bm u}_o-{\bm u}_o\|\leq |\langle {\bm u},{\bm u}_o\rangle|\, \lesssim \veps$ for $n$ large enough. Thus, on $\mathcal{E}_{\delta,\veps}$, for all $n$ large enough
\begin{align*} \langle \widetilde{{\bm u}}_o, {\rm Adj } \widetilde{{\bm u}}_o \rangle &=  \langle \widetilde{{\bm u}}_o, {\rm Adj }_o \widetilde{{\bm u}}_o\rangle +np \langle \widetilde{\bm u}_o, {\bm e}_n\rangle^2 -p \|\widetilde{\bm u}_o\|^2\\
& \geq \langle {\bm u}_o, {\rm Adj}_o {\bm u}_o\rangle -2\|\widetilde{\bm u}_o-{\bm u}_o\|\cdot \| {\rm Adj}_o\| +np \langle \widetilde{{\bm u}}_o, {\bm e}_n\rangle^2 - p \geq  \lambda_{{\rm Adj}_o} - C\veps \sqrt{np},\end{align*}
where $C$ is a positive constant. As $\bar E_{\delta,\veps}\cap \bar F_\delta \subset \{\lambda_{{\rm Adj}_o} \geq (\lambda+\delta)\sqrt{np}\}$ and $\wt {\bm u}_o$ is orthogonal to ${\bm u}$, this shows that for any $\delta>0$ there exists $\veps>0$ small enough such that for $n$ large enough $\mathcal{E}_{\delta,\veps} \subset \{\lambda_2({\rm Adj})\geq \lambda \sqrt{np}\}$. By \eqref{lbE} this ends the proof. 
\end{proof}

\section{Concentration for convex functions with flat gradients}\label{sec:concflat}
In this section we investigate the concentration property of a convex Lipschitz function of a sparse Wigner matrix with bounded coefficients. It is known from Talagrand's inequality \cite[Corollary 4.7]{Ledouxmono} or Marton's conditional transportation inequality \cite[Theorem 8.6]{BLM} that such functions enjoy sub-Gaussian concentration inequalities. We show that when such a function has  ``flat'' subgradients in the sense that they  belong to a set of low rank matrices associated to delocalised vectors, then one can reduce the variance factor in the sub-Gaussian concentration inequality according to the delocalisation strength. To formulate this notion we need to fix some notation.
%
%, in the sense that it spans in the set of low-rank matrices associated to delocalised vectors, then one can reduce the variance factor in the sub-Gaussian concentration inequality according to the delocalisation strength. 
%
%\corAB{In this section we develop concentration bounds for convex Lipschitz functions of a sparse Wigner matrix with bounded coefficients that have flat subgradients in the sense that they  belong to a set of lowrank matrices associated to delocalised vectors. To formulate this notion we need to fix some notation.} 
%, then one can reduce the variance factor in the sub-Gaussian concentration inequality according to the delocalisation strength. 
%
%
For any $k\geq 1$, $\eps>0$, define the subset $\mathcal{F}_\eps^k$ of ``flat'' matrices of rank at most $k$ by
\begin{equation} \label{flatmatrices} \mathcal{F}_{\eps}^k := \Big\{ \sum_{\ell=1}^k \theta_\ell v_\ell v_\ell^{\sf T} : {v_\ell} \in \mathcal{D}_\veps, |\theta_\ell|\leq 1, \  \forall \ell\in[k]\Big\}, \quad \text{ where } \mathcal{D}_\veps:= \{ w \in \mathbb{B}^n: \|w\|_\infty \le \veps\}. 
\end{equation}
We work in this section with a slightly more general model compared to the one of the sparse Wigner matrices introduced in Definition \ref{defmodel}, and allow the entries to have different distributions, as long as they remain uniformly bounded. Unless mentioned otherwise in this section we let $G$ to be a random symmetric matrix independent of $\Xi$ such that $(G_{i,j})_{i\leq j}$ are independent, centered and bounded by $1$. Set $\widehat{X} = G\circ \Xi$ and $X= \widehat{X}/\sqrt{np}$. 
By \cite[Theorem 8.6]{BLM}, it follows that if $f : \mathcal{H}_n \to \RR$ is a convex $1$-Lipschitz function with respect to the Hilbert-Schmidt norm, then for any $t>0$, 
\begin{equation}\label{eq:subG-talagrand}
\PP\big( |f(X) -\EE f(X)| >t \big) \leq 2e^{-\frac{t^2 np}{32}}.
\end{equation}
Our next result shows that when the subgradients of $f$ belong to the subset $\mathcal{F}_\eps^k$ one can improve the above concentration inequality for the upper tail and obtain a variance factor scaling as $\veps^2$.

\begin{Pro}\label{strongconc}
Let $\veps \in (0,1/3)$ and $n\geq 2$ such that $np \geq 4 \log n$.  Let $f : \mathcal{H}_n \mapsto \RR$ be a convex function such that there exists a measurable function $\zeta : \mathcal{H}_n \to \mathcal{F}_{\eps}^k$ satisfying
\begin{equation} \label{convf}   
f(K) - f(K') \leq \langle \zeta(K), K-K'\rangle, \qquad \forall K,K' \in \mathcal{H}_n.
\end{equation}
There exist constants $\gamma_0, p_0>0$ depending on $k$ such that if $p\leq p_0$,
\begin{equation}\label{eq:eps-strongconc}
\frac{\gamma_0}{\sqrt{\log(1/p)}}\leq t \leq \frac{1}{\gamma_0}  \text{ and } \frac{\log(t^2 \log(1/p))}{t\sqrt{\log(1/p)}} \leq \veps <1,
\end{equation} 
then
\begin{equation}\label{eq:strongconc} 
\PP\big( f(X) - \EE f(X) > t \big) \leq 2\exp\Big(-\frac{t^2np }{\gamma_0 \veps^2}\Big).
\end{equation}
\end{Pro}

Compare \eqref{eq:strongconc} with the standard concentration bound \eqref{eq:subG-talagrand}. 
The improved concentration inequality in Proposition \ref{strongconc} relies on a generalised moments inequality \cite[Theorem 2]{BBLM} (see also \cite[Theorem 15.5]{BLM}) that in a sense generalises the Efron-Stein inequality to higher moments.  It builds on the idea that the upper tail behaviour of a function $f : \RR^{q} \to \RR$ of independent random variables $(Y_1,\ldots,Y_{q})$ can be estimated through the random variable
\begin{equation}\label{eq:V+}
{V^+: = V^+(f):}=\sum_{i=1}^q \EE' (f(Y)-f({Y'_{(i)}}))_+^2,
\end{equation}
where $Y'_{(i)} =(Y_1,\ldots,Y_{i-1},Y_i',Y_{i+1},\ldots,Y_q)$ with $Y_i'$ an independent copy of $Y_i$ and $\EE'(\cdot)$ denotes the expectation with respect to the randomness of $Y'$ (one can regard $V^+$ as a kind of ``local'' Lipschitz constant).
%
%
%
%
%Likewise, our approach relies on the generalized moments inequality of Boucheron-Lugosi-Massart \cite[Theorem 15.5]{BLM}. Such a moments inequality builds upon the idea that much of the upper tail of a function $f : \RR^n \to \RR$ of independent random variables $(X_1,\ldots,X_n)$ can be estimated through the random variable $V^+$,
%\[ 
%\corAB{V^+: = V^+(f):}=\sum_{i=1}^n \EE' (f(X)-f({X'_{(i)}}))_+^2,\]
%where $X'_{(i)} =(X_1,\ldots,X_{i-1},X_i',X_{i+1},\ldots,X_n)$ with $X_i'$ an independent copy of $X_i$. One can regard $V^+$ as a kind of ``instantaneous'' or ``local'' Lipschitz constant. 
By the celebrated Efron-Stein inequality (see \cite[Theorem 3.1]{BLM}) the expectation of $V^+$ provides an upper bound on the variance of $f(X)$. Further, by \cite[Theorem 15.4]{BLM} the bound $V^+\leq c$ a.s.\, implies a sub-Gaussian upper tail bound with variance factor a multiple of $c$. In Lemma \ref{enhancedconcentration} below we reformulate \cite[Theorem 15.5]{BLM} and show that if $V^+(f)$ is small with overwhelming probability, then one obtains an improved sub-Gaussian upper tail bound.

%$V^+(f)$ can explain the exponential decay of the upper tail of $f(X)$. From \cite[Theorem 15.4]{BLM}, we know that 

Improved concentration inequalities using such ideas had been used in \cite{LMZ} to obtain uniform concentration inequalities for the spectral radius of Erd\H{o}s-Rényi graphs.
However, in our case for the functions $f$ that are relevant to us, it turns out that $V^+(f)$ is not small with overwhelming probability. Nevertheless, owing to the boundedness assumption of the entries of $G$ it is possible to decompose $f(X)$ into $f_1(X)+f_2(X)$ such that $V^+(f_1)$ is small with overwhelming probability and $f_2(X)$ has negligible deviations at the exponential scale $np$. 
%because of our convexity assumption, $V^+$ is controlled by a sum of independent $\dBer(p)$ variables weighted by the square of ${X}$-measurable entries of flat matrices in \eqref{defG}. 

Thus, our first step towards proving Proposition \ref{strongconc} is to show an improved sub-Gaussian upper tail bounds for convex functions of sparse Wigner matrices with independent bounded coefficients in the case where the gradient spans in the set $\mathcal{G}_{\eps,M}^k$ that we describe now. 
For any integer $k\geq 1$, and $M\geq 1$, $\eps>0$, define the subset 
\begin{equation} \label{defG} 
\mathcal{G}_{\eps,M}^k := \Big\{ \sum_{\ell=1}^k \theta_\ell v_\ell w_\ell^{\sf T} : (v_\ell, w_\ell) \in  (\mathcal{D}_{\eps} \times \mathcal{D}_{M/\sqrt{np}}) \cup (\mathcal{D}_{\veps} \times \mathcal{D}_{M/\sqrt{np}}), |\theta_\ell|\leq 1, \  \forall \ell\in[k]\Big\}.\end{equation}
We will prove the following concentration inequality.
\begin{Pro}\label{improvedconc} Let $\gamma \in (0,1)$, $M\geq 1$ and $1\geq \veps^2\geq 3p$ be such that $\log(M^2/\veps^2)/\log(1/p) \leq \gamma$ and $np\log(1/p) /M^2 \geq 4 \log n$. 
Let $f : \mathcal{H}_n \mapsto \RR$ be a convex function such that there exists a measurable function $\zeta : \mathcal{H}_n \mapsto \mathcal{G}_{\eps,M}^k$ satisfying
\begin{equation} \label{convexf}  f(K) - f(K') \leq \langle \zeta(K), K-K'\rangle, \qquad \forall K,K' \in \mathcal{H}_n.\end{equation}
There exists a constant $\gamma_0>0$ depending on $k$ and $\gamma$ such that if 
\begin{equation}\label{eq:improvedconc-assump}
 \frac{\gamma_0}{np}  \leq t^2 \leq  \frac{1}{\gamma_0} \text{ and }  \eta : = \eps M/\sqrt{\log(1/p)}  \leq e^{-1},
\end{equation}
then
\[ \PP\big( f(X) - \EE f(X) >t \big) \leq \exp\Big(-\frac{t^2np}{\gamma_0 \max\big( \eta   \log ( \frac{1}{\eta}\big),\eps^2\big) }\Big).\]
\end{Pro}
The above inequality tells us that whenever $\eps \ll 1$ and $\eps M \ll  \sqrt{\log(1/p)}$, the speed of the deviations of $f(X)$ is much larger than $np$.
To prove Proposition \ref{improvedconc}, we will need the following lemma. It will allow us to show that $V^+(f)$, for $f$ as in Proposition \ref{improvedconc}, is small with overwhelming probability, as long as $\veps \ll1$ and $\veps M \ll \sqrt{\log(1/p)}$. 

\begin{Lem} \label{controlBernoullisum}Let $\gamma \in (0,1)$. There exists $\wt \gamma=\wt \gamma(\gamma)>0$ such that for any $M\geq 1$, and $\veps,\delta>0$ with  $\delta\geq \max(\veps^2,3p)$ and $\frac{\log(M^2/\delta)}{\log(1/p)} \leq \gamma$, 
 \[ \PP\Big( \sup_{v\in \mathcal{D}_\eps , w\in \mathcal{D}_{M/\sqrt{np}} }\sum_{i,j \in [n]} \xi_{i,j} v_i^2 w_j^2 > \wt \gamma \delta  \Big)\leq  n^2e^{- \delta np \frac{\log(1/p)}{(\veps M)^2}}.\]
\end{Lem}

The proof of Lemma \ref{controlBernoullisum} requires the following two results. The first one is an adaptation of Boucheron-Bousquet-Lugosi-Massart generalised moments inequality \cite[Theorem 2]{BBLM}. The second one lists some properties of the $\log$-Laplace transform and its Legendre-Fenchel transform of a centered $\dBer(p)$ random variable. Proofs of these two results are postponed to Appendix \ref{app:aux-res}.

\begin{Lem}\label{enhancedconcentration}
Let $Y=(Y_1,\ldots,Y_{N})$ be a vector of independent random variables taking values in a set $\mathcal{X}$ and $f : \mathcal{X}^N \mapsto \RR$ be a measurable function. Denote by $Z = f(Y_1,\ldots, Y_N)$. Assume that $V^+(f)\leq1$ almost surely. Let $\veps>0$ and define $E_\veps:= \big\{ V^+(f) \leq \veps^2\big\}$. Let  $q_0 := \log (\PP (E_\veps^c))/ \log(\veps)$.
Then there exists a numerical constant $c>0$ such that, for any $\sqrt{2} \leq t/\sqrt{c \veps^2} \leq \sqrt{q_0}$,
\begin{equation}\label{eq:enhancedconcentration}
 \PP\big( Z- \EE Z > t\big)  \leq e^{-\frac{t^2}{2c\veps^{2}}}.
\end{equation}
\end{Lem}

\begin{Lem}\label{lem:Lam-bd}
Let $\zeta$ be a $\dBer(p)$ random variable. For $\theta, x \in \RR$ define
\begin{equation}
\Lambda_p(\theta):= \log \EE \exp(\theta(\zeta - p)) \quad \text{ and } \quad \Lambda_p^*(x):= \sup_{\theta' \in \RR} \{ \theta'x - \Lambda_p(\theta')\}.
\end{equation}
Let $p<1/2$. Then the following properties hold.
\begin{enumerate}
\item[(i)]  For any $\theta \in \RR$,  $\Lambda_p(\theta) \leq (1-2p)\theta^2/[4 \log \big(\frac{1-p}{p}\big)]$.
\item[(ii)]  For any $-p\leq x \leq Cp\leq \frac{1}{2}-p$, $\Lambda^*_p(x) \geq o_C(1) \frac{x^2}{p}$ as $C \to +\infty$.
\item[(iii)] For any $x\geq Cp$, $\Lambda^*_p(x) \geq o_C(1) x \log \frac{x}{p}$ as  $C \downarrow e$.
\item[(iv)] For any $C>0$, $\frac{1}{p}\Lambda_p^*(Cp) \, = (1+o(1)) ( C+1) \log (C+1) -C$, as $p\to 0$. 
\end{enumerate}
\end{Lem}

\begin{proof}[Proof of Lemma \ref{controlBernoullisum}] 
Fix $M\geq 1$, $\veps,\delta>0$ such that $M/\sqrt{np} \leq \veps$, $\delta \geq \max(\veps^2,3p)$ and $\log(M^2/\delta)\leq \gamma \log(1/p)$. The idea is to split the sum $\sum_{i,j} \xi_{i,j}v_i^2 w_j^2$ according to a dyadic partition of  the coordinates of $v$ and $w$ in order to take advantage of the concentration of empirical means of Bernoulli random variables, and finally to proceed with union bounds. 
Let $k_0:= \max\{ k \in \NN : \veps \leq 2^{-k}\}$, and $\ell_0 :=\max\big\{k\in \NN : M/ \sqrt{np} \leq 2^{-k}\big\}$.
Define for any $k\in \NN$ and $v\in \mathcal{D}_\eps$,  $J_v^k:=\{ i\in [n] : 2^{-(k+1)} < |v_i| \leq  2^{-k}\}$.
We set for any $v\in\mathcal{D}_\eps$ and $k\geq k_0$, $ \alpha_v^k:= \sum_{i \in J_v^k} v_i^2$,  and $I_v:=\{k\geq k_0:{\alpha_v^k} \geq r_k\}$, where 
  \[ r_k :=  \begin{cases}
16 \delta (k-k_0+2)^{-2} & \text{ if } k_0\leq k <\ell_0,\\
16 \delta (k-\ell_0+2)^{-2} & \text{ if } k\geq \ell_0.
  \end{cases}\]   
 Let $v\in \mathcal{D}_\eps$ and $w\in\mathcal{D}_{M/\sqrt{np}}$. We next observe that the choice of the sequence $r_k$ is made so that the contribution of $\sum \xi_{i,j} v_i^2w_j^2$ when $i$ or $j$ runs in the set of ``bad'' indices  $I_v^c$ and $I_w^c$ respectively, is negligible. More precisely, observe that as $v,w\in\mathbb{B}^n$,
\begin{equation} \label{badsums} \sum_{(k,\ell) \notin I_v\times I_w} \sum_{i\in J_v^k, j\in J_w^\ell} \xi_{i,j} v_i^2 {w_j}^2 \leq \sum_{k\notin I_v} \alpha_v^k   +\sum_{\ell\notin I_w} \alpha_w^\ell  \leq 64\delta\sum_{k\geq 1} k^{-2}  = \frac{32\pi^2\delta}{3}, \end{equation}
and
\begin{equation} \label{goodsum}  \sum_{(k,\ell) \in I_v\times I_w} \sum_{i\in J_v^k, j\in J_w^\ell} \xi_{i,j} v_i^2 {w_j}^2\leq\sum_{(k,\ell) \in I_v\times I_w} 2^{-2(k+\ell)} \sum_{i\in J_v^k, j\in J_w^\ell} \xi_{i,j}.\end{equation}
Now, for any $k\in I_v$ and $\ell \in I_w$, 
\begin{equation} \label{boundsJ}  2^{2(k+1)} \alpha_v^k\geq  |J_v^k|\geq r_k 2^{2k}, \ 2^{2(\ell+1)} \alpha_w^\ell \geq |J_w^\ell|\geq r_\ell 2^{2\ell}.\end{equation}
One can check that  $k\mapsto k^{-2} 2^{2k}$ on $[2,+\infty)$ is increasing, and as a consequence $k\mapsto r_k 2^{2k}$ as well. This implies that for any $k\in I_v$, $|J_v^k|\geq r_{k_0} 2^{2k_0} \geq \delta \veps^{-2}$. Note that as $w\in \mathcal{D}_{M/\sqrt{np}}$, $\alpha_w^\ell=0$ for $\ell <\ell_0$. Therefore, for any $\ell \in I_w$, $|J_w^\ell|\geq r_{\ell_0} 2^{2\ell_0} \geq  \delta np/M^2$. 
Now, for fixed sets $S, T \subset [n]$ and $\eta \ge 3p$ ($3$ is arbitrary, any numerical constant strictly greater than $e$ works), we claim that
\begin{equation} \label{boundsum} \PP\Big( \sum_{i\in S, j\in T} \xi_{i,j} > 2(p+\eta)|S|\cdot |T|\Big) \leq e^{- \frac{1}{2}|S|\cdot|T|  \Lambda_p^*(\eta)}  \leq e^{-c \eta |S| \cdot |T| \log(\eta/p)},\end{equation}
where $c>0$ is a numerical constant. Indeed, we can write $\sum_{i\in S,j\in T} \xi_{i,j} \leq 2\sum_{(i,j) \in E} \xi_{i,j}$, 
where $E := \big((S\cap T)^2\cap \{(i,j) \in[n]^2: i\leq j\} \big)\cup (S\setminus T \times T) \cup (S \times T\setminus S)$. One can check that $(\xi_{i,j})_{(i,j)\in E}$ is a family of independent Bernoulli random variables, and that $(1/2) |S|\cdot|T| \leq |E|\leq |S|\cdot|T|$.
Therefore the first inequality in \eqref{boundsum} follows from Chernoff's inequality. To derive the second inequality we use 
Lemma \ref{lem:Lam-bd}(iii). 
%we deduce that for any $\eta \geq 3p$, 
% \[ \PP\Big( \sum_{i\in S, j\in T} \xi_{i,j} > 2(p+\eta) |S|.|T|\Big) \leq e^{-c \eta |S|.|T| \log(\eta/p)},\]
% where $c$ is a numerical constant (3 is arbitrary, any numerical constant strictly greater than $e$ would work). 

Next let $s,t\in \NN$ such that $\delta \veps^{-2}\leq s$ and $t\geq \delta np/M^2$.
As for any $\ell\in[n]$, there are at most $(en/\ell)^\ell$ subsets of $[n]$ of size $\ell$, we get from \eqref{boundsum} by a union bound, for any $\eta\geq 3p$,
\[ \PP\Big( \sup_{|S| = s, |T| =t} \sum_{i\in S  j\in T} \xi_{i,j} > 2(p+\eta)st \Big) \leq e^{2(s\vee t) \log \frac{{e} n}{s\vee t} } e^{-c \eta st \log(\eta/p)}.\]
Since $M\geq 1$ and $\log(M^2/\delta) \leq \gamma \log(1/p)$, we have 
\[ \frac{\log(\delta/p)}{\log(M^2/(\delta p))} \geq \frac{\log(\delta /(M^2p))}{\log(M^2/(\delta p))}\geq  \frac{1-\gamma}{1+\gamma}.\] 
We deduce that there exists $\gamma'>0$ depending on $\gamma$ such that for any $\eta \geq\max( \frac{\gamma'}{s \wedge t},\delta)$, 
\[  c  \eta (s\wedge t) \log (\eta/p)- 2\log \frac{{e}n}{s\vee t} \geq  c\eta (s\wedge t) \log (\eta/p) -2 \log \Big(\frac{{e}M^2}{\delta p}\Big)\geq \frac{c}{2}  \eta (s\wedge t)  \log (\delta/p).\]
Again, as $M\geq 1$ and $\log(M^2/\delta) \leq \gamma \log(1/p)$, we have $\log(\delta/p)/\log(1/p)\geq (1-\gamma) >0$. Fix $r>0$.
At the price of enlarging $\gamma'>0$, we get for any $\eta\geq \max\big( \frac{\gamma' r np}{st}, \delta \big)$,  
\[  \frac{c}{2} \eta (s \vee t) \cdot (s \wedge t) \log (\delta/p)= \frac{c}{2} \eta s t \log (\delta/p) \geq rnp \log(1/p).\]
Thus we have shown that there exists a  constant $\gamma'>0$ depending on $\gamma$ such that
\[ \PP\Big( \sup_{|S| = s, |T| =t} \sum_{i\in S, j\in T} \xi_{i,j} > 2(p+\eta_{s,t}) st \Big) \leq e^{-rnp\log(1/p)},\]
where $\eta_{s,t} = \max\big(\frac{\gamma'}{s \wedge t}, \frac{\gamma'rnp}{st}, \delta, 3p\big)$. 
Using another union bound we obtain 
\[  \PP\bigg( \bigcup_{t \geq \delta np/M^2 \atop  s\geq \delta \veps^{-2}}\Big\{\sup_{|S| = s, |T| =t} \sum_{i\in S, j\in T} \xi_{i,j} > 2(p+\eta_{s,t}) st\Big\}\bigg) \leq   n^2e^{-r np \log(1/p)}.\]
Let $\mathcal{E}$ denote the event on the left hand side above.  Recall that for any $k\in I_v$, $|J_v^k|\geq \delta \veps^{-2}$ and any  $\ell\in I_w$, $|J_w^\ell|\geq \delta np/M^2$. For such  $k,\ell$, set $\eta_{k,\ell}$ to be a short hand for $\eta_{|J_v^k|,|J_w^\ell|}$. On $\mathcal{E}^c$, using \eqref{goodsum} we get
\begin{align*}
 \sum_{(k,\ell) \in I_v\times I_w} \sum_{i\in J_v^k, j\in J_w^\ell} \xi_{i,j} v_i^2 {w_j}^2 &  \leq \sum_{(k,\ell) \in I_v\times I_w} 2^{1-2(k+\ell)} |J_v^k|.|J_w^\ell| (p+\eta_{k,\ell}).
 \end{align*}
To complete the proof we need to bound the right hand side of the above. To this end, we use the bound $\eta_{k,\ell} \leq \gamma'/(|J_v^k| \wedge |J_w^\ell|) + \gamma' rnp/(|J_v^k|.|J_w^\ell|) + \delta+3p$, and compute the contributions of each term.
The upper bounds on $|J_v^k|$ and $|J_w^\ell|$ from  \eqref{boundsJ}, and $\sum_k \alpha_v^k=\sum_\ell \alpha_w^\ell=1$ yield 
 \[  
 \sum_{(k,\ell) \in I_v\times I_w} 2^{-2(k+\ell)} |J_v^k|.|J_w^\ell| \leq 16, \quad \sum_{(k,\ell) \in I_v\times I_w} 2^{-2(k+\ell)} \leq 2^{-2(k_0+\ell_0-2)}\leq 144 \frac{\eps^2M^2}{np}, \]
 and 
 \begin{multline*} 
 \sum_{(k,\ell) \in I_v\times I_w} 2^{-2(k+\ell)} |J_w^\ell| \vee |J_v^k| \le \sum_{(k,\ell) \in I_v\times I_w} 2^{-2(k+\ell)} (|J_w^\ell| + |J_v^k|)  \\
\leq \Big(\sum_{k\geq k_0} 2^{-2k}\Big)\Big( \sum_{\ell } 4\alpha_w^\ell\Big) + \Big(\sum_{\ell\geq \ell_0} 2^{-2\ell}\Big)\Big( \sum_{k } 4\alpha_v^k\Big) \leq128 \eps^2.
 \end{multline*}
%Moreover,  $\sum_{(k,\ell) \in I_v\times I_w} 2^{-2(k+\ell)} \leq 2^{-2(k_0+\ell_0-2)}\leq 144 \frac{\eps^2M^2}{np}$. 
Using all the estimates above and the fact that $\delta \geq \max(\veps^2,3p)$, we get
    \[  \sum_{(k,\ell) \in I_v\times I_w} \sum_{i\in J_v^k, j\in J_w^\ell} \xi_{i,j} v_i^2 {w_j}^2 \leq32(\delta+4p)+288\gamma'\big( \veps^2+ r(\veps M)^2\big)\leq \gamma''(\delta + r(\veps M)^2),\]
where $\gamma''$ is constant depending on $\gamma$. Combining the above inequality with \eqref{badsums} and choosing $r = \delta/(\veps M)^2$, this concludes the proof. 
\end{proof}

Equipped with Lemma \ref{controlBernoullisum}, we can now give a proof of Proposition \ref{improvedconc}.
\begin{proof}[Proof of Proposition \ref{improvedconc}]
Let $\widehat{X}' := G' \circ \Xi'$ be an independent copy of $\widehat{X}$. For any $(i,j)\in[n]^2$, $i\leq j$, let $\widehat{X}_{(i,j)}'$ be the matrix with the same entries as $\widehat{X}$ except that the $(i,j)$ and $(j,i)$ coefficients are replaced by $\widehat{X}_{i,j}'$. Define the random variable $V_+$ as 
\[ V_+ =  \sum_{i\leq j} \EE'\big[\big( f(\widehat{X})-f(\widehat{X}_{(i,j)}')\big)_+^2\big],\]
where $\EE'$ means that the expectation is with respect to $\widehat{X}'$.
For any $1\leq i \leq j\leq n$ we have by \eqref{convexf} 
\[  f(\widehat{X}) - f(\widehat{X}'_{(i,j)}) \leq \langle \zeta(\widehat{X}), \widehat{X}-\widehat{X}'_{(i,j)}\rangle = (\zeta(\widehat{X})_{i,j} + \zeta(\widehat{X})_{j,i} {\bf 1}_{i \ne j})\big(\xi_{i,j} G_{i,j} - \xi_{i,j}' G'_{i,j}\big).\]
Therefore, 
\begin{equation}\label{boundV+}
 V_+ \leq 2 \sum_{i, j \in [n]} (\zeta(\widehat{X})_{i,j})^2 \EE'\big[( \xi_{i,j}G_{i,j} - \xi'_{i,j}G'_{i,j} )^2\big].\end{equation}
Using the bound $|\xi_{i,j}G_{i,j}-\xi'_{i,j}G'_{i,j}|\leq 2$ and the fact that $\zeta(\widehat{X}) \in \mathcal{G}^k_{\veps,M}$, one obtains $V_+\leq 8\|\zeta(\widehat{X})\|_2^2 \leq 8k$ almost surely. Using Lemma \ref{controlBernoullisum}, we will show that $V_+$ is actually much smaller with overwhelming probability. Coming back at \eqref{boundV+} and writing $\xi_{i,j}G_{i,j} - \xi'_{i,j}G'_{i,j}  = \xi_{i,j}(G_{i,j}-G'_{i,j}) + (\xi_{i,j}-\xi_{i,j}')G'_{i,j}$, we get 
\begin{align*}
 \EE'[(\xi_{i,j}G_{i,j}-\xi_{i,j}'G'_{i,j})^2] &\leq 2\EE'\big[\big(\xi_{i,j}(G_{i,j}-G'_{i,j})\big)^2\big] + 2 \EE'\big[\big((\xi_{i,j}-\xi'_{i,j})G'_{i,j}\big)^2\big]\\
& \leq 8\xi_{i,j} + 2\EE'[(\xi_{i,j}-\xi_{i,j}')^2] \leq 10 \xi_{i,j} + 2p.
\end{align*}
where in the last step we used $\EE'[(\xi_{i,j}-\xi_{i,j}')^2] \leq \xi_{i,j}+p$. %, we obtain $\EE'(\xi_{i,j}G_{i,j}-\xi_{i,j}'G'_{i,j})^2  \leq 10 \xi_{i,j} + 2p$.
Using the fact that $\zeta(\widehat{X}) \in \mathcal{G}_{\veps,M}^k$, and as a consequence $\| \zeta(\widehat{X})\|_2^2 \leq k $, we get from \eqref{boundV+}
\[  V_+ \leq 20 k \sup_{v \in \mathcal{D}_\eps, w\in\mathcal{D}_M/\sqrt{np}} \sum_{i, j \in [n]} v_i^2w_j^2 \xi_{i,j} + 4kp.  \]
By Lemma \ref{controlBernoullisum}, there exists  $\gamma'\geq 1$ depending on $\gamma$  such that for any $\delta \geq \max(\veps^2,3p)$ and $\log(M^2/\delta)\leq \gamma \log(1/p)$,
\begin{equation} \label{tailV+}\PP\big( V_+ \geq 24 k\gamma'\delta  \big) \leq n^2e^{-\delta np \frac{\log(1/p)} { (\veps M)^2}}.\end{equation}
Assume  $\eta := \eps M/\sqrt{\log(1/p)}\leq e^{-1}$ and take $\delta = \max( \eta \sqrt{\log(1/\eta)},\eps^2)$. Since $\veps^2\geq 3p$ and $\log(M^2/\veps^2)\leq \gamma \log(1/p)$, we can apply the above inequality \eqref{tailV+}.  As we assumed $np \log(1/p)/M^2\geq 4 \log n$ and $\delta \geq \veps^2$, this yields
\begin{equation} \label{probaV+} \PP\big( V_+ \geq 24 k\gamma'  \delta \big) \leq e^{- \frac{\delta}{2\eta^2 } np}.\end{equation}
Let $E^c$ denote the event on the left-hand side and set $q_0 := 2\log(1/\PP(E^c))/\log(1/(3\gamma'\delta))$. 
Applying Lemma \ref{enhancedconcentration} to $f/\sqrt{8k}$ ($V_+ \le 8 k$ almost surely), it follows that there exists a constant $\gamma''>0$ depending on $\gamma$ and $k$ such that
\begin{equation}\label{concbetterf} \PP\big( f(\widehat X) -\EE f(\widehat X) >t \big) \leq e^{-  \frac{t^2}{2\gamma'' \delta }},\end{equation}
for any $2 \leq t^2 /(\gamma'' \delta)  \leq q_0$. From \eqref{probaV+} and $\gamma'\geq 1$, we have $q_0\geq  \delta  np/[\eta^2 \log(1/\delta)]$.
As in addition $\delta \leq 1$,  \eqref{concbetterf} holds in particular when $2\gamma'' \leq t^2 \leq \gamma'' \delta^2 np/[\eta ^2\log(1/\delta)]$.
Finally, we observe that as $\eta \leq e^{-1}$,
\[ \frac{\delta^2}{\eta^2 \log(1/\delta)} \geq  \frac{\log (1/\eta)}{\log\Big(\frac{1}{\eta \sqrt{\log(1/\eta)}}\Big) } \geq 1, \]
which ends the proof. 
\end{proof}

The second step towards proving Proposition \ref{strongconc} is the following result which says that the deviations of the quadratic form induced by $X$ over vectors whose support is of size much smaller than $np$ are negligible at the exponential scale $np$. We will use this result repetitively in the proof of the upper bound of Theorem \ref{theo-main-weight} to identify the structure of the top eigenvector. 
\begin{lemma}\label{locvect} Let $X = G\circ \Xi/\sqrt{np}$,  $G$ and $\Xi$ are independent symmetric random matrices where $(G_{i,j})_{i\leq j}$ are independent centered and bounded by $1$ and $\Xi$ has i.i.d.~$\dBer(p)$ entries on and above the diagonal. There exist numerical constants $p_0, C>0$ such that if $p\leq p_0$ then for any $\veps\in(0,1)$ and $t \geq  \veps  \sqrt{C \log(1/(\veps^2 p))/\log(1/p)}$,
\begin{equation}\label{eq:locvect}
\PP\Big( \sup_{v,w \in \mathcal{S}_\veps} \langle v, X w\rangle \geq t \Big) \leq e^{- \frac{t^2}{C} np \log(1/p)}, \, \text{ where } \mathcal{S}_{\veps} := \{u \in \mathbb{B}^n: |\mathrm{supp}(u)| \le \veps^2 np\}.
\end{equation}
%is the set of all vectors of $\mathbb{B}^{n}$ with support of size at most $\veps np$.
\end{lemma}
\begin{proof}
Fix $v,w\in\mathbb{B}^n$. By Chernoff's inequality, we have for any $\theta \geq 0$, 
\begin{equation} \label{chernoffineq}\PP\big( \langle v, (G \circ \Xi ) w \rangle \geq t \sqrt{np} \big) \leq e^{-\theta np t} \prod_{i<j} \EE \Big( e^{2 \theta \sqrt{np} v_i w_j G_{i,j} \xi_{i,j}}\Big) \prod_{i=1}^n \EE \Big(e^{\theta \sqrt{np} v_i w_i G_{i,i} \xi_{i,i}}\Big).\end{equation}
We claim that there exists $p_0>0$ such that if $Z$ is a centered random variable bounded by $1$ and $\zeta$ an independent Bernoulli random variable with parameter $p\leq p_0$, then for any $s \in \RR$,
\begin{equation}\label{subgaussdiluted} \log \EE \big(e^{s Z \zeta}\big) \leq \frac{s^2}{\log(1/p)}.\end{equation}
Indeed, by Lemma \ref{lem:Lam-bd} there exists $p_0>0$ such that for any $p\leq p_0$  and for any $\theta\in \RR$,  $\Lambda_p(\theta) \leq \theta^2 /(2 \log(1/p))$. Let $\EE_\zeta$ denote the expectation with respect to $\zeta$. Then, 
\[ \EE_\zeta\big( e^{s Z\zeta}\big) \leq e^{\frac{s^2 Z^2}{2\log(1/p)} + ps Z}\leq e^{\frac{s^2}{2\log(1/p)}+ps Z}.\] 
%Integrating the above inequality with respect to $Z$, it yields
%\[ \EE\big( e^{s Z\zeta}\big) \leq e^{\frac{s^2}{2\log(1/p)}}\EE\big(e^{s p Z}\big).\]
By Hoeffding's lemma (see \cite[Lemma 2.2]{BLM}) we have $\EE\big(e^{s p Z}\big) \leq \exp(s^2 p^2/2)$. Therefore integrating the above inequality with respect to $Z$ and using that $p^2 \leq 1/\log(1/p)$ for any $p\in(0,1)$, we get $\EE\big( e^{s Z\zeta}\big) \leq \exp(s^2/\log(1/p))$,
 which proves the claim \eqref{subgaussdiluted}. %Coming back to 
Now, by \eqref{chernoffineq}--\eqref{subgaussdiluted} and the fact that $v,w\in\mathbb{B}^n$, 
 \begin{equation} \label{chernoff}
\log \PP\big( \langle v, (G\circ \Xi) w \rangle \geq t \sqrt{np} \big)\leq {- np \sup_{\theta \ge 0}\Big\{\theta t - \frac{2\theta^2}{\log(1/p)})\Big\}} = -\frac{t^2}{8} np \log(1/p).
 \end{equation}
% Optimizing on $\theta$, we obtain:
%\begin{equation} \label{chernoff} \PP\big( \langle v, (G\circ \Xi) w \rangle \geq t\sqrt{np} \big) \leq e^{-\frac{t^2}{8} np \log(1/p)}.\end{equation}
Let $\mathcal{N}$ be a $1/4$-net of $\mathcal{S}_\eps$ for the $\ell^2$-norm and $m=\lfloor \eps^2 np\rfloor$. 
%For $p$ small enough such that $m\leq e^{-1}n$, 
Note that there are at most $m(en/m)^m$ subsets $\mathcal{J} \subset[n]$ such that $|\mathcal{J}|\leq m$. Since $\mathcal{S}_\veps$ is the union of the corresponding unit balls $\mathbb{B}^\mathcal{J}$, it follows from \cite[Corollary 4.1.15]{AGM} that one can find such a net $\mathcal{N}$ with $\# \mathcal{N} \le m(e n/m)^m 12^m \leq (n/m)^{3m}$, for $p$ sufficiently small. Then, using a similar argument as in \cite[Lemma 4.4.1]{HDP}, we find
\begin{equation} \label{netineq}\sup_{v,w\in\mathcal{S}_\veps} \langle v,( G\circ \Xi )w \rangle\leq 2 \sup_{v,w \in \mathcal{N}} \langle v, (G\circ \Xi)w \rangle. \end{equation}
%Note that for $p$ small enough, $m(n/m)^m 12^m \leq (n/m)^{3m}$. 
Using \eqref{chernoff}, \eqref{netineq}, the bound on $\#\mathcal{N}$, and a union bound, we get
\[ \PP\Big( \sup_{v,w \in \mathcal{S}_\veps} \langle v, (G\circ \Xi)  w\rangle \geq t\sqrt{np} \Big) \leq \Big( \frac{1}{\veps^2 p}\Big)^{3\veps^2 np} e^{-\frac{t^2}{32} np \log(1/p)}\leq e^{-\frac{t^2}{64} np \log(1/p)},\]
for any $t^2 \geq 192 \veps^2 \log(1/(\veps^2 p))/\log(1/p)$. This completes the proof. 
\end{proof}

It remains to prove Proposition \ref{strongconc}. The proof will use Proposition \ref{improvedconc} and will involve careful choices of certain parameters.

\begin{proof}[Proof of Proposition \ref{strongconc}] %Let $M\geq 1$ and $\eps \geq \log(t^2 \log (1/p))/[t\sqrt{\log(1/p)]}$. \corAB{The precise choice of $M$ will be spelled out later.}
For any $w\in \RR^n$, write $w = \widehat{w}+\widecheck{w}$, where $\widecheck{w}_i= w_i\Car_{|w_i| >M/\sqrt{ np}}$ %and $\hat{w}_i = w_i \Car_{|w_i|\leq M/\sqrt{np}}$ 
for any $i\in[n]$, where $M \ge 1$ is some constant to be determined below. 
For %$\zeta(K) \in \mathcal{F}_\eps^k$, 
$\zeta(K) =\sum_{\ell=1}^k \theta_\ell w_\ell w_\ell^{\sf T} \in \mathcal{F}_\eps^k$, with $\vep$ as in \eqref{eq:eps-strongconc}, %with $|\theta_\ell|\leq 1$ and $w_\ell \in \mathcal{D}_\veps$ for any $\ell\in [k]$, 
we let %decompose $\zeta(K)$ into  $\zeta(K) = \zeta_1(K)+\zeta_2(K)$,
%where 
$\zeta_2(K) :=\sum_{\ell=1}^k \theta_\ell \widecheck{w}_\ell {\widecheck{w}_\ell}^{\sf T}$ and $\zeta_1(K):= \zeta(K)-\zeta_2(K)$. %= \sum_{\ell=1}^k \theta_\ell (\hat{w}_\ell {\check{w}_\ell}^{\sf T} + \check{w}_\ell {\hat{w}_\ell}^{\sf T}+ \hat{w}_\ell {\hat{w}_\ell}^{\sf T})$. 
Because of the assumption \eqref{convf} we can write 
\[ f(K') = \sup_{K \in \mathcal{H}_n} \{f(K) + \langle \zeta(K), K'-K\rangle\}, \ \forall K'\in \mathcal{H}_n.
\]
Define %$f_1$ and $f_2$ as
\[f_1(K') :=\sup_{K \in \mathcal{H}_n} \{f(K) - \langle \zeta_2(K),K\rangle +  \langle \zeta_1(K), K'-K\rangle\} \text{ and } f_2(K') :=\sup_{K \in \mathcal{H}_n} | \langle \zeta_2(K),K'\rangle |, 
\]
\text{for $K' \in \mathcal{H}_n$} (note that $f_1$ and $f_2$ depend implicitly on $\veps$ and $M$). 
%Clearly $|f -f_1|\leq f_2$. 

Now fix $t_0 = \sqrt{\pi C_0 k^2}$ for some large constant $C_0 < +\infty$, to be specified below. Assume $t_0/\sqrt{\log(1/p)}\leq t \leq 1/t_0$, where $p \le p_0$ and $p_0 \in (0,1)$ is some small parameter (depending on $t_0$) to be chosen later. Set $M=t_0/t$ for the rest of the proof.
Recalling the definition of $\mathcal{S}_\veps$ from \eqref{eq:locvect} we see that $\widecheck{w}_\ell \in \mathcal{S}_{1/M}$. Further $|\theta_\ell| \le 1$ for $\ell \in [k]$. Therefore, $f_2(K') \leq  k \sup_{u,v \in \mathcal{S}_{1/M}} \langle u,K'v\rangle$ for $K' \in \cH_n$.  By our choice of $M$, if $p_0$ is sufficiently small, we have $(1/M) \sqrt{C \log(M^2/p)/\log(1/p)} \leq t/4$, where $C$ is as in Lemma \ref{locvect}. Therefore, by Lemma \ref{locvect}, for any $s\geq t/4$,
\begin{equation} \label{ineqf22} \PP\big( f_2(X)> s\big ) \leq e^{-\frac{s^2 np}{C k^2} \log(1/p)}.\end{equation}
Integrating the above inequality yields $\EE f_2(X) \leq \frac{t}{4} +  \sqrt{ \pi C k^2/\log(1/p)}$.
By taking $C_0$ (and hence $t_0$) large enough we obtain $\EE f_2(X)\leq t/3$. Since $|f-f_1|\leq f_2$, this implies that
\begin{equation} \label{compexpect} 
\EE f(X) \leq \EE f_1(X) + t/3.
\end{equation} 
%Moreover, as $\zeta_2(K) = \sum_{\ell=1}^k \widecheck{w}_\ell \widecheck{w}_\ell^{\sf T}$, with $\widecheck{w}_\ell$ such that $|\mathrm{supp}(\widecheck{w}_\ell)| \leq  np/M^2$ for any $\ell\in[k]$, we can write
%\[ \forall K'\in \mathcal{H}_n, \ f_2(K') \leq  k \sup_{u,v \in \mathcal{S}_{1/M^2}} \langle u,K'v\rangle,\]
%%where $\mathcal{S}_{1/M^2} = \{ u\in\mathbb{B} :  |\mathrm{supp}(u)|\leq np/M^2\}$. 
%On the one hand, from Lemma \ref{locvect}, we know that there exists a numerical constant $c>0$ such that for any $s\geq (1/M) \sqrt{c\log(M^2/ p)/\log(1/p)}$,
%\begin{equation} \label{ineqf2} \PP\big( f_2(X)> s\big ) \leq e^{-\frac{s^2 np}{ck^2} \log(1/p)}.\end{equation}
On the other hand, note that by construction $\zeta_1(K) \in \mathcal{G}_{\eps,M}^k$ for any $K\in\mathcal{H}_n$, where $\mathcal{G}_{\eps,M}^k$ is defined in \eqref{defG}. 
%
%By Proposition \ref{improvedconc}, we know that if for some $\gamma \in (0,1)$, $\log(M^2/\eps^2)\leq \gamma \log(1/p)$ and $np \log(1/p) /M^2 \geq 4 \log n$, then there exists $\gamma_0>0$ depending on $k$ and $\gamma$ such that \[ \frac{\gamma_0}{np}  \leq t^2 \leq  \frac{1}{\gamma_0} \text{ and }  \eta : = \eps M/\sqrt{\log(1/p)} <1,\] 
%implies
%\begin{equation} \label{ineqf1} \PP\big( f_1(X) - \EE f_1(X) >t/3\big) \leq \exp\Big(-\frac{t^2np}{\gamma_0 \max\big(\eta  \log ( \frac{1}{\eta}\big),\eps^2 \big) }\Big).\end{equation}
%Let $t_0/\sqrt{\log(1/p)}\leq t \leq 1/t_0$ and $p\leq p_0$, with $t_0>1$ and $p_0>0$ to be chosen later. For $p_0$ small enough and $t_0$ large enough, if we take $M = t_0/t$, then  $(1/M) \sqrt{c \log(M^2/p)/\log(1/p)} \leq t/4$. With this choice, we get by \eqref{ineqf2}, that for any $s\geq t/4$,
%\begin{equation} \label{ineqf22} \PP\big( f_2(X)> s\big ) \leq e^{-\frac{s^2 np}{ck^2} \log(1/p)}.\end{equation}
%Integrating the above inequality yields $\EE f_2(X) \leq \frac{t}{4} +  \sqrt{ \pi c k^2/\log(1/p)}$.
%By taking $t_0$ large enough depending on $k$, we obtain $\EE f_2(X)\leq t/3$. Since $|f-f_1|\leq f_2$, this implies that
%\begin{equation} \label{compexpect} \EE f(X) \leq \EE f_1(X) + t/3.\end{equation} 
Set $\delta = \eps /\sqrt{\log(1/p)}$. %Recall that $\eps \geq \log(t^2 \log (1/p))/[t\sqrt{\log(1/p)]}$. 
Assuming that $np \geq 4 \log n$, $t_0/\sqrt{\log(1/p)}\le t \le 1/t_0$, and $\log(t^2 \log(1/p))/[t\sqrt{\log(1/p)}]\leq \veps$, one can check that provided $t_0$ is sufficiently large, and $p_0$ small enough depending on $t_0$, all the assumptions on parameters $M, \vep$, and $\eta$ needed to apply Proposition \ref{improvedconc} are satisfied. Therefore, noting that $x \mapsto x \log(1/x)$ is increasing if $x \in (0,e^{-1}]$, we deduce from Proposition \ref{improvedconc} that there exists some $\wt \gamma_0 >0$, depending only on $k$, such that if $\wt\gamma_0/np  \leq t^2 \leq 1/{\wt \gamma_0}$,
%
%
% we deduce that with the choice of $M =t_0/t$ made and at the price of taking $p_0$ smaller, depending on $t_0$, we have $\log(M^2/\eps^2) \leq (1/2) \log(1/p)$ and $np \log(1/p)/M^2 \geq 4 \log n$. Using the fact that $t_0\geq 1$, we deduce by \eqref{ineqf1} that there exists $\gamma_0>0$ depending on $k$ and $t_0$ such that if
%\begin{equation}\label{ranget} \frac{\corAB{\wt\gamma_0}}{np}  \leq t^2 \leq  \frac{1}{\corAB{\wt \gamma_0}}, \notag %\text{ and }  \eps < 1, \notag
%\end{equation}
then
%\[\PP\big( f_1(X) - \EE f_1(X) >t/3\big) \leq \exp\Big(-\frac{t^2np}{\gamma_0 \max\big((\delta /t) \log ( \frac{t}{\delta}\big),\eps^2 \big) }\Big),\]
%where $\delta = \eps /\sqrt{\log(1/p)}$. One can check that for $p_0$ small enough, depending on $t_0$, the lower bound $\veps \geq \log(t^2 \log (1/p))/[t\sqrt{\log(1/p)]}$ implies that $\veps^2\geq (\delta/t) \log(t/\delta)$. This yields for the range of $t$ as in \eqref{ranget}, 
\begin{equation} \label{ineqf11} \PP\big( f_1(X) - \EE f_1(X) >t/3\big) \le \exp\Big(-\frac{t^2np}{\gamma_0 \max\big((\delta /t) \log ( \frac{t}{\delta}\big),\eps^2 \big) }\Big) \leq \exp\Big(-\frac{t^2np}{\wt \gamma_0 \eps^2}\Big),\end{equation}
where in the last step we used that $\veps^2\geq (\delta/t) \log(t/\delta)$, which is a consequence of the upper bound on $\veps$ given in \eqref{eq:eps-strongconc}. Since $np \ge \log(1/p)$, this implies that \eqref{ineqf11} continues to hold for all $t$ satisfying \eqref{eq:eps-strongconc} with $\gamma_0:=\wt \gamma_0 \vee t_0^2$. Thus
%
%Note that as we assumed $n\geq2$ and $np \geq \log n$, we have $np \geq \log(1/p)$. Thus, 
combining \eqref{ineqf22}, applied for $s=t/3$, \eqref{compexpect}, and \eqref{ineqf11}, and using a union bound, we deduce that, for any $t$ and $\eps$ as in \eqref{eq:eps-strongconc},  %such that 
%\[ \frac{\gamma_0\vee t_0^2}{\log(1/p)} \leq t^2 \leq  \frac{1}{\gamma_0 \vee t_0^2} \text{ and } \frac{\log(t^2 \log (1/p))}{t\sqrt{\log(1/p)}} \leq \veps <1,\]
%then
\[\PP\big( f(X) - \EE f(X) >t \big) \leq 2\exp\Big(-\frac{t^2np}{\gamma_0'}  [\veps^{-2} \wedge \log(1/p)]\Big),\]
where $\gamma_0'$ is a positive constant depending on $k$. Since $\veps \geq t_0^{-1} \log(t_0^2)/\sqrt{\log(1/p)}$ (use $t_0/\sqrt{\log(1/p)} \le t \le t_0^{-1}$), this ends the proof. 
\end{proof}

\section{Large deviation upper bound for matrices with bounded entries}\label{sec:bdd-entry}

The goal of this section is to prove the large deviation upper bound of Theorem \ref{theo-main-weight} under the additional assumption that the entries of $G$ are bounded. Throughout this section we fix $R>0$ such that for any $i,j \in [n]$, $|G_{i,j}|\leq R$. 

%\begin{Pro}\label{upperbound:boundedcase} Assume that the entries of $G$ are bounded. For any $\lambda>2$, 
%\[ \limsup_{n\to+\infty} \frac{1}{np} \log \PP\big( \lambda_X \geq \lambda  \big) \leq - I(\lambda),\]
%where $I(\lambda) = h_L(\lambda/m(\lambda))$. 
%Further, for any $t>0$, 
%\[  \lim_{n\to +\infty} \PP\big(\max_{i\in [n]}\|X_i\|^2 \geq \lambda/m(\lambda)-t \mid  \lambda_X\geq \lambda \big) = 1.\]  
%\end{Pro}

%In the following, we fix $R>0$ such that for any $i,j \in [n]$, $|G_{i,j}|\leq R$. 

\subsection{Block decomposition}
Let ${\bm u}$ be an eigenvector of $X$ associated to $\lambda_X$ with unit norm and denote by $k_{\text{top}}$ its multiplicity.  For any $\veps>0$ we want to find a random subset $J$ with size at most $O(\veps^{-2})$ with overwhelming probability, such that 
\begin{equation} \label{condJ}  J \supset \{i\in[n] : |{\bm u}_i|\geq \eps\} \text{ and }   \lambda_{X^{(J)}}<\lambda_X,\end{equation}
where we remind the reader that $X^{(J)}$ is the submatrix of $X$ spanned by the rows and columns in $J^c$. To construct this set, we will use the following lemma, which says that one can decrease the top eigenvalue of a matrix by removing certain rows and columns. As we could not locate the proof of this result in the literature, we include it below.

\begin{lemma}\label{top-ei-submatrix} Let $Y$ be a real symmetric matrix of size $n$ and $k$ be the multiplicity of $\lambda_Y$. Then, there exists a  subset $\mathcal{J} \subset [n]$ of size $k$ such that $\lambda_{Y^{(\mathcal{J})} } < \lambda_Y$.
%where $Y^{(\mathcal{J})}$ is the submatrix of $Y$ spanned by the rows and columns in $\mathcal{J}^c$.
\end{lemma}
\begin{proof}
Let $w_1,\ldots,w_k$ be a collection of linearly independent eigenvectors associated to $\lambda_Y$. It follows that there exists $\mathcal{J} \subset [n]$ with $|\mathcal{J}|=k$ such that the matrix $(w_j(i))_{i \in \mathcal{J}, j \in [k]}$ is invertible, where $w_j(i)$ denotes the $i$-th entry of $w_j$. We claim that $\lambda_{Y^{(\mathcal{J})}}<\lambda_Y$. Note that it always holds that $\lambda_{Y^{(\mathcal{J})}}\leq \lambda_Y$. Arguing by contradiction, assume that the equality holds. This implies that there exists an eigenvector $w$ of $Y$ supported on $\mathcal{J}^c$.
This follows from the fact if $\widetilde{w}\in\RR^{\mathcal{J}^c}$ is a unit vector such that $\langle \widetilde{w},Y^{(\mathcal{J})}\widetilde{w} \rangle = \lambda_{Y}$ then $\langle \widetilde{w}, Y^{(\mathcal{J})}\widetilde{w}\rangle = \langle w, Yw\rangle = \sup_{\|v\|^2 = 1} \langle v, Yv\rangle$,
where $w$ is the extension of $\widetilde{w}$ to $[n]$ by setting to zero its entries in $\mathcal{J}$. This equality implies that $w$, as an optimiser of the quadratic form defined by $Y$ on the sphere, should satisfy the critical point equation, and therefore should be an eigenvector of $Y$.  
Now, we show that $w,w_1,\ldots,w_k$ are linearly independent. If $\theta_1,\ldots,\theta_k$ are scalars such that $w = \sum_{j=1}^k \theta_j w_j$, then as $w$ is supported on $\mathcal{J}^c$, we have in particular for any  $i \in \mathcal{J}$,  $0 = \sum_{j=1}^k \theta_j w_j(i)$. 
Since we chose $\mathcal{J}$ such that $(w_j(i))_{i\in\mathcal{J},j\in[k]}$ is invertible, these equalities  imply that $\theta_1=\cdots = \theta_k= 0$. As $k$ is the multiplicity of $\lambda_Y$ we reach a contradiction.
\end{proof}

Define now $J\subset [n]$ as the subset 
\begin{equation} \label{choiceJ}  J  = \{i\in[n] : |{\bm u}_i|\geq \eps\}\cup J',\end{equation}
where $J'$ is the subset obtained from Lemma \ref{top-ei-submatrix} such that $\# J' = k_{\text{top}}$. Decompose the top eigenvector and the matrix $X$ as 
\begin{equation} \label{decompo}  {\bm u} = \begin{pmatrix} {\bm v}^\veps \\ {\bm w}^\veps \end{pmatrix}, \quad X = \begin{pmatrix}
X_{J} &  \widecheck{X}_{ J}^{\sf T} \\
\widecheck{X}_{J}  &  X^{(J)}
\end{pmatrix},\end{equation}
where ${\bm v}^{\veps} \in \RR^J$, ${\bm w}^\veps \in \RR^{J^c}$. Note that by construction ${\bm w}^\veps \in \mathcal{D}_\veps$, and therefore by Chebychev's inequality %, we know that 
\begin{equation} \label{boundsizeJ}k_{\text{top}} \leq \#J \leq  \veps^{-2}+ k_{\text{top}}.\end{equation}
Moreover, by Lemma \ref{top-ei-submatrix}, $\lambda_{X^{(J)}}<\lambda_X
$ almost surely.

%,  $X_J$ and $X^{(J)}$ are respectively the $J\times J$ and $J^c\times J^c$ minors of $X$, and $\widecheck{X}_J$ the submatrix spanned by the rows of $X$ in $J^c$ and the columns in $J$. 
\subsection{Multiplicity of the top eigenvalue}
We will now show that with high probability, $J$ is of size only of order $\veps^{-2}$. To prove this, we will  rely on the fact that on a large deviation event of the top eigenvalue, its multiplicity can be at most of order $1$ with high probability. 
More precisely, we will show the following result.
\begin{Lem}\label{boundmulti}Let $\delta >0$. 
\begin{equation} \label{claim1-boundmulti} \lim_{\eps \to 0} \limsup_{n\to +\infty} \frac{1}{np} \log \PP\big( k_{\text{top}} \geq \eps^{-2}, \lambda_X \geq 2+\delta \big) = -\infty.\end{equation}
As a consequence, for any $\delta>0$, and with $J$ as in \eqref{choiceJ},
\begin{equation} \label{claim2-boundmulti} \lim_{\eps \to 0} \limsup_{n\to +\infty} \frac{1}{np} \log \PP\big( \#J \geq 2\eps^{-2}, \lambda_X \geq 2+\delta \big) = -\infty.
\end{equation}
\end{Lem}

To prove Lemma \ref{boundmulti}, we will show that with high probability, there are at most $O(1)$ eigenvalues greater than $2$. This fact is a consequence of the following concentration inequality. Denote for any interval $I \subset \RR$ by $\mathcal{N}(I)$ the number of eigenvalues of $X$ in $I$.
\begin{Lem}\label{conc-multi}
Let $\delta >0$. There exists $n_{\delta,R}$ depending on $\delta$ and $R$ such that for any $n\geq n_{\delta,R}$ and $k\geq 1$, 
\[ \PP\Big( \mathcal{N}\big([2+\delta, +\infty)\big) \geq k\Big) \leq e^{-c \frac{k \delta^2 np}{R^2}},\]
where $c>0$ is a numerical constant.
\end{Lem}
\begin{proof}
Define $f(x) := \frac{1}{\delta}(x-2)_+$ for any $x\in\RR$. As $f(x)\geq 1$ for $x\geq 2+\delta$, we have
\[ \PP\big( \mathcal{N}\big([2+\delta, +\infty)) \geq k\big) \leq \PP\Big( \sum_{i=1}^k f(\lambda_i) \geq k\Big),\]
where $\lambda_1\geq \cdots\geq \lambda_n$ are the eigenvalues of $X$ in a non increasing order. Note that as $f$ is non decreasing
\[ \EE\Big(\sum_{i=1}^k f(\lambda_i)\Big) \leq k \EE f(\lambda_1) \leq k\Big( \frac{1}{4} + \frac{1}{\delta} \EE \Big(( \lambda_1-2)_+ \Car_{\lambda_1 \geq 2+ \delta/4} \Big) \Big)\leq   k\Big( \frac{1}{4} + \frac{
4}{\delta^2} \EE[( \lambda_1-2)_+^2]\Big) \le k/2,\]
for all large $n$, where the last inequality follows from the expectation bound $\EE \lambda_1 \le 2(1+o(1))$ (see \cite[Theorem 2.7]{BeBoKn}), and the variance bound $\Var(\lambda_1)=o(1)$ where the latter is a consequence of \eqref{eq:subG-talagrand}, applied for $f(X)=\lambda_1(X)$, and an integration by parts.
%
%Using \cite[Theorem 2.7]{BeBoKn}  and the fact that $np \gg \log n$, we have $\EE \lambda_1 \leq 2(1+o(1))$. Moreover, integrating the  concentration inequality \eqref{conc-sp-rad} it follows that   $\EE (\lambda_1-\EE\lambda_1)^2_+ =O(R^2/np)$. Thus, for $n$ large enough $\EE\big(\sum_{i=1}^k f(\lambda_i)\big) \leq k/2$,
%
This allows us to write 
\[ \PP\big( \mathcal{N}\big([2+\delta, +\infty)) \geq k\big) \leq \PP\Big( \sum_{i=1}^k \big(f(\lambda_i) -\EE f(\lambda_i) \big)\geq k/2\Big).\]
Finally using \cite[Proposition 3.2]{NLAu} we get the claim.
\end{proof}

The proof of Lemma \ref{boundmulti} is now immediate.
\begin{proof}[Proof of Lemma \ref{boundmulti}]
For any $\delta>0$, we have  $\{\lambda_X\geq 2+\delta\} \subset \{\mathcal{N}([2+\delta,+\infty)) \geq k_{\text{top}}\}$. Therefore
\[ \PP\big( k_{\text{top}} \geq \eps^{-2}, \lambda_X \geq 2+\delta \big) \leq \PP\big( \mathcal{N}\big([2+\delta, +\infty)) \geq \veps^{-2}\big).\]
Using Lemma \ref{conc-multi}, we get the first claim \eqref{claim1-boundmulti}. The second claim follows from \eqref{boundsizeJ}.
\end{proof}

\subsection{Large deviations of the maximal degree}\label{dev-deg-max}
In this section we show that for any random subset $\bar J \subset [n]$ of bounded size,  the large deviation of the operator norm of $\widecheck{X}_{\bar J}$ is dominated by the deviation of the maximal degree among the vertices in $\bar J$. 
\begin{Lem}\label{normXcheck:deg}For any $n\in \NN$, let $\bar J \subset [n]$ be a random subset of bounded size.
For any $t>0$, 
\[ \lim_{n\to +\infty} \frac{1}{np} \log  \PP \big( \|\widecheck{X}_{\bar J}\|^2 \geq \max_{i \in \bar J} \|X_i\|^2+ t\big) = -\infty.\]
\end{Lem}
Assume for the moment that Lemma \ref{normXcheck:deg} holds. It tells us in particular that the large deviation upper tail of $\|\widecheck{X}_{\bar J}\|^2$ can be read off from the one of the maximal degree. Now, since $np\gg \log n$, the large deviations of the maximal degree are given by the ones of the degree of an individual vertex, which has for rate function $h_L$ according to Lemma \ref{LDP-deg}.
Thus, Lemmas \ref{normXcheck:deg} and \ref{LDP-deg} together with the lower semicontinuity of $h_L$ gives the following large deviation result. 

\begin{Pro}\label{GD:rectanglebounded}
For any $\bar J\subset [n]$ a random subset of bounded size and  $t>1$, 
\[ \limsup_{n\to +\infty} \frac{1}{np} \log  \PP \big( \|\widecheck{X}_{\bar J}\|^2 \geq  t \big) \leq -h_L(t).\]
\end{Pro}

It remains to prove Lemma \ref{normXcheck:deg}. To this end, we first show that for a given vector $v$ with bounded support, $\|Xv\|^2$ is dominated by the mean of the degrees of the vertices in the support of $v$ weighted by the square of the entries of $v$. This amounts to say that in the underlying random graph $\Xi$, for any finite set of vertices,  their neighbourhoods are essentially disjoint at the exponential scale $np$. 
\begin{Lem}\label{largedegree} For any $n\in \NN$, let  $\mathcal{J} \subset [n]$ be a subset of bounded size. For any $v\in \mathbb{B}^{\mathcal{J}}$ and $t>0$, 
\[\lim_{n\to +\infty} \frac{1}{np} \log  \PP \big( \|X v\|^2 \geq \sum_{i \in \mathcal{J}} \|X_i\|^2 v_i^2+ t\big) = -\infty.\]
\end{Lem}

\begin{proof}Observe 
\begin{align*}
\|Xv\|^2 =\big \|\sum_{i\in \mathcal{J}} v_i X_i \big\|^2& = \sum_{i\in \mathcal{J}} v_i^2 \|X_i\|^2 + \sum_{i\neq j\in \mathcal{J}} v_i v_j  \langle X_i,X_j\rangle \leq \sum_{i\in \mathcal{J}} v_i^2 \|X_i\|^2  + \#\mathcal{I} \max_{i\neq j}|\langle X_i,X_j\rangle|.
\end{align*}
Since $\#\mathcal{J} $ is bounded, it suffices to show that for any $t>0$,  
\begin{equation} \label{claimdeg} \lim_{n\to +\infty} \frac{1}{np} \log  \PP \big(  \max_{i\neq j \in [n]} |\langle X_i,X_j \rangle| >t\big) = -\infty.\end{equation}
Let $1\leq i\neq j \leq n$. As $|G_{k,\ell}|\leq R$ for any $k,\ell$, we have $|\langle X_i,X_j\rangle| \leq R^2 \sum_{k =1}^n \xi_{i,k} \xi_{j,k}/np$. As $(\xi_{i,k}\xi_{j,k})_{k \notin \{i,j\}}$ are i.i.d. Bernoulli of parameter $p^2$, we obtain by Chernoff's inequality and Lemma \ref{lem:Lam-bd}(iii), that for any $t>0$ and for $n$ large enough,
\[ \PP\big( \sum_{k \notin \{i,j\}} \xi_{i,k} \xi_{j,k}  > t np \big) \leq e^{-ct n p \log(t/p) },\]
where $c$ is a positive numerical constant.
Using a union bound and the fact that $np \gg \log n$, this yields the claim \eqref{claimdeg}.
\end{proof}

We are now ready to prove Lemma \ref{normXcheck:deg}.
\begin{proof}[Proof of Lemma \ref{normXcheck:deg}] Let $\veps>0$ such that $\#\bar J \leq \veps^{-1}$ almost surely. 
Since $np \gg \log n$ and 
\[ \log \#\big\{\mathcal{J}\subset [n] : \# \mathcal{J}\leq \veps^{-1} \big\} = O(\log n),\]
it suffices to show the statement for a deterministic subset $\mathcal{J} \subset [n]$ such that $\# \mathcal{J}\leq \veps^{-1}$. By Lemma \ref{largedegree}, we know that for any $v\in \mathbb{S}^{\mathcal{J}}$ and $t>0$,
\begin{equation} \label{normindiv}  \lim_{n\to +\infty} \frac{1}{np} \log  \PP \big( \|\widecheck{X}_\mathcal{J} v\|^2 \geq \max_{i \in \mathcal{J} } \|X_i\|^2+t\big) = -\infty.\end{equation}
It remains to perform a net argument. Let $\eta \in (0,1)$ and  $\mathcal{N}_\eta$ a $\eta$-net for the $\ell^2$ norm of $\mathbb{S}^{\mathcal{I}}$. By  \cite[Corollary 4.1.15]{AGM}, there exists such a net of size at most $(3/\eta)^{1/\veps}$. Besides by  \cite[Lemma 4.4.1]{HDP}, 
\[ \|\widecheck{X}_\mathcal{J}\| \leq  \frac{1}{1-\eta} \sup_{v\in\mathcal{N}_\eta} \|\widecheck{X}_\mathcal{J} v\|. \]
Using \eqref{normindiv}, this gives for any $t>0$ and $\eta \in (0,1)$, 
\[  \lim_{n\to +\infty} \frac{1}{np} \log  \PP \big( \|\widecheck{X}_\mathcal{I}\|^2 \geq (1-\eta)^{-2}(\max_{i \in \mathcal{I} } \|X_i\|^2+t)\big) = -\infty.\]
Using the exponential tightness of the maximum degree guaranteed by Lemma \ref{LDP-deg} and a union bound, this ends the proof. 
\end{proof}

\subsection{Localisation of the top eigenvector} We show that on the upper tail large deviation event of the top eigenvalue, the top eigenvector must localise in the sense that the $\ell^2$ weight of the entries of order at least $\veps$ is lower bounded by a constant. More precisely, we prove that ${\bm v}^\veps$ defined in \eqref{decompo} has a non-trivial $\ell^2$ norm at the large deviation scale.
\begin{Lem}\label{mass-loc}For any $\lambda>2$, there exists $\eta_\lambda>0$ increasing in $\lambda$ such that 
\[\limsup_{\veps \to 0} \limsup_{n\to +\infty} \frac{1}{np} \log \PP\big( \|{\bm v}^\veps\| \leq \eta_\lambda, \ \lambda_X >\lambda  \big) < -I(\lambda).\]
\end{Lem}

To  prove Lemma \ref{mass-loc} we first show that a large value of the inner product $\langle v,Xw\rangle$ cannot be achieved by using delocalised unit vectors. This result uses in a {\em crucial way} the fact that the entries are bounded. Recall the definition of $\mathcal{D}_\eps$ from \eqref{flatmatrices}.
\begin{Lem}\label{conc-sp-rad-deloc}
For any $t>0$,
\[\lim_{\eps \to 0}  \limsup_{n\to +\infty} \frac{1}{np} \log  \PP\Big( \sup_{v,w \in \mathcal{D}_\veps} \langle v, Xw\rangle >2+t \Big) = -\infty.\]
%where $\mathcal{D}_\veps = \{v \in\mathbb{B}^n : \|v\|_\infty \leq \veps\}$.
\end{Lem}

\begin{proof}Define for any $K \in \mathcal{H}_n$,  $f(K) := \sup_{v,w\in\mathcal{D}_\veps} \langle v,Kw\rangle$. 
Clearly $f$ is a convex function. Note that for any $K\in\mathcal{H}_n$, $\langle v,K w\rangle = \langle K, V \rangle$ where $V = \frac{1}{2} (vw^{\sf T} + wv^{\sf T})$. Since $4V= (v+w)(v+w)^{\sf T} - (v-w)(v-w)^{\sf T}$, we deduce that $V$ belongs to the set of flat matrices of rank at most $2$, $\mathcal{F}_{\veps}^2$ defined in \eqref{flatmatrices}. Since  $\mathcal{D}_{\veps}$ is compact, the supremum defining $f(K)$ is achieved for any $K$, and it follows that $\partial f(K) \cap \mathcal{F}_{\veps}^2 \neq \emptyset$. It is then not difficult to show that there exists a measurable choice of sub differential that belongs to $\cF_\veps^2$. That is, there exists $\zeta : \mathcal{H}_n\mapsto \mathcal{H}_n$ measurable such that  $\zeta(K) \in\partial f(K) \cap \mathcal{F}_\veps^2$ for any $K\in \mathcal{D}_\veps$ (see Lemma \ref{selection} and its proof). As a consequence, for all $K,K' \in \mathcal{H}_n$,
\[  f(K)-f(K')\leq \langle \zeta(K),K-K'\rangle.\] 
It follows from Lemma \ref{strongconc} that for any $t>0$, 
\[ \lim_{\veps \to 0} \limsup_{n\to +\infty}\frac{1}{np} \log  \PP\big( f(X)-\EE f(X) >t \big) =-\infty.\]
Besides, as $np \gg \log n$, we know from  \cite[Theorem 2.7]{BeBoKn} that $\EE \sup_{v,w \in \mathcal{D}_\veps} \langle v, Xw\rangle \leq \EE \|X\| =2(1+o(1))$.
This ends the proof of the claim. 
\end{proof}

We are now ready to give a proof of Lemma \ref{mass-loc}.
\begin{proof}[Proof of Lemma \ref{mass-loc}]

%By Lemma \ref{conc-sp-rad-deloc}, we deduce that on a large deviation event of the top eigenvalue, ${\bm u}$ has at least an entry of order at least $\eps$ with overwhelming probability. Further, we show now that, with high probability, the $\ell^2$-weight of these entries is bounded from below by a constant. 
Expanding $\lambda_X = \langle ({\bm v}^\veps+{\bm w}^\veps), X ({\bm v}^\veps+{\bm w}^\veps)\rangle$, on the event $\{\#J\leq 2\veps^{-2}\}$, we get 
\[ \lambda_X = 2\langle \widecheck{X}_J{\bm v}^\veps, {\bm w}^\veps\rangle + \langle {\bm w}^\veps, X^{(J)} {\bm w}^\veps\rangle+\langle {\bm v}^\veps, \widecheck{X}_J {\bm v}^\veps\rangle \le 2\|\widecheck{X}_J\|.\|{\bm v}^\veps\| + \sup_{w \in  \mathcal{D}_\eps} \langle w, Xw \rangle+o(1),\]
where the last inequality follows from the fact that the entries of $G$ are bounded. %we have on the event where $\#J\leq 2\veps^{-2}$, 
%\[ \lambda_X \leq 2\|\widecheck{X}_J\|.\|{\bm v}^\veps\| + \sup_{w \in  \mathcal{D}_\eps} \langle w, Xw \rangle+o(1).\]
Thus, for any $t,r>0$, and all large $n$, 
\[ \{\|{\bm v}^\veps\| \leq t /(4r),  \ \lambda_X >2+t,\ \#J\leq 2\veps^{-2} \} \subset \{ \|\widecheck{X}_J\| > r\}\cup  \{ \sup_{w \in  \mathcal{D}_\eps} \langle w, Xw \rangle \geq 2 +t/4\}.\]
By Lemma \ref{GD:rectanglebounded}, it follows that
\begin{equation} \label{choicenorm} \sup_{\eps>0} \limsup_{n\to +\infty} \frac{1}{np} \log \PP\big( \|\widecheck{X}_J\| > 2+t \big)  \leq -h_L((2+t)^2)< - h_L\Big(\frac{2+t}{m(2+t)}\Big)=I(2+t),\end{equation}
where we used that for any $\lambda>2$, $\lambda m(\lambda) >1$ and that $h_L$ is increasing on $[1,+\infty)$.
Thus, taking $r = t+2$, we get using Lemmas \ref{conc-sp-rad-deloc} and \ref{boundmulti},  and \eqref{choicenorm} that for any $t>0$, 
\[ \limsup_{\veps \to 0} \limsup_{n\to +\infty} \frac{1}{np} \log \PP\big(\|{\bm v}^{\veps}\|\leq \eta_t, \lambda_X>2+t\big) <-I(2+t),\]
where $\eta_t = t/(4(t+2))$ for any $t>0$. Finally, one can check that $\eta_t$ is increasing in $t \in \RR_+$. 
\end{proof}

\subsection{Concentration of the resolvent}\label{section:concresolv}
%From the eigenvalue-eigenvector equation and the decomposition of ${\bm u}$ and $X$ as in \eqref{decompo}, we have
%\begin{equation} \label{eq-top-ei} \begin{cases}
%&X_{ J}  {\bm v}^\veps + \widecheck{X}_{J}^{\sf T} {\bm w}^\veps =\lambda_X  {\bm v}^\veps,\\
%& \widecheck{X}_{J}{\bm v}^\veps  + X^{( J)}  {\bm w}^\veps = \lambda_X {\bm w}^\veps.
%\end{cases}
%\end{equation}
%By definition of $J$ (see \eqref{choiceJ}), $\lambda_{X^{(J)}} <\lambda_X$. From the second line of \eqref{eq-top-ei} we find that ${\bm w}^\veps = (\lambda_X -X^{( J)})^{-1} \widecheck{X}_{J}{\bm v}^\veps$. 
%Now taking the inner product with ${\bm w}^\veps$ in the first line gives the following: 
%\begin{equation} \label{eiev:eq} \begin{cases}
%&\langle  {\bm v}^\veps, X_{J} {\bm v}^\veps \rangle  + \langle \widecheck{X}_{J}  {\bm v}^\veps, (\lambda_X -X^{(J)})^{-1}\widecheck{X}_{J}  {\bm v}^\veps \rangle =\lambda_X \| {\bm v}^\veps\|^2,\\
%& {\bm w}^\veps = (\lambda_X -X^{(J)})^{-1} \widecheck{X}_{J}{\bm v}^\veps.
%\end{cases}
%\end{equation}
The goal of this section is to prove the following concentration result of the resolvent. 
\begin{Pro}\label{concresolvfinalbounded}For any $\lambda>2$, $t>0$, 
\[\lim_{\eps \to 0} \limsup_{n\to +\infty} \frac{1}{np} \log  \PP\Big( \langle \widecheck{X}_J {\bm v}^\veps, (\lambda_X -X^{(J)})^{-1}\widecheck{X}_J{\bm v}^\veps  \rangle -m(\lambda_X) \|\widecheck{X}_J\|^2 \cdot\|{\bm v}^\veps\|^2>t, \lambda_X\geq \lambda \Big)< - I(\lambda).\]
\end{Pro}
To prove Proposition \ref{concresolvfinalbounded}, we will take advantage of the fact that for a deterministic symmetric matrix $Y$, if there exists $\mathcal{J}\subset [n]$ such that $\lambda_{Y^{(\mathcal{J})}} <\lambda_Y$, then the top eigenvalue of $Y$ is the solution of an equation involving an optimization problem on $\mathbb{S}^{\mathcal{J}}$ described in the following lemma. 
\begin{Lem}\label{topeioptim} Let $Y \in \mathcal{H}_n$ and $\mathcal{J}\subset [n]$ such that $\lambda_{Y^{(\mathcal{J})}} <\lambda_Y$. Then, 
\begin{equation}\label{eq:topeioptim}
\lambda_Y = \sup_{v \in \mathbb{S}^\mathcal{J}} ( \langle \widecheck{Y}_\mathcal{J}v, (\lambda_Y -Y^{(\mathcal{J})})^{-1}\widecheck{Y}_\mathcal{J}v\rangle + \langle v, Y_{\mathcal{J}} v\rangle\big).
\end{equation}
Moreover, the supremum is achieved for any vector $v\in \mathbb{S}^\mathcal{J}$ which can be extended to a top eigenvector of $Y$. 
\end{Lem}
\begin{proof} Using the same computation as in \eqref{eq-top-ei-sec2-1} one can check that if $v\in \mathbb{S}^{\mathcal{J}}$ admits an extension as a top eigenvector of $Y$, then 
\[ \lambda_Y  =  \langle \widecheck{Y}_\mathcal{J}v, (\lambda_Y -Y^{(\mathcal{J})})^{-1}\widecheck{Y}_\mathcal{J}v\rangle + \langle v, Y_\mathcal{J} v\rangle.\]
 As $\lambda_{Y^{(\mathcal{J})}} <\lambda_Y$, any top eigenvector of $Y$ has a non-zero projection on $\RR^{\mathcal{J}}$. Therefore, there exists some vector $v\in\mathbb{S}^{\mathcal{J}}$ which can be extended to a top eigenvector, and consequently the supremum in \eqref{eq:topeioptim} is at least as large as  $\lambda_Y$. It remains to prove the other direction.
 
 By a compactness argument, the supremum is achieved and let $v_*$ be a maximiser. It should satisfy the critical point equation which reads $\widecheck{Y}_\mathcal{J}^{\sf T} (\lambda_Y -Y^{(\mathcal{J})})^{-1} \widecheck{Y}_\mathcal{J} v_*+Y_\mathcal{J}v_* = \theta v_*$ for some $\theta \in \RR$. Denoting $w_*=(\lambda_Y -Y^{(\mathcal{J})})^{-1} \widecheck{Y}_\mathcal{J} v_*$, we get the equations 
\[ \begin{cases} Y_\mathcal{J}v_* +  \widecheck{Y}_\mathcal{J}^{\sf T} w_*= \theta v_*\\
\widecheck{Y}_{\mathcal{J}}v_* + Y^{(\mathcal{J})} w_* = \lambda_Y w_*.\end{cases}\]
If $u_*^{\sf T} =(v_*^{\sf T},w_*^{\sf T})$, then we find $\langle u_*, Y u_*\rangle = \theta  + \lambda_Y\|w_*\|^2$. The inequality $\langle u_*, Y u_*\rangle \leq \lambda_Y \|u_*\|^2$, as  $\|v_*\|^2=1$, yields $\theta \leq \lambda_Y$. Since $\theta$ is the value of the supremum, this ends the proof.
\end{proof}

To prove Proposition \ref{concresolvfinalbounded} we need to set up a few more notation.
Fix  $\mathcal{J} \subset [n]$ a subset of bounded size, $\lambda>2$, and  a realisation of $\widecheck{X}_\mathcal{J}$ and $X_{\mathcal{J}}$. For $K \in \cH_{\cJ^c}^\lambda$ (recall the relevant notation from Section \ref{sec:notation}) we set 
%Let $\mathcal{H}_{\mathcal{J}^c}$ be the set of symmetric matrices of size $\mathcal{J}^c\times \mathcal{J}^c$ and define for any $K \in \mathcal{H}_{\mathcal{J}^c}$ such that $\lambda_K< \lambda$,
\[ \phi_{\lambda,v}(K):= \langle \widecheck{X}_\mathcal{J}v, (\lambda -K)^{-1}\widecheck{X}_\mathcal{J}v\rangle + \langle v, X_\mathcal{J} v\rangle,  \forall v\in\RR^{\mathcal{J}}, \text{ and }  \phi_\lambda(K) := \sup_{v \in \mathbb{S}^\mathcal{J}}  \phi_{\lambda,v}(K).\]
Observe that since $\lambda_{X^{(J)}} <\lambda_X$ by definition of $J$, Lemma \ref{topeioptim} yields that on the event where $J=\mathcal{J}$,  $\phi_{\lambda_X}(X^{(\mathcal{J})}) = \phi_{\lambda_X,\widetilde{{\bm v}}^\veps}(X^{(\mathcal{J})})$, where $\widetilde{{\bm v}}^\veps := {\bm v}^\veps/\|{\bm v}^\veps\|$. Therefore, to prove Proposition \ref{concresolvfinalbounded}, it suffices to show that the upper tail of $\phi_{\lambda}(X^{(\mathcal{J})})$ concentrates uniformly in $\lambda$ and $\mathcal{J}$ around its expectation.   
By \cite[Exercise I.I.15]{Bhatia}, we know that for any given $w \in \RR^{\mathcal{J}^c}$, $K\mapsto \langle w, (\lambda -K)^{-1} w\rangle$ is convex on $\cH^\lambda_{\cJ^c}$.  In order to show that $\phi_\lambda(X^{(\mathcal{J})})$ concentrates with a large deviation speed much larger than $np$, we need to extend it suitably on the whole space $\cH_{\cJ^c}$. To this end, 
define for any $r,\eps>0$ the set $E_\lambda^{\veps,r}$ of matrices $K$ such that there exists a ``flat'' subdifferential of $\phi_\lambda$ at $K$, that is
\begin{equation*} E_\lambda^{\veps,r} := \big\{K \in \mathcal{H}_{\mathcal{J}^c}^\lambda: \exists v \in \mathbb{S}^\mathcal{J}, \phi_\lambda(K) = \phi_{\lambda,v}(K), \nabla \phi_{\lambda,v}(K) \in r\mathcal{F}_\veps^1\big\},\end{equation*}
where $\mathcal{F}^1_{\veps}$ is defined in \eqref{flatmatrices}.
For any $K\in E_\lambda^{\veps,r}$, we set
\[
\Theta(K):= \{\nabla \phi_{\lambda,v}(K) \in  r\mathcal{F}_\veps^1: v\in\mathbb{S}^{\mathcal{J}} \text{ and } \phi_{\lambda,v}(K)=\phi_\lambda(K) \}. 
\]
%\corFA{One can observe that by Danskin's formula \cite[Theorem 10.22]{Clarke} that $\overline{\mathrm{co}}(\Theta (K)) = \partial \phi_\lambda(K)\cap r \mathcal{F}_\veps^1$, where $\overline{\mathrm{co}}(.)$ denotes the closed convex hull and $\partial \phi_\lambda$ the subdifferential of $\phi_\lambda$}. 
Now extend ${\phi_\lambda}_{|E_\lambda^{\veps,r}}$ to $\mathcal{H}_{\mathcal{J}^c}$ by setting 
\begin{equation} \label{defphitilde} 
\widetilde{\phi}_{\lambda}^\veps(K') := \sup_{K\in E_\lambda^{\veps,r}} \sup_{\zeta \in \Theta(K)}\big\{\phi_\lambda(K) + \langle \zeta, K'-K\rangle \big\}, \, \forall K' \in \mathcal{H}_{\mathcal{J}^c}.
\end{equation}
Note that $\widetilde{\phi}_\lambda^\veps$ extends ${\phi_\lambda}_{|E_\lambda^{\veps,r}}$ in such a way that $\widetilde{\phi}_\lambda^\veps$ is convex and $\widetilde{\phi}_\lambda^\veps \leq {\phi}_\lambda$ pointwise on the subset $\cH^\lambda_{\cJ^c}$.

To prove Proposition \ref{concresolvfinalbounded}, in a first step, we upper bound the expectation of $\widetilde{\phi}_{\lambda}^\veps(X^{(\mathcal{J})})$ given $\widecheck{X}_\mathcal{J}$ and $X_\mathcal{J}$. We need to define the following event. Fix $C \ge 1$ and $\cJ \subset [n]$, and set
\begin{equation}
\Omega_C:= \big\{\#\{(i,j) \in [n] \times \mathcal{J} : \xi_{i,j}=1\}\leq Cnp \text{ and }  \|\widecheck{X}_\mathcal{J}\|^2 \leq C\big\}. 
\end{equation}
\begin{Pro} \label{computation-expec}
Let $C \ge 1$, $\delta \in (0,1)$, and $\lambda > 2+2\delta$. Denote $\EE_{\mathcal{J}^c}$ to be the expectation with respect to $X^{(\mathcal{J})}$. Then
%Let $C\geq 1$. On the event where $\#\{(i,j) \in [n] \times \mathcal{J} : \xi_{i,j}=1\}\leq Cnp$, and  $\|\widecheck{X}_\mathcal{J}\|^2 \leq C$, 
 \[ \EE_{\mathcal{J}^c}\big[\widetilde{\phi}_{\lambda}^\veps(X^{(\mathcal{J})})\big] \leq m(\lambda)\|\widecheck{X}_\mathcal{J}\|^2+o(1), \qquad \text{a.s.~on  } \Omega_C.
 \]
% where $\EE_\mathcal{J}$ denotes the expectation with respect to $X^{(\mathcal{J})}$.
\end{Pro}

The proof of Proposition \ref{computation-expec} is postponed to Section \ref{sec:loc-law}. It uses an isotropic local law, concentration bounds, and a net argument.  We also need the following exponential tightness result.

\begin{Lem} \label{expotight}
Under the same setup of Theorem \ref{theo-main-weight}
 \[
 \lim_{M \to +\infty} \limsup_{n\to+\infty}\frac{1}{np} \log \PP( \| X\| > M) = -\infty.\]
 \end{Lem}

%The difference between Lemmas \ref{expotightnessA} and \ref{expotightnessB}, \corAQ{and Lemma \ref{expotight}} is that the first two identify the correct constant for the typical value of $\|X\|$ without providing an optimal speed of decay for its tail probability, while the last one provides  the correct speed of decay for the tail probability without pinning down the exact asymptotic value of $\EE\|X\|$. Both these aspects will be used in the proof of the large deviation lower bound. 

Lemma \ref{expotight} was essentially proved in \cite{BR17}. See \cite[Theorem 1.7]{BR17}. 
It was not stated in the form that we need. Nevertheless, one can check that the same proof yields Lemma \ref{expotight}. So we do not repeat it here. Let us point to the reader that Lemma \ref{expotight} does not need the entries of $G$ to be uniformly bounded, and therefore it will also be used later in Section \ref{sec:ld-ub-gen} when we consider the unbounded case. We are now ready to give a proof of Proposition \ref{concresolvfinalbounded}.

\begin{proof}[Proof of Proposition \ref{concresolvfinalbounded}]
By definition (see \eqref{defphitilde}) $\widetilde{\phi}_\lambda^\veps$ is convex on $\mathcal{H}_{\mathcal{J}^c}$ and coincides with $\phi_\lambda$ on $E_\lambda^{\veps,r}$. 
 Since $\mathcal{F}^1_{\veps}$ is compact and $\Theta(K)\subset r\mathcal{F}_\veps^1$ for any $K\in E_\lambda^{\veps,r}$, we deduce from \eqref{defphitilde} that $r\mathcal{F}^1_\veps \cap \partial \wt \phi_\lambda^\veps (K')\neq \emptyset$ for any $K'\in\mathcal{H}_{\mathcal{J}^c}$. 
Therefore there exists $\zeta : \mathcal{H}_{\mathcal{J}^c} \mapsto  \mathcal{H}_{\mathcal{J}^c}$ a measurable selection for the multivalued mapping $\partial \wt \phi_\lambda^\veps (\cdot) \cap r \mathcal{F}_\veps^1$, that is, $\zeta(K) \in \partial \wt \phi_\lambda^\veps(K)\cap r\mathcal{F}_\veps^1$ for any $K\in \mathcal{H}_{\mathcal{J}^c}$ (see Lemma \ref{selection}). Thus, 
\[ \widetilde{\phi}^\veps_\lambda(K)-\widetilde{\phi}^\veps_\lambda(K')\leq \langle \zeta(K), K-K'\rangle, \qquad \forall K, K' \in \mathcal{H}_{\mathcal{J}^c}, \ \]
Applying Proposition \ref{strongconc} to $\widetilde{\phi}_\lambda^\veps/r$, we get that for any sufficiently small $t>0$ and $n$ large enough
\begin{equation} \label{concflat} \PP_\mathcal{J}\big( \widetilde{\phi}^\veps_\lambda(X^{(\mathcal{J})}) -\EE_\mathcal{J}  \widetilde{\phi}^\veps_\lambda(X^{(\mathcal{J})}) >t \big) \leq 2 e^{-\frac{t^2 np}{\gamma_0 (r\veps)^2}},\end{equation}
 where $\gamma_0 >0$ is a numerical constant. %a positive constant and $\PP_\mathcal{J}$ and $\EE_\mathcal{J}$ denote the probability and expectation with respect to ${X}^{(\mathcal{J})}$.  
 Using Proposition \ref{computation-expec}, we obtain for any $C,t>0$, 
 \begin{equation} \label{concunifEc} \lim_{\veps \to 0} \limsup_{n\to+\infty} \frac{1}{np} \log \PP\big( \Omega_C\cap \big\{\widetilde{\phi}^\veps_\lambda(X^{(\mathcal{J})})> m(\lambda)\|\widecheck{X}_\mathcal{J}\|^2 + t\big\}\big) = -\infty,\end{equation}
 %where $\Omega_C$ is the event where $\#\{(i,j) \in [n] \times \mathcal{J} : \xi_{i,j}=1\}\leq Cnp$, and  $\|\widecheck{X}_\mathcal{J}\|^2 \leq C$.
Next we claim that
 \begin{equation} \label{claimEc} \lim_{C\to +\infty} \limsup_{n\to +\infty} \frac{1}{np} \log \PP( \Omega_C^c) = -\infty.\end{equation}
 Indeed, by Lemma \ref{expotight}, $\|\widecheck{X}_\mathcal{J}\|$ is exponentially tight. Moreover, using Binomial tail estimates and the fact that $\mathcal{J}$ is a set of bounded size it yields,
    \[ \lim_{C\to +\infty} \limsup_{n\to +\infty} \frac{1}{np} \log \PP\big(\#\{(i,j) \in [n] \times \mathcal{J} : \xi_{i,j}=1\}> Cnp\big) = -\infty.\]
 Thus the claim \eqref{claimEc} follows. Combining \eqref{claimEc} and \eqref{concunifEc}, we get for any $t>0$ and $\mathcal{J}\subset [n]$ of bounded size, 
  \begin{equation}\label{concpoint}  \lim_{\veps \to 0} \limsup_{n\to+\infty} \frac{1}{np} \log \PP\big( \widetilde{\phi}^\veps_\lambda(X^{(\mathcal{J})})> m(\lambda)\|\widecheck{X}_\mathcal{J}\|^2 + t\big) = -\infty.\end{equation}
 It remains now to perform a union bound on $\mathcal{J}$ and a net argument on $\lambda$. 
We first perform a net argument over $\lambda$. For $\delta>0$ 
%and define $\Gamma_{\delta,r}^{\lambda,\mathcal{J}}=\{\lambda  \leq \lambda_X \leq \lambda + \delta, \|{\bm v}^\veps\|^2 >1/r, J=\mathcal{J}\}$. 
we show that 
\begin{equation}\label{eq:Gamma-delta-r} 
\Gamma_{\delta,r}^{\lambda,\mathcal{J}} :=\{\lambda  \leq \lambda_X \leq \lambda + \delta, \|{\bm v}^\veps\|^2 >1/r, J=\mathcal{J}\} \subset \{\phi_{\lambda_X}(X^{(\mathcal{J})}) \leq \widetilde{\phi}_{\lambda+\delta}^\veps(X^{(\mathcal{J})}) + \delta r\}.
\end{equation}
On the event $J=\mathcal{J}$, we have on the one hand by Lemma \ref{topeioptim}  that $\phi_{\lambda_X}(X^{(\mathcal{J})}) = \phi_{\lambda_X,\widetilde{{\bm v}}^{\veps} }(X^{(\mathcal{J})})$, where $\widetilde{{\bm v}}^{\veps} = {\bm v}^\veps/\|{\bm v}^\veps\|$, and on the other hand by \eqref{eq-top-ei-sec2-1}, 
\[\nabla \phi_{\lambda_X,\widetilde{{\bm v}}^{\veps} }(X^{(\mathcal{J})}) = \frac{1}{\|{\bm v}^\veps\|^2} {\bm w}^\veps {{\bm w}^\veps}^{\sf T}.\]
Therefore, on the event  $\Gamma_{\delta,r}^{\lambda,\mathcal{J}}$, $X^{(\mathcal{J})} \in E_{\lambda_X}^{\veps,r}$. Observe that this is equivalent to say that on the event $\Gamma_{\delta,r}^{\lambda,\mathcal{J}}$, $K = X^{(\mathcal{J})} +(\lambda+\delta)- \lambda_X \in E_{\lambda+\delta}^{\veps, r}$. Moreover, on this event 
we have $\nabla \phi_{\lambda_X,\widetilde{{\bm v}}^{\veps}}(X^{(\mathcal{J})}) = \nabla \phi_{\lambda+\delta,\widetilde{{\bm v}}^{\veps}}(K) \in \Theta(K)$ and $\phi_{\lambda+\delta}(K) = \phi_{\lambda}(X^{(\mathcal{J})})$. Hence,  on $\Gamma_{\delta,r}^{\lambda,\mathcal{J}}$ we have 
\[ \widetilde{\phi}_{\lambda+\delta}^\veps(X^{(\mathcal{J})}) \geq  \phi_{\lambda+\delta}(K)+ \langle X^{(\mathcal{J})} - K, \nabla \phi_{\lambda+\delta,\widetilde{{\bm v}}^{\veps}}(K)\rangle=  \phi_{\lambda_X}(X^{({\mathcal{J}})})+ \langle X^{(\mathcal{J})} - K, \nabla \phi_{\lambda_X,\widetilde{{\bm v}}^{\veps}}(X^{(\mathcal{J})})\rangle.\]
Using that on $\Gamma_{\delta,r}^{\lambda,\mathcal{J}}$,  $\lambda_X\geq \lambda$, this yields
\[ \widetilde{\phi}^\veps_{\lambda+\delta}(X^{(\mathcal{J})}) \geq \phi_{\lambda_X}(X^{(\mathcal{J})})-(\lambda+ \delta - \lambda_X) \|{\bm w}^\veps\|^2/\|{\bm v}^\veps\|^2 \geq \phi_{\lambda_X}(X^{(\mathcal{J})}) - \delta r. \]
This yields \eqref{eq:Gamma-delta-r}. Now, if we take $r=1/\eta_\lambda^2$ where $\eta_\lambda$ is as in Lemma \ref{mass-loc}, then we get using \eqref{concpoint} and the fact that $m$ is non increasing, that for any $t>0$ and $0<\delta \leq t \eta_\lambda^{2}$,
  \begin{align*} 
 \lim_{\veps \to 0} \limsup_{n\to+\infty} \frac{1}{np} \log \PP\big( {\phi}_{\lambda_X}(X^{(\mathcal{J})})> m(\lambda_X)\|\widecheck{X}_{\mathcal{J}}\|^2 + 2t, \lambda\leq \lambda_X\leq  \lambda+\delta, J=\mathcal{J}\big) & < -I(\lambda).\end{align*} 
%In particular, for any $t>0$ and $0<\delta \leq t \eta_\lambda/2$,
%  \[ \lim_{\veps \to 0} \limsup_{n\to+\infty} \frac{1}{np} \log \PP\big( {\phi}_{\lambda_X}(X^{(\mathcal{J})})> m(\lambda_X)\|\widecheck{X}_{\mathcal{J}}\|^2 + t, \lambda\leq \lambda_X\leq  \lambda+\delta, J=\mathcal{J}\big) < -I(\lambda),\]
As the above estimate is true for any $t>0$, $\lambda>2$ and $\eta_\lambda$ is increasing in $\lambda$, we deduce using a net argument and the exponential tightness of $\lambda_X$ given by Lemma \ref{expotight}, that for any $\lambda>2$,
  \[ \lim_{\veps \to 0} \limsup_{n\to+\infty} \frac{1}{np} \log \PP\big( {\phi}_{\lambda_X}(X^{(\mathcal{J})})> m(\lambda_X)\|\widecheck{X}_{\mathcal{J}}\|^2 + t,  \lambda_X\geq \lambda, J=\mathcal{J} \big) < -I(\lambda).\]
 Finally, as there are at most $e^{O(\veps^{-2}\log n)}$  subsets $\mathcal{J}$ such that $|\mathcal{J}|\leq 2\veps^{-2}$ and for any such subset $\phi_{\lambda_X}(X^{(\mathcal{J})}) \leq  \langle \widecheck{X}_{\mathcal{J}} {\bm v}^\veps, (\lambda_X -X^{(\mathcal{J})})^{-1}\widecheck{X}_{\mathcal{J}}{\bm v}^\veps \rangle+2R\veps^{-2}/\sqrt{np}$, we conclude the proof by using a union bound, Lemma \ref{boundmulti} and the fact that $np \gg \log n$.
 \end{proof}

\subsection{Proofs of the large deviation upper bound in the bounded case and of Theorem \ref{thm:bdd-case}}
%In this section, we show that upon a large deviation of the top eigenvalue, there exists a vertex of degree at least $\lambda_X/m(\lambda_X)$. 
%\begin{Pro}\label{degmaxdev} For any $t>0$ and $\lambda>2$,
%\[ \limsup_{n\to +\infty} \frac{1}{np}  \log \PP\big( \max_{i\in [n]} \|X_i\|^2 < \lambda/m(\lambda)-t, \lambda_X\geq \lambda \big) <-I(\lambda).\]
%As a consequence,
%\[ \lim_{n\to +\infty}  \PP\big( \max_{i\in [n]} \|X_i\|^2 \geq  \lambda/m(\lambda)-t \mid \lambda_X\geq \lambda \big) =1.\]
%\end{Pro}

%\begin{proof}
We begin with the proof of Theorem \ref{thm:bdd-case}. By \eqref{eq-top-ei-sec2-1}, we have 
\[ \langle  {\bm v}^\veps, X_{J} {\bm v}^\veps \rangle  + \langle \widecheck{X}_{J}  {\bm v}^\veps, (\lambda_X -X^{(J)})^{-1}\widecheck{X}_{J}  {\bm v}^\veps \rangle =\lambda_X \| {\bm v}^\veps\|^2.\]
Using Proposition \ref{concresolvfinalbounded} and the fact that $\langle  {\bm v}^\veps, X_{J} {\bm v}^\veps \rangle  =o(1)$ with overwhelming probability by Lemma \ref{boundmulti}, we get that for any $t>0$, and $\lambda>2$,
\[ \lim_{\veps \to 0} \limsup_{n\to+\infty} \frac{1}{np} \log \PP\big( \lambda_X\|{\bm v}^\veps\|^2> m(\lambda_X) \|\widecheck{X}_J\|^2\|{\bm v}^\veps\|^2 +t, \lambda_X\geq \lambda\big) <-I(\lambda).\]
Using  that $\lambda_X$ is exponentially tight, $\lambda \mapsto \lambda/m(\lambda)$ is increasing and Lemma \ref{mass-loc}, we get for any $t>0$, 
\[ \lim_{\veps \to 0} \limsup_{n\to+\infty} \frac{1}{np} \log \PP\big( \lambda/m(\lambda)> \|\widecheck{X}_J\|^2+t, \lambda_X\geq \lambda\big) <-I(\lambda).\]
Combining the above estimate with Lemma \ref{normXcheck:deg}, this ends the proof of the first claim. The second claim is an immediately consequence of Lemma \ref{mass-loc}, and as we know that the large deviation lower bound holds with rate function $I$, the third claim follows.  
%\end{proof}

%\subsection{Large deviation upper bound in the bounded case}
We finally give the proof of large deviation upper bound in the bounded case. Let $\lambda>2$. Theorem \ref{thm:bdd-case} shows that for any $t>0$, with overwhelming probability, on the large deviation event $\{\lambda_X\geq \lambda\}$, there exists a vertex with degree at least $\lambda/m(\lambda)-t$. Since the large deviation of the maximal degree has rate function $h_L$ by Lemma \ref{LDP-deg}, this ends the proof. 
\qed
%\section{Large deviation upper bound for the general case}\label{sec:ld-ub-gen}

 \section{Large deviation upper bound for the general case}\label{sec:ld-ub-gen}
In this section we prove that the upper tail large deviation event of the top eigenvalue in the unbounded case is dominated by the emergence of  either a clique of vertices with large edge weights or by one vertex having a large vertex weight and a high degree.

\subsection{Truncation and decoupling}\label{section:decoupl}
We will build upon the large deviation principle for the top eigenvalue in the bounded case. We start by defining an appropriate decomposition of $X$ into $A+B$, where $A$ is a sparse Wigner matrix with bounded entries, and then show that we can replace $A$ by a sparse Wigner matrix with bounded entries {\em independent} of $B$. In order to carry out these arguments, we first perform a reduction of Theorem \ref{theo-main-weight} in the case where the distributions of the entries of $G$ satisfy the following additional assumption. 
\begin{Hypo}[Density and extended support]\label{hypo:ex-supp}
The laws of the entries of $G$ have densities with respect to the Lebesgue measure and have extended supports, where a probability measure $\mu$ on $\RR$ is said to have an {\em extended support} if $\inf \mathrm{supp}(\mu) = -\infty$ and $\sup \mathrm{supp}(\mu) = +\infty$.
\end{Hypo}

\begin{lemma}\label{reductionextsupp}
If the large deviation upper bound of Theorem \ref{theo-main-weight} holds for any sparse Wigner matrix satisfying Assumptions \ref{hypo:subg-entry} and \ref{hypo:ex-supp} then it continues to hold for any sparse Wigner matrices satisfying {\em only} Assumption \ref{hypo:subg-entry}. 
\end{lemma}

\begin{proof}
Let $X, G$, and $\Xi$ be as in Theorem \ref{theo-main-weight}.
%be a Wigner matrix satisfying Assumption \ref{hypo:subg-entry} and $\Xi$ an independent symmetric random matrix with i.i.d.~$\dBer(p)$ entries on and above the diagonal. 
Let $\Gamma$ be a symmetric random matrix, independent of $G$ and $\Xi$, with i.i.d.~random variables variables on and above the diagonal such that the common law of those random variables has a symmetric density (around zero) with respect to the Lebesgue measure, with tails lighter than a Gaussian and of unit variance, and has an extended support. 
Define for any $\veps>0$, $G^\veps:= G+\veps \Gamma$, and $X^\veps := G^\veps \circ \Xi /\sqrt{np}$. Clearly $G^\veps$ satisfy Assumptions \ref{hypo:subg-entry} and \ref{hypo:ex-supp}
%have extended support and are absolutely continuous with respect to the Lebesgue measure 
for any $\veps>0$. 
%Let $\Lambda_{i,j}^\veps$ be the log-Laplace transform of $G^\veps_{i,j}$ for any $i,j\in [n]$. As $\Lambda_{i,j}^\veps(\theta) = \Lambda_{i,j}(\theta) + \theta^2 \veps^2/2$ for any $\theta \in \RR$, 
%\corAB{Clearly} $G^\veps$ satisfies Assumption \ref{hypo:subg-entry}. 
Note that $\Var (G^\veps_{1,2}) = 1+\veps^2$ and that the entries of $\Gamma$ having a tail lighter than the Gaussian means that \eqref{subgaussian} continues to hold  for $(G^\eps)_{i,j}$ with the same $\alpha$ and $\beta$, for any $\eps >0$. Hence applying Theorem \ref{theo-main-weight} to $X^\veps/\sqrt{1+\veps^2}$ gives  for any $t>2\sqrt{1+\veps^2}$  
\[ \limsup_{n\to +\infty} \frac{1}{np} \log \PP\big(\lambda_{X^\veps}>t\big) \leq - I_\veps(t),\]
where 
\[ I_\veps (t) := \min\Big\{\widehat{I}_\veps(t_\veps), \frac{1}{4\beta_\eps m(t_\veps)^2}\Big\}, \ \widehat{I}_\veps(t_\veps) := \inf \big\{ \frac{r^2}{2\alpha_\eps} + h_{L_\veps}(s) : r+m(t_\veps)s = t_\veps, r\geq 0,s\geq 1\big\},\] 
with $t_\veps := t/\sqrt{1+\veps^2}$, $\alpha_\eps:=\alpha/(1+\eps^2)$, $\beta_\eps:=\beta/(1+\eps^2)$, and $L_\veps(\theta) := \EE[\exp(\theta (G^\veps_{1,2})^2/(1+\veps^2))]$. 

Upon observing that $(G^\eps)_{1,2}^2 \le (1+\eps^2) (G_{1,2}^2+ \Gamma_{1,2}^2)$ and using the fact that the lighter than a Gaussian tail assumption of $\Gamma_{1,2}$ implies that $\EE[\exp(\theta \Gamma_{1,2}^2)] < \infty$ for any $\theta \in \RR$, we derive that $\mathcal{D}_L \subset \mathcal{D}_{L_\eps}$ for any $\eps >0$. It also follows by the dominated convergence theorem that $L_\veps$ converges to $L$ pointwise on $\mathcal{D}_L$, as $\veps \downarrow 0$. As $L_\veps$ is increasing on $\RR$ for any $\veps>0$, Dini's theorem further implies that $L_\veps$ converges uniformly to $L$ on compact subsets of $\mathcal{D}_L$. Therefore
\begin{equation}\label{eq:h-liminf}
\liminf_{\veps \downarrow 0} h_{L_\veps}(x) \ge \liminf_{\veps \downarrow 0} \sup_{\theta \in \mathcal{D}_L} \{\theta x - L_\veps(\theta) +1\} = h_L(x). 
\end{equation}
On the other hand, by Jensen's inequality and the fact that $\Gamma_{1,2}$ is centered, we have $L_\veps(\theta) \ge L(\theta/(1+\veps^2))$ for any $\theta \ge 0$ which in turn gives
\begin{equation}\label{eq:h-limsup}
\limsup_{\veps \downarrow 0} h_{L_\veps}(x) \le \limsup_{\veps \downarrow 0} h_L(x(1+\veps^2)) = h_L(x), \quad \text{ for } x \ge 1,
\end{equation}
where the last step also uses that $h_L$ is continuous. The pointwise convergence of $h_{L_\veps}$ to $h_L$ as $\veps \downarrow 0$, given by \eqref{eq:h-liminf}-\eqref{eq:h-limsup}, together with the fact that $h_{L_\veps}$ is non decreasing on $[1,+\infty)$ (recall Lemma \ref{lem:hL-prop}(e)), imply that the convergence is uniform on compact subsets of $[1, +\infty)$.  One can check that this entails that $I_\veps(t) \to I(t)$ as $\veps \downarrow 0$, for any $t>2$. Besides $|\lambda_{X^\veps} - \lambda_X|\leq  \veps \|\Gamma \circ \Xi /\sqrt{np}\|$, so that by Lemma \ref{expotight}, it follows that $(\lambda_{X^\veps})_{\veps>0}$, is an exponentially good approximation of $\lambda_X$ (see \cite[Definition 4.2.14]{DZ}). Since $I$ is lower semicontinuous, we deduce from \cite[Theorem 4.2.16]{DZ} that the large deviation upper bound of Theorem \ref{theo-main-weight} holds for $X$.
\end{proof}

Equipped with Lemma \ref{reductionextsupp} from now on we assume that the entries of $G$  satisfy Assumptions \ref{hypo:subg-entry}  and \ref{hypo:ex-supp}. With these additional assumptions, we define a truncation of the entries of $G$ in a way that preserve their zero mean property. This will be crucial in our decoupling argument in order to guarantee the resulting decoupled sparse matrices to be independent. To achieve this, we rely on the following lemma.

\begin{lemma}\label{tronc}
Let $\mu$ be a centered probability measure on $\RR$ that is absolutely continuous with respect to the Lebesgue measure, and with an extended support. There exist continuous functions $a,b : \RR_+\mapsto \RR_+$ such that 
\[ \int_{-a(R)}^{b(R)} x d\mu(x) = 0, \ \forall R>0, \qquad \text{ and }\qquad  a(R), b(R) \to +\infty \text{ as } R\to +\infty.
\]
%and $a(R), b(R) \to +\infty$ as $R\to +\infty$. 
\end{lemma}
\begin{proof}
Let $a>0$. Define $F_a : b \in \RR_+ \mapsto \int_{-a}^b x d\mu(x)$. Note that $F_a$ is continuous since $\mu$ is absolutely continuous.  Further $F_a(0)<0$, $\lim_{b\to +\infty} F_a(b) >0$ by the assumptions that $\mu$ is centered and has an extended support. By the mean value theorem, there exists $b=b(a)>0$ such that $F_a(b)=0$. Moreover, one can check that $a\in \RR_+\mapsto b(a)$ is continuous and using the fact that $\mu$ is centered and has an extended support, $b(a) \to+\infty$ as $a\to +\infty$. 
\end{proof}

Let $\mu_{i,j}$ denote the law of $G_{i,j}$ for $i,j\in [n]$, $a_{i,j}$, $b_{i,j}$ be the functions given by Lemma \ref{tronc} for $\mu_{i,j}$, and $I_{i,j}(R) := [-a_{i,j}(R),b_{i,j}(R)]$. For any $R>0$, define $\widetilde{A}^R$ to be the matrix with its entries given by
\begin{equation}\label{defA} \widetilde{A}_{i,j}^R := G_{i,j} \Car_{G_{i,j} \in I_{i,j}(R)},  \forall i,j \in [n], 
\end{equation}
and $A^R := \widetilde{A}^R \circ \Xi/\sqrt{np}$. Further set $\widetilde{B} := G-\widetilde{A}^R$ and $B^R:= \widetilde{B}^R \circ \Xi/\sqrt{np}$. To ease the notation, when there is  no risk of confusion, we will drop the $R$-dependence from these notation and this convention will be adopted for most of the results and in their  proofs that follow.

Next we argue that we can replace $A^R$ by a symmetric matrix $H^R$ independent of $B^R$, such that $H^R:=(\widetilde{H}^R\circ \Xi')/\sqrt{np}$, where $\widetilde{H}^R$ and $\Xi'$ are independent symmetric matrices, $\Xi' :=(\xi'_{i,j})_{i,j\in [n]}$ has same law as $\Xi$, the variables $(\widetilde{H}^R_{i,j})_{i\leq j}$ are independent and $\widetilde{H}^R_{i,j}$ is distributed according to the law of $G_{i,j}$ conditioned on $G_{i,j} \in I_{i,j}(R)$. 
Let $E := \{(i,j) \in [n]^2 : G_{i,j} \in I_{i,j}(R)]\}$ and define $\widehat{A}^R$ to be the matrix with entries
\begin{equation} \label{defAprime} \widehat{A}^R_{i,j} := \Car_{(i,j) \in E} A^R_{i,j} + \Car_{(i,j) \notin E} H^R_{i,j}, \quad \forall i,j \in [n].
\end{equation}
One can check that $\widehat{A}$ is independent of $B^R$ and that it has the same law as $H^R$. Here we take the advantage of our special truncation of the entries. {\em Observe that we could have done the same construction with a simpler truncation, e.g.~choosing $I_{i,j}$ to be symmetric, and then recenter the truncated entries, but for such a truncation, unless the distribution of $G$ is symmetric around zero, the independence does not hold}.
With the above notation we have the following result.

\begin{Pro}
\label{decoupling}Let $X^R = \widehat{A}^R+B^R$, where $\widehat{A}$ is defined in \eqref{defAprime}.
For any $t>0$, 
\[ \lim_{R\to +\infty}\limsup_{n\to +\infty} \frac{1}{np} \log \PP\big( \| X^R - X \| > t \big) =-\infty.\]
\end{Pro}
\begin{proof}
 We can write $\widehat{A}^R-A^R  = (\widetilde{H} \circ \Xi'')/\sqrt{np}$,  where $\Xi'' = (\xi''_{i,j})_{i,j\in [n]}$ such that $\xi''_{i,j}= \xi'_{i,j} \Car_{G_{i,j} \notin I_{i,j}(R)}$. Note that $\widetilde{H}^R$ and $\Xi''$ are independent. Since $\xi_{i,j}'' \stackrel{d}{=} \dBer(q)$ 
 %are Bernoulli random variables 
 with $ q = o_R(1) p$, the exponential tightness result of Lemma \ref{expotight} yields that for any $t>0$,
\[ \lim_{R\to+\infty}\limsup_{n\to +\infty} \frac{1}{np} \log \PP\big( \|\widetilde{H}^R \circ \Xi''\| > t\sqrt{np} \big) =-\infty,\]
which ends the proof of the claim.
\end{proof}

 %The interest of the above 
Proposition \ref{decoupling} allows us to focus on proving a large deviation upper bound for $\lambda_{Z^R}$ where $Z^R = H^R+B^R$, for which one can take advantage of the independence between $H^R$ and $B^R$. Indeed, $Z^R$ has the same law as $X^R$ by construction and by Weyl's inequality, $|\lambda_{X^R}-\lambda_{X}| \leq \| X^R-X\|$. Thus Proposition \ref{decoupling} yields that $(\lambda_{Z^R})_{R>0}$ is an exponentially good approximation of $\lambda_X$, and by \cite[Theorem 4.2.16]{DZ}  we can retrieve the large deviation upper bound rate function of $\lambda_X$ as the limit of the ones for $\lambda_{Z^R}$ as $R\to +\infty$. Since $\|Z^R\|\leq \|H^R\| +\|B^R\|$, using Lemma \ref{expotight} we note that $\|Z^R\|$ is exponentially tight for any $R>0$. Moreover, by the exponential tightness of $\|X\|$ and Proposition  \ref{decoupling}, we deduce that $\|Z^R\|$, $R\to +\infty$ are uniformly exponentially tight. More precisely we have the following result. 
 \begin{Lem}[Exponential tightness] \label{expotightZ}
 For any $R>0$, 
   \begin{equation} \label{expotightZR} \lim_{t \to +\infty} \limsup_{n\to +\infty }\frac{1}{np} \log \PP( 
\|Z^R\|>t) = -\infty.\end{equation} 
Further,
  \begin{equation} \label{expotightunif}  \lim_{t \to +\infty} \limsup_{R\to +\infty} \limsup_{n\to +\infty }\frac{1}{np} \log \PP( 
\|Z^R\|>t) = -\infty.\end{equation}
 
 \end{Lem}

\subsection{Block matrix decomposition} Let ${\bm u}$ be a measurable choice of a unit eigenvector of $Z^R$ associated to the top eigenvalue $\lambda_{Z^R}$. 
Similarly as in the bounded case, we will decompose the matrix $Z^R$ in blocks according to a certain set $J$ of vertices where the top eigenvector ${\bm u}$ localises. Unlike the bounded case, we choose a different localisation threshold since we are expecting to see a new large deviation scenario:~the emergence of a clique of sub entropic size. We incorporate as well for technical reasons to the set $J$ vertices with large degree with respect to the network $B^R$. More precisely, define for any $\veps>0$,
\begin{equation}\label{defJ} J := \big\{j \in [n] :  \|B^R_j\|\geq \delta\big\} \cup\big\{i\in[n] : |{\bm u}_i|\geq \eps^{-1}  \sqrt{\frac{\log (1/p)}{np}} \big\},\end{equation}
where $B^R_j$ denotes 
%the $\ell^2$-norm of 
the $j^{\text{th}}$-column of $B^R$, and $\delta$ is chosen to be any function of $\veps$ such that $\delta = o(\veps)$ as $\vep \to 0$. Observe that by Chebychev's inequality we have 
\begin{equation}\label{eq:J-ubdd-bd} 
\#\big\{i\in [n] : |{\bm u}_i| \geq \veps^{-1} \sqrt{\frac{\log(1/p)}{np}} \big\} \leq  \frac{\veps^2 np}{\log(1/p)}.
\end{equation}
As we will see, adding the indices $j$ for which $\|B^R_j\|$ is atypically large only increases the size of $J$ by an additive constant with overwhelming probability. This yields the following lemma.
\begin{lemma}\label{cardJ}For any $\veps>0$ and $R$ large enough,
\[  \lim_{n\to +\infty}\frac{1}{np} \log \PP\big( \#J > 2\veps^2 np /\log(1/p)\big) =-\infty.\]
\end{lemma}
A consequence of the above result is that the set $J$ carries {\em no entropy} at the large deviation scale $np$, meaning that $\log$ of the number of possible realisations of $J$ is bounded by
\begin{equation} \label{entropyJ} \log \#\big\{ \mathcal{J} \subset[n] :\# \mathcal{J} \leq 2\veps^{2} np/\log(1/p) \big\} = O(\veps^2 np).\end{equation}
The above estimate follows readily from the fact that there are at most $(en/m)^m$ subsets of $[n]$ with size $m$. This estimate is the motivation behind our choice of the localisation threshold $\veps^{-1} \sqrt{\log(1/p)/np}$. In particular, \eqref{entropyJ} will allow us to derive estimates with $J$ replaced by a fixed deterministic subset $\mathcal{J}$ and then take union bounds over possible realisations of $J$.

\begin{proof}[Proof of Lemma \ref{cardJ}]Since $np \gg \log(1/p)$, by \eqref{eq:J-ubdd-bd} it suffices to prove that for any $\delta>0$ and $R$ large enough,
\[  \lim_{m \to +\infty}  \limsup_{n\to +\infty}\frac{1}{np} \log \PP\big( \#\big\{j\in[n] : \|B_j\|\geq \delta \big\}\geq m\big) =-\infty.\]
As the number of subsets of $[n]$ of size $m$ is at most $e^{O(m\log n)}$ and $np \gg \log n$,  
%we can assume that the first $m$ vertices have a large degree. In other words, 
it is sufficient to prove that for any $\delta>0$, and $R$ large enough,
\[ 
\lim_{m \to +\infty} \lim_{n\to +\infty}\frac{1}{np} \log \PP\big( \forall j \in [m], \ \|B_j\|\geq \delta\big) =-\infty.\]
%By Chernoff inequality, we find for any $\theta>0$,
%\begin{equation}\label{chernoffJ} \log \PP\big( \sum_{j=1}^{m} \|B_j\|^2>\delta^2 m \big) \leq -\theta \delta^2mnp + \sum_{j=1}^{m} \sum_{i\geq j} \log \EE (e^{2 \theta \widetilde{B}_{i,j}^2\xi_{i,j}}).\end{equation}
%For any $\theta >0$ and $i,j \in [n]$, 
%\begin{equation} \label{loglaplaB2} \log \EE ( e^{2\theta \widetilde{B}_{i,j}^2 \xi_{i,j}}) = \log (1-p+p\EE (e^{2\theta \widetilde{B}_{i,j}^2})) \leq p \EE(e^{2\theta \widetilde{B}_{i,j}^2}-1)\leq  p\EE( \Car_{G_{i,j} \notin I_{i,j}(R)} e^{2\theta G_{i,j}%^2}).\end{equation}
%As $G$ is sub-Gaussian we deduce that for $\theta$ small enough, independent of $(i,j)$, $\EE( \Car_{G_{i,j} \notin I_{i,j}(R)} e^{2\theta G_{i,j}^2}) = o_R(1)$. 
Using a similar argument as in \eqref{loglaplaB2}, we deduce that for $\theta$ small enough, independent of $(i,j)$, $\log \EE(e^{2\theta \widetilde{B}_{i,j}^2\xi_{i,j} }) = o_R(1)p$. 
 With this choice of $\theta$, Chernoff's inequality and \eqref{loglaplaB2} give
\begin{equation}\label{loglaplaB2-new}  
\log \PP\big( \sum_{j=1}^{m} \|B_j\|^2>\delta^2 m\big) \leq -(\theta \delta^2 - o_R(1))m np, 
\end{equation}
which ends the proof of the claim.
\end{proof}
With our choice of subset $J$ defined in \eqref{defJ}, we split the top eigenvector ${\bm u}$ into two parts ${\bm v}^{\veps} \in \RR^{ J}$ and ${\bm w}^{\veps}\in \RR^{{\bm J}^c}$ such that ${\bm u}^{\sf T} = ({{\bm v}^{\veps}}^{\sf T}, {{\bm w}^{\veps}}^{\sf T})$. This decomposition of the space $\RR^n$ entails the following block decomposition of $Z^R$:
\begin{equation} \label{decomp-Z} Z^R = \begin{pmatrix}
Z_{J} &  \widecheck{Z}_{ J}^{\sf T} \\
\widecheck{Z}_{ J}  &  Z^{( J)}
\end{pmatrix},\end{equation}
where we refer the reader yet again to Section \ref{sec:notation} for the notational conventions. %$Z_{J}$ is the $J\times J$ principal minor of $Z^R$, and $Z^{(J)}$ is the matrix spanned by the rows and columns in $J^c$, whereas $\widecheck{Z}_J$ is of size $J^c \times J$. 

\subsection{Localisation of the top eigenvector} \label{section:loc:unbound} We now move on  to prove that on the upper tail large deviation event of the top eigenvalue, any top eigenvector of $Z^R$ must localise in the sense that the $\ell^2$-weight of of the sub vector consisting of entries that are much greater than $ \sqrt{\log(1/p)/np}$ is bounded from below by a constant depending on the deviation of the top eigenvalue. 

\begin{Lem}\label{loc:unbounded}
For any $\lambda>2$, 
\[\lim_{\eta \to 0} \limsup_{R\to +\infty}\limsup_{\veps \to 0} \limsup_{n\to+\infty} \frac{1}{np} \log \PP\big( \|{\bm v}^\veps\|\leq \eta, \lambda_{Z^R}\geq \lambda) =-\infty.\]
%The same statement holds for the matrix $X$ and a top eigenvector of $X$. 
\end{Lem}

Similarly as in the bounded case, this localisation result builds on the fact that the supremum of $\langle w, Z^R w\rangle$ when $w$ is a delocalised unit vector, meaning that  $\|w\|_\infty \leq \veps^{-1} \sqrt{\log(1/p)/np}$,  cannot exceed $2$ with overwhelming probability. This is the content of the following lemma. Recall the definition of $\mathcal{D}_\veps$ from \eqref{flatmatrices}.

\begin{Lem} \label{quadra:deloc:unbounded}
For $\vep >0$ set $M_\veps: = \veps^{-1} \sqrt{\log(1/p)}$. Then for any $t>0$,
\[ \lim_{R\to +\infty} \limsup_{\veps \to 0} \limsup_{n\to +\infty} \frac{1}{np} \log \PP\Big( \sup_{w\in \mathcal{D}_{M_\veps/\sqrt{np}}} \langle w, Z^R w\rangle >2+t \Big) = -\infty.\]
%where $M_\veps = \veps^{-1} \sqrt{\log(1/p)}$. 
\end{Lem}

%Once the above Lemma \ref{quadra:deloc:unbounded} is proven, Lemma \ref{loc:unbounded} follows readily using the same argument as in the proof of Lemma \ref{mass-loc} in the bounded case. The only difference is that one should use the exponential tightness of $\|\widecheck{Z}_J\|$, guaranteed by Proposition \ref{expotight}, instead of the large deviation bound of Lemma \ref{normXcheck:deg}.
Postponing the proof of Lemma \ref{quadra:deloc:unbounded} for later we first show how it yields Lemma \ref{loc:unbounded}. %this result entails that ${\bm v}^{\veps}$ must carry some non trivial $\ell^2$-weight.
\begin{proof}[Proof of Lemma \ref{loc:unbounded}]
Expanding $\lambda_Z = \langle {\bm u}, Z {\bm u}\rangle$, we can write $\lambda_Z = \langle {\bm v}^\veps, Z_J {\bm v}^\veps\rangle +2\langle {\bm w}^\veps, \widecheck{Z}_J{\bm v}^\veps\rangle + \langle {\bm w}^\veps, Z^{(J)} {\bm w}^\veps\rangle$. 
As a consequence, 
\[ \lambda_Z \leq \|Z\| \|{\bm v}^\veps\|^2 + 2 \|Z\| \|{\bm v}^\veps\| + \sup_{w \in \mathcal{D}_{M_\veps/\sqrt{np}}} \langle w, Z w\rangle.\]
%where $M_\veps=\veps^{-1} \sqrt{\log(1/p)}$. 
Thus if $\lambda>2$, $r\geq 1$ and $\eta \in (0,1)$, we have 
\[ \big\{\|{\bm v}^\veps\|\leq \eta, \lambda_Z\geq \lambda\big\} \subset \{\lambda-3r\eta \leq \sup_{w \in \mathcal{D}_{M_\veps/\sqrt{np}}} \langle w, Z w\rangle\big\}\cup\big\{\|Z\|\geq r\big\}.\]
Combining Lemmas \ref{expotightZ} and \ref{quadra:deloc:unbounded} %and the exponential tightness of $\|Z\|$ from Lemma \ref{expotightZ}, 
we get the claim. 
\end{proof}

It now remains to prove Lemma \ref{quadra:deloc:unbounded}. Since $H^R$ has bounded entries, by Lemma \ref{conc-sp-rad-deloc} we already know that the statement holds for $H^R$ instead of $Z^R$. Thus, it suffices to show that the supremum of $\langle w, B^Rw \rangle$ where $w \in \mathcal{D}_{M_\veps}$ is negligible with overwhelming probability. We prove the following stronger statement which will be used later in this section.

\begin{lemma} \label{step1}
For any $t>0$ and $M_\veps$ as in Lemma \ref{quadra:deloc:unbounded},
\begin{equation} \label{probaBloc} \lim_{R\to +\infty} \limsup_{\veps,\veps' \to 0 \atop \veps'/\veps   \to 0} \limsup_{n\to +\infty}  \frac{1}{np} \log  \PP\Big( \sup_{ v\in \mathcal{D}_{\veps'}, w \in \mathcal{D}_{M_\veps /\sqrt{np} }} \langle v, B^Rw\rangle > t\Big) = -\infty.\end{equation}
%where $M_\veps = \veps^{-1} \sqrt{\log(1/p)}$.
\end{lemma}
This results says that a large deviation of the operator norm of $B$ should come from a localised structure: In order to see an atypically large inner product $\langle v, B^Rw\rangle$, either $v$ should have entries that are of order $1$ or $w$ should put weight on a set of vertices of size much smaller than $np/\log(1/p)$.

As we will use a concentration argument, we need to regularize the distribution of $\widetilde{B}^R$ in a first step by adding an independent centered random symmetric matrix $\widetilde{C}^R$ whose entries have a small variance such that the distribution of $\widetilde{B}^R + \widetilde{C}^R$ satisfies a log-Sobolev inequality on $\mathcal{H}_n$. In the following lemma we prove the existence of such a matrix $\widetilde{C}^R$. 

\begin{lemma}\label{regularizelogSob}
For $R$ large enough, independent of $n$, there exists $\widetilde{C}^R$ a centered symmetric random matrix independent of $\widetilde{B}^R$ such that 
\begin{enumerate}
\item[ (a)] $(\widetilde{C}^R_{i,i})_{i\in [n]}$ and $(\widetilde{C}^R_{i,j})_{i<j}$ are independent families of i.i.d.~random variables.
\item[ (b)] For any $(i,j) \in [n]^2$, $\EE [(\widetilde{C}^R_{i,j})^2]=o_R(1)$.
\item[ (c)] The distribution of $\widetilde{B}^R+\widetilde{C}^R$ satisfies a log-Sobolev inequality on $\mathcal{H}_n$ with respect to the Hilbert-Schmidt norm and with a constant $c_R$ independent of $n$.
\end{enumerate}
\end{lemma}
\begin{proof}
As log-Sobolev inequalities tensorise (see \cite[Corollary 5.7]{Ledouxmono}) it suffices to show that for $(i,j) \in \{(1,1), (1,2)\}$, there exists a centered random variable $\widetilde{C}^R_{i,j}$ with the property (b)
%of small variance and independent of $\widetilde{B}^R$ 
such that the distribution of $\widetilde{B}^R_{i,j}+\widetilde{C}^R_{i,j}$ satisfies a log-Sobolev inequality on $\RR$. Recall from \eqref{defA} that $\widetilde{B}_{i,j}^R= G_{i,j}\Car\{G_{i,j} \notin [-a_{i,j}(R),b_{i,j}(R)]\}$. Let $c_{i,j} := \min(a_{i,j}(R), b_{i,j}(R))/4$. Define $\beta_{i,j} = \alpha$ if $i=j$, and $\beta_{i,j}=\beta$ otherwise, and fix $\beta'_{i,j}> \beta_{i,j}$. Now, for $R$ large enough set $V_{i,j}$ to be a symmetric function on $\RR$ defined as 
\[ V_{i,j}(x) := \begin{cases} \frac{c_{i,j} x^2}{2} & \text{ if } |x|\leq c_{i,j},\\
\frac{x^2}{2\beta_{i,j}'} + k_{i,j}|x|+\ell_{i,j} & \text{ otherwise,}
\end{cases}\]
where $\ell_{i,j}\le 0 \le k_{i,j}$ are chosen so that $V_{i,j} \in C^1(\RR)$. Since $V_{i,j}'$ is monotonically increasing the function $V_{i,j}$ is convex.
Define the probability measure $\mu_{i,j} (dx) \propto \exp(-V_{i,j}(x)) dx$ %be the probability measure 
on $\RR$. %with density with respect to the Lebesgue measure proportional to $e^{-V_{i,j}}$. 
Clearly $\mu_{i,j}$ has zero mean and one can easily check that $\int x^2 d\mu_{i,j}(x) =o_R(1)$ since $c_{i,j} \to +\infty$ as $R\to +\infty$. We will show that $\mu_{i,j}$ satisfies a log-Sobolev inequality, and  that if $\upsilon_{i,j}$ is the distribution of $\widetilde{B}^R_{i,j}$ then $\upsilon_{i,j} \ast \mu_{i,j}$ also satisfies  a log-Sobolev inequality.  

As for $R$ large enough $V'_{i,j}(z+h) - V_{i,j}'(z-h)\geq 2h/\beta'_{i,j}$ for any $z\in\RR$ and $h\geq 0$, it follows that $V_{i,j}$ is strongly convex with convexity parameter bounded from below by $1/\beta'_{i,j}$. By \cite[Proposition 3.1]{BoLe} we deduce that $\mu_{i,j}$ satisfies a log-Sobolev inequality with constant $2\beta'_{i,j}$. Now, from \cite[Theorem 2.1]{WW} and Jensen's inequality, we know that a sufficient condition for $\upsilon_{i,j} \ast \mu_{i,j}$ to satisfy a log-Sobolev inequality is that for some $\lambda>1$,
\begin{equation} \label{expocondlogsob} \int \exp\Big( \frac{\lambda \beta_{i,j}'}{2} |V_{i,j}'(x) -V_{i,j}'(x-z)|^2 \Big) d\upsilon_x(z) d\mu_{i,j}(x) <+\infty,\end{equation}
where $\upsilon_x(dz) \propto e^{-V_{i,j}(x-z)} d\upsilon_{i,j}(z)$. In order to prove \eqref{expocondlogsob} we note that $\supp(\upsilon_{i,j}) \subset \{0\}\cup [-a_{i,j}(R),b_{i,j}(R)]^c$. Since $c_{i,j} <\min(a_{i,j}(R),b_{i,j}(R))/2$, if $z \notin [-b_{i,j}(R),a_{i,j}(R)]$, then for any $x\in \RR$, $x$ and $x-z$ cannot be both in $[-c_{i,j},c_{i,j}]$. As a consequence, one can check that for any $z \in [-a_{i,j}(R),b_{i,j}(R)]^c\cup\{0\}$ and $x\in \RR$, $|V'_{i,j}(x)-V'_{i,j}(x-z)|\leq |z|/\beta'_{i,j} +\kappa_{i,j}$, where $\kappa_{i,j}>0$ depends on $R$. Moreover, $x\mapsto V_{i,j}(x)-x^2/(2\beta_{i,j}')-k_{i,j}|x|$ is uniformly bounded. Thus, by \eqref{expocondlogsob}, it suffices to show  that there exists $\lambda>1$ such that 
\begin{equation} \label{expocondlogsob2} \int \frac{\exp\big( \frac{1}{2\beta_{i,j}'}(\lambda z^2 - (x-z)^2- x^2)) }{\int \exp\big(-\frac{(x-z')^2}{2\beta_{i,j}'}-k_{i,j}|x-z'| \big)d\upsilon_{i,j}(z')} d\upsilon_{i,j}(z)dx<+\infty.\end{equation}
Arguing as in the proof of \cite[Theorem 2.1]{WW}, we can localize the integral in the denominator on a interval $[-L,L]$ such that  $\upsilon_{i,j}([-L,L]) \geq 1/2$, and lower bound up to a multiplicative constant the integrand by $\exp(-x^2/(2\beta_{i,j}')-C|x|)$, where $C>0$ depend on $L$. Using the bound $(x-z)^2\geq [t/(1-t)] x^2 -tz^2$ for $x,z\in\RR$ and $t\in(0,1)$, we get that \eqref{expocondlogsob2} is implied by the condition that there exists $\lambda>1$ and $t\in(0,1)$ such that
\begin{equation} \label{expocondlogsob3} \int \exp\Big(\frac{1}{2\beta_{i,j}'}\big((\lambda+t)z^2 -\frac{t}{1-t} x^2\big)+C|x|\Big) d\upsilon_{i,j}(z) dx <+\infty.\end{equation}
Since $\int \exp(z^2/(2\beta'')) d \upsilon(z)<+\infty$ for any $\beta''>\beta_{i,j}$ (by \eqref{subgaussian}) we finally obtain that there indeed exist $\lambda>1$ and $t\in(0,1)$ such that \eqref{expocondlogsob3} holds. This completes the proof. %, which  finally yields that $\upsilon_{i,j} \ast \mu_{i,j}$ satisfies a log-Sobolev inequality.
\end{proof}

With Lemma \ref{regularizelogSob} now at our disposal we can give a proof of Lemma \ref{step1}. 
\begin{proof}[Proof of Lemma \ref{step1}]
 Let $\widetilde{C}^R$ be a symmetric random matrix independent of $\widetilde{B}^R$ satisfying the conditions of Lemma \ref{regularizelogSob}  and define $\widetilde{T}^R := \widetilde{B}^R+ \widetilde{C}^R$. For any realisation of the graph $\Xi$ and $\veps, \veps' >0$ define the function 
\[ f^{\veps,\veps'}_\Xi(K) :=  \sup_{v\in \mathcal{D}_{\veps'}, w \in \mathcal{D}_{M_\veps/\sqrt{np}} } \langle v, (K\circ \Xi)w\rangle, \ \forall K\in\mathcal{H}_n.\]
We will prove that for any $\theta\geq 0$, 
\begin{equation} \label{concregularization} \lim_{R \to +\infty}\limsup_{\veps,\veps' \to 0 \atop \veps'/\veps  \to 0} \limsup_{n\to +\infty} \frac{1}{np} \log \EE \big[ e^{\theta \sqrt{np}f^{\veps,\veps'}_\Xi(\widetilde{T}^R) }\big] \leq 0 \end{equation}
%Once the above statement is proven, the result of Lemma \ref{step1} follows. Indeed, 
%
Using $\EE \widetilde{C}^R=0$ and Jensen's inequality we have for any $\theta \geq 0$,
\[ \EE\big[ e^{\theta \sqrt{np} f_{\Xi}^{\veps,\veps'}(\widetilde{B}^R)}\big] \leq \EE\big[ e^{\theta \sqrt{np} f_{\Xi}^{\veps,\veps'}(\widetilde{T}^R)}\big].\] 
Therefore the above inequality together with \eqref{concregularization} and Chernoff's inequality yield \eqref{probaBloc}. So we only need to prove  \eqref{concregularization}. 
To this end, fixing a realisation of $\Xi$ for now, we can write for any $K,K'\in\mathcal{H}_n$,
\[ f^{\veps,\veps'}_\Xi(K) -f^{\veps,\veps'}_\Xi(K') \leq L_\Xi \| K- K'\|_2,
\text{ where } L_\Xi^2 := \sup_{v\in \mathcal{D}_{\veps'},w \in \mathcal{D}_{M_\veps/\sqrt{np}} } \sum_{i,j} \xi_{i,j} v_i^2 w_j^2.\]
Since, by Lemma \ref{regularizelogSob}, for $R$ large enough $\widetilde{T}^R$ satisfies a log-Sobolev inequality with some constant $c_{R}>0$, we deduce by \cite[Eqn.~(5.8)]{Ledouxmono} that for any $\theta \geq 0$, 
\begin{equation} \label{conccond} \EE_{\widetilde{T}^R} \big[e^{\theta \sqrt{np} f^{\veps,\veps'}_\Xi(\widetilde{T}^R) } \big] \leq e^{\frac{np \theta^2 L_{\Xi}^2}{2c_R}+\theta \sqrt{np} \EE_{\widetilde{T}^R}  f^{\veps,\veps'}_\Xi(\widetilde{T}^R)}.
\end{equation}
%where $\PP_{\widetilde{T}^R}$ and $\EE_{\widetilde{T}^R}$ denote respectively the probability and expectation with respect to ${\widetilde{T}^R}$.
On the other hand, we know by Lemma \ref{controlBernoullisum} that for any $\delta^2\geq \max({\veps'}^2, 3p)$ and any $n$ large enough\begin{equation} \label{controlLip} \PP\big( L_\Xi > \delta) \leq e^{-\frac{\delta^2}{2(\veps' M_\veps)^2} np \log(1/p)} = e^{-\frac{\delta^2}{2(\veps'/\veps)^2} np}.\end{equation}
If $\veps =o(1)$ and $\veps'/\veps=o(1)$, then we can find $\delta =o(1)$ such that $\delta \gg \veps'/\veps$ and $\delta^2\geq \max({\veps'}^2,3p)$, and fix this choice of $\delta$ for the rest of the proof. 
Now, observe that the function $\Xi \mapsto \EE_{\widetilde{T}^R} f^{\veps,\veps'}_\Xi(\widetilde{T}^R)$ is  convex  and that for any symmetric matrices $\Xi=(\xi_{i,j})_{i,j}$, $\Xi'=(\xi'_{i,j})_{i,j}$, 
\begin{align*} \EE_{\widetilde{T}^R}  f^{\veps,\veps'}_\Xi(\widetilde{T}^R)- \EE_{\widetilde{T}^R}  f^{\veps,\veps'}_{\Xi'}(\widetilde{T}^R) \leq \EE_{\widetilde{T}^R}  \| \widetilde{T}^R \circ (\Xi-\Xi') \|  & \leq \EE_{\widetilde{T}^R}  \|\widetilde{T}^{R} \circ (\Xi-\Xi') \|_2\\
& \leq \sqrt{ \sum_{i,j} \EE(\widetilde{T}^R_{i,j})^2 (\xi_{i,j}-\xi_{i,j}')^{2}},
\end{align*}
implying that it is $\sigma_{R}$-Lipschitz with respect to the Hilbert-Schmidt norm, where $\sigma^2_{R} := \max(\EE(\widetilde{T}^R_{1,1})^2),\EE(\widetilde{T}^R_{1,2})^2)$. %, then we deduce from the above inequality that $\Xi \mapsto \EE_{\widetilde{T}^R}  f^{\veps,\veps'}_\Xi(\widetilde{C}^R)$ is $\sigma_{R}$-Lipschitz with respect to the Hilbert-Schmidt norm. 
Therefore using \cite[Theorem 8.6]{BLM}, we get that for any $\theta \geq 0$, 
\begin{equation} \label{conccondf} \EE\Big(e^{\theta \sqrt{np} \big( \EE_{\widetilde{T}^R}  f^{\veps,\veps'}_\Xi(\widetilde{T}^R)- \EE f^{\veps,\veps'}_{\Xi} (\widetilde{T}^R) \big)}\Big)\leq \exp\Big(\frac{\theta^2\sigma_{R}^2 np}{8 }\Big).\end{equation}
Note that $\sigma_{R} \to 0$ as $R\to +\infty$. Thus, combining \eqref{conccond}, \eqref{conccondf}, and \eqref{controlLip} together with the fact that $L_\Xi\leq 1$ a.s., we deduce that in order to prove \eqref{concregularization} it suffices to show that 
\begin{equation} \label{claimeq} \lim_{R \to +\infty}\sup_{\veps,\veps'>0} \limsup_{n\to +\infty} \frac{1}{\sqrt{np}} \EE f_{\Xi}^{\veps,\veps'}(\widetilde{T}^R) =0.\end{equation}
Turning to do that using Seginer's theorem  \cite[Theorem 1.1]{Seginer} and a symmetrization argument \cite[Section 4]{HWX16} (see also \cite[Theorem 3.5]{LLV}), we know that 
\begin{equation} \label{expecf} \EE f^{\veps,\veps'}_\Xi(\widetilde{T}^R) \leq  \EE \|\widetilde{T}^R \circ \Xi\|  \lesssim \EE \max_{1\leq i\leq n} \|(\widetilde{T}^R \circ \Xi)_i \|. \end{equation}
%where $\kappa$ is a numerical constant. 
Using again the sub-Gaussian concentration property of $\widetilde{T}^R$, ensured by Lemma \ref{regularizelogSob}, together with the fact that $\Xi \mapsto \EE_{\widetilde{T}^R} \|(\widetilde{T}^R \circ \Xi)_i\|$ is a convex $\sigma_{R}$-Lipschitz function with respect to the Hilbert-Schmidt norm we find that for any $t>0$ and $i\in[n]$,
\begin{multline} \label{concconddeg}
\PP\big(  \| (\widetilde{T}^R \circ \Xi)_i\| - \EE \| (\widetilde{T}^R \circ \Xi)_i\| >t \big) 
\le  
\PP\big( \EE_{\widetilde{T}^R} \| (\widetilde{T}^R \circ \Xi)_i\| - \EE \| (\widetilde{T}^R \circ \Xi)_i\| >t/2 \big) \\
+ \EE\big[\PP_{\widetilde{T}^R} \big( \|(\widetilde{T}^R \circ \Xi )_i \| - \EE_{\widetilde{T}^R}  \|(\widetilde{T}^R \circ \Xi )_i \|> t/2 \big)\big] \leq  2 \exp\Big(-\frac{t^2}{8(c_{R}\vee \sigma_R^2)}\Big). 
\end{multline}
%On the other hand, $\Xi \mapsto \EE_{\widetilde{T}^R} \|(\widetilde{T}^R \circ \Xi)_i\|$ is a convex $\sigma_{R}$-Lipschitz function with respect to the Hilbert-Schmidt norm so that using again \cite[Theorem 8.6]{BLM}, we get for any $t>0$ and $i \in [n]$, 
%\begin{equation} \label{concdeg} \PP\big( \EE_{\widetilde{T}^R} \| (\widetilde{T}^R \circ \Xi)_i\| - \EE \| (\widetilde{T}^R \circ \Xi)_i\| >t \big) \leq \exp\Big(- \frac{t^2}{2 \sigma_{R}^2}\Big).\end{equation}
%Since $\sigma_{R} \to 0$ as $R\to +\infty$ and $c_{R}\geq 1$, we deduce that for $R$ large enough, $\sigma_{R}^2 \leq c_{R}$. Thus, combining \eqref{concconddeg} and \eqref{concdeg}, we get for any $R$ large enough, and $t>0$ and $i\in[n]$, 
%\[ \PP\big(  \| (\widetilde{T}^R \circ \Xi)_i\| - \EE \| (\widetilde{T}^R \circ \Xi)_i\| >t \big) \leq 2e^{- \frac{t^2}{8 c_{R}}}.\]
It is a classical fact that a sub-Gaussian tail such as \eqref{concconddeg} bound entails 
\begin{equation}\label{eq:subg-max-norm} \EE \max_{i\in [n]}  \big( \|(\widetilde{T}^R \circ \Xi )_i \| - \EE  \|(\widetilde{T}^R\circ \Xi )_i \|\big) =O\big(\sqrt{(c_{R}\vee \sigma_R^2 \log n})\big).
\end{equation} 
On the other hand, using Cauchy-Schwarz inequality we get $\max_{i \in [n] }\EE \big[\|(\widetilde{T}^R \circ \Xi)_i \| \big] \leq \sigma_{R}  \sqrt{np}$.
This observation together with \eqref{expecf} and \eqref{eq:subg-max-norm} 
%we get for all $R$ large enough,
%\[ \sup_{\veps,\veps'>0}\EE f^{\veps,\veps'}_\Xi(\widetilde{T}^R) =  O\big(\sqrt{c_{R} \log n} + \sigma_{R} \sqrt{np}\big),\]
%which 
gives the claim \eqref{claimeq} as $\lim_{R \to +\infty}\sigma_{R}=0$. %as $R\to +\infty$.
%
%We now show how to conclude the proof from \eqref{concregularization}. For any $R>0$, we have the inequality  $f^{\veps,\veps'}_\Xi(\widetilde{B}^R) -f^{\veps,\veps'}_\Xi(\widetilde{C}^R) \leq  \| \Gamma \circ \Xi\|/R$.
%Thus, using the exponential tightness of $\|\Gamma \circ \Xi\|$ shown in Lemma \ref{expotight}, we deduce that for any $t>0$, 
%\[ \lim_{R \to +\infty} \sup_{\veps,\veps'>0} \limsup_{n\to +\infty }\frac{1}{np} \log \PP\big( f^{\veps,\veps'}_\Xi(\widetilde{B}^R)- f^{\veps,\veps'}_\Xi(\widetilde{C}^R) >t\sqrt{np} \big) = -\infty.\]
%Together with \eqref{concregularization}, this ends the proof. 
\end{proof}

\subsection{Exponential equivalent of the eigenvalue-eigenvector equation}
%With the decomposition of the top eigenvector ${\bm u}$ into ${\bm u}^{\sf T} =({{\bm v}^{\veps}}^{\sf T},  {{\bm w}^{\veps}}^{\sf T})$, and the block-decomposition \eqref{decomp-Z} of $Z$, we have the following eigenvalue-eigenvector equation
%\begin{equation} \label{eievZ}\begin{cases}
%&Z_{J} {\bm v}^{\veps} + \widecheck{Z}_J^{\sf T} {\bm w}^{\veps}  = \lambda_Z {\bm v}^{\veps}\\
%& \widecheck{Z}_J {\bm v}^{\veps} + Z^{(J)} {\bm w}^{\veps}  = \lambda_Z {\bm w}^{\veps}.
% \end{cases}\end{equation}
%When $\lambda _Z \notin \mathrm{Sp}(H^{(J)})$, we can rewrite these equations as 
%\begin{equation} \label{eievZ}\begin{cases}
%&Z_{J} {\bm v}^{\veps} + \widecheck{Z}_J^{\sf T} {\bm w}^{\veps}  = \lambda_Z {\bm v}^{\veps}\\
%& {\bm w}^{\veps}  = (\lambda_Z -H^{(J)})^{-1} (\widecheck{Z}_J {\bm v}^{\veps} + B^{(J)} {\bm w}^{\veps}).
% \end{cases}\end{equation}
%Taking the inner product with ${\bm v}^{\veps}$ in the first equation and replacing ${\bm w}^{\veps}$ by its expression in the second equation, we get 
%\begin{equation} \label{eievmastereq} \langle {\bm v}^{\veps}, Z_J {\bm v}^{\veps} \rangle + \langle {\bm x}^\veps, (\lambda_Z -H^{(J)})^{-1} {\bm x}^{\veps}\rangle - \langle {\bm w}^\veps, B^{(J)} {\bm w}^{\veps} \rangle = \lambda_Z \|{\bm v}^{\veps}\|^2,\end{equation}
%where ${\bm x}^\veps =  (\widecheck{Z}_J {\bm v}^{\veps} + B^{(J)} {\bm w}^{\veps})$. 
Recall \eqref{eq-top-ei-sec2-2} and the definition of $J$ from \eqref{defJ}. The goal of this section is to find an exponential equivalent of the equation in \eqref{eq-top-ei-sec2-2} involving terms with tractable large deviation behaviours. More precisely, we will show the following proposition.

\begin{Pro} \label{expoequi:ineq}For any $t,\veps>0$, $R\geq 1$ and $\lambda>2$, define the event
\begin{equation} \label{eq:Om-eps-R-t}
\Omega_{\veps,R,t}:= \Big\{\langle {\bm v}^{\veps}, Z_{{J}} {\bm v}^{\veps} \rangle+m(\lambda) \Big( \sum_{k\in \gJ}{\bm u}_k^2  (\|(\widecheck{H}_{J})_{k}\|^2-1)_+ + \|\widecheck{B}_{ J} {\bm  v}^{\veps} \|^2\Big) \geq  \|{\bm v}^{\veps}\|^2/m(\lambda) -t\Big\},
\end{equation}
where $\gJ := \{ k \in J  : |{\bm u}_k|\geq \veps\}$. Then, for any $\delta>0$ such that $\lambda \geq 2+\delta$, 
\begin{equation} 
\lim_{R\to +\infty}\limsup_{\veps \to 0} \limsup_{n\to +\infty} \frac{1}{np} \log \PP\big( \Omega_{\veps,R, t}^c \cap \Theta_{\lambda,\delta} \big) = -\infty, \, \text{ where } \Theta_{\lambda,\delta} := \{\lambda_{H^{(J)}} \leq \lambda -\delta \leq \lambda \leq \lambda_Z\}.
\end{equation}
\end{Pro}
Proposition \ref{expoequi:ineq} shows that the large deviation upper tail event of $\lambda_Z$, intersected with the event that $\lambda_Z$ is away from ${\rm Spec}(H^{(J)})$,  enforces an inequality, at the large deviation scale, involving $\langle {\bm v}^{\veps}, Z_J {\bm v}^{\veps}\rangle$, which one can think as the `clique contribution', the `degrees' $\|(\widecheck{H}_J)_k\|^2$ in the network $H$, and $\|\widecheck{B}_J {\bm v}^{\veps}\|^2$.  %The last term in the inequation, $\|\widecheck{B}_J {\bm v}^{\veps}\|^2$, is a term that although not negligible, have sub-optimal deviations as 
We will see later that although the last term is non negligible it corresponds to a sub optimal scenario and thus will ultimately disappear from the final expression for the rate function. 

The proof of Proposition \ref{expoequi:ineq} require several ingredients. Proofs of those are carried out over the next few subsections.

\subsubsection{$B^{(J)}$ is spectrally negligible over completely delocalised vectors} In a first step, using that $J^c$ does not contain any high degree vertices with respect to the network $B$, we show that $\|B^{(J)} w\|=o(1)$ when $w$ is a completely delocalised unit vector, that is  $\|w\|_\infty \ll \sqrt{\log(1/p)/np}$. 
\begin{lemma}  
\label{nodevBJ}
For any $t >0$ and $\vep >0$, with $M_\veps$ as in Lemma \ref{quadra:deloc:unbounded},
%For any $\veps>0$, denote by $M_\veps=\eps^{-1}\sqrt{\log(1/p)}$. Then, for any $t>0$, 
\[ \lim_{R\to +\infty}\limsup_{\eps \to 0} \limsup_{n\to +\infty} \frac{1}{np}\log \PP\big(\sup_{w\in \mathcal{D}_{M_\veps/\sqrt{np}}} \|B^{(J)}w\|>t \big) = -\infty.\]
\end{lemma}
\begin{proof}
Let $w'\in \mathbb{B}^{J^c}$. Write $w'=x+y$ with $x_i:=w'_i\Car\{|w_i'|>\veps'\}$, for any $i\in J^c$,
where $\veps' >0$ is a parameter to be chosen later. Note that for any $w\in\mathbb{B}^{J^c}$, we have by Cauchy-Schwarz inequality,
\[ 
\langle x,B^{(J)} w\rangle \leq \sum_{i\in J^c} |x_i| \|B_i\|\leq \|x\|_1\max_{i\in J^c} \|B_i\|.
\]
%where $B^{(i)}$ denotes the $i^{\text{th}}$ column of $B$. 
Since $\|x\|_{2}\leq 1$ and for any $i\in J^c$ we have $|x_i|\geq \veps'$, we deduce that $\|x\|_1\leq 1/\veps'$. By construction (see \eqref{defJ}), $J^c$ does not contain any vertex of high degree with respect to the network $B$. In particular $\max_{i\notin J} \|B_i\|\leq \delta$. Thus $\langle x,B^{(J)} w\rangle \leq \delta/\veps' $.
So, we have proven that %for any $w\in\mathcal{D}_{M_\veps/\sqrt{np}}$,
\[ \sup_{w \in \mathcal{D}_{M_\veps/\sqrt{np}}} \|B^{(J)}w\|^2
\leq \delta/\veps'+\sup\big\{\langle y, Bw\rangle : y\in \mathcal{D}_{\veps'},w\in\mathcal{D}_{M_\veps/\sqrt{np}} \big\}.\]
Since $\delta =o(\eps)$, as $\eps \to 0$, we can find $\veps'>0$ such that $\delta \ll \veps' \ll \eps$.
%With this choice, the claim follows from 
Applying Lemma \ref{step1} with these choices the claim follows.
\end{proof}

\subsubsection{Concentration of the resolvent} The most important step to move from \eqref{eq-top-ei-sec2-2} to Proposition \ref{expoequi:ineq} is the following concentration bound.

\begin{lemma}\label{conc:unbounded} 
Let $t, \delta, \lambda$, and $\Theta_{\lambda, \delta}$ be as in Proposition \ref{expoequi:ineq}. Set $\PP_{\Theta_{\lambda,\delta}}(\cdot)= \PP( \cdot\cap \Theta_{\lambda,\delta})$. Then 
%For any $t,\delta>0$ and $\lambda>2+\delta$, 
\[\lim_{R\to +\infty}\limsup_{\veps \to 0} \limsup_{n\to +\infty} \frac{1}{np}\log\PP_{\Theta_{\lambda,\delta}}\Big( \langle {\bm x}^\veps, (\lambda_{Z^R}-H^{(J)})^{-1} {\bm x}^\veps\rangle - m(\lambda_{Z^R})\|{\bm x}^\veps\|^2> t \Big) = -\infty.\]
%where $\PP_{\Theta_{\lambda,\delta}}$ denotes the measure $\PP( .\cap \Theta_{\lambda,\delta})$ with $\Theta_{\lambda,\delta} = \{\lambda_{(H^R)^{(J)}} \leq \lambda -\delta \leq \lambda \leq \lambda_{Z^R}\}$.
\end{lemma}
The proof of Lemma \ref{conc:unbounded} builds on the idea developed in Section \ref{sec:concflat} that convex functions with flat gradients of sparse Wigner matrices with bounded entries satisfy improved sub-Gaussian upper tail bounds. %More precisely, we will show the following lemma. 
%\corAB{It may further be added that} Lemma \ref{conc:unbounded} shares similarities to Proposition \ref{concresolvfinalbounded} in the bounded case.  Nevertheless, 
Observe that %in order to} deal with the new equation \corAB{\eqref{eq-top-ei-sec2-2}}, 
in the lemma above we have restricted our concentration estimate to the event where there is a gap in the spectrum, in the sense that $\lambda_{H^{(J)}}+\delta \leq \lambda_Z$ for some $\delta>0$. %We could not afford such a restriction in our treatment of the bounded case. Now that 
Since the bounded case is settled we can work on this event %make this assumption since 
as we can compute the probability of the complement event $\{\lambda_{H^{(J)}}>\lambda-\delta\}$.

 Turning to prove Lemma \ref{conc:unbounded} we first show  that if $Y$ is a sparse Wigner matrix, in the sense of Definition \ref{defmodel}, with bounded entries, then $f_{\lambda, x}(Y):= \langle x, (\lambda-Y)^{-1}x\rangle$ (note that the domain of $f_{\lambda, x}$ is $\cH_n^\lambda$) concentrates when $x$ belongs to a set  of `good' directions, to be defined, and that the concentration can be made uniform in these directions if one considers vectors close to a given low dimensional subspace. 
To this end, define for any $\lambda \in \RR$, $\eps,r>0$ and $K\in \mathcal{H}_n^{\lambda}$ %such that $\lambda_K< \lambda$,
\begin{equation} \label{defE} 
\mathcal{V}_{\lambda}^{\eps,r}(K) := \big\{ x \in \mathbb{B}^n :  \| (\lambda-K)^{-1}x\|_\infty \leq \veps,  \  \| (\lambda-K)^{-1}x\|_2 \leq r
\big\}. \end{equation}
Since $\nabla_K f_{\lambda, x}(K) = (\lambda-K)^{-1} x ((\lambda-K)^{-1} x)^{\sf T}$, the set $\mathcal{V}_{\lambda}^{\eps,r}(K)$ consists of directions in which the gradient at $K$ of the map $K' \mapsto f_{\lambda, x}(K')$ is flat, in the sense that it belongs to $r^2 \mathcal{F}_{\veps/r}^1$ (see \eqref{flatmatrices} for a definition). 
Finally, denote by $\mathfrak{S}$ the set of  subspaces $\mathcal{W}$ of $\RR^n$ such that 
\begin{equation} \label{condL1} \mathrm{dim}(\mathcal{W})\leq  np/\log(1/p) \text{ and } \ \|x\|_1/\|x\|_2\leq  np/\sqrt{\log(1/p)}=:a_n, \, \forall x \in \mathcal{W}. \end{equation}
With this notation we have the following proposition. 
\begin{Pro}\label{concresolvunif}Let $\lambda,r>0$, and $\delta \in (0,1)$ be such that $\lambda > 2+\delta$. Let $Y$ be a sparse Wigner matrix with bounded entries. Fix $\mathcal{W} \in \mathfrak{S}$. For $\lambda', t,\eta >0$ and $K \in \cH^{\lambda'}_n$ define
\[
F_{\lambda', \eta}^{\eps, r, \mathcal{W}}(K):= \sup_{x  \in \mathcal{G}_{\lambda',\eta}^{\eps,r, \mathcal{W}}(K)} \big(f_{\lambda', x}(K) -m(\lambda') \|x\|^2\big) \quad \text{ and } \quad \mathcal{G}_{\lambda',\eta}^{\eps,r, \mathcal{W}}(K) := \{x\in \mathcal{V}_{\lambda'}^{\eps,r}(K) : d(x,\mathcal{W}) \leq \eta\},
\]
where $d(y,E) := \inf_{w\in E} \|y-w\|$ for any $y \in \RR^n$ and $E \subset \RR^n$. Then, for any $t>0$ and $\eta >0$ small enough
\[\lim_{\eps \to 0} \limsup_{n\to +\infty} \frac{1}{np} \log \sup_{\mathcal{W} \in \mathfrak{S}}  \PP\Big( \sup_{\lambda ' \geq \lambda}F_{\lambda', \eta}^{\eps, r, \mathcal{W}}(Y) >t, \lambda_{Y} <\lambda -\delta \Big)= - \infty.\]
%where $\mathcal{G}_{\lambda',\eta}^{\eps,r}(Y) = \{x\in \mathcal{V}_{\lambda'}^{\eps,r}(Y) : d(x,\mathcal{W}) \leq \eta\}$, and $d$ is defined as 
%\[ \forall x \in \RR^n, \forall E\subset \RR^n, \ d(x,E) = \inf_{w\in E} \|x-w\|.\]
\end{Pro}
%Finally \corAB{the observation} that all our tail estimates are uniform in $\mathcal{W}\in \mathfrak{S}$ \corAB{yields the proposition}. 
Let us add here that the probability bound being uniform over $\mathcal{W}\in \mathfrak{S}$ will be important later when we apply Proposition  \ref{concresolvunif} to a random minor of $H$.

To prove Proposition \ref{concresolvunif} 
%we will start from the concentration inequality for convex functions with flat gradients that we proved in Proposition \ref{strongconc} and the isotropic local law of Proposition \ref{prop:conc-loc-law}, to 
we first show that for a given $\lambda>2$ and $x \in  \mathcal{V}_{\lambda}^{\eps,r}(Y)$, the probability that $f_{\lambda, x}(Y)$ is greater than $m(\lambda)\|x\|^2$ decreases at the exponential scale faster than $np$. %Owing to the smoothness of our functionals in the parameters $x$ and $\lambda$ we  \corAB{can then} perform \corAB{a} double net argument to obtain the desired uniform concentration statement. Finally \corAB{the observation} that all our tail estimates are uniform in $\mathcal{W}\in \mathfrak{S}$ \corAB{yields the proposition}. \corAB{Let us add here that} this last fact will be important when we apply Proposition  \ref{concresolvunif} to a random minor of $H$.  
To state and prove this concentration result we need to regularise $f_{\lambda, x}$ which in turn requires some more notation. %Let $\lambda \in \RR$ and $x\in\mathbb{B}^n$.
%Define for any $K\in\mathcal{H}_n$ such that $\lambda_K<\lambda$, $f_{\lambda,x}(K) = \langle x, (\lambda-K)^{-1}x\rangle$.
Recall \eqref{defE} and for any $\delta,\eps,r>0$ define $E_{\lambda,x}^{\eps,r} := \big\{K \in \mathcal{H}^{\lambda-\delta}_n :  x \in \mathcal{V}_\lambda^{\eps,r}(K)
\big\}$.
%where is defined in \eqref{defE}. 
In words $E_{\lambda,x}^{\eps,r} $ is the set of symmetric matrices for which $x$ is a `good' direction for the resolvent at $\lambda$.
%For any $K\in E_{\lambda,x}^{\eps,r}$, let $\zeta_{\lambda,x}(K) \in \mathcal{H}_{n}$ be the gradient of $f_{\lambda,x}$ at $K$, that is
%\begin{equation} \label{defzeta} \zeta_{\lambda,x}(K)  = (\lambda -K)^{-1} x \big((\lambda-K)^{-1} x\big)^{\sf T}.\end{equation}
Extend $f_{\lambda,x}$ to the whole space $\mathcal{H}_{n}$ by setting 
\begin{equation} \label{defftilde} \widetilde{f}^{\eps,r}_{\lambda,x}(K) := \sup_{K'\in E_{\lambda,x}^{\eps,r}} \big\{f_{\lambda,x}(K') + \langle \nabla f_{\lambda,x}(K'), K-K'\rangle \big\}, \  \forall K \in \mathcal{H}_{n}.
\end{equation}
By \cite[Exercise I.I.15]{Bhatia}, we know that $f_{\lambda,x}$ is convex on its domain. Therefore, $\widetilde{f}_{\lambda,x}^{\eps,r}$ coincides with $f_{\lambda,x}$ on $E_{\lambda,x}^{\veps,r}$. Moreover, $\widetilde{f}_{\lambda,x}^{\eps,r}$ is by construction convex and $r^2$-Lipschitz with respect to the Hilbert-Schmidt norm on $\mathcal{H}_{n}$ since $\|\nabla f_{\lambda,x}(K')\|_2\leq r^2$ for any $K'\in E_{\lambda,x}^{\veps,r}$. Note that although not explicitly mentioned, $E_{\lambda,x}^{\veps,r}$ and $\widetilde{f}_{\lambda,x}^{\veps,r}$ both depend implicitly on $\delta$.

 Applying the concentration inequality for functions with flat gradients of Proposition \ref{strongconc} and using the isotropic local law we get  the following lemma. Its proof is deferred to Section \ref{sec:loc-law}.
Recall the definition of $a_n$ from \eqref{condL1}. 
\begin{Lem}\label{concresolvent} Let $\delta>0$ and $\lambda>2+\delta$. Assume $Y$ is a sparse Wigner matrix with bounded entries. Then for any $t,r>0$,
\[ \lim_{\veps \to 0} \limsup_{n\to +\infty}\frac{1}{np} \log \sup_{\|x\|_1\leq a_n \atop x \in 2\mathbb{B}^n}\PP\big( \widetilde{f}^{\eps,r}_{\lambda,x}(Y) - m(\lambda)\|x\|^2 >t \big) =-\infty. \]
%where $\widetilde{f}^{\eps,r}_{\lambda,x}$ is defined in \eqref{defftilde} and $a_n = np/\sqrt{\log(1/p)}$.
\end{Lem}

%\begin{proof}Without loss of generality, we can assume that the entries of $Y$ are bounded by $1$. As $E_{\lambda,x}^{\eps,r}$ is compact and both $f_{\lambda,x}$ and $\zeta_{\lambda,x}$ are continuous, we deduce that the supremum in \eqref{defftilde} is achieved.
%Thus for any $K,K' \in \mathcal{H}_{n}$, 
%\[ \widetilde{f}^{\eps,r}_{\lambda,x}(K) -\widetilde{f}^{\eps,r}_{\lambda,x}(K') \leq \langle \zeta_{\lambda,x}(\widehat{K}),K-K'\rangle,\]
%where $\widehat{K} \in E^{\eps,r}_{\lambda,x}$ depends on $K$ in a measurable way and $\zeta_{\lambda,x} (\widehat{K})$ is as in \eqref{defzeta}. Since $\widehat{K} \in E^{\eps,r}_{\lambda,x}$, it follows that $\zeta_{\lambda,x}(\widehat{H})$ is a rank one flat matrix. More precisely,  $\zeta_{\lambda,x}(\widehat{K}) \in r^2 \mathcal{F}^1_{\eps/r}$, where $\mathcal{F}^k_{\veps'}$ is defined in \eqref{flatmatrices}. Applying Proposition \ref{strongconc} to $\widetilde{f}^{\eps,r}_{\lambda,x}/r^2$ and using the fact that $np \gg \log n$, we get the claim. 
%\end{proof}

Combining  Lemma \ref{concresolvent}  together with a net argument on the vectors in $\mathcal{V}_{\lambda}^{\eps,r}(Y)$ close to a low dimensional subspace $\mathcal{W} \in \mathfrak{S}$, we get the following proposition. We draw the attention of the reader to the fact that $\mathcal{V}_{\lambda}^{\eps,r}(Y)$ being a set of a random vectors the net argument will not be completely straightforward.
\begin{Pro}\label{concresolvfinal}
With the same setting and notation of Proposition \ref{concresolvunif} %we have
for any $t>0$, %For any $t,r,\delta, \lambda>0$ such that $\lambda>2+\delta$ and $\eta$ small enough
\[ \lim_{\eps \to 0} \limsup_{n\to +\infty} \frac{1}{np} \log \sup_{\mathcal{W} \in \mathfrak{S}} \PP\big(F_{\lambda, \eta}^{\eps, r, \mathcal{W}}(Y) > t, \lambda_{Y} <\lambda -\delta \big)= - \infty.\]
\end{Pro}

\begin{proof}
Starting from Lemma  \ref{concresolvent} and fixing $\mathcal{W} \in \mathfrak{S}$ we first perform %for a given subspace $\mathcal{W} \in \mathfrak{S}$ 
a net argument on $x\in \mathbb{B}_\mathcal{W}:=\mathcal{W}\cap \mathbb{B}^n$. Without loss of generality assume that $\delta \in (0,1)$.  Since $\|(\lambda -K)^{-1} \| \leq 1/\delta$ for any $K \in \mathcal{H}^{\lambda-\delta}_n$,  we have for all $x,x'\in \mathbb{B}^n$, 
\begin{equation} \label{stabilityE} E_{\lambda,x}^{\eps,r}\subset E_{\lambda,x'}^{\veps + s',r +s'} \text{ and } \
 |f_{\lambda,x}(K) -f_{\lambda,x'}(K) |\leq 2s', \ \forall K \in \mathcal{H}^{\lambda-\delta}_n, 
\end{equation}
with $s' = \delta^{-1} \|x-x'\|_2$. For any $\eta\in(0,1)$ let $\mathcal{N}_\eta(\mathcal{W})$ be an $\eta$-net for the $\ell^2$-norm of $2\mathbb{B}_{\mathcal{W}}$. We claim that the two stability estimates  of \eqref{stabilityE}  yield for any $K\in \mathcal{H}^{\lambda-\delta}_n$,
\begin{equation} \label{stability} 
F_{\lambda, \eta}^{\eps, r, \mathcal{W}}(K)\leq \sup_{x' \in \mathcal{N}_\eta(\mathcal{W})} \big(\widetilde{f}_{\lambda,x'}^{\eps+s,r+s}(K)-m(\lambda) \|x'\|^2\big) + 12s,\end{equation}
where $s= 2\delta^{-1}\eta$.
% Indeed, if  for a matrix $K\in \mathcal{H}_{\lambda-\delta}$ the set $ \mathcal{G}_{\lambda,\eta}^{\veps,r}(K) $ is empty, then there is nothing to prove. Assume from now on that this set is nonempty and
%let $x\in \mathcal{G}_{\lambda,\eta}^{\veps,r}(K)$. 
Indeed, if $K\in\mathcal{H}^{\lambda-\delta}_n$ and  $x\in \mathcal{G}_{\lambda,\eta}^{\veps,r, \mathcal{W}}(K)$, then this means that $K\in E_{\lambda,x}^{\eps,r}$ and that there exists $x''\in \mathcal{W}$ such that $\|x-x''\|\leq \eta$. Since $\|x\|\leq 1$ this entails that $\|x''\|\leq 2$ as $\eta\leq 1$, and we deduce that there exists $x'\in \mathcal{N}_\eta(\mathcal{W})$ such that $\|x-x'\|\leq 2\eta$. Using \eqref{stabilityE} we get
\[ f_{\lambda,x}(K)\leq f_{\lambda,x'}(K) + 4\delta^{-1}\eta.\]
Moreover,  $K\in E_{\lambda,x'}^{\veps+s,r+s}$ (by \eqref{stabilityE} again) %with $s= 2\delta^{-1}\eta$, 
which yields $f_{\lambda,x'}(K) = \widetilde{f}_{\lambda,x'}^{\eps+s,r+s}(K)$. Noting that $m(\lambda) \leq 1$ and $\delta \in (0,1)$, this ends the proof of the claim \eqref{stability}. 

By \cite[Corollary 4.1.15]{AGM} we know that for any $\eta \in (0,1)$ we can find such a set $\mathcal{N}_\eta(\mathcal{W})$ with
\[ \log | \mathcal{N}_{\eta}(\mathcal{W}) | \leq \mathrm{dim}(\mathcal{W}) \log(3/\eta) = o(np),\]
where we used the fact that  $\mathrm{dim}(\mathcal{W}) \leq np/\log(1/p)$.
So using a union bound and Lemma  \ref{concresolvent}, we get for any $t>0$, and $\eta>0$ small enough %denoting by $s=2\delta^{-1}\eta$,
\begin{equation} \label{unionbound} \lim_{\eps \to 0} \limsup_{n\to +\infty} \frac{1}{np} \log \sup_{\mathcal{W} \in \mathfrak{S}} \PP\big(\sup_{x'\in \mathcal{N}_{\eta}(\mathcal{W}) }\big(\widetilde{f}^{\eps+s,r+s}_{\lambda,x'}(Y)-m(\lambda) \|x'\|^2\big) >t)= - \infty.\end{equation}
Combining this with \eqref{stability} ends the proof. 
\end{proof}
Finally we perform a net argument on $\lambda$  to obtain the uniform concentration result of Proposition  \ref{concresolvunif}.

\begin{proof}[Proof of Proposition  \ref{concresolvunif}]
Fix $\lambda>2+\delta $ and $K\in\mathcal{H}^{\lambda -\delta}_n$. Let $\mathcal{N}_\eps$ be the standard $\eps$-net of the interval $(\lambda, \lambda +\eps^{-1})$. Pick $\lambda_1' \ge \lambda$ and $\lambda'_2 \in \mathcal{N}_\eps$ such that $0\leq \lambda_1' -\lambda_2' \leq \eps$.
%Define for any $\lambda' > \lambda$,
%\[  F_{\lambda',\eta}^{\eps,r}(K) =\sup\big\{f_{\lambda',x}(K) - m(\lambda')\|x\|^2 :  x \in \mathcal{G}_{\lambda',\eta}^{\veps,r}(K) \big\},\]
%where $f_{\lambda',x}$ is as in \eqref{defflambda}. 
%Let $\lambda_1'\geq \lambda_2' \geq \lambda$. 
Since $\|(\lambda_1'-K)^{-1} - (\lambda_2'-K)^{-1}\|\leq q:= \delta^{-2} \eps$ we have the inclusion
\[ \mathcal{V}_{\lambda_1'}^{\eps,r}(K) \subset \mathcal{V}_{\lambda_2'}^{\eps+q,r+q}(K),\]
%where $s =  |\lambda_2'-\lambda_1'| \delta^{-2}$, 
and as a consequence $ \mathcal{G}_{\lambda'_1,\eta}^{\veps,r, \mathcal{W}}(K)  \subset  \mathcal{G}_{\lambda'_2,\eta}^{\veps+q,r+q, \mathcal{W}}(K) $. Therefore using the facts that $\lambda' \mapsto f_{\lambda',x}(K)$ is non increasing on $(\lambda,+\infty)$ for any $x\in\RR^n$ %, we get
%\[  \sup \{f_{\lambda'_1,x}(K) : x\in \mathcal{G}_{\lambda'_1,\eta}^{\veps,r}(K)\} \leq  \sup\{f_{\lambda'_2,x}(K) : x\in \mathcal{G}_{\lambda'_2,\eta}^{\veps+s,r+s}(K)\}.\]
and that $m$ is $\delta^{-2}$-Lipschitz on  $(2+\delta,+\infty)$ we have 
\begin{equation*} %\label{stabF} 
\sup_{\lambda ' \in (\lambda,\lambda +\eps^{-1})} F_{\lambda',\eta}^{\eps,r, \mathcal{W}}(Y) \leq \sup_{\lambda' \in \mathcal{N}_\eps} F_{\lambda',\eta}^{\eps+q,r+q, \mathcal{W}}(Y) +  q,
\end{equation*}
on the event $Y \in \mathcal{H}^{\lambda-\delta}_n$. %such that for any $\lambda'_1 \in (\lambda,\lambda +\eps^{-1})$, there exists $\lambda_2' \in \mathcal{N}_\eps$ such that $0\leq \lambda_1' -\lambda_2' \leq \eps$. One can find such a set with cardinal 
%\corAB{So} $|\mathcal{N}_\eps| = O(1/ \eps^2)$. 
Using Proposition \ref{concresolvfinal}, a union bound, and the fact that $|\mathcal{N}_\eps| = O(1/ \eps^2)$ we get for any $t>0$ and $\eta$ small enough
\begin{multline*}
\limsup_{\eps \to 0} \limsup_{n\to +\infty} \frac{1}{np }\log \sup_{\mathcal{W} \in \mathfrak{S}}  \PP\big( \sup_{\lambda' \in (\lambda,\lambda+\eps^{-1})} F_{\lambda',\eta}^{\eps,r, \mathcal{W}}(Y) >t, Y \in \mathcal{H}^{\lambda-\delta}_n\big) \\
\le \lim_{\eps \to 0} \limsup_{n\to +\infty} \frac{1}{np }\log \sup_{\mathcal{W} \in \mathfrak{S}} \PP\big( \sup_{\lambda' \in \mathcal{N}_\eps} F_{\lambda',\eta}^{\eps +q,r+q, \mathcal{W}}(Y) >t, Y \in \mathcal{H}^{\lambda-\delta}_n\big) = -\infty.
\end{multline*}
%where $q = \delta^{-2} \veps$. 
%Now, from \eqref{stabF} we deduce that if $Y\in \mathcal{H}^{\corAB{\lambda-\delta}}$, then
%\[ \sup_{\lambda ' \in (\lambda,\lambda +\eps^{-1})} F_{\lambda',\eta}^{\eps,r, \corAB{\mathcal{W}}}(Y) \leq \sup_{\lambda' \in \mathcal{N}_\eps} F_{\lambda',\eta}^{\eps+q,r+q, \corAB{\mathcal{W}}}(Y) +  q.\]
%Therefore, we get that for any $t>0$ and $\eta$ small enough
%\[ \lim_{\eps \to 0} \limsup_{n\to +\infty} \frac{1}{np }\log \sup_{\mathcal{W} \in \mathfrak{S}}  \PP\big( \sup_{\lambda' \in (\lambda,\lambda+\eps^{-1})} F_{\lambda',\eta}^{\eps,r, \corAB{\mathcal{W}}}(Y) >t, Y \in \mathcal{H}^{\corAB{\lambda-\delta}}\big) = -\infty.\]
On the other hand, observe that on the event $Y \in \mathcal{H}^{\lambda-\delta}_n$ we have
\[ \sup_{\lambda ' \geq \lambda +\eps^{-1}} F_{\lambda',\eta}^{\eps,r, \mathcal{W}}(Y) \leq \sup_{\lambda' \geq \lambda+\eps^{-1}}\sup_{x \in \mathbb{B}^n }f_{\lambda',x}(Y)  \leq \eps.\]
The last two observations end the proof.
\end{proof}

We can now give a proof of Lemma \ref{conc:unbounded}. 
\begin{proof}[Proof of Lemma \ref{conc:unbounded}] 
Recall that $a_n =np/\sqrt{\log(1/p)}$ and let $\wt{\mathcal{N}}:=\{\mathcal{J} \subset [n]: |\mathcal{J}|\leq 2\eps^2 np/\log(1/p)\}$. 
Define for any $t,r,\veps,\eta>0$, $\mathcal{J} \in \wt{\mathcal{N}}$, and $\mathcal{W}$ a subspace of $\RR^{\mathcal{J}^c}$, 
\[ \mathcal{F}_{t,\mathcal{W}}^{\mathcal{J},\veps,\eta} := \Big\{\sup_{\lambda' \geq \lambda} 
F_{\lambda', \eta}^{\eps, r, \mathcal{W}}(H^{(\mathcal{J})})
%\sup_{x  \in \mathcal{G}_{\lambda',\eta}^{\eps,r}(H^{(\mathcal{J})}) }\big[\langle x, (\lambda'-H^{(\mathcal{J})})^{-1}x  \rangle -m(\lambda') \|x\|^2\big] 
>t, \lambda-\delta>\lambda_{H^{(\mathcal{J})}}\Big\}.\]
%where $\mathcal{G}_{\lambda',\eta}^{\eps,r}(H^{(\mathcal{J})})=\{x \in  \mathcal{V}_{\lambda'}^{\eps,r}(H^{(\mathcal{J})}) : d(x,\mathcal{W})\leq \eta \}$ and $\mathcal{V}_{\lambda'}^{\eps,r}(H^{(\mathcal{J})})$ is defined in \eqref{defE}.
Using that $\log |\wt{\mathcal{N}}| =O(\veps^2 np)$ and applying Proposition \ref{concresolvunif} to $H^{(\mathcal{J})}/\sigma_R$, where $\sigma_R^2 := \EE H_{1,2}^2 =1+o_R(1)$ and $\mathcal{J}\in \wt{\mathcal{N}}$, we get that for $\eta$ small enough and $R$ large enough
 \begin{equation}\label{concgooddir}\lim_{\eps \to 0} \limsup_{n\to +\infty} \frac{1}{np} \log \sup_{\mathcal{J} \in \wt{\mathcal{N}}}\sup_{\mathcal{W} \in \mathfrak{S}_{\mathcal{J}}}  \PP(\mathcal{F}_{t,\mathcal{W}}^{\mathcal{J},\veps,\eta} )= - \infty,\end{equation}
 where $\mathfrak{S}_{\mathcal{J}}$ is the set of subspaces $\mathcal{W}$ of $\RR^{\mathcal{J}^c}$ satisfying \eqref{condL1}. 
%where $\PP_{H^{(\mathcal{J)}}}$ denotes the probability with respect to $H^{(\mathcal{J})}$ and $\mathfrak{S}$ is defined in \eqref{condL1}.
Now define the event $\mathcal{G}_{\mathcal{J}} := \{\sum_{i\in[n],j\in \mathcal{J}} (\xi_{i,j} +\xi_{i,j}') \leq a_n^2\}$. Using a binomial tail bound together with the fact that $np |\mathcal{J}|\leq 2\veps^2 a_n^2$ for any $\mathcal{J} \in \wt{\mathcal{N}}$ and $a_n^2\gg np$, we obtain that for $\veps$ small enough
\begin{equation} \label{Gwhp}   \limsup_{n\to +\infty}\frac{1}{np} \log\sup_{\mathcal{J} \in \wt{\mathcal{N}}} \PP( \mathcal{G}_{\mathcal{J}}^c) =-\infty.\end{equation}
On the event ${G}_{\mathcal{J}}$, $\widecheck{Z}_\mathcal{J}$ has at most $a_n^2$ non zero entries. It follows that for any $v\in \RR^{\mathcal{J}}$, $|\supp(\widecheck{Z}_\mathcal{J}v)| \le a_n^2$ and therefore $\|\widecheck{Z}_\mathcal{J}v\|_1 \leq  a_n \|\widecheck{Z}_\mathcal{J} v\|_2$. Besides ${\rm rank}(\widecheck{Z}_\mathcal{J}) \le |\mathcal{J}|\leq np/\log(1/p)$ for $\veps$ small enough. Thus the subspace $\mathcal{W}_{\mathcal{J}} := \mathrm{ran}(\widecheck{Z}_\mathcal{J}) \in \mathfrak{S}_\mathcal{J}$ on $\mathcal{G}_{\mathcal{J}}$. As the tail estimate \eqref{concgooddir} is uniform in $\mathcal{W} \in \mathfrak{S}_\mathcal{J}$ and $(H^{(\mathcal{J})}, \widecheck{Z}_\mathcal{J})$ are independent,
%, we can integrate \eqref{concgooddir} on the event $G_{\mathcal{J}}$, 
the estimate \eqref{Gwhp} together with Lemma \ref{cardJ}, and a union bound yield for any $t >0$ and $\eta$ small enough
%\[ \lim_{\eps \to 0} \limsup_{n\to +\infty} \frac{1}{np} \log  \PP(F_{t,\mathcal{W}_{\mathcal{J}}}^{\mathcal{J},\veps,\eta})= - \infty.\]
%Since there is at most $e^{O(\veps^2 np)}$ subsets $\mathcal{J}\subset [n]$ of size at most $2\veps^2 np/\log(1/p)$ by \eqref{entropyJ}, we get by a union bound and  Lemma \ref{cardJ}, that for any $t>0$ and $\eta$ small enough
\begin{equation} \label{expoequiveq} 
\lim_{\eps \to 0} \limsup_{n\to +\infty} \frac{1}{np} \log  \PP(\mathcal{F}_{t,\mathcal{W}_{J}}^{J,\veps,\eta} ) 
\le \lim_{\eps \to 0} \limsup_{n\to +\infty} \frac{1}{np} \log  \sup_{\mathcal{J} \in \wt{\mathcal{N}}} \PP(\mathcal{F}_{t,\mathcal{W}_{\mathcal{J}}}^{\mathcal{J},\veps,\eta}) 
+  \lim_{\eps \to 0} \limsup_{n\to +\infty} \frac{1}{np} \log |\wt{\mathcal{N}}| = - \infty.
\end{equation}
%, where $M_\veps=\veps^{-1}\sqrt{\log (1/p)}$. 
Recall that ${\bm x}^\veps = \widecheck{Z}_J{\bm v}^{\veps} + B^{(J)} {\bm w}^{\veps}$, where by definition ${\bm w}^\veps \in \mathcal{D}_{M_\veps/\sqrt{np}}$ and fix $\kappa\geq 1$. When $\lambda_Z>\lambda_{H^{(J)}}$ obviously one has $\lambda_Z \notin {\rm Spec}(H^{(J)})$. Hence, on the event where $\|{\bm x}^\veps\| \le \kappa$, by \eqref{eq-top-ei-sec2-2} we have $\kappa^{-1} {\bm x}^\veps \in \mathcal{V}_{\lambda_Z}^{\veps, 1}(H^{(J)})$ all large $n$ (recall \eqref{defE}). Therefore, upon denoting
\begin{equation}\label{eq:dfn-G-cK}
\Gamma_{t,\veps} := \{f_{\lambda_Z, {\bm x}^\eps}(H^{(J)})- m(\lambda_Z)\|{\bm x}^\veps\|^2\leq  t, \lambda_Z>\lambda_{H^{(J)}}\} \end{equation}
and 
\begin{equation}  \label{eq:dfn-G-cK2} \mathcal{K}_{\veps,\eta}^\kappa := \{\|\widecheck{Z}_J\|\leq \kappa/2, \ %\sup_{w \in \mathcal{D}_{M_\veps/\sqrt{np}}} 
\|B^{({J})}{\bm w}^\veps\|\leq \eta \},
\end{equation}
we deduce from \eqref{expoequiveq} and the fact that on $\mathcal{K}_{\veps,\eta}^\kappa$, for $\eta  \in (0,1/2]$, we have $\|{\bm x}^\veps\|\leq \kappa$ and  $d({\bm x}^\veps, \mathcal{W}_J) \leq \eta$, that for any $t>0$, $\kappa\geq 1$ and $\eta>0$ small enough,
%we have $(\lambda_Z-H^{(J)})^{-1} {\bm x}^\veps = {\bm w}^{\veps}$. Therefore, for $n$ large enough, $\|(\lambda_Z -H^{(J)})^{-1} {\bm x}^\veps\|_\infty \leq M_\veps/\sqrt{np} \leq \veps$ and $\|(\lambda_Z-H^{(J)})^{-1}{\bm x}^\veps\|_2\leq 1$. Moreover, on the event $\corAB{\mathcal{K}}_{\veps,\eta}^\kappa$, we have for $0<\eta \leq  \kappa$, $\|{\bm x}^\veps\|\leq 2\kappa$ and  $d({\bm x}^\veps, \mathcal{W}_J) \leq \eta$. Thus, by \eqref{expoequiveq} we obtain for any $\kappa>0$ and $\eta$ small enough
\[ \lim_{\eps \to 0} \limsup_{n\to +\infty} \frac{1}{np} \log  \PP_{\Theta_{\lambda,\delta}}(\Gamma_{t,\veps}^c \cap \mathcal{K}_{\veps,\eta}^\kappa )= - \infty.
\]
 %where %$\Gamma_{t,\veps} := \{\corAB{f_{\lambda_Z, {\bm x}^\eps}(H^{(J)})}
 %\langle {\bm x}^\veps, (\lambda_Z-H^{(J)})^{-1} {\bm x}^\veps\rangle 
 %- m(\lambda_Z)\|{\bm x}^\veps\|^2\leq  t, \lambda_Z>\lambda_{H^{(J)}}\}$.
%Using %that $\|Z\|$ is exponentially tight
Finally using Lemmas \ref{expotightZ} and \ref{nodevBJ} (as ${\bm w}^\veps \in \mathcal{D}_{M_\veps/\sqrt{np}}$) we get the claim. 
%
%\corAB{Next, for $\kappa>0$ define} $\corAB{\mathcal{K}}_{\veps,\eta}^\kappa := \{\|\widecheck{Z}_J\|\leq \kappa, \ \sup_{w \in \mathcal{D}_{M_\veps/\sqrt{np}}} \|B^{({J})}w\|\leq \eta \}$. 
\end{proof}

\subsubsection{Exponential equivalent of  $\|\widecheck{Z}_{\mathcal{J}}v\|^2$}
We now move to the analysis of the large deviation upper tail of $\|{\bm x}^\veps\|$, where ${\bm x}^\veps = \widecheck{Z}_J{\bm v}^{\veps} +B^{(J)} {\bm w}^\veps$. As $\|B^{(J)} {\bm w}^\veps\| =o(1)$ by Lemma \ref{nodevBJ}, it boils down to understand the deviations of $\|\widecheck{Z}_J {\bm v}^\veps\|$. %Since $J$ carries no entropy we can \corAB{essentially} assume \corAB{that} both $J$ and ${\bm v}^\veps$ are deterministic. \corAB{Thus it suffices to} prove the following exponential equivalent. 

\begin{Pro}\label{GD-norme-weight-unbound}
Fix $t \in (0,1)$. Let $\Upsilon_{t,\veps}:=\Upsilon_{t,\veps}({\bm J}, {\bm v}^\veps)$, where for any $\cJ \subset [n]$ and $v \in \mathbb{B}^\cJ$
%$\mathcal{J} \subset [n]$ such that $|\mathcal{J}|\leq \corAB{2} \veps^2 np /\log(1/p)$ and $v\in\mathbb{S}^{\mathcal{J}}$. For any $t>1$, 
\[
\Upsilon_{t,\veps}(\cJ, v):=\big\{\|\widecheck{Z}_{\cJ} {v}\|^2 \leq \sum_{k\in \mathcal{I}_v} v_k^2 (\|\widecheck{H}_\cJ^{(k)}\|^2-1)_+  +\|\widecheck{B}_\cJ{v}\|^2 +\|{v}\|^2+t\big\}, \text{ and } \mathcal{I}_v:=\{k : |v_k|\geq \veps\}. 
\]
Then 
\begin{equation}\label{eq:GD-norme-weight-unbound}
\limsup_{\veps \to 0} \limsup_{n\to +\infty} \frac{1}{np} \log  \PP (\Upsilon_{t,\veps}(J, {\bm v}^\veps)^c)= \limsup_{\veps \to 0} \limsup_{n\to +\infty} \frac{1}{np} \log  \sup_{\cJ \subset [n]} \sup_{v \in \mathbb{B}^\cJ}\PP (\Upsilon_{t,\veps}(\cJ, v)^c) = -\infty.
\end{equation}
%\big( \|\widecheck{Z}_{\mathcal{J}}v\|^2 > \sum_{k\in \mathcal{I}_v} v_k^2 (\|\widecheck{H}_{\mathcal{J}}^{(k)}\|^2-1)_+ + \|\widecheck{B}_{\mathcal{J}} v\|^2+t \big)=-\infty,\]
%where $\widecheck{H}_J^{(k)}$ denotes the vector $(H_{i,k})_{i \notin J}$ and $\mathcal{I}_v=\{k\in \mathcal{J} : |v_k|\geq \veps\}$. 
\end{Pro}

%This result is the analogue of Proposition \ref{GD:rectanglebounded} but in the unbounded setting. As already noticed in the bounded case, the rate function involved in the above proposition corresponds to the square of the $\ell^2$-norm of a column vector of the matrix $Z$. In other words, $h_L$ is the rate function of the large deviation of the ``degree'' of a vertex in the weighted network defined by $Z$. 
%
%When $h_L$ is strictly convex, which corresponds to the case where $x_\star =+\infty$, one can further say that the upper tail large deviation event of $\|\widecheck{Z}_J\|$ is dominated by the large deviation of the maximal degree in $J$. In the case where $h_L$ is not strictly convex, there is a competing scenario which has the same cost as the one of having a large degree, which corresponds to have high conductances between two disjoint subsets of vertices of size of order $1$.  

\begin{proof}
Observe that the first equality in \eqref{eq:GD-norme-weight-unbound} follows from Lemma \ref{cardJ}, a union bound, and a net argument. So we only prove the second equality. Let $\mathcal{J}\subset [n]$ and $v\in\mathbb{B}^{\mathcal{J}}$.
To ease the notation, write $\widecheck{Z}$ instead of $\widecheck{Z}_\mathcal{J}$, and similarly for $\widecheck{B}_\mathcal{J}$ and $\widecheck{H}_{\mathcal{J}}$. 
 %Let $v\in \mathbb{S}^{\mathcal{J}}$. 
 We first prove that with overwhelming probability $\widecheck{H}v$ and $\widecheck{B} v$ are almost orthogonal. Since $H$ and $B$ are independent, and $\widetilde{H}$ has entries bounded by some $\kappa_R>0$, we deduce by \eqref{chernoff}  that  
\begin{equation} \label{ineqcondB} \PP\big( \langle \widecheck{B}v,\widecheck{H}v \rangle > t\|\widecheck{B}v\| \big) \leq \sup_{w \in \mathbb{S}^{\mathcal{J}^c}} \PP\big( \langle \widecheck{H} v,w \rangle \geq t\big) \le \exp\Big(-\frac{c t^2}{ \kappa_R^2} np  \log(1/p)\Big), \end{equation}
for any $t>0$, where $c>0$ is a numerical constant. 
Besides, we can write for any $t,s>0$,
\begin{equation}\label{ineqcondB-new} \PP\big( \langle \widecheck{B}v,\widecheck{H} v \rangle \geq s \big) \leq \PP\big(\langle \widecheck{B}v,\widecheck{H} v \rangle > (s/t) \|\widecheck{B}v\| \big) +\PP\big( \|\widecheck{B}\| \geq  t\big). 
\end{equation}
Using \eqref{ineqcondB} and the fact that  $\|B\|$ is exponentially tight, by Lemma \ref{expotight}, we obtain from \eqref{ineqcondB-new} by first letting $n\to +\infty$ and then letting $t\to +\infty$, that for any $s>0$,
\begin{equation} \label{indepBH}  \lim_{n\to +\infty}\frac{1}{np} \log \PP\big( \langle \widecheck{B}v,\widecheck{H} v \rangle \geq s \big) =- \infty.
\end{equation}
%meaning that $\widecheck{B}v$ and $\widecheck{H}v$ are almost orthogonal at the large deviation scale. 
For $\veps>0$ write $v = x^\veps+y^\veps$ where  $x_i^\veps := v_i \Car_{|v_i| > \veps}$
%, $y_i^\veps=v_i \Car_{|v_i| \leq  \veps}$ 
for any  $i\in \mathcal{J}$. Note that $x^\veps$ and $y^\veps$ have disjoint supports. This yields that $\widecheck{H}x^\veps$ and $\widecheck{H}y^\veps$ are independent. Thus the same argument as for \eqref{indepBH} gives
\begin{equation} \label{orthoH} \lim_{n\to+\infty} \frac{1}{np} \log \PP\big( \langle \widecheck{H}x^\veps,\widecheck{H}y^\veps\rangle>s\big) =-\infty, \qquad \text{ for any } s>0.
\end{equation}
Since Lemma \ref{largedegree} tells us that $\|\widecheck{H} x^\veps\|^2$ is dominated by the weighted sum of the degrees of the vertices in the support of $x^\veps$, it remains to show that $\|\widecheck{H} y^\veps\|^2$ concentrates at the exponential scale, that is for any $s>0$ and $R$ large enough
\begin{equation} \label{claimnormH}  \lim_{\veps \to 0} \limsup_{n\to +\infty} \frac{1}{np} \log \PP\big( \|\widecheck{H} y^\veps\|^2>s+\|y^\veps\|^2\big) = -\infty.\end{equation}
Writing for any $y' \in \mathbb{B}^{\mathcal{J}^c}$, $y' = v'+w'$ with  $v'_i := y'_i \Car_{|y'_i|> \veps^{-1}/\sqrt{np}}$%, and $w_i' = y'_i \Car_{|y'_i|\leq \veps^{-1}/\sqrt{np}}$ 
for any $i\notin \mathcal{J}$, we 
get the inequality
\begin{equation} \label{ineg:normA} \| \widecheck{H}y^\veps\| = \sup_{y' \in \mathbb{B}^{\cJ^c}} \langle y', \widecheck{H}y^\veps\rangle \leq \sup_{v' \in \mathcal{S}_{\veps}} \langle v' ,\widecheck{H}y^\veps \rangle + \sup_{w' \in \mathcal{D}_{\veps^{-1}/\sqrt{np}}} \langle w', \widecheck{H}y^\veps\rangle,\end{equation}
where $\mathcal{S}_\veps$ is as in \eqref{eq:locvect}. %:= \{v'\in \mathbb{S}^{\mathcal{J}^c} : |\mathrm{supp}(v')|\leq \veps^2 np\}$.
%On the one hand, 
By \eqref{chernoff} and a union bound there exists a constant $C<+\infty$ such that for any $t>C R \veps$,
\begin{equation} \label{neglisparse} \lim_{n\to+\infty} \frac{1}{np}\log \PP\Big( \sup_{v' \in \mathcal{S}_{\veps}} \langle v' , \widecheck{H}y^\veps\rangle >t \|y^\veps\|\Big) = -\infty.\end{equation}
On the other hand, using a similar argument as in the proof of Lemma \ref{conc-sp-rad-deloc} and the fact that  $\|y^\veps\|_\infty \leq \veps$, we deduce from Proposition \ref{improvedconc} that for any $t>0$, $\veps$ small enough and  $n$ large enough,
\begin{equation} \label{concA} \PP\Big( \sup_{w' \in \mathcal{D}_{\veps^{-1}/\sqrt{np}} } \langle w', \widecheck{H}y^\veps\rangle -\EE \sup_{w' \in \mathcal{D}_{\veps^{-1}/\sqrt{np} }} \langle w', \widecheck{H}y^\veps\rangle >t \Big) \leq e^{- \frac{t^2}{\gamma_0 \veps^2} np},\end{equation}
where $\gamma_0$ is a numerical constant. Besides $\EE (\widetilde{H}_{1,2})^2=1+o_R(1)$, which implies that 
\begin{equation} \label{expectA} \EE \sup_{w' \in \mathcal{D}_{\veps^{-1}/\sqrt{np} }} \langle w', \widecheck{H}y^\veps\rangle  \leq \EE \|\widecheck{H}y^\veps\|\leq \big( \EE \|\widecheck{H}y^\veps\|^2\big)^{1/2} \leq \|y^\veps\|(1+o_R(1)).\end{equation}
Combining \eqref{ineg:normA}--\eqref{expectA} we get the claim \eqref{claimnormH}. This finally ends the proof.
\end{proof}

\subsubsection{Proof of Proposition \ref{expoequi:ineq}}
Fix $t>0$ and $\eta >0$ small. Recall the definition of $\Upsilon_{t,\veps}$ from Proposition \ref{GD-norme-weight-unbound}. Recall also \eqref{eq:dfn-G-cK}-\eqref{eq:dfn-G-cK2}.
%
%Define for any $t,\veps,\eta \in (0,1)$, the events 
%\[ \Gamma_{t,\veps} := \big\{\langle {\bm x}^\veps, (\lambda_Z-H^{({ J})})^{-1} {\bm x}^\veps\rangle - m(\lambda_Z)\|{\bm x}^\veps\|^2\leq  t, \lambda_{H^{({ J})}}<\lambda_Z\big\},\]
%\[ \Upsilon_{t,\veps} =\big\{\|\widecheck{Z}_{J} {\bm v}^\veps\|^2 \leq \sum_{k\in I} {\bm u}_k^2 (\|\widecheck{H}_J^{(k)}\|^2-1)_+  +\|\widecheck{B}_J{\bm v}^\veps\|^2 +\|{\bm v}^\veps\|^2+t\big\}, \quad K_{t,\veps,\eta} = \big\{ \|B^{(J)} {\bm w}^\veps\| \leq t\eta, \ \|{Z}\|\leq \eta^{-1} \big\}.\]
%Coming back at the equation 
By \eqref{eq-top-ei-sec2-2} %, we get 
on %the event 
$\Gamma_{t,\veps}\cap \mathcal{K}_{\veps,t\eta}^{\eta^{-1}}$,
\[\lambda_Z \|{\bm v}^\veps\|^2\leq  \langle {\bm v}^\veps, Z_J {\bm v}^\veps\rangle  + m(\lambda_Z) \|{\bm x}^\veps\|^2 +2t \le \langle {\bm v}^\veps, Z_J {\bm v}^\veps\rangle +m(\lambda_Z) \|\widecheck{Z}_J{\bm v}^\veps\|^2 +5 t,\]
where the last step follows from the fact that $\|{\bm x}^\veps\|^2 \leq \|\widecheck{Z}_J{\bm v}^\veps\|^2 +2\|Z\|.\|B^{(J)}{\bm w}^\veps\| + \|B^{(J)} {\bm w}^\veps\|^2$. %, we have on the event $\Gamma_{t,\veps}\cap K_{t,\veps,\eta}$,
%\[ \lambda_Z \|{\bm v}^\veps\|^2 \leq  \langle {\bm v}^\veps, Z_J {\bm v}^\veps\rangle +m(\lambda_Z) \|\widecheck{Z}_J{\bm v}^\veps\|^2 +5 t,\]
This together with the fact that $m$ is non increasing on $[2, +\infty)$, that $\lambda' -m(\lambda') = 1/m(\lambda')$ and $m(\lambda') \leq 1$ for $\lambda' \ge 2$ yield that on $\wh \Theta:=\Theta_{\lambda,\delta} \cap \Upsilon_{t,\veps} \cap\Gamma_{t,\veps} \cap \mathcal{K}_{\veps,t\eta}^{\eta^{-1}}$,
\[ \frac{\|{\bm v}^\veps\|^2}{m(\lambda)}  \leq  \langle {\bm v}^\veps, Z_J {\bm v}^\veps\rangle+m(\lambda) \Big( \sum_{k\in \mathfrak{J} } {\bm u}^2_k (\|\widecheck{H}_J^{(k)}\|^2-1)_+ + \|\widecheck{B}_J {\bm v}^\veps\|^2\Big) +6 t,\]
where $\gJ := \{ k \in J  : |{\bm u}_k|\geq \veps\}$.
%where we used the fact that $m$ is non-increasing and that $\lambda -m(\lambda) = 1/m(\lambda)$.  
Now combining Lemmas \ref{expotightZ}, \ref{nodevBJ}, and \ref{conc:unbounded},  and Proposition \ref{GD-norme-weight-unbound} we find the desired bound on $\PP(\wh \Theta^c)$. This completes the proof.
%, \ref{nodevBJ} and the fact that $\|Z\|$ is exponentially tight by Lemma \ref{expotight}, we get the claim. 
\qed

\subsection{Large deviation estimates}
In this section we analyze the large deviations of the random variables appearing in Proposition \ref{expoequi:ineq}, $\langle {\bm v}^\veps,Z_J{\bm v}^\veps\rangle$, $\|\widecheck{H}_J^{(k)}\|^2$, and $\|\widecheck{B}_J {\bm v}^\veps\|^2$. Since up to fixing $J$ and ${\bm v}^\veps$ these random variables are independent, it is sufficient to compute their large deviation rate functions separately.

 \subsubsection{Large deviations of $\langle v,Z^R v\rangle $} 
In this short section we prove the following tail bound estimate.
 \begin{Pro}\label{devspradJ}  For any $t,\veps'>0$, 
\[ \limsup_{R\to +\infty}\limsup_{n\to+\infty} \frac{1}{np} \sup_{v \in \mathbb{B}^n} \log \PP\Big( \langle v, Z^R v\rangle> t \sqrt{\beta + \big(\frac{\alpha}{2}-\beta\big)\|v\|_4^4 + \veps'}\Big) \leq - \frac{t^2}{4}.\]
\end{Pro}
The extra parameter $\veps'$  in the above proposition is only meant to give some security in the case where $\beta =0$. 
\begin{proof}
By Proposition \ref{decoupling} it suffices to prove the same tail bound for $X$ instead of $Z$. 
Let $v\in \mathbb{B}^n$. For any $\theta \geq 0$, 
\[ \EE (e^{\theta np \langle v, X v\rangle} )=\prod_{i=1}^n \EE \big( e^{\theta \sqrt{np} v_i^2 \xi_{i,i} G_{i,i} }\big) \prod_{i < j} \EE \big( e^{2\theta \sqrt{np} v_i v_j \xi_{i,j} G_{i,j}}\big).\]
Let $M>0$. Note that if $|\lambda |>M$ then $\log \EE(e^{\lambda \xi_{i,j} G_{i,j}}) \leq (\beta/2+o_M(1))\lambda^2$, for any $i\neq j$ and $\EE(e^{|\lambda| \xi_{i,i} G_{i,i}}) \leq (\alpha/2+o_M(1))\lambda^2$. 
If $|\lambda|\leq M$ %, we can write 
using the concavity of the $\log$ and the fact that $\EE G_{i,j} =0$,
\[ 
 \log \EE(e^{\lambda \xi_{i,j}G_{i,j}}) = \log (1 -p + p\EE e^{\lambda G_{i,j}}) %\leq %p(\EE e^{\lambda G_{i,j} }-1) 
\le p\EE ( e^{\lambda G_{i,j}} - 1-\lambda G_{i,j}) \le \lambda^2 p \EE ( G_{i,j}^2 e^{\lambda G_{i,j}}) \leq \lambda^2 p e^{O(M^2)},
\]
where in the second last step we used  $e^x-1-x \leq x^2e^x$ for any $x\in \RR$ and in the last step we used that $G_{i,j}$ is a sub-Gaussian random variable. %, it yields
%\[  \log \EE(e^{\lambda X_{i,j}}) \leq   \lambda^2 p \EE ( G_{i,j}^2 e^{\lambda G_{i,j}}) \leq \lambda^2 p e^{O(M^2)},\]
%for any $|\lambda|\leq M$. 
Thus, %we can write 
for $n$ large enough, $i \ne j \in [n]$, and any $\lambda \in \RR$,
\[ 
\log \EE(e^{\lambda \xi_{i,j} G_{i,j}}) \leq (\beta/2+o_M(1))\lambda^2 \quad  \text{ and } \quad \log \EE(e^{|\lambda| \xi_{i,i} G_{i,i}}) \leq (\alpha/2+o_M(1))\lambda^2.\]
As $v \in\mathbb{B}^n$, we find
\begin{align*}
 \log \EE (e^{\theta np \langle v, Xv\rangle }) &\leq np \Big( \beta\sum_{i\neq j} v_i^2v_j^2 + \frac{\alpha}{2}\sum_{i=1}^n v_i^4   +o_M(1) \Big) \theta^2 = np\Big( \beta + \big(\frac{\alpha}{2}-\beta\big) \|v\|_4^4 +o_M(1)\Big)\theta^2.
\end{align*}
Using Chernoff's inequality and then taking a $\limsup$ over $n$ first followed by a $\limsup$ over $M$ %\to +\infty$, 
we get the claim. 
\end{proof}

\subsubsection{Tail estimate of $\|\widecheck{B}_\mathcal{J}v\|^2$}
We show the following lemma. 
\begin{Lem}
\label{devBhat} For any $t>0$,
\begin{equation}\label{eq:devBhat}
 \limsup_{R\to +\infty} \limsup_{n\to +\infty} \frac{1}{np} \sup_{ \mathcal{J} \subset [n]} \sup_{v\in \mathbb{B}^\mathcal{J}}\log \PP\big( \|\widecheck{B}_\mathcal{J} v\|^2>t\big) \leq - \frac{t}{2\beta}.
 \end{equation}
\end{Lem}

We start first by proving this tail bound in the case where $\mathcal{J}$ is of bounded size. 
\begin{lemma}\label{boundB2} For any $t>0$ and $m\in \NN$, 
\[\limsup_{R\to +\infty} \limsup_{n\to +\infty} \frac{1}{np}  \sup_{ \mathcal{I} \subset [n] \atop \# \mathcal{I} \leq m } \sup_{u\in \mathbb{B}^\mathcal{I}}\log  \PP\big(\|\widecheck{B}_\mathcal{I}u\|^2 > t\big) \leq -\frac{t }{2\beta}.\]
\end{lemma}
\begin{proof}
Fix $m\in \NN$ and let $u \in \mathbb{B}^\mathcal{I}$ such that $|\mathcal{I}|\leq m$.
For any $i\notin \mathcal{I}$ and $\theta \in \RR$  define 
\[  
\Phi_i(\theta) := \log \EE \big( e^{\theta np \big( \sum_{j\in \mathcal{I}} B_{i,j} u_j \big)^2 } \big)= \log \EE \big( e^{\sqrt{2 \theta } \Gamma  \sum_{j\in \mathcal{I}} \widetilde{B}_{i,j}\xi_{i,j} u_j} \big),
\]
where $\Gamma$ is an independent standard Gaussian. Observe that it suffices to show that 
for any $\theta < 1/(2\beta)$, 
$\sup_{i\notin \mathcal{I}}\Phi_i(\theta) =o_R(1)p$ uniformly in $u\in \mathbb{B}^{\mathcal{I}}$ and $\mathcal{I}$ such that $|\mathcal{I}|\leq m$ (one then uses Chernoff's inequality). 
%By independence we have for any $i\notin \mathcal{I}$ and $\theta \geq 0$, 
%\[
%\Phi_i(\theta) 
% \leq  \log \EE_\Gamma \prod_{j\in \mathcal{I}} \big( 1-pq + pe^{f(\sqrt{2\theta} \Gamma u_j)}\big), 
%\]
%where for any $\zeta \in \RR$, $f(\zeta)=\log \EE \big(e^{\zeta G_{1,2,}} \Car_{G_{1,2} \notin I_{1,2}( R)}\big)$ and $q = \PP(G_{1,2}\notin I_{1,2}(R))$.
Fix $\theta <1/(2\beta)$. Let $r,s>1$ such that $\frac{1}{r} + \frac{1}{s}=1$ and $s \theta <1/(2\beta)$. Recall that $\widetilde{B}_{i,j} = G_{i,j} \Car_{G_{i,j}\notin I_{i,j}(R)}$. Using independence and Hölder's inequality, %we get 
\[ \Phi_i(\theta) \leq \log \EE \prod_{j \in \mathcal{I}} (1-pq +pq^{\frac{1}{r} }e^{\frac{1}{s}\Lambda_{1,2}(\sqrt{2s^2\theta} \Gamma u_j)}),\]
where $q = \PP(G_{1,2}\notin I_{1,2}(R))=o_R(1)$.
Expanding the product and using that $\log(a+x) \le x$ for $a \in (0,1]$ and $x \ge 0$ we get 
\begin{align} \Phi_i(\theta) %&\leq  \log \sum_{E \subset \mathcal{I} } (1-pq)^{|\mathcal{I} \setminus E|} (pq^{\frac{1}{r}})^{|E|}  \EE \big(e^{\frac{1}{s}\sum_{j\in E}\Lambda_{1,2}(\sqrt{2s^2\theta} \Gamma u_j)}\big)\nonumber\\
& \leq \sum_{E \subset \mathcal{I} \atop |E|\geq 1} (1-pq)^{|\mathcal{I} \setminus E|} (pq^{\frac{1}{r}})^{|E|} \EE \big( e^{\frac{1}{s}\sum_{j\in E}\Lambda_{1,2}(\sqrt{2s^2\theta} \Gamma u_j)}\big).\label{boundPhi}
\end{align}
Let $\sigma^2 := \sup_{\theta \neq 0 } \Lambda_{1,2}(\theta)/\theta^2$. Note that for any $M>0$ and $\zeta \in \RR$, we can write $\Lambda_{1,2}(\zeta) \leq \sigma^2 M^2 + (\beta/2)(1+o_M(1)) \zeta^2$. Therefore, for any $E\subset \mathcal{I}$, $E\neq \emptyset$, as $\|u\| \le 1 $,
\[ \frac{1}{s}\sum_{j\in E}\Lambda_{1,2}(\sqrt{2s^2\theta} \Gamma u_j) \leq  \sigma^2 M^2 \frac{|E|}{s}+\beta \theta s \Gamma^2 (1+o_M(1)),\]
which yields for $M$ large enough 
\[ \EE \big(e^{\frac{1}{s}\sum_{j\in E}\Lambda(\sqrt{2s^2\theta} \Gamma u_j)}\big) \leq \frac{e^{\frac{\sigma^2 M^2}{s} |E|}}{\sqrt{1-2\beta \theta s(1+o_M(1))}} \leq C^{|E|},\]
where $C$ is some positive constant. Coming back to \eqref{boundPhi}, we find denoting $m' = |\mathcal{I}|$, 
\begin{align*} \Phi_i(\theta) &\leq \sum_{E \subset \mathcal{I} \atop |E|\geq 1} (1-pq)^{|\mathcal{I}\setminus E|} (Cpq^{\frac{1}{r}})^{|E|} = (1-pq +C pq^{\frac{1}{r}})^{m'}-(1-pq)^{m'} \leq C m'(1+Cpq^{\frac{1}{r}})^{m'-1} p q^{\frac{1}{r}}.
\end{align*}
We deduce that $\sup_{i\notin \mathcal{I}} \Phi_i(\theta)=o_R(1) p$ uniformly in $|\mathcal{I}|\leq m$. This completes the proof.
%Using Chernoff's inequality, this ends the proof. 
\end{proof}

We are now ready to give a proof of Lemma \ref{devBhat}.

\begin{proof}[Proof of Lemma \ref{devBhat}]
Let $\eta>0$ and write $v = v^\eta+{v}_\eta$, where $v^\eta(i)=v_i\Car_{|v_i|>\eta}$, 
%$v_\eta(i)=v_i \Car_{|v_i|\leq \eta}$ 
for any $i\in\mathcal{J}$. 
By Lemma \ref{boundB2}, we know that for any $t,\eta>0$,\begin{equation}\label{boundBa} \limsup_{n\to +\infty} \frac{1}{np}\sup_{\mathcal{J} \subset [n]} \sup_{v\in \mathbb{B}^\mathcal{J}} \log \PP\big( \|\widecheck{B}_\mathcal{J}v^\eta\|>t\|v^\eta\|\big) \leq - \frac{t^2}{2\beta}.\end{equation}
We claim that similarly for any $t>0$,
\begin{equation} \label{secondclaim}\limsup_{R\to +\infty}  \limsup_{\eta \to 0}\limsup_{n\to +\infty} \frac{1}{np}\sup_{\mathcal{J} \subset [n]} \sup_{v\in \mathbb{B}^\mathcal{J}} \log \PP\big( \|\widecheck{B}_\mathcal{J}{v}_\eta\|>t\|{v}_\eta\|\big) \leq - \frac{t^2}{2\beta}.\end{equation}
Assume for the moment that the above claim is true. Observe that if we denote by $\hat{v}^\eta=v^\eta/\|v^\eta\|$ and $\hat{v}_\eta=v_\eta/\|v_\eta\|$, then 
\begin{equation} \label{CS}
\|\widecheck{B}_\cJ v\| \leq  \|\widecheck{B}_\mathcal{J} v^\eta \|+\|\widecheck{B}_\mathcal{J} v_\eta\| = \|v^\eta\|\|\widecheck{B}_\mathcal{J} \hat{v}^\eta\|+\|v_\eta\|\|\widecheck{B}_\mathcal{J} \hat{v}_\eta\|  \leq \sqrt{ \|\widecheck{B}_\mathcal{J}\hat{v}^\eta\|^2+\|\widecheck{B}_\mathcal{J}\hat{v}_\eta\|^2}.\end{equation}
Since $v^\eta$ and $v_\eta$ have disjoint supports, $\widecheck{B}_\mathcal{J}\hat{v}^\eta$ and $\widecheck{B}_\mathcal{J} \hat{v}_\eta$ are independent. 
Using \eqref{boundBa}--\eqref{secondclaim} and the contraction principle (see also Proposition \ref{UTindep}), \eqref{eq:devBhat} follows from  \eqref{CS}. 
%for any $t>0$,
%\[ \limsup_{R\to +\infty}\limsup_{\eta \to 0} \limsup_{n\to+\infty} \frac{1}{np} \sup_{v\in\mathbb{B}^\mathcal{J}}\log \PP\big( \|\widecheck{B}_\mathcal{J}\hat{v}^\eta\|^2+\|\widecheck{B}_\mathcal{J}\hat{v}_\eta\|^2 \geq t^2 \big) \leq -\frac{t^2}{2\beta}.\]
%By \eqref{CS} and the triangular inequality, this ends the proof.
Hence it remains to show \eqref{secondclaim}. For any $z \in \mathbb{S}^{\mathcal{J}^c}$ we can write $z = x+y$ where $x_i=z_i\Car_{|z_i| >M/\sqrt{np}}$, %$y_i =z_i \Car_{|z_i|\leq M/\sqrt{np}}$, 
with $M=\sqrt{\log(1/p)/\eta}$. 
This decomposition gives the bound
\[ \|\widecheck{B}_\mathcal{J}v_\eta\|  \leq \sup_{x \in \mathcal{S}_{1/M^2}} \langle x,\widecheck{B}_\mathcal{J} v_\eta\rangle+\sup_{y \in \mathcal{D}_{M/\sqrt{np}}, y' \in\mathcal{D}_\eta } \langle y,{B}y'\rangle.
\]
By Lemma \ref{step1}, we know that for any $\delta>0$, 
\[ \lim_{R\to+\infty}\limsup_{\eta \to 0} \limsup_{n\to +\infty} \frac{1}{np} \log \PP\big(\sup_{y \in \mathcal{D}_{M/\sqrt{np}}, y' \in \mathcal{D}_\eta} \langle y,{B}y'\rangle>\delta \big) = -\infty.\]
Thus, it suffices to show that for any $t>0$, 
\[ \limsup_{\eta \to 0 } \limsup_{n\to +\infty} \frac{1}{np}\sup_{\mathcal{J}\subset [n]}\sup_{v\in\mathbb{B}^{\mathcal{J}}} \log \PP\big(\sup_{x \in \mathcal{S}_{1/M^2}} \langle x,\widecheck{B}_\mathcal{J}v_\eta\rangle>t\big) \leq -\frac{t^2}{2\beta}.\]
Using yet another net argument and a union bound we can reduce ourselves to prove that for any $t>0$,
\[ \limsup_{R\to+\infty} \limsup_{n\to +\infty} \frac{1}{np}\sup_{\mathcal{J} \subset [n]} \sup_{x \in \mathbb{B}^{\mathcal{J}^c}, x'\in \mathbb{B}^\mathcal{J}}\log \PP\big(\langle x,\widecheck{B}_\mathcal{J}x'\rangle>t\big) \leq -\frac{t^2}{2\beta}.\]
Since the entries of $\widecheck{B}_\cJ$ are centered i.i.d.~random variables, and for any $\lambda \in \RR$, $\EE (e^{\lambda \widetilde{B}_{1,2}}) \leq 1+ \EE (e^{\lambda G_{1,2}}) \le 2\EE (e^{\lambda G_{1,2}})$ repeating the same argument as in the proof of Proposition \ref{devspradJ} we obtain the above bound. We omit  the details. %This follows from a similar argument as in the proof of Lemma \ref{devvector}, using the fact that for any $\lambda \in \RR$, $\EE (e^{\lambda \widetilde{B}_{1,2}}) \leq \corAB{1+ \EE (e^{\lambda G_{1,2}}) \le} 2\EE (e^{\lambda G_{1,2}})$.
\end{proof}

\subsubsection{Tail estimate of a degree in the network $H$}
In the next lemma, we give the large deviation upper tail rate function of the $\ell^2$-norm of a column of $H$. 
%We note that this upper tail estimate can be made uniform in $R>0$, a fact that we will exploit in the next section. 
\begin{Lem}\label{degH} For any $k \in [n]$ and $t>1$, 
\[ 
\limsup_{R \to +\infty} \limsup_{n\to+\infty} \frac{1}{np} \log \PP\big( \| {H}_{k} \|^2>t \big) \leq - h_L(t).\]
\end{Lem}
\begin{proof}
For any $\theta \geq 0$ define $ \Lambda_{n,R}(\theta) := (np)^{-1}\EE( \exp(np \theta \|{H}_k\|^2))$.
Using the same argument as in the proof of Lemma \ref{LDP-deg}, it follows that 
\[  \lim_{n\to +\infty} \Lambda_{n,R}(\theta) =  L_R(\theta) -1,\]
where $L_R(\theta) := \EE (e^{\theta \widetilde{H}_{1,2}^2})$. Since $\lim_{R \to +\infty} L_R(\theta)=L(\theta)$ the proof follows upon using Chernoff's inequality by first letting $n \to +\infty$, followed by letting $R \to +\infty$, and then optimising over $\theta$.
%\corAQ{\sout{ and a continuity argument (e.g.~see the proof of Lemma \ref{reductionextsupp})}}. 
%
%Observing that $\EE(e^{\theta \widetilde{H}_{1,2}^2}) \leq \EE(e^{\theta G_{1,2}^2})$ and using Chernoff's inequality ends the proof. 
\end{proof}

\subsection{Variational principle of the upper bound upper tail large deviation}
In this section, we show that the large deviation upper tail rate function of $\lambda_X$ can be expressed as the solution of a certain variational problem that we now describe.
Define for any $\lambda>2$, $k\in \NN$, $k\geq 1$ and $\mu>0$ %the function 
\begin{equation} \label{defJk} \Phi_{k,\lambda}(\mu) := \inf\Big\{\frac{r^2}{4} + \sum_{i=1}^{k} h_{L}(d_i) + \frac{t}{2\beta} : H_{\lambda}\big(r,t,\kappa, (d_i)_{i\leq k}\big) \geq \mu, (r,t,\kappa, (d_i)_{i\leq k}) \in \mathcal{D}_k\Big\},\end{equation}
where $\mathcal{D}_k := \RR_+^2 \times [0,1] \times [1,+\infty)^k $, and $H_{\lambda}$ is defined as 
\[ 
H_{\lambda}(r,t,\kappa, (d_i)_{i\leq k}):= r \sqrt{\beta + \big(\frac{\alpha}{2}-\beta\big)\kappa^2} + m(\lambda) \kappa \Big(\sum_{i=1}^k (d_i-1)^2\Big)^{1/2} + m(\lambda) t.\]
The goal of this section is to prove the following proposition. 
\begin{Pro}\label{varpbUT}Fis any $\lambda>2$. Set $\Phi_{\lambda} := \inf_{k\geq 1} \Phi_{k,\lambda}$. For any $t>0$ we have
\[  \limsup_{n\to +\infty} \frac{1}{np} \log \PP\big( \lambda_X\geq \lambda) \leq - \Phi_{\lambda}(1/m(\lambda)-t),\]
%where $J_{\lambda} = \inf_{k\geq 1} J_{k,\lambda}.$
%
\end{Pro}

\begin{proof} Let $\lambda>2$. By Proposition \ref{decoupling} it suffices to show that for any $t>0$,
\[  \limsup_{R\to +\infty}\limsup_{n\to +\infty} \frac{1}{np} \log \PP\big( \lambda_{Z^R}\geq \lambda) \leq - \Phi_{\lambda}(1/m(\lambda)-t).
\]
Observe that as $h_L$ is finite everywhere (see Lemma \ref{lem:hL-prop}) $\Phi_{\lambda}(1/m(\lambda))<+\infty$.  By Lemma \ref{loc:unbounded} %we know that upon a large deviation of the top eigenvalue, the top eigenvector should localise. As a consequence, we deduce that 
there exists a $\chi(\lambda)>0$ such that 
\begin{equation} \label{loctopev}\limsup_{R\to +\infty}  \limsup_{\veps \to 0} \limsup_{n\to+\infty} \frac{1}{np} \log \PP\big( \|{\bm v}^\veps\|^2\leq \chi(\lambda), \lambda_{Z^R}\geq \lambda) \leq - \Phi_{\lambda}(1/m(\lambda)).\end{equation}
For any $t>0$, define $\wh \Omega_{\veps, R, t}:= \Omega_{\vep, R, t\chi(\lambda)}$ (recall \eqref{eq:Om-eps-R-t})
%the event
%\[ \Omega_{t,\veps,R} = \big\{\langle {\bm v}^\veps, Z_{{J}} {\bm v}^\veps \rangle+m(\lambda) \big[ \sum_{k\in I}{\bm u}^2_k  (\|(\widecheck{H}_{J})_{k}\|^2-1)_+ + \|\widecheck{B}_{ J} {\bm  v}^\veps \|^2\big] \geq  \|{\bm v}^\veps\|^2/m(\lambda) - \chi(\lambda)t\big\},\]
%where $I=\{k\in [n] : |{\bm u}_k|\geq \veps\}\subset J$.
Using Cauchy-Schwarz inequality and dividing by $\|{\bm v}^\veps\|^2$, we obtain that on ${\wh\Omega_{\veps, R,t}}\cap \{\|{\bm v}^\veps\|^2>\chi(\lambda)\}$ 
\begin{equation} \label{ineqgoodevent} 1/m(\lambda) \leq  \langle{\bm \hat{\bm v}^\veps}, Z_J{\bm \hat{\bm v}^\veps}\rangle+m(\lambda)\|{\bm \hat{\bm v}^\veps}\|_4^2 \big(\sum_{k\in \mathfrak{J}}   (\|(\widecheck{H}_J)_{k}\|^2-1)_+^2\big)^{1/2} +m(\lambda) \|\widecheck{B}_J{\bm \hat{\bm v}^\veps}\|^2 +t,\end{equation}
where ${\bm\hat{\bm v}^\veps} := {\bm v^\veps /\|{\bm v^\veps}\|}$. 
It is not clear so far that the Cauchy-Schwarz inequality used here is sharp. %One can note though that 
Since the two scenarios that we believe to lead the large deviation event are on the one hand a planted clique with diverging size, corresponding to $\|{\bm \hat{\bm v}^\veps}\|_4 =0$, and on the other a large degree vertex, corresponding to $\|{\bm \hat{\bm v}^\eps}\|_4 =1$, we are at least not loosing any information when of one these two strategies are at play. To compute an upper bound on the probability of $\wh\Omega_{\veps, R,t}$, we will first fix $J$, $\mathfrak{J}$ and $\hat{\bm v}^\veps$, and later take union bounds. To this end, let $\mathcal{J}\subset [n]$, $\mathcal{I}\subset \mathcal{J}$  such that $|\mathcal{J}|\leq 2\veps^2 np/\log(1/p)$ and $|\mathcal{I}| = \lfloor \veps^{-2}\rfloor $, and $v\in \mathbb{S}^\mathcal{J}$. Define
\[ Z_{R,\veps}^{(1)} := \frac{\langle v, Z_\mathcal{J} v\rangle}{\sqrt{\beta + \big(\frac{\alpha}{2}-\beta\big) \kappa_v^2}}, \ Z_{R,\veps}^{(2)}:= \Big(\sum_{k\in \mathcal{I}}  (\|(\widecheck{H}_\mathcal{J})_{k}\|^2-1)_+^2\Big)^{1/2}, \ Z_{R,\veps}^{(3)}:=\| \widecheck{B}_\mathcal{J} v\|^2,\]
where $\kappa_v := \|v\|_4^2$. 
Since $\|(\widecheck{H}_\mathcal{J})_{k}\|^2$, $k\in\mathcal{I}$ are independent, and $h_L$ is an increasing good rate function on $[1,+\infty)$, using Proposition \ref{degH} and a generalisation of the contraction principle (cf.~Proposition \ref{UTindep}) we get  %(in the special case where $X_{n,\delta}^{(k)}$ does not depend on $\delta$ and $f=g \equiv 1$), that for any $t,R,\veps>0$,
\begin{equation} \label{UTsumdeg}
\limsup_{\veps \to 0} \limsup_{R \to +\infty} \limsup_{n\to+\infty} \frac{1}{np} \log \PP\big( Z^{(2)}_{R,\veps}> t \big) \leq - \liminf_{\veps \to 0} K_{\lfloor \veps^{-2} \rfloor}(t) \le - K(t), 
\end{equation}
where $K_\ell(t) := \inf\{\sum_{k=1}^\ell h_L(d_k) : \sum_{k=1}^\ell (d_k-1)^2 \geq t, \forall 1\leq k \leq \ell, d_k\geq 1\}$ for any $\ell\in\NN$
%. This yields for any $t>0$,
%\begin{equation} \label{UTsumdeg} \limsup_{R\to +\infty \atop \veps \to 0} \limsup_{n\to+\infty} \frac{1}{np} \log \PP\big( Z^{(2)}_{R,\veps}> t \big) \leq - K(t), \end{equation}
and $K:= \inf_{\ell \geq 1} K_\ell$. 

As $Z_\mathcal{J}$, $(\widecheck{H}_\mathcal{J})_{k}$, $k\in\mathcal{I}$, and $\widecheck{B}_\mathcal{J}$ are independent so are %the random variables 
$Z_{R,\veps}^{(k)}$, $k=1,2,3$. %are independent. 
Therefore, by Propositions \ref{devspradJ}, \ref{devBhat}, and \eqref{UTsumdeg}, and  %give the upper tail large deviation rate functions of each of these random variables. By 
Proposition \ref{UTindep} (applied with the positive part of $Z_{R,\veps}^{(1)}$), we deduce that for any $\mu>0$,
\[ \limsup_{R\to +\infty \atop \veps \to 0} \limsup_{n\to +\infty} \frac{1}{np} \log \PP\big( T_{\mathcal{J},\mathcal{I}}^{R,\veps}({v}) >\mu\big) \leq - \Phi_{\lambda}(\mu),\]
where 
$T_{\mathcal{J},\mathcal{I}}^{R,\veps}({v}) := Z_{R,\veps}^{(1)} \sqrt{\beta + \big(\frac{\alpha}{2}-\beta\big) \kappa_v^2}  +m(\lambda)(\kappa_v Z_{2,R}+  Z_{3,R})$. 
Since $|\mathcal{J}|\leq 2\veps^2 np/\log(1/p)$, the map $\mathcal{I} \mapsto T_{\mathcal{J},\mathcal{I}}^{R,\veps}(v)$ is non decreasing function in the cardinality of $\mathcal{I}$, and $np \gg \log n$, using a net argument and a union bound we derive that for any $\mu>0$,
\begin{multline*}  \limsup_{R\to +\infty \atop \veps \to 0} \limsup_{n\to +\infty} \frac{1}{np} \log \PP\Big( \sup_{|\mathcal{J}| \leq 2\veps^2np /\log(1/p) \atop |\mathcal{I} |\leq \veps^{-2}}\sup_{v \in \mathbb{S}^{\mathcal{J}}} T_{\mathcal{J},\mathcal{I}}^R(v)>\mu\Big) \\
 \le \limsup_{R\to +\infty \atop \veps \to 0} \limsup_{n\to +\infty} \frac{1}{np} \log \PP( \sup_{v \in \mathbb{S}^{\mathcal{J}}} T_{\mathcal{J},\mathcal{I}}^{R,\veps}(v)>\mu) \leq - \Phi_{\lambda}(\mu).
\end{multline*}
%Using \eqref{entropyJ} and again the fact that $np \gg \log n$, we can perform a union bound on $\mathcal{J}$ and $\mathcal{I}$, which yields that for any $\mu>0$, 
%\[  \limsup_{R\to +\infty \atop \veps \to 0} \limsup_{n\to +\infty} \frac{1}{np} \log \PP\Big( \sup_{|\mathcal{J}| \leq \veps^2np /\log(1/p) \atop |\mathcal{I} |\leq \veps^{-2}}\sup_{v \in \mathbb{S}^{\mathcal{J}}} T_{\mathcal{J},\mathcal{I}}^R(v)>\mu\Big) \leq - J_{\lambda}(\mu).\]
%In the above estimate we relaxed the condition on the cardinality of $\mathcal{I}$ using the fact that $\mathcal{I} \mapsto T_{\mathcal{J},\mathcal{I}}^{R,\veps}(v)$ is increasing for the inclusion. 
Combining the above estimate together with \eqref{loctopev} and Lemma \ref{cardJ}, upon using \eqref{ineqgoodevent}  we get that for any $t>0$,
\[ \limsup_{\veps \to 0} \limsup_{R\to +\infty} \limsup_{n\to +\infty} \frac{1}{np} \log \PP(\wh\Omega_{\veps, R,t})  \leq - \Phi_{\lambda}(1/m(\lambda)-t).\]
By Proposition \ref{expoequi:ineq} for any $\delta>0$ such that $\lambda >2+\delta$ and $t>0$ we further deduce that
\begin{equation}\label{UTcondH} \limsup_{\veps \to 0} \limsup_{R\to +\infty} \limsup_{n\to +\infty} \frac{1}{np} \log \PP(\lambda_{H^{(J)}} \leq \lambda -\delta \leq \lambda \leq \lambda_{Z^R}) \leq - \Phi_{\lambda}\big(1/m(\lambda)-t\big).\end{equation}
Besides, using the result of Theorem \ref{theo-main-weight} in the bounded case, and a union bound on $J$, we have for any $R,\veps>0$, 
\begin{equation}\label{UTcondH2} \limsup_{n\to +\infty}  \frac{1}{np} \log  \PP\big(\lambda_{H^{(J)}} >\lambda- \delta \big) \leq -I_{R}((\lambda-\delta)/\sigma_R),\end{equation}
where $I_R(\mu) :=  h_{L_R}(\mu/m(\mu))$, $\sigma_R^2 = \EE( H_{1,2}^2)$, and $L_R(\theta ) := \EE (e^{\theta H_{1,2}^2/\sigma_R^2})$ for any $\mu\geq 2$ and $\theta \geq 0$.  One can check that 
\[ \lim_{\delta \to 0} \lim_{R\to +\infty} I_{R}((\lambda-\delta)/\sigma_R) = h_L(\lambda/m(\lambda)).\]
Since $\Phi_{\lambda} (1/m(\lambda)-t) \leq \Phi_{\lambda} (1/m(\lambda)) \leq \Phi_{1,\lambda}(1/m(\lambda)) \le h_L(\lambda/m(\lambda))$, the proof is complete by putting together \eqref{UTcondH}--\eqref{UTcondH2}. 
\end{proof}
\subsection{Analysis of the optimisation problem}
 The goal of this section is to solve the optimisation problem given by the function $\Phi_{\lambda}$. %More precisely, 
 We prove the following result.
  \begin{Pro}\label{optimsol}
 For any $\lambda>2$, and $0 < \mu \leq 1/m(\lambda)$,
 \[ \Phi_{\lambda}(\mu) =   \min\Big( \frac{\mu^2}{4\beta}, \inf\big\{\frac{r^2}{2\alpha} +h_L(1+d) : r + m(\lambda)s \geq \mu : r,d\geq 0\big\}\Big).\]
 \end{Pro}
This result is the last piece of the proof of the upper bound of Theorem \ref{theo-main-weight}. Indeed, Proposition \ref{optimsol} yields in particular that $\Phi_\lambda$ is lower semicontinuous as it is the minimum of two lower semicontinuous functions. Combining Propositions \ref{optimsol} and \ref{varpbUT}, we get the claimed upper bound rate function $I$ defined in Definition \ref{def:rate-fn}. In a first step we start by optimising on the degrees $d_1,\ldots,d_k$, i.e.~the minimum is achieved by $\Phi_{1,\lambda}$ (recall \eqref{defJk}), and show that the minimum is achieved in dimension 1, which accounts for the fact that the large deviation event is dominated by the emergence of at most atypically large degree. 
\begin{Lem}\label{reducdim1}
%For any $\veps>0$, 
%\[ J_\veps(\lambda) = \inf \big\{\frac{r^2}{4} + h_L(1+s) +\frac{t}{2\beta} :  \phi_\lambda(r,s,t,\kappa)\geq 1/m(\lambda), r,s,t\geq 0,  \kappa \in [0,1] \big\}\]
%where $\phi_\lambda(r,s,t,\kappa) = r\sqrt{\beta + \big( \frac{\alpha}{2}-\beta\big)\kappa^2} + m(\lambda) \kappa s +m(\lambda) t$.
For any $k\in \NN$ and $s\geq 0$, 
\[   \inf \big\{\sum_{i=1}^k h_L(d_i) : \sum_{i=1}^k (d_i-1)^2 \geq s^2, d_i\geq 1, 1\leq i \leq k \big\} =  h_L(1+s).\]

\end{Lem}
\begin{proof}
First we claim that for any $k\in \NN$ and $s\geq0$,
\begin{equation} \label{optimpb}  \inf \big\{\sum_{i=1}^k h_L(d_i) : \sum_{i=1}^k (d_i-1)^2 \geq s^2, d_i\geq 1, 1\leq i \leq k \big\} =  \inf_{1\leq \ell \leq k } \ell h_L\Big( 1+ \frac{s}{\sqrt{\ell}}\Big).\end{equation}
Since $h_L$ is a good rate function, the infimum above is achieved at some vector $d^* \in \RR^k$. Without loss of generality, we can assume that $d^* = (d_1^*,\ldots, d_\ell^*,1,\ldots, 1)$ where $\ell\leq k$ and $d_i^*>1$ for any $i\leq \ell$. As $h_L(1)=0$ the vector $(d_1^*,\ldots, d_\ell^*)$ is a minimizer of the optimization problem \eqref{optimpb} in dimension $\ell$ which lies in the interior of the domain $[1,+\infty)^\ell$. By the multiplier's rule (e.g.~\cite[Theorem 9.1]{Clarke}), there exists $(\gamma,\theta) \neq (0,0)$ such that $\gamma \in \{0,1\}$, $\theta \leq 0$, 
\[
\theta \big(\sum_{i=1}^\ell (d_i^*-1)^2-s^2\big)=0, \quad \text{ and } \quad 
\gamma h_L'(d_i^*) +2\theta(d_i^*-1)=0, \, i \in [\ell].
\]
If $\gamma=0$, then as $\theta \neq 0$, $d^*_i=1$ for any $i$, contradicting the definition of $d^*$. Therefore, $\gamma =1$. If $\theta =0$ then $h_L'(d^*_i)=0$, which entails by Lemma  \ref{lem:hL-prop}  (e) that $d^*_i=1$ for any $i$ which is again contradictory. Thus $\theta \neq 0$. As a consequence $\sum_{i=1}^\ell (d_i^*-1)^2=s^2$.
By Lemma \ref{lem:hL-prop}(e) %, $h_L'(1)=0$, $h_L'$ in strictly concave on $(-\infty,x_\star)$ and $h_L'(x) =1/(2\beta)$ for $x\geq x^*$. Thus, 
the equation $h_L'(d) = -2\theta(d-1)$ can only have  at most one solution strictly greater than $1$.  Since $d^*_i>1$ for $i\in [\ell]$, it follows that $d^*_i= d_j^*$ for $i,j\in[\ell]$. As $\sum_{i=1}^\ell(d_i^*-1)^2=s^2$, we deduce that $d^*_i= 1+s/\sqrt{\ell}$, yielding the claim \eqref{optimpb}.

By \eqref{optimpb} it now suffices to show that %Secondly, we claim that 
\begin{equation} \label{optimbghL}  \inf_{\ell \geq 1 } \ell h_L\Big( 1+ \frac{s}{\sqrt{\ell}}\Big) = h_L(1+s).\end{equation}
This is straightforward. Indeed, fixing $\ell \geq 1$ and denoting $f(s):= \ell h_L(1+s/\sqrt{\ell}) - h_L(1+s)$, $s \ge 0$, we observe that $f(0)=0$, and by concavity of $h_L'$ and the fact that $h_L'(1)=0$ we have $f'(s) = \sqrt{\ell} h_L'(1+s/\sqrt{\ell}) - h_L'(1+s) \ge 0$. This yields \eqref{optimpb} completing the proof.
% 
%We will show that $f : x\in \RR_+^* \mapsto h_L(1+x)/x^2$ is non-increasing, thus proving the claim \eqref{optimbghL}. Indeed, by Lemma  \ref{prophL}, $f$ is differentiable. We can write $f'(x) = g(x)/x^3$ for any $x>0$ where $g(x) = xh_L(1+x) -2h_L'(1+x)$. Again by Lemma  \ref{lem:hL-prop} we know that $h_L'$ is differentiable except potentially at $x^*$. For any $x\in \RR_+\setminus\{x_\star-1\}$, 
%\[ g'(x) = xh_L''(1+x) - h_L'(1+x) \leq 0,\]
%since $h_L'$ is non-decreasing and concave. As $g$ is continuous, it follows that $g$ is non-increasing. Besides $g(0) = 0$. Therefore $g\leq 0$ and as a consequence $f$ is non-increasing. 
\end{proof}

In the next step, we optimize on $\kappa$ and show that the infimum is achieved either for $\kappa=1$ or $\kappa=0$, corresponding respectively to the high degree scenario and the planted clique of diverging size.
\begin{Lem}\label{optimkappa} Define for any $\mu\geq 0$, 
\[\widehat{\Phi}_\lambda(\mu) := \inf \big\{\frac{r^2}{4} + h_L(1+s) :   r\sqrt{\beta + \big( \frac{\alpha}{2}-\beta\big)\kappa^2} + m(\lambda) \kappa s \ge \mu, r,s\geq 0,  \kappa \in [0,1] \big\}.\]
Then,
\[ \widehat{\Phi}_\lambda(\mu) = \min\Big( \frac{\mu^2}{4\beta}, \inf\big\{\frac{r^2}{2\alpha} +h_L(1+s) : r + m(\lambda)s \geq \mu : r,s\geq 0\big\}\Big).\]
\end{Lem}

\begin{proof}
First observe that if $\beta \leq \alpha/2$ then the function $\kappa \mapsto r\sqrt{\beta +(\alpha/2-\beta)\kappa^2} + m(\lambda) \kappa s$ is increasing on $\RR_+$ and thus achieves its supremum  for $\kappa=1$. The claim immediately follows as $h_L$ is non decreasing on $[1,+\infty)$. Assume from now on that $\beta >\alpha/2$ and 
define for any $\kappa,s\geq 0$, 
\[ \phi(\kappa,s) := \frac{(\mu-m(\lambda) \kappa s)^2}{4 (\beta- \big( \beta - \frac{\alpha}{2}\big)\kappa^2)} + h_L(1+s).\]
With this notation %we can write 
$\widehat{\Phi}_\lambda(\mu) = \inf\big\{\phi(\kappa,s) : \kappa \in [0,1], 0\leq m(\lambda)\kappa s\leq \mu\big\}$. We first show that we can restrict our attention to the set of parameters $(\kappa,s)$ where $s\leq s_\star:= (\beta-\alpha/2)\mu/(\beta m(\lambda))$. Indeed, for any $\kappa \in [0,1]$ and $s\geq0$,
\begin{equation}\label{dervphikappa} \mathrm{sg}\big(\frac{\partial \phi}{\partial \kappa}(\kappa,s)\big) = \mathrm{sg}\Big(\big(\beta- \frac{\alpha}{2}\big)\frac{\kappa (\mu-m(\lambda)\kappa s)}{(\beta - \big(\beta- \frac{\alpha}{2}\big)\kappa^2\big)} -m(\lambda) s\Big) = \mathrm{sg}\Big(\kappa \big(\beta -\frac{\alpha}{2}\big) \mu-m(\lambda) s\beta\Big),\end{equation}
where we denote by $\mathrm{sg}$ the sign of a real number. It yields that $\phi(\cdot,s)$ is non increasing on $[0,1]$ for any $s\geq s_\star$, and as a consequence its infimum is achieved for $\kappa=1$. 

This entails that it suffices to show that the infimum
%\begin{equation} \label{reducoptim}
$\inf\big\{\phi(\kappa,s) : \kappa \in [0,1], 0\leq s\leq s_\star\big\}$ %,\end{equation}
is achieved for $\kappa\in \{0,1\}$. To this end, for $\kappa \in [0,1]$ let $s(\kappa) \in [0, s_\star]$ be such that $\varphi(\kappa, \cdot)$, when viewed as a function on $[0,s_\star]$, is uniquely minimized at $s(\kappa)$. To show the existence of such a unique minima note that 
\[ \frac{\partial \phi}{\partial s} (\kappa, s) = h'_L(1+s) - \frac{m(\lambda) \kappa}{2(\beta -\big( \beta-\frac{\alpha}{2}\big)\kappa^2)} (\mu-m(\lambda) \kappa s).\]
If $\kappa=0$, as $h_L'$ vanishes only at $1$ and is non decreasing, and so the unique minima is at $s(0)=0$. Assume $\kappa \ne 0$. Then again using that $h_L'$ is non decreasing we find that $(\partial \phi/\partial s)(\kappa,.)$ is strictly increasing, and thus the existence of a unique minima $s(\kappa)$ follows. Moreover, as $(\partial \phi/\partial s)(\kappa,0) \le 0$ notice that 
\begin{equation}\label{eq:s-kappa}
\frac{\partial \phi}{\partial s} (\kappa, s_\star) \le 0 \Leftrightarrow s(\kappa)=s_\star \quad \text{ and } \quad  s(\kappa) < s_\star \Rightarrow \frac{\partial \phi}{\partial s} (\kappa, s(\kappa))=0. 
\end{equation}
We next claim that the map $\kappa \mapsto s(\kappa)$ is a non decreasing continuous function. The continuity follows from the fact that the map $\partial \phi/\partial s$ is continuous on $[0,1]\times \RR_+$, the characterization of $s(\kappa)$ given in \eqref{eq:s-kappa}, and the fact that $(\partial \phi/\partial s)(\kappa,\cdot)$ is strictly increasing for $\kappa >0$. To derive the non decreasing property of $s(\kappa)$ observe that for any $\kappa \in [0,1]$ and $s\ge 0$, we have $\mathrm{sg}((\partial^2 \phi/\partial \kappa \partial s) (\kappa,s)) = - \mathrm{sg}(\psi(s,\kappa))$,
where $\psi(s,\kappa) := \big(\beta-\frac{\alpha}{2}\big) \mu \kappa^2 + (\mu-2m(\lambda) s\kappa)\beta$. Thus, for any $s\leq s_\star$ and $\kappa \in[0,1]$, 
\[ \psi(\kappa,s) \geq \psi(\kappa,s_\star) = \mu\big((\beta -\alpha/2)(\kappa^2 -2\kappa)+\beta\big) \geq \frac{\alpha}{2} \geq 0,\]
showing that $(\partial\varphi /\partial s)(\cdot, s)$ is non increasing. Now fix $\kappa < \kappa' \in [0,1]$. If $s(\kappa)=s_\star$ then by \eqref{eq:s-kappa}, and as $(\partial\varphi /\partial s)(\cdot, s)$ is non increasing, we obtain that $s(\kappa')=s_\star$. On the other hand, if $s(\kappa) < s_\star$ and $s(\kappa') < s(\kappa)$ then, using that $(\partial \varphi/\partial s)(\kappa', \cdot)$ is strictly increasing, ($\partial \phi/\partial s)(\cdot,s(\kappa))$ is non increasing, and \eqref{eq:s-kappa} we get
\[
0=\frac{\partial \varphi}{\partial s} (\kappa', s(\kappa')) < \frac{\partial \varphi}{\partial s} (\kappa', s(\kappa)) \le \frac{\partial \varphi}{\partial s} (\kappa, s(\kappa)) =0,
\]
yielding a contradiction. Observe that the above argument also gives a $\kappa_\star \in (0,1]$ such that $s(\kappa) < s_\star$ for $\kappa < \kappa_\star$ and otherwise $s(\kappa)=s_\star$.
As $h_L'$ is differentiable, except potentially at $x_\star$ and $h_L'$ is non decreasing (see Lemma \ref{lem:hL-prop}), we obtain that $(\partial^2 \phi/\partial s^2)(\kappa, s) >0$ for any $s \neq x_\star -1$ and $\kappa\in [0,1]$. Besides, for any $s \ge 0$ the equation $(\partial \phi/\partial s)(\kappa,s) =0$ has at most two solutions in $\kappa$. Denoting $\mathcal{K}$ to be the set of $\kappa$ such that  $(\partial \phi/\partial s)(\kappa,x_\star -1) =0$, %. We then have $|\mathcal{K}|\leq 2$. 
by the implicit function theorem it follows that $\kappa \mapsto s(\kappa)$ is differentiable on $(0,\kappa_\star) \setminus \mathcal{K}$, and hence so is the map $\kappa \mapsto \Psi(\kappa):=\varphi(\kappa, s(\kappa))$. % is differentiable on the same set. 
Furthermore
%is equivalent to a quadratic equation in $\kappa$. Thus, for a given $s\geq 0$, there is at most $2$ values of $\kappa$ such that $(\partial \phi/\partial s) (\kappa,s) =0$. Let $\mathcal{K}$ be the set of values $\kappa$ such that  $(\partial \phi/\partial s)(\kappa,x_\star -1) =0$. We then have $|\mathcal{K}|\leq 2$. By the implicit functions theorem it follows that $\kappa \mapsto s(\kappa)$ is differentiable on $(0,\kappa_\star)\setminus \mathcal{K}$ and for any $\kappa \in(0,\kappa_\star)\setminus \mathcal{K}$, 
%\begin{equation} \label{eqderiv} \frac{\partial^2 \phi}{\partial \kappa \partial s} (\kappa,s(\kappa)) +s'(\kappa) \frac{\partial^2 \phi}{\partial s^2}(\kappa, s(\kappa))=0.\end{equation}
% By \eqref{eqderiv}, this entails that $s'(\kappa) \geq 0$ for any $\kappa \in (0,\kappa_\star)\setminus \mathcal{K}$. As $\kappa \mapsto s(\kappa)$ is continuous and $|\mathcal{K}|\leq 2$, we deduce that $s(\kappa)$ is non-decreasing in $\kappa$. Besides, for any $\kappa \in (0,\kappa_\star)\setminus \mathcal{K}$, 
\begin{equation} \label{dervphi} \Psi'(\kappa) = \frac{\partial \phi}{\partial \kappa}(\kappa, s(\kappa)) + s'(\kappa) \frac{\partial \phi}{\partial s}(\kappa,s(\kappa))= \frac{\partial \phi}{\partial \kappa}(\kappa, s(\kappa)),\end{equation}
where we used that $(\partial \phi/\partial s)(\kappa, s(\kappa))=0$ for $\kappa < \kappa_\star$ and trivially $s'(\kappa)=0$ for $\kappa \in (\kappa_\star, 1)\setminus \mathcal{K}$. %Moreover, the above equation \eqref{dervphi} remains valid when $\kappa \geq \kappa_\star$ as $s(\kappa) =s_\star$.
 %Now, using the first equality of \eqref{dervphikappa} and again the fact that $(\partial \phi/\partial s)(\kappa, s(\kappa))=0$, it follows that for any $\kappa \in (0,\kappa_\star)\setminus \mathcal{K}$, 
Therefore by the first equality in \eqref{dervphikappa} and \eqref{dervphi}, for any $\kappa \in [0,1]\setminus (\mathcal{K}\cup\{\kappa_\star\})$,
\begin{equation}\label{eq:sg-phi}
\mathrm{sg}(\Psi'(\kappa))=\mathrm{sg}\Big(2\big(\beta-\frac{\alpha}{2}\big)  h_L'(1+s(\kappa)) -m(\lambda)^2s(\kappa)\Big). %, \quad \corAB{\kappa \in [0,1]\setminus \mathcal{K}}.
\end{equation}
Using that $h_L'$ is concave by  Lemma \ref{lem:hL-prop}(e) and the facts that $\kappa \mapsto s(\kappa)$ is non decreasing and  that $(\partial \varphi/\partial s)(\kappa, s(\kappa)) \le 0$ for $\kappa \ge \kappa_\star$ we deduce that the LHS of \eqref{eq:sg-phi} is either nonpositive on $[0,1]\setminus \mathcal{K}$ or nonnegative on $[0,\kappa_{\star \star}]\setminus \mathcal{K}$ and nonpositive otherwise, for some $\kappa_{\star\star} \in (0,1]$. Since $\#\mathcal{K}$ is $O(1)$, and $\kappa \mapsto \Psi(\kappa)$ is continuous we find that either it is nonincreasing or it is nondecreasing on $[0,\kappa_{\star\star}]$ and then nonincreasing. In both these cases $\inf\{\Psi(\kappa): \kappa \in [0,1]\}=\Psi(0) \wedge \Psi(1)$. This ends the proof. %which by the definition of $s(\kappa)$ implies that  
\end{proof}

Finally, we add the linear term $t/(2\beta)$ to the optimisation problem and show that the infimum is achieved for $t=0$. This will prove that the $\| \widecheck{B}_J {\bm v}^\veps\|$ term in the exponential equivalent of Proposition \ref{expoequi:ineq} leads indeed to a sub optimal large deviation scenario. 
\begin{Lem}\label{finaltoptim} For any $\mu\geq0$, $\lambda\geq 2$.
% \begin{equation} \label{defJtilde}  \widetilde{J}_\lambda(\mu) := \min\Big( \frac{\mu^2}{4\beta}, \inf\big\{\frac{r^2}{4\alpha} +h_L(1+s) : r + m(\lambda)s \geq \mu : r,s\geq 0\big\}\Big).\end{equation}
%Then,
\begin{equation} \label{reducinf} 
\wh{\Phi}_\lambda(1/m(\lambda)) = \inf\big\{ \wh{\Phi}_\lambda(\mu)+\frac{t}{2\beta} : \mu + m(\lambda) t \geq 1/m(\lambda)\big\} .\end{equation}
\end{Lem}

\begin{proof}
We will prove that on the one hand for any $w\geq 0$, 
\begin{equation}\label{claim1inf} \inf\big\{h_L(1+s) +\frac{t}{2\beta}:  s + t\geq w\big\} = h_L(1+w).\end{equation}
and on the other hand
\begin{equation} \label{claim2inf}\inf\big\{\frac{\mu^2}{4\beta} +\frac{\nu}{2\beta}: \mu+m(\lambda) \nu\geq 1/m(\lambda)\big\} = \frac{1}{4\beta m(\lambda)^2}.\end{equation}
Once these two claims are proven, it follows immediately by Lemma \ref{optimkappa} that the infimum in \eqref{reducinf} is achieved for $t=0$.
%\corAB{and thus the proof completes upon invoking Lemma \ref{optimkappa}}. 

Now observe that as $h_L'$ is non decreasing and $\lim_{x \to +\infty} h_L'(x)=1/(2\beta)$ by Lemma \ref{lem:hL-prop}, the function $s\mapsto h_L(1+s) + (w-s)/(2\beta)$ is non increasing for any $w\geq0$. This yields \eqref{claim1inf}. The second claim  \eqref{claim2inf} follows from the fact that $\mu\mapsto \mu^2/(4\beta) + (1-m(\lambda)\mu)/(2\beta m(\lambda)^{2})$ is non increasing on $[0,1/m(\lambda)]$. 
\end{proof}

Proposition \ref{optimsol} is now obvious from Lemmas \ref{reducdim1}, \ref{optimkappa}, and \ref{finaltoptim}. We omit the details.%We are now ready to give a proof of Proposition \ref{optimsol}.

%\begin{proof}[Proof of Proposition \ref{optimsol}]
%By Lemma \ref{reducdim1} it follows that 
%\[ J_\lambda(1/m(\lambda)) = \inf\{ \widehat{J}_\lambda(\mu) + \frac{t}{2\beta} : \mu +m(\lambda) t \geq1/m(\lambda) \}.\]
%Lemma \ref{optimkappa} gives us that for any $\mu \geq 0$, $\widehat{J}_\lambda(\mu) = \widetilde{J}_\lambda(\mu)$, where $\widetilde{J}_\lambda$ is defined in \eqref{defJtilde}. Finally, using Lemma \ref{finaltoptim} we obtain that $J_\lambda(1/m(\lambda))= \widetilde{J}_\lambda(1/m(\lambda))$. 
%\end{proof}

\section{Local laws for supercritical sparse Wigner matrices}\label{sec:loc-law}
In this section we provide the proofs of Propositions \ref{prop:conc-loc-law} and \ref{computation-expec}, and Lemma \ref{concresolvent}. A key to these proofs is an isotropic local law for supercritical Wigner matrices under Assumption \ref{hypo:subg-entry} and in the whole regime $np \gg \log n$. Before stating our results we introduce the following notation. 
%For ease in writing, 
Fixing $\vep_0 \in (0,1/2)$, we let 
\begin{equation}\label{error}
\Upsilon_n:= \sqrt{\frac{\log n}{(np) \wedge n^{\vep_0}}}. 
\end{equation}
Set $\mathbb{H}:=\{z' \in \CC: \Im z'  >0\}$ and for any $\kappa, \vep \in (0,1)$,
\[
\mathbb{H}_\kappa:= \{z \in \mathbb{H}: |\Re z| \le \kappa^{-1} \text{ and } |\Re z \pm 2| \ge \kappa\}, \quad \mathbb{H}_{\kappa, \vep}:= \mathbb{H}_{\kappa, \vep}^n:= \{z \in \mathbb{H}_\kappa: \Im z \in [n^{-1+\vep}, \kappa^{-1}]\}.
\]
Fixing a sequence $(s_n)_{n\in \NN}$ of positive numbers such that $\lim_{n \to +\infty} s_n=0$ and $s_n \ge n^{-1+\veps}$ for all $n \in \NN$, and a numerical constant $\mathfrak{C}>0$ %going to $0$ as $n \to +\infty$, 
we further define
\[
\wt{\mathbb{H}}_\kappa:= \big\{ z \in \mathbb{H}_{\kappa}:  s_n \vee ((np) \wedge n^\veps)^{-\mathfrak{C}} \le  \Im z \le \kappa^{-1} \big\}, \quad \mathcal{U}_n:=\big\{u\in \mathbb{S}^{n-1} : \|u\|_1 \leq s_n np\big\}.
\]
For $W \in \cH_n$ and $z \in \mathbb{H} \cup (\RR \setminus {\rm Spec}(W))$ we define {$\cR(z,W):= (z -W)^{-1}$}.
%where ${\rm Spec}$ denotes the spectrum.  
When the choice of $z$ and/or $W$ will be clear from the context we will suppress the dependencies of these variables in the notation $\cR(\cdot, \cdot)$. For $\T \subset [n]$, $W \in \cH_n$, and $z \in \mathbb{H} \cup (\RR \setminus {\rm Spec}(W))$ we further let $\cR^{(\T)}:= \cR^{(\T)}(z,W):= (z -W^{(\T)})^{-1}$.
%\begin{Def}
%For a matrix $W \in \cH^n$ and $z \in \mathbb{H}:=\{z' \in \CC: \Im z'  >0\}$ we let \corAB{$R(z,W):= (z -W)^{-1}$}. Fix $\kappa, \vep \in (0,1)$ and set 
%\[
%\mathbb{H}_\kappa:= \{z \in \mathbb{H}: |\Re z| \le \kappa^{-1} \text{ and } |\Re z \pm 2| \ge \kappa\}, \mathbb{H}_{\kappa, \vep}:= \mathbb{H}_{\kappa, \vep}^n:= \{z \in \mathbb{H}_\kappa: \Im z \ge n^{-1+\vep}\},
%\]
%and 
%\[
%\wt{\mathbb{H}}_\kappa:= \bigg\{ z \in \mathbb{H}_\kappa: \Im z \ge \sqrt{\frac{\log n}{np}} \bigg\}. 
%\] 
%\end{Def}

To reach the isotropic local law for supercritical Wigner matrices we first prove the following entry wise local law. 
\begin{The}[Entry wise local law]\label{thm:loc-law}
Let $p$ be such that $ \log n \ll np \ll n$. Fix $\kappa \in (0,1)$ and $\vep \in (2\vep_0, 1)$. Then under Assumption \ref{hypo:subg-entry}, for any $M >0$ there exists a constant $C < +\infty$ such that
\[
\PP\bigg(\sup_{z \in \mathbb{H}_{\kappa, \vep}} \max_{\T \subset [n]: |\T| \lesssim 1}\max_{i,j \in [n]\setminus \T} |\cR^{(\T)}_{i,j}(z, X) - m(z) \updelta_{i,j}| \ge C \Upsilon_n \bigg) \le n^{-M},
\]
for all large $n$. Moreover, the same result holds for ${\rm Adj}_o/\sqrt{np}$ instead of $X$.
\end{The}

Building on Theorem \ref{thm:loc-law} one can prove the following result. 

\begin{The}[Isotropic local law]\label{thm:loc-law-iso}
%For any $z \in \wt{\mathbb{H}}_\kappa$, 
Under the same setup of Theorem \ref{thm:loc-law} we have
\[
\lim_{n \to +\infty}\sup_{z \in \wt{\mathbb{H}}_\kappa} \sup_{u \in \cU_n} \big| \EE \langle u, \cR(z, X) u \rangle - m(z) \big| =0.
\]
Moreover, the same result holds for ${\rm Adj}_o/\sqrt{np}$ instead of $X$.
\end{The}

%Results similar to Theorems \ref{thm:loc-law} and \ref{thm:loc-law-iso} in `sparse' settings have appeared in the literature. For example, local laws analogous to Theorem \ref{thm:loc-law} are derived in \cite{EKHY13, HKM19} for the adjacency matrix of supercritical Erd\H{o}s-R\'enyi graphs (both centred and non-centred versions), while under a boundedness assumption on the entries of $G$ (recall Definition \ref{defmodel}) local laws were derived in \cite{DZ19}. However, none of these results readily apply for the sparse Wigner matrices considered here in the entire supercritical regime. 

%On the other hand, a stronger version (see also Remark \ref{rem:iso-loc-law} below) of the isotropic law (Theorem \ref{thm:loc-law-iso}) was proved in \cite[Theorem 2.1]{BHY17} for deformed Wigner matrices. It was also noted in \cite[Section 4]{BHY17} that using \cite[Theorem 8.2]{BKY17} an isotropic local law holds for the centred adjacency matrix of spare Erd\H{o}s-R\'enyi graph when $u$ is orthogonal to the constant vector and $p$ is such that $\log(np) \gtrsim \log n$. We need the isotropic local law for vectors $u$ that are not necessarily orthogonal to the constant vector and in the entire supercritical regime. 

As already mentioned in Section \ref{sec:outline-bdd} the results available in literature do not yield Theorems \ref{thm:loc-law} and \ref{thm:loc-law-iso} in the generality we need. These two results are proved using adaptations of the general strategies developed in \cite{BHY17} and  \cite{HKM19} together with a careful counting argument.

%respectively, and are included later in this section. %These proofs involve adaptations of the  \cite{BHY17, EKHY13, HKM19}. %Readers familiar with these latter works may skip these proofs. 

%For example, \cite{EKHY13, HKM19} derive these results for the adjacency matrix of supercritical Erd\H{o}s-R\'enyi graphs (both centred and non-centred versions), while under a boundedness assumption on the entries of $G$ (recall Definition \ref{defmodel}) local laws were derived in \cite{DZ19}. In a different direction, such results were proved in \cite[Theorem 2.1]{BHY17} for deformed Wigner matrices. It was also noted in \cite[Section 4]{BHY17} that using \cite[Theorem 8.2]{BKY17} an isotropic local law holds for the centred adjacency matrix of spare Erd\H{o}s-R\'enyi graph when $u$ is orthogonal to the constant vector and $p$ is such that $\log(np) \gtrsim \log n$. Since 

 %we need to adapt the general strategies used in \cite{BHY17} and  \cite{HKM19}, respectively and perform a more delicate counting argument. 

\begin{Rem}\label{rem:iso-loc-law}
A couple more comments regarding the isotropic local law result in Theorem \ref{thm:loc-law-iso} are in order. The reader can note from its proof that under the additional assumption that the entries of $G$ have a symmetric distribution on $\RR$ the result would continue to hold for any $u \in \mathbb{S}^{n-1}$ such that $\|u\|_1 \le (np)^{O(1)}$, and thus for $p$ such that $\log (np) \gtrsim \log n$ one does not need any upper bound $\|u\|_1$ for the isotropic local law to hold.  The proof will also indicate that it is amenable to yield a bound on $\EE(|\langle u, \cR(z) u \rangle - m(z)|^{2k})$ for any $k \in \NN$. Since we do not need such a bound we have not pursued it here.  
\end{Rem}

\subsection{Proofs of Propositions \ref{prop:conc-loc-law} and \ref{computation-expec}, and Lemma \ref{concresolvent}}In this section we use Theorem \ref{thm:loc-law-iso} to prove these results.
%Propositions \ref{prop:conc-loc-law}. 

\begin{proof}[Proof of Proposition \ref{prop:conc-loc-law}]
We prove the required result only for $X$. The proof for ${\rm Adj}_o/\sqrt{np}$ being identical is omitted. Let $\lambda>2$ and $u\in \mathbb{B}^n$. Recall that %definition of $f_{\lambda,u}(K)$ for any $K\in\mathcal{H}_n$ such that $\lambda_K<\lambda$, as 
$f_{\lambda,u}(K)=\langle u,(\lambda-K)^{-1} u\rangle$ for $K \in \cH_n^\lambda$. Fix $\delta>0$ such that $\lambda \geq 2+2\delta$. 
As $f_{\lambda,u}$ is convex and $\delta^{-2}$-Lipschitz with respect to the Hilbert-Schmidt norm on $\mathcal{H}_n^{2+\delta}$, %(recall our notation in \ref{sec:notation}), 
there exists an extension $\widetilde{f}_{\lambda,u}$ of $f_{\lambda,u}$ to $\mathcal{H}_n$ such that $\widetilde{f}_{\lambda,u}$ is convex and  $\delta^{-2}$-Lipschitz. Namely, $\wt{f}_{\lambda,u}(K')=\sup\{ f_{\lambda,u}(K) +\langle \nabla f_{\lambda,u}(K),K'-K\rangle : K\in \mathcal{H}_n^{2+\delta}\}$ for $K'\in\mathcal{H}_n$. In a first step we show that 
\begin{equation} \label{controlmean}  \sup_{u\in \mathcal{U}_n} \big| \EE \widetilde{f}_{\lambda,u}(X) - m(\lambda) \big| =o(1).\end{equation}
Let $z= \lambda + {\rm i} \upeta$, where $\upeta=(\log \log n)^{-1}\vee s_n$ implying that $z \in \wt{\mathbb{H}}_\kappa$. Since  $\widetilde{f}_{\lambda, u}(K)=f_{\lambda, u}(K)$ for any $K \in \cH^{2+\delta}_n$, by triangle inequality we get that 
\begin{multline}\label{eq:iso-loc-pf1}
 \big|\EE \big[\widetilde{f}_{\lambda,u}(X) \big]-m(\lambda)\big| \le  \EE\big[\big|\widetilde f_{\lambda, u}(X)\big| \Car_{\cE^c}\big] + \EE\big[\big|\langle u, \cR(z, X) u \rangle\big|\Car_{\cE^c} \big] + |m(z) - m(\lambda)|\\
 + \EE \big[\big|\langle u, \cR(z, X) u \rangle - \langle u, \cR(\lambda, X) u \rangle\big|\Car_\cE\big]
 +  \big|\EE \langle u, \cR(z, X) u \rangle - m(z) \big|,
\end{multline}
where $\mathcal{E} := \{X\in \mathcal{H}^{2+\delta}_n\}$. The third and fourth terms on the right-hand side of the above inequality are bounded by $\delta^{-2} \upeta$, while the last term goes to $0$ as $n\to +\infty$ by Theorem \ref{thm:loc-law-iso} uniformly in $u\in \mathcal{U}_n$. For the two first terms, note that on the one hand as $\widetilde{f}_{\lambda,u}$ is convex and $\delta^{-2}$-Lipschitz and $\widetilde{f}_{\lambda,u}(0) = f_{\lambda,u}(0)=\lambda^{-1}$, $0$ being the zero matrix, for any $K\in \mathcal{H}_n^{2+\delta}$, we have
\[ \widetilde{f}_{\lambda,u}(K) \leq  \widetilde{f}_{\lambda,u}(0) + \delta^{-2} \|K\|_2 = \lambda^{-1} + \delta^{-2}\|K\|_2.\]
We deduce by Cauchy-Schwarz inequality and the fact that $\EE \|X\|_2^2=n$, that  $\EE\big[ \big| \widetilde{f}_{\lambda,u}(X)\big|\Car_{\mathcal{E}^c}\big] \leq \lambda^{-1} \PP(\mathcal{E}^c) + \delta^{-2} \sqrt{n}\PP(\mathcal{E}^c)^{1/2}$. We claim that, for any $C >0$, 
\begin{equation}\label{eq:cE-bd}
\limsup_{n \to +\infty} \frac{1}{\log n}\log \PP(\mathcal{E}^c) \le - C. 
\end{equation}
Since $np \gg \log n$ the bound \eqref{eq:cE-bd} implies that 
 $\EE\big[ \big| \widetilde{f}_{\lambda,u}(X)\big|\Car_{\mathcal{E}^c}\big]=o(1)$ uniformly in $u\in\mathbb{B}^n$. On the other hand, $|\langle u, \mathcal{R}(z,X) u\rangle |\leq \upeta^{-1}$, so that by our choice of $\upeta$, it follows that $ \EE\big[\big|\langle u, \cR(z, X) u \rangle\big|\Car_{\cE^c} \big] =o(1)$  uniformly in $u\in\mathbb{B}^n$. 

Thus to complete the proof of \eqref{controlmean} we need to derive \eqref{eq:cE-bd}. Recall the matrices $A^R$ and $B^R$ from Section \ref{section:decoupl}. Since the entries of $A^R$ are bounded, from \cite[Example 8.7]{BLM}, \cite[Theorem 2.7]{BeBoKn}, and the fact that $\EE (A_{1,2}^R)^2=1+o_R(1)$, it follows that
\begin{equation}\label{eq:cE-bd-1}
\limsup_{n \to \infty} \frac{1}{np} \log \PP(\|A^R\| \ge 2+\delta/2) <0,
\end{equation}
for $R$ large enough.
By \eqref{loglaplaB2-new} and the integration by parts formula, for any $\theta >0$ sufficiently small, we derive
\begin{multline}\label{eq:Bnorm2}
\EE\left[\|B_i^R\|^{N}\right] \le \delta_*^N + \int_{\delta_*}^{+\infty} N t^{N-1} \exp(-np \theta t^2 + o_R(1) np) dt \\
\le \delta_*^N + \int_{0}^{+\infty} N t^{N-1} \exp(-np \theta t^2/2) dt \le \delta_*^{N} + \sqrt{2\pi N} \left(\frac{N}{\theta np}\right)^{N/2},
\end{multline} 
for any $\delta_* >0$, $N \in \NN$, and all large $R$ (depending on $\delta_*$). By a symmetrisation argument and Seginer's theorem \cite[Theorem 1.1]{Seginer} we have $\EE[\|B^R\|^{\log n}] \le n \wt C^{\log n} \max_{i \in [n]} \EE\|B_i^R\|^{\log n}$, where $\wt C < +\infty$ is some numerical constant. Therefore, as $np \gg \log n$, by \eqref{eq:Bnorm2} and Markov's inequality 
\begin{equation}\label{eq:cE-bd-2}
\limsup_{R \to +\infty} \limsup_{n \to +\infty} \frac{1}{\log n}\log \PP(\|B^R\| \ge \delta/2) \le -\log(\delta/2) + \log(\delta_*) + \log (e \wt C) \le \frac12\log(\delta_*),
\end{equation}
where the last step follows upon choosing $\delta_*$ sufficiently small (depending on $\delta$). Since $\delta_* >0$ is can be chosen to be arbitrarily small the bound \eqref{eq:cE-bd} follows from \eqref{eq:cE-bd-1} and \eqref{eq:cE-bd-2}.

We now move on to prove \eqref{eq:conc-loc-law}.
%the concentration inequality of Proposition \ref{prop:conc-loc-law}. 
Since $(G_{i,j} \xi_{i,j})_{i\leq j}$ are independent sub-Gaussian random variables with uniformly bounded sub-Gaussian constants, by \cite[Theorem 1.3]{HT21}, there exists $c>0$ such that for any $t>\sqrt{\log n/(c np \delta^4)}$
\begin{equation}\label{eq:subg-conc}
\PP\Big(\big| \widetilde{f}_{\lambda, u}(X) - \EE [\widetilde{f}_{\lambda, u}(X)]  \big| \ge  t \Big) \leq \exp \Big(-\frac{c\delta^4 t^2 np}{ \log n}\Big).
\end{equation}\label{eq:subg-conc-n}
Using \eqref{controlmean}, we deduce that for any $t>0$ and $\lambda\geq 2+2\delta$, and $n$ large enough, 
\begin{equation} \label{concresolR} \PP\Big(\big| {f}_{\lambda, u}(X) - m(\lambda) \big| \ge  t, \lambda_{X}\leq 2+\delta \Big) \leq \exp \Big(-\frac{c\delta^4 t^2 np}{\log n}\Big), \quad u \in \mathcal{U}_n.\end{equation}
To obtain \eqref{eq:conc-loc-law} from \eqref{eq:subg-conc-n} we follow the following two step strategy:
As $\lambda \in (2+2\delta, +\infty) \mapsto f_{\lambda,u}(X)$ and $m$ are $\delta^{-2}$-Lipschitz on the event where $\lambda_X\leq 2+\delta$, using a net argument and a union bound (and possibly shrinking the constant $c$) we obtain \eqref{eq:conc-loc-law} with $\sup_{\lambda \ge 2+2\delta}$ replaced by $\sup_{\lambda \in ( 2+2\delta,2+2\delta+4t^{-1})}$. 

In a second step we observe that on $\{\lambda_X\leq 2+\delta\}$ we have $|f_{\lambda,u}(X)|\vee m(\lambda) \leq t/4$ for any $\lambda \geq 2+\delta +4t^{-1}$. This together with the first step and a triangle inequality proves  \eqref{eq:conc-loc-law}. %, and $m(\lambda) \leq t/4$ as well for $\lambda \geq 2+4t^{-1}$, 
% 
%We will now perform a net-argument over $\lambda\geq 2+2\delta$. Let $\mathcal{N}$ be a $t\delta^2/4$-net of $(2+2\delta,2+2\delta +4t^{-1})$ of size $|\mathcal{N}|\leq 16 (\delta t)^{-2}$. Since $\lambda \in (2+2\delta, +\infty) \mapsto f_{\lambda,u}(X)$ and $m$ are $\delta^{-2}$-Lipschitz on the event where $\lambda_X\leq 2+\delta$, we get that for $n$ large enough
%\[ \PP\Big(\sup_{\lambda \in ( 2+2\delta,2+2\delta+4t^{-1})} \big| {f}_{\lambda, u}(X) - m(\lambda) \big| \ge  t, \lambda_{X}\leq 2+\delta \Big) \leq 16 (\delta t)^{-2}  \exp\Big(-\frac{c\delta^4 t^2 np}{4 \log n}\Big), \quad u \in \mathcal{U}_n.\]
%Finally, observe that on the event where $\lambda_X\leq 2+\delta$, $|f_{\lambda,u}(X)|\leq t/4$ for any $\lambda \geq 2+\delta +4t^{-1}$, and $m(\lambda) \leq t/4$ as well for $\lambda \geq 2+4t^{-1}$, which ends the proof.
\end{proof}

%Let
%\[
%\cE:= \big\{ \|X^{([k])}\| \le 2 +\delta/2\big\} \qquad \text{ and } \qquad \Omega:= \big\{ \|Y\|_1 \vee \|Y'\|_1 \le C \sqrt{np}\big\}. 
%\]

%\end{proof}

\begin{proof}[Proof of Proposition \ref{computation-expec}] Recall the definition of $\widetilde{\phi}_\lambda^\veps$ from \eqref{defphitilde}.
Let $\delta\in (0,1)$ such that $\lambda>2+2\delta$. We show in a first step that 
\begin{equation}\label{step1expec}  \EE_{\mathcal{J}^c}\big[ \Car_{\{\|X^{(\mathcal{J})}\|>2+\delta\}} \widetilde{\phi}_{\lambda}^\veps(X^{(\mathcal{J})})\big]  =o(1)(\|\widecheck{X}_\mathcal{J}\|^2+1).\end{equation}
We claim that for any $K'\in \mathcal{H}_{\mathcal{J}^c}$, 
\begin{equation} \label{claimouside} \widetilde{\phi}_\lambda^\veps(K') \leq \lambda^{-1}\|\widecheck{X}_\mathcal{J}\|^2 +\|X_\mathcal{J}\|+r \|K'\|_2.\end{equation}
 Indeed, let $K'\in \mathcal{H}_{\mathcal{J}^c}$, $K\in E_\lambda^{\veps,r}$ and $\zeta \in\Theta(K)$. This means that there exists $v\in \mathbb{S}^\mathcal{J}$ such that $\phi_{\lambda,v}(K) = \phi_\lambda(K)$ and $\zeta = \nabla \phi_{\lambda,v}(K)\in r \mathcal{F}_\veps^1$. As $\phi_{\lambda, v}$ is convex, $\phi_{\lambda,v}(K)\leq \phi_{\lambda,v}(0) +  \langle \zeta, K\rangle \leq \lambda^{-1}\|\widecheck{X}_\mathcal{J}\|^2 +\|X_\mathcal{J}\|+ \langle \zeta, K\rangle$. Since $\phi_{\lambda,v}(K) = \phi_{\lambda}(K)$ and  $\|\zeta\|_2 \le r$, this gives the claimed bound \eqref{claimouside}. As the entries of $G$ are bounded, by \cite[Example 8.7]{BLM} and \cite[Theorem 2.7]{BeBoKn}, we obtain %Using the same argument as in Lemma \ref{expotightnessA}, we know that 
\begin{equation} \label{concsp}  \limsup_{n\to +\infty} \frac{1}{np} \log \PP\big( \|\widecheck{X}^{(\mathcal{J})}\|>2+\delta \big) <0.\end{equation}
By Cauchy-Schwarz inequality, $\EE \|X^{(\mathcal{J})}\|_2^2 =O(n)$ and the fact that $np \gg \log n$, we further  get
\begin{equation}\label{concsp-n}
\EE_{\mathcal{J}^c}\big[ \Car_{\{\|X^{(\mathcal{J})}\|>2+\delta\}} \| X^{(\mathcal{J})}\|_2 \big] =o(1).
\end{equation}
Besides, as $\mathcal{J}$ is of bounded size, $\|X_{\mathcal{J}}\|=o(1)$. Therefore, by \eqref{claimouside}-\eqref{concsp-n} the bound \eqref{step1expec} follows.
In a second step, we will prove that 
\begin{equation}\label{claimbulk} \EE_{\mathcal{J}^c}\big[\Car_{\{\|X^{(\mathcal{J})}\|\leq 2+\delta\}} \psi_\lambda(X^{(\mathcal{J})}) \big]\leq m(\lambda) \|\widecheck{X}_{\mathcal{J}}\|^2(1+o(1)),\end{equation}
% To show \eqref{claimbulk}, we will use Proposition  \ref{prop:conc-loc-law} $(ii)$ and a concentration argument from \cite{LDPei}.
%Define $\mathcal{H}^{2+\delta}_{\mathcal{J}^c} = \{K\in\mathcal{H}_{\mathcal{J}^c} : \lambda_K\leq 2+\delta\}$ and for any $v\in \mathbb{S}^{\mathcal{J}}$, 
where for $K\in\mathcal{H}^{2+\delta}_{\mathcal{J}^c}$  
\[  
\psi_{\lambda,v}(K) := \langle \widecheck{X}_\mathcal{J}v, (\lambda -K)^{-1}\widecheck{X}_\mathcal{J}v\rangle, \text{ for } v\in \mathbb{S}^{\mathcal{J}},  \text{ and } \psi_\lambda(K) := \sup_{v \in \mathbb{S}^\mathcal{J}}  \psi_{\lambda,v}(K).
\]
 As $\mathcal{J}$ is of bounded size and the entries of $G$ are bounded, on $\Omega_{\cJ}:= \{\|X^{(\mathcal{J})}\|\leq 2+\delta\}$ we have $\phi_\lambda(X^{(\mathcal{J})})\leq  \psi_{\lambda}(X^{(\mathcal{J})}) + o(1)$. Since  $\widetilde{\phi}_\lambda^\veps(K) \leq \phi_\lambda(K)$ for any $K \in \cH^{\lambda}_n$ we note that \eqref{step1expec} and \eqref{claimbulk} yield the required result.
 
We now move on to the proof of \eqref{claimbulk}. Since $\psi_{\lambda,v}$ is convex and $\delta^{-2} \|\widecheck{X}_\mathcal{J}\|$-Lipschitz with respect to the Hilbert-Schmidt norm on $\mathcal{H}^{2+\delta}_{\mathcal{J}^c}$, there exists $\wt{\psi}_{\lambda,v}$ a convex extension of ${\psi}_{\lambda,v}$ to $\mathcal{H}_{\mathcal{J}^c}$ that is also $\delta^{-2} \|\widecheck{X}_\mathcal{J}\|$-Lipschitz with respect to the Hilbert-Schmidt norm.
By \cite[Proposition 4.2]{LDPei}, for any $t>0$, almost surely on $\Omega_C$ we have
 %it suffices to show \eqref{claimbulk} with $\psi_\lambda$ instead of $\widetilde{\phi}_\lambda$.
%\corAB{To this end, set} $\widetilde{\psi}_{\lambda,v}(K) := \sup_{K' \in \corAB{\mathcal{H}^{2+\delta}_{\cJ^c}}} \big\{\psi_{\lambda,v}(K') + \langle \nabla \psi_{\lambda,v}(K'), K-K'\rangle\big\}$ for $K\in\mathcal{H}_{\mathcal{J}^c}$, is a convex extension of $\psi_{\lambda,v}$. 
%\corAB{As $\widetilde{\psi}_{\lambda,v} = \psi_{\lambda,v}$ on $\mathcal{H}^{2+\delta}_{\mathcal{J}^c}$, by} \cite[Proposition 4.2]{LDPei} \corAB{we find} that for any $t>0$, 
%\[ \PP_{\mathcal{J}}\big( \widetilde{\psi}_{\lambda,v}(X^{(\mathcal{J})}) - \EE_{\mathcal{J}}    \widetilde{\psi}_{\lambda,v}(X^{(\mathcal{J})}) >t \big) \leq e^{-\frac{t^2 \delta^4np }{32 R^2}},
%\]
%where $\widetilde{\psi}_{\lambda,v}(K) = \sup_{K' \in \mathcal{H}_{2+\delta}} \big\{\psi_{\lambda,v}(K') + \langle \nabla \psi_{\lambda,v}(K'), K-K'\rangle\big\}$ for any $K\in\mathcal{H}_{\mathcal{J}^c}$, is a convex extension of $\psi_{\lambda,v}$.
%As $\widetilde{\psi}_{\lambda,v} = \psi_{\lambda,v}$ on $\mathcal{H}^{2+\delta}_{\mathcal{J}^c}$, this yields in particular that 
\begin{multline}\label{concsup}  \PP_{\mathcal{J}^c}\Big( {\psi}_{\lambda,v}(X^{(\mathcal{J})}) - \sup_{v'\in \mathbb{S}^{\mathcal{J}}}\EE_{\mathcal{J}^c} \widetilde{\psi}_{\lambda,v'}(X^{(\mathcal{J})}) >t , X^{(\mathcal{J})} \in \mathcal{H}^{2+\delta}_{\mathcal{J}^c}\Big) \\
\leq \PP_{\mathcal{J}^c}\big( \widetilde{\psi}_{\lambda,v}(X^{(\mathcal{J})}) - \EE_{\mathcal{J}^c}    \widetilde{\psi}_{\lambda,v}(X^{(\mathcal{J})}) >t \big) \le e^{-\frac{t^2 \delta^4np }{32 (CR)^2}},\end{multline}
where $\PP_{\mathcal{J}^c}$ denotes the probability with respect to the randomness of $X^{(\cJ)}$. 
Let $\kappa>0$ such that $\# \mathcal{J}\leq \kappa$ almost surely and set $\eta:= (np)^{-1/(4\kappa)} \ll 1$. Note that when $X^{(\mathcal{J})} \in \mathcal{H}^{2+\delta}_{\mathcal{J}^c}$ the map $v\mapsto \psi_{\lambda,v}(X^{(\mathcal{J})})$ is $2\delta^{-1}\|\widecheck{X}_\mathcal{J}\|^2$-Lipschitz with respect to the $\ell^2$-norm.
By  \cite[Corollary 4.1.15]{AGM}, there exists a $\delta \eta/2$--net for the $\ell^{2}$-norm of $\mathbb{S}^{\mathcal{J}}$ of size at most $(6(\delta\eta)^{-1})^\kappa$. By a union bound it follows that  for any $t>0$, 
\[ \PP_{\mathcal{J}^c}\Big( {\psi}_{\lambda}(X^{(\mathcal{J})}) - \sup_{v\in \mathbb{S}^{\mathcal{J}}}\EE_{\mathcal{J}^c} \widetilde{\psi}_{\lambda,v}(X^{(\mathcal{J})}) >t+\eta \|\widecheck{X}_\mathcal{J}\|^2, X^{(\mathcal{J})} \in \mathcal{H}^{2+\delta}_{\mathcal{J}^c}\Big) \leq  \big(6 (\delta \eta)^{-1}\big)^{\kappa}e^{-\frac{t^2 \delta^4np }{32 (CR)^2}},\]
almost surely on $\Omega_C$.
Integrating this inequality, by our choice of $\eta$, 
%we get for any $\eta>0$,
%\[ \EE_{\mathcal{J}} \big[\Car_{\{\|X^{(\mathcal{J})}\|\leq 2+\delta\}}\psi_{\lambda}(X^{(\mathcal{J})})\big] - \sup_{v\in \mathbb{S}^{\mathcal{J}}}\EE_{\mathcal{J}} \widetilde{\psi}_{\lambda,v}(X^{(\mathcal{J})})\leq \eta \|\widecheck{X}_\mathcal{J}\|^2 + o(1) (6(\delta \eta)^{-1})^\kappa.\]
%This gives, by taking an appropriate sequence $\eta=\eta_n$ depending on $n$ and going to $0$ as $n$ goes to infinity, 
\begin{equation} \label{expectunif}
\EE_{\mathcal{J}^c} \big[\Car_{\{\|X^{(\mathcal{J})}\|\leq 2+\delta\}} \psi_{\lambda}(X^{(\mathcal{J})})\big] \leq \sup_{v\in \mathbb{S}^{\mathcal{J}}}\EE_{\mathcal{J}^c} \widetilde{\psi}_{\lambda,v}(X^{(\mathcal{J})})+ o(1)(\|\widecheck{X}_\mathcal{J}\|^2+1), \quad \text{a.s.~on $\Omega_C$}.
\end{equation}
%It remains to show that on the event $\Omega_C$ where  $\#\{(i,j) \in [n] \times \mathcal{J} : \xi_{i,j}=1\}\leq Cnp$, and  $\|\widecheck{X}_\mathcal{J}\|^2 \leq C$,  
On the other hand note that on %the event 
$\Omega_C$, 
$|\supp(\widecheck{X}_\mathcal{J}v)| \le Cnp$, and thus $\|\widecheck{X}_\mathcal{J}v\|_1\leq \sqrt{Cnp} \|\widecheck{X}_\mathcal{J}v\|_2$ for any $v\in\mathbb{S}^{\mathcal{J}}$. %\corAB{Therefore using Theorem \ref{thm:loc-law-iso} and arguing similarly as in the proof of \eqref{controlmean} we get} 
Therefore using \eqref{controlmean} we get, a.s.~on $\Omega_C$,
\begin{equation}  
\sup_{v\in \mathbb{S}^{\mathcal{J}}}\EE_{\mathcal{J}} \widetilde{\psi}_{\lambda,v}(X^{(\mathcal{J})}) \leq  m(\lambda)\|\widecheck{X}_{\mathcal{J}}\|^2(1+o(1)). \notag
\end{equation}
%Using Proposition \ref{prop:conc-loc-law} $(ii)$, we obtain \eqref{expectunif}, which ends the proof. 
This together with \eqref{expectunif} yields \eqref{claimbulk} completing the proof. 
\end{proof}

\begin{proof}[Proof of Lemma \ref{concresolvent}]
Applying Proposition \ref{strongconc}  and using the same argument as in the proof of \eqref{concflat} we obtain %by applying Proposition \ref{strongconc} 
that for any $t>0$, 
\[ \lim_{\veps \to 0} \limsup_{n\to +\infty}\frac{1}{np} \log \sup_{x \in  2\mathbb{B}^n}\PP\big( \widetilde{f}^{\eps,r}_{\lambda,x}(Y) - \EE  \widetilde{f}^{\eps,r}_{\lambda,x}(Y) >t \big) =-\infty.\]
Next note that for any $\veps,r>0$, 
%$E_{\lambda,x}^{\veps,r} \subset \mathcal{H}_{\lambda-\delta}$. Thus, 
$\widetilde{f}^{\veps,r}_{\lambda,x} \leq \widetilde{f}_{\lambda,x}$, where $\widetilde{f}_{\lambda,x}$ is as in the proof of Proposition \ref{prop:conc-loc-law}. Thus, by \eqref{controlmean}, $\EE \widetilde{f}_{\lambda,x}^{\veps, r}(Y) \leq \EE \wt{f}_{\lambda,x}(Y) \le m(\lambda)\|x\|^2(1+o(1))$, for all large $n$. This completes the proof.
%Using Proposition \ref{prop:conc-loc-law} $(i)$, we finally get the claim. 
%
\end{proof}

\subsection{Proof of Theorem \ref{thm:loc-law-iso}}
The proof of Theorem \ref{thm:loc-law-iso} requires the following concentration bounds. Recall $\Upsilon_n$ from \eqref{error}. We recall also that the sub-Gaussian constant of a real random variable $Y$ is defined as $\|Y\|_{\psi_2}:= \inf\{ t>0 : \EE (e^{t Y^2})\leq 1\}$.

\begin{Lem}\label{lem:loc-law-con-bd}
Let $\{Z_i\}_{i \in [n]}$ be a sequence of centered i.i.d.~sub-Gaussian random variables and $\{\xi_i\}_{i \in [n]}$ be a sequence of i.i.d.~$\dBer(p)$ random variables. Set $Y_i:= Z_i \xi_i/\sqrt{np}$, $i \in [n]$. Let $\{\wt Y_i\}_{i \in [n]}$ be an independent copy of $\{Y_i\}_{i \in [n]}$. Fix $\vep \in (2\vep_0,1)$, $M >0$, and $C_0 < +\infty$. 

\begin{enumerate}

\item[(a)] For any $a_1, a_2, \ldots, a_n \in \CC$ deterministic with 
\[
\Big(\frac1n\sum_{i=1}^n |a_i|^2\Big)^{1/2} \le n^{-\vep} \qquad \text{ and } \qquad {\max_{i\in[n]} |a_i|} \le C_0,
\]
there exists some constant $C<+\infty$ depending only the sub-Gaussian norm of $Z_1$, $C_0$, $\vep_0$, and $M$ such that, for all large $n$,
\begin{equation}\label{eq:conc-lin-sum}
\PP\Big(\Big|\sum_{i=1}^n a_i Y_i\Big| \ge C \Upsilon_n\Big) \le n^{-M}.
\end{equation}
\item[(b)] Let $np \gtrsim \log n$. For any $a_1, a_2, \ldots, a_n \in \CC$ deterministic, bounded in moduli by $C_0$, there exists some constant $C<+\infty$ depending only the sub-Gaussian norm of $Z_1$, $C_0$, and $M$ such that, for all large $n$,
\[
\PP\Big(\Big|\sum_{i=1}^n a_i (|Y_i|^2 - \EE |Y_i|^2)\Big| \ge C \Upsilon_n\Big) \le n^{-M}.
\]
\item[(c)] For any deterministic $a_{i,j}, i, j \in [n]$ with
\[
\Big(\max_i\frac1n\sum_{j=1}^n |a_{i,j}|^2\Big)^{1/2} \bigvee \Big(\max_j\frac1n\sum_{i=1}^n |a_{i,j}|^2\Big)^{1/2} \le n^{-\vep} \qquad \text{ and } \qquad {\max_{i,j\in[n]} |a_{i,j}|} \le C_0,
\]
there exists some constant $C<+\infty$ depending only the sub-Gaussian norm of $Z_1$, $C_0$, $\vep$, and $M$ such that, for all large $n$,
\[
\PP\Big(\Big|\sum_{i,j=1}^n a_{i,j} Y_i\wt Y_j \Big| \ge C \Upsilon_n^2\Big) \le n^{-M},\quad 
\PP\Big(\Big|\sum_{i \ne j \in [n]} a_{i,j} Y_i Y_j \Big| \ge C \Upsilon_n^2\Big) \le n^{-M}.
\]
\end{enumerate}
\end{Lem}

\begin{proof}
We consider only the case $a_i \in \RR, i \in [n]$. The proof for the complex case does not require any new ideas except for the fact that it is notationally heavy compared to its real counterpart. So we omit it. 
Set $\gamma:= n^{-\vep}$ and $\psi:= C_0 /\sqrt{np}$. 
%\[
%\gamma:=\Big(\frac1n\sum_{i=1}^n |a_i|^2\Big)^{1/2}  \qquad \text{ and } \qquad \psi:=\frac{\max_i |a_i|}{\sqrt{np}}. 
%\]
We will prove that for any $r \in 2\NN$
\begin{equation}\label{eq:mom-lin-sum}
\Big\| \sum_{i=1}^n a_i Y_i\Big\|_r \lesssim \bigg(\frac{r}{1+2 \log_+(\psi/\gamma)} \bigvee 1\bigg) \cdot (\gamma \vee \psi) \cdot \sqrt{r}.
\end{equation}
Observe that, for $r \asymp \log n$, the first term in the RHS of \eqref{eq:mom-lin-sum} can be trivially bounded by a $O(\log n)$ term, while for $np \le n^{2\vep_0}$ it is $O(1)$. Therefore, for $r \asymp \log n$, we deduce that $\|\sum_i a_i Y_i\|_r \lesssim \Upsilon_n$. Using this together with Markov's inequality, applied with $r \asymp \log n$, we obtain \eqref{eq:conc-lin-sum}. 

Turning to prove \eqref{eq:mom-lin-sum} we let $\mathfrak{P}(r)$ to be the set of partitions of $[r]$ and $\mathfrak{P}_{\ge 2}(r) \subset \mathfrak{P}(r)$ to be the subset that consists of partitions with all blocks having size at least two. For any $\Pi \in \mathfrak{P}(r)$ we write $|\Pi|$ to denote the number of blocks in $\Pi$. For a block $\pi \in \Pi$ we use $|\pi|$ to denote the number of elements in that block. Using the fact that $\EE Y_i=0$ we have that 
\begin{equation*}
\Big\| \sum_{i=1}^n a_i Y_i\Big\|_r^r = \sum_{i_1, i_2, \ldots, i_r=1}^n \Big(\prod_{j=1}^r a_{i_j}\Big) \EE \Big(\prod_{j=1}^r Y_{i_j} \Big) \le \sum_{\Pi \in \mathfrak{P}_{\ge 2}(r)} \sum_{s \in [n]^{|\Pi|}} \prod_{\pi \in \Pi} |a_{s_\pi}|^{|\pi|} \EE|Y_{s_\pi}|^{|\pi|}.
\end{equation*}
Using the fact that $Y_i=Z_i \xi_i/\sqrt{np}$, where $Z_i$ is a sub-Gaussian random variable and $\xi_i$ is a $\dBer(p)$ random variable, for any $\Pi \in \gP_{\ge 2}(r)$ we find that
\begin{multline}\label{eq:mom-lin-sum1}
\sum_{s \in [n]^{|\Pi|}} \prod_{\pi \in \Pi} |a_{s_\pi}|^{|\pi|} \EE|Y_{s_\pi}|^{|\pi|} \le \EE|Y_1|^r \prod_{\pi \in \Pi} \sum_{s \in [n]}  |a_{s}|^{|\pi|} \\
\le \|Y_1\|_{\psi_2}^{r} r^{r/2}  \prod_{\pi \in \Pi} \Big(\sum_{s \in [n]}  |a_{s}|^2\Big)\big(\max_{s \in [n]} |a_s|\big)^{|\pi|-2} \frac{1}{n (np)^{\frac12|\pi|-1}} \le  \|Y_1\|_{\psi_2}^{r} r^{r/2} \gamma^{2|\Pi|} \psi^{r - 2 |\Pi|},
\end{multline}
where the first step uses H\"older's inequality. We point out to the reader that in the case when $\sqrt{np} Y_i$ is a centred $\dBer(p)$ random variable a bound analogous to \eqref{eq:mom-lin-sum1}, without the factor $r^{r/2}$ was obtained in the proof of \cite[Proposition 3.1]{HKM19}. Therefore repeating the rest of the arguments of that proof one obtains \eqref{eq:mom-lin-sum}. 

To prove part (c) one is required to do similar adaptations of those of \cite[Proposition 3.3]{HKM19} and \cite[Proposition 3.4]{HKM19} to obtain
\begin{equation}\label{eq:mom-prod-sum1}
\Big\| \sum_{i,j=1}^n a_{i,j} Y_i\wt Y_j\Big\|_r \lesssim \bigg(\frac{r}{1+2 \log_+(\wt \psi/\gamma)} \bigvee 1\bigg)^2 \cdot (\gamma \vee \wt\psi) \cdot {r}.
\end{equation}
and
\begin{equation}\label{eq:mom-prod-sum2}
\Big\| \sum_{i\ne j \in [n]} a_{i,j} Y_iY_j\Big\|_r \lesssim \bigg(\frac{r}{1+2 \log_+(\wt \psi/\gamma)} \bigvee 1\bigg)^2 \cdot (\gamma \vee \wt \psi) \cdot {r},
\end{equation}
for any $r \in 2\NN$, where $\wt \psi:= C_0/np$. These bounds together with Markov's inequality yield the desired results. Further details are omitted. Finally, part (b) follows from \cite[Theorem 3.1]{zhou-spHW} and the fact that $np \gtrsim \log n$.
\end{proof}

%We are now ready to prove Theorem \ref{thm:loc-law-iso}. The strategy behind the proof is borrowed from that of \cite[Theorem 2.1]{BHY17}. Since our setup is different from that of \cite[Theorem 2.1]{BHY17} adaptations are required. 

\begin{proof}[Proof of Theorem \ref{thm:loc-law-iso}]
First we prove the required result for $X$.
To ease the notation, we drop the dependency of $\mathcal{R}(z,X)$ in $X$, and write $\mathcal{R}(z)$ instead, and similarly for $\mathcal{R}^{\mathbb{T}}(z)$, where $\mathbb{T} \subset [n]$.
We begin by noting that
\begin{equation}\label{eq:loc-iso-decomp}
\EE \langle u, \cR(z) u \rangle = \EE \Big[ \sum_{i \in [n]} u_i^2 \cR_{i,i}(z) \Big] + \EE\Big[\sum_{i \ne j \in [n]} u_i u_j \cR_{i,j}(z)\Big].
\end{equation}
Since $z \in \wt{\mathbb{H}}_\kappa$ implies that $\upeta:= \Im z \ge n^{-1}$, using the trivial bound $\|\cR(z)\| \le \upeta^{-1}$ and applying Theorem \ref{thm:loc-law}, as $u \in \mathbb{S}^{n-1}$, we find that
\begin{equation}\label{eq:loc-iso-diag}
 \EE \Big[ \sum_{i \in [n]} u_i^2 \cR_{i,i}(z) \Big] \to m(z), \qquad \text{ as } n \to +\infty,
\end{equation}
uniformly for $u \in \mathcal{U}_n$. Next for $i_1, i_2, i_3, i_4 \in [n]$ we let $\I:= \{i_1, i_2, i_3, i_4\}$ and set 
\[
\widecheck{\Omega}_{\I}:= \bigg\{ \max_{i,j \in [n]\setminus \I} |\cR^{(\I)}_{i,j}(z) - m(z) \updelta_{i,j}| \le C \Upsilon_n \bigg\},
\]
where $C$ is as in Theorem \ref{thm:loc-law} with $M=4\gf$, and $\gf$ some large integer to be determined below. By Theorem \ref{thm:loc-law}, as $\|u\|_1 \le n^{1/2}$ and $\|\cR(z)\| \le \upeta^{-1}$ we observe that
\[
\sup_{u \in \mathcal{U}_n} \bigg|\EE\Big[\sum_{i_1 \ne i_2 \in [n]}\sum_{i_3 \ne i_4 \in [n]}  u_{i_1} u_{i_2} u_{i_3} u_{i_4} \cR_{i_1,i_2}(z) \cR_{i_3,i_4}(z){\bf 1}_{\widecheck \Omega_{\I}^c}\Big]\bigg| \le n^4 \max_{\I} \PP(\widecheck \Omega_{\I}^c) \to 0
\]
as $n \to +\infty$. 
Hence, by \eqref{eq:loc-iso-decomp}, \eqref{eq:loc-iso-diag}, and Cauchy-Schwarz inequality it suffices to show that
\begin{equation}\label{eq:cR-off-diag}
\sup_{u \in \mathcal{U}_n}  \EE\Big[\sum_{i_1 \ne i_2 \in [n]}\sum_{i_3 \ne i_4 \in [n]}  u_{i_1} u_{i_2} u_{i_3} u_{i_4} \cR_{i_1,i_2}(z) \cR_{i_3,i_4}(z){\bf 1}_{\widecheck \Omega_{\I}}\Big]\to 0,
\end{equation}
as $n \to +\infty$. %To this end, for $i_\ell \ne j_\ell \in [n]$, $\ell =1,2$, we set $\T:=\T(i_\ell, j_\ell, \ell \in [2]):=\{i_1, j_1, i_2, j_2\}$. 
To compute the sum in the RHS of \eqref{eq:cR-off-diag} we split the sum according to $|\I|$. Observe that, when the sum is restricted over all those indices for which $|\I| =2$ the claim \eqref{eq:cR-off-diag} is immediate due to Theorem \ref{thm:loc-law}. So for the rest of the proof we will focus only on summing those indices for which $|\I| >2$. %When convenient and there is no scope of confusion we will suppress the dependencies in ${\bm i}$ and ${\bm j}$ in $\widecheck{\Omega}_{{\bm i}, {\bm j}}$ and $\T_{{\bm i}, {\bm j}}$ and write $\widecheck \Omega$ and $\T$, respectively.

Fix $i_1\ne i_2, i_3 \ne i_4 \in [n]$. Using the Schur complement formula, and the fact that $m(z)+m(z)^{-1} = z$ we obtain
\[
\cR_\I=  (\sD - \sE)^{-1},
\]
where $\sD:= \sD(z):= m(z)^{-1} {\rm Id} - \sE_1$, $\sE:= \sE_2+\sE_3$, 
\[
\sE_1:= (m_n^{(\I)}(z) - m(z)){\rm Id}, \quad \sE_2:= X_\I, \quad \sE_3:= \widecheck{X}_\I^{\sf T} \cR^{(\I)} \widecheck{X}_\I - m_n^{(\I)} {\rm Id},
\]
and $m_n^{(\I)}:=m_n^{(\I)}(z):=n^{-1} {\rm Tr}\, \cR^{(\I)}(z)$. %Define $\widecheck \Omega_\ell := \{\|\sE_\ell\| \le \wt C \Upsilon_n\}$, for $\ell \in [3]$, and $\wt C< +\infty$ is some constant determined below. 
Note that, as $np \gg \log n$ we have $\Upsilon_n \ll 1$, and $|m(z)| \asymp 1$ for any $z \in \wt{\mathbb{H}}_{\kappa}$ (see \cite[Lemma 3.2]{EKHY13}) we obtain that $\sD$ is invertible on the event $\widecheck{\Omega}_\I$ for $n$ large enough. Therefore, by the resolvent expansion, on the same event, upon writing $\widehat \sE:= \sD^{-1} \sE$, we obtain that 
\begin{equation}\label{eq:cRdef}
\cR=\sD^{-1} + \cR_1+\cR_2,
\end{equation}
where
\begin{equation}\label{eq:cR1def}
\cR_1:=\cR_1(z):=\sum_{a=1}^{\mathfrak{f}} \widehat\sE^a \sD^{-1} \qquad \text{ and } \qquad \cR_2:=\cR_2(z):= \widehat \sE^{\gf+1} (\sD - \sE)^{-1},
\end{equation}
for some $\gf \in \NN$ to be determined below. Notice that $\sD^{-1} = \wh m(z) {\rm Id}$, where $\wh m(z) := m(z)/ (1- (m_n^{(\I)}(z) - m(z)) m(z))$. On the event $\widecheck \Omega_\I$ we find that $|\wh m(z)| \asymp 1$ (use that $|m(z)| \asymp 1$) for $z \in \wt{\mathbb{H}}_\kappa$. Now fix any $a \in [\gf]$ and observe that
\begin{equation}\label{eq:cR1cR2-pre1}
(\wh \sE^a \sD^{-1})_{i_1, i_2} \cdot (\cR_2)_{i_3,i_4} = \wh m(z)^{a+\gf +2} \sum_{i_0 \in \I} \Big[\sum_{\wt \uppi, \uppi_{i_0}} \prod_{i \le j \in \I} \sE_{i,j}^{h_{i,j}}\Big] \cdot (\sD - \sE)^{-1}_{i_0, i_4},
\end{equation}
where the second sum in the RHS is over all possible paths $\wt\uppi$ of  length $a$ in the complete graph over $\I$ with starting and ending points $i_1$ and $i_2$, respectively, and paths $\uppi_{i_0}$ of length $\gf +1$ in the complete graph over $\I$ with starting and ending points $i_3$ and $i_0$, respectively. Here $h_{i,j}:=h_{i,j}(\wt \uppi, \uppi_{i_0})$ denotes the number of times the edge between $i$ and $j$ is traversed by paths $\wt \uppi$ and $\uppi_{i_0}$. Notice that for any $\wt \uppi$ and $\uppi_{i_0}$ as above
\begin{equation}\label{eq:sumhij}
\sum_{i \le j \in \I} h_{i,j} = a+\gf+1.
\end{equation}
Next observe that, on the event $\widecheck \Omega_\I$
using that $|m(z)| \asymp 1$ we obtain that
\[
\max_{i,j\in [n] \setminus \I} |\cR_{i,j}^{(\I)}(z)| \lesssim 1,
\]
and hence by Ward identity (see \cite[Eqn.~(4.4)]{HKM19} we further derive that on $\widecheck \Omega_\I$, as $z \in \wt{\mathbb{H}}_\kappa$,   
\[
\bigg(\max_i \frac1n\sum_{j \notin \I}|\cR_{i,j}^{(\I)}(z)|^2\bigg)^{1/2} \bigvee \bigg(\max_j \frac1n\sum_{i \notin \I}|\cR_{i,j}^{(\I)}(z)|^2\bigg)^{1/2} = \sqrt{\frac{\max_{i \notin \I} \Im \cR_{i,i}^{(\I)}(z)}{n\upeta}} \le n^{-\vep},
\]
for any $\vep >0$ sufficiently small. Therefore, the bounds \eqref{eq:mom-prod-sum1}-\eqref{eq:mom-prod-sum2}, and an argument analogous to that in \cite[Proposition 3.2]{HKM19} yield the bound  
\begin{equation}\label{eq:E3mom2}
 \max_{i,j \in \I}\EE\Big[((\sE_{3})_{i,j})^{2h} \Big| X^{(\I)}\Big] \lesssim ((np) \wedge n^{\vep})^{-h},
\end{equation}
for any $h \in \NN$ with $h \asymp 1$ on the event $\widecheck{\Omega}_\I$. Furthermore, due to independence we have that
\begin{equation}\label{eq:E2mom}
 \max_{i,j \in \I}\EE\Big[((\sE_{2})_{i,j})^{2h}\Big| X^{(\I)}\Big] \lesssim n^{-1} \cdot (np)^{1-h} \ll (np)^{-h},
 \end{equation}
 as $p \ll 1$, again for any $h \in \NN$ with $h \asymp 1$. Thus by \eqref{eq:sumhij}, \eqref{eq:E3mom2}-\eqref{eq:E2mom}, H\"older's inequality, %\sout{\corAQ{\eqref{eq:sumhij}}} 
 and Theorem \ref{thm:loc-law} we deduce that 
\begin{equation}
\EE\Big[\Big(\prod_{i \le j \in \I} \sE_{i,j}^{2h_{i,j}}\Big){\bf 1}_{\widecheck \Omega_\I} \Big] = \EE\bigg[\EE\Big[\prod_{i \le j \in \I} \sE_{i,j}^{2h_{i,j}}\Big| X^{(\I)}\Big] {\bf 1}_{\widecheck \Omega_\I}\bigg] %+ \EE\bigg[\EE\Big[\prod_{i \le j \in \T} \sE_{i,j}^{2h_{i,j}}\Big| X^{(\T)}\Big] {\bf 1}(X^{(\T)}\notin \widecheck \Omega_0) \bigg]
\lesssim ((np) \wedge n^{\vep})^{-\gf}. %+ n^{2\gf} \PP(X^{(\T)} \notin \widecheck{\Omega}_0) \lesssim ((np) \wedge n^{\vep})^{-\gf}. 
\end{equation}
Hence, by Cauchy-Schwarz inequality, \eqref{eq:cR1cR2-pre1}, and the fact that $|\wh m(z)| \lesssim 1$ on $\widecheck \Omega_\I$ and $\|\mathcal{R}(z)\|\leq \upeta$, and the lower bound $\upeta$ we further see that %, as $z \in \wt{\mathbb{H}}_\kappa$,
\begin{multline}\label{eq:cR1cR2-pre}
\Big|\EE\Big[ (\wh \sE^a \sD^{-1})_{i_1, i_2} \cdot (\cR_2)_{i_3,i_4} {\bf 1}_{\widecheck\Omega_{\I}}\Big]\Big| \\
\lesssim \max_{i_0 \in \I} \max_{\wt\uppi, \uppi_{i_0}} \sqrt{\EE\Big[\Big(\prod_{i \le j \in \I} \sE_{i,j}^{2h_{i,j}}\Big){\bf 1}_{\widecheck \Omega_{\I}} \Big]} \cdot\sqrt{\EE \Big[|\wh m(z)|^{2(a+\gf +2)}\|\cR(z)\| {\bf 1}_{\widecheck \Omega_{\I}}\Big]} \\\lesssim  ((np) \wedge n^{\vep})^{-\gf/2} \cdot \upeta^{-1/2} \ll (np)^{-4},
\end{multline}
for $\gf \in \NN$ sufficiently large, any $a \in [\gf]$. %and any $i_1,j_1, i_2, j_2 \in [n]$. 
Fix this choice of $\gf$ for the rest of the proof. Since $\|u\|_1 \ll np$, the bound \eqref{eq:cR1cR2-pre} in turn implies that
\begin{multline}\label{eq:cR1cR2}
 \bigg|\EE\Big[\sum_{i_1 \ne i_2 \in [n]}\sum_{i_3 \ne i_4 \in [n]}  u_{i_1} u_{i_2} u_{i_3} u_{i_4} (\cR_{1})_{i_1,i_2}(\cR_{2})_{i_3,i_4}{\bf 1}_{\widecheck \Omega_{\I}}\Big]\bigg| \\
 \lesssim \|u\|_1^4 \max_{a \in [\gf]}\max_{\I} \Big|\EE\Big[ (\wh \sE^a \sD^{-1})_{i_1, i_2} \cdot (\cR_2)_{i_3,i_4} {\bf 1}_{\widecheck \Omega_{\I}} \Big]\Big| \to 0,
\end{multline}
as $n \to +\infty$. Note that a similar argument further yields that 
\begin{equation}\label{eq:cR1cR2-new}
\EE\Big[\sum_{i_1 \ne i_2 \in [n]}\sum_{i_3 \ne i_4 \in [n]}  u_{i_1} u_{i_2} u_{i_3} u_{i_4} (\cR_{\ell})_{i_1,i_2}(\cR_{\ell'})_{i_3,i_4} {\bf 1}_{\widecheck \Omega_{\I}} \Big] \to 0,
\end{equation}
as $n \to +\infty$, for $\ell=2$ and $\ell'=1,2$. Since $i_{2\ell-1}\ne  i_{2\ell}$ for $\ell=1,2$ and $\sD$ is a diagonal matrix we also have that $\cR_{i_{2\ell-1}, i_{2\ell}} = (\cR_1)_{i_{2\ell-1}, i_{2\ell}} + (\cR_2)_{i_{2\ell-1}, i_{2\ell}}$ (see \eqref{eq:cRdef}). Therefore, by \eqref{eq:cR1cR2}-\eqref{eq:cR1cR2-new} and recalling \eqref{eq:cR1def} we deduce that it is enough to prove the following for any $a_1, a_2 \in [\gf]$ and $k=3,4$:
\begin{equation}\label{eq:cR1cR1}
 \lim_{n \to +\infty}\EE\bigg[\sum_{i_1 \ne i_2 \in [n]}\sum_{i_3 \ne i_4 \in [n]}  u_{i_1} u_{i_2} u_{i_3} u_{i_4} \wh m(z)^{a_1+a_2+2} (\sE^{a_1})_{i_1,i_2}(\sE^{a_2})_{i_3,i_4} {\bf 1}_{\widecheck \Omega_{\I}} {\bf 1}_{|\I|=k} \bigg] =0.
\end{equation}
We will prove that %\corFA{\sout{for $i_{2\ell-1}\ne  i_{2\ell} \in [n]$ with $\ell=1,2$}}, 
\begin{equation}\label{eq:cR1cR1-new}
\Big|\EE\Big[ (\sE^{a_1})_{i_1,i_2}(\sE^{a_2})_{i_3,i_4} \Big| X^{(\I)}\Big] \Big| {\bf 1}_{\widecheck \Omega_{\I}} {\bf 1}_{|\I|=k} \lesssim \corEpty{(np)^{-\lfloor \frac{k}{2}\rfloor}  \cdot (n\upeta)^{-\frac12\lceil \frac{k}{2}\rceil}}\qquad \text{ for } k=3,4. 
\end{equation}
Notice that $u \in \mathcal{U}_n \subset \mathbb{S}^{n-1}$ implies that $\|u\|_1 \le n^{1/2} \wedge (s_n np)$ and hence $\|u\|_1^4 \le n \cdot (s_n np)^2$. The latter bound on $\|u\|_1$ together with \eqref{eq:cR1cR1-new}, the bound 
on $|\wh m(z)|$ on the event $\widecheck\Omega_\I$, and the lower bound on $\upeta$ for $z \in \wt{\mathbb{H}}_\kappa$ yield \eqref{eq:cR1cR1}. 

We now move on to the proof of \eqref{eq:cR1cR1-new}. This requires some careful counting arguments. We will represent the LHS of \eqref{eq:cR1cR1-new} as a sum of `weights' of a collection of graphs. It needs some notation. 

To this end, we let $\mathfrak{P}$ to be the set of all finite sequence of paths on $\I$ (self loops and backtracking allowed). For a sequence of paths $\gG \in \gP$ and $i \le j \in \I$ we set $e(i,j)$ to be the number of times the edge $(i,j)$ appears in $\gG$. By a slight abuse of notation we write $(i,j) \in \gG$ to denote that the edge $(i,j)$ belongs to the sequence of paths $\gG$. Further let ${\bm n}(\gG)$ to be the number of distinct edges in $\gG$ and write $|\gG|$ to denote the total length of the paths in $\gG$. %Set $\gG_2:=\Pi_2(\Pi)$ to be the weighted graph obtained from $\Pi$ by putting a weight $(\sE_2)_{i,j}^{e_2(i,j)}$ for every distinct edge $(i,j)$ appearing in $\Pi$.  
Let
\[
\varpi_\ell(\gG):= \prod_{i \le j \in \I} (\sE_\ell)_{i,j}^{e(i,j)}, \qquad \ell=2,3 \text{ and } \gG \in \gP.
\]
Since $\sE=\sE_2+\sE_3$, expanding the LHS of \eqref{eq:cR1cR1-new} and using the joint independence of $\sE_2$ and $\sE_3$, conditioned on $X^{(\I)}$, we see that
\begin{equation}\label{eq:cR1cR1-new1}
\EE\Big[ (\sE^{a_1})_{i_1,i_2}(\sE^{a_2})_{i_3,i_4} \Big| X^{(\I)}\Big] = \sum_{\gG, \wt \gG} \EE\Big[\varpi_2(\wt \gG)\big|X^{(\I)}\Big] \cdot \EE\Big[\varpi_3( \gG)\big| X^{(\I)}\Big],
\end{equation}
where the sum is over those $\gG, \wt \gG \in \gP$ such that the alternating concatenation of paths in $\gG$ and $\wt \gG$ is the union of a path from $i_1$ to $i_2$ and a path from  $i_3$ to $i_4$, and $|\gG| + |\wt \gG| =a_1+a_2$. To bound the LHS of \eqref{eq:cR1cR1-new} it suffices to bound each of the summand appearing in the RHS of \eqref{eq:cR1cR1-new1}. 

The bound on the first term in the summand is easy. Indeed, by the definition of $X$, as its entries are centred, we have that 
\begin{equation}\label{eq:varpi2}
\Big|\EE\Big[ \varpi_2(\wt \gG) \Big| X^{(\I)}\Big]\Big| \lesssim \Big(\frac{1}{n}\Big)^{{\bm n}(\wt \gG)} \cdot \Big(\frac{1}{np}\Big)^{|\wt \gG|/2 -  {\bm n}(\wt \gG)} \cdot \prod_{\substack{i \le j \in \I\\ (i,j) \in \wt \gG}} {\bf 1}(e(i,j) \ge 2). 
\end{equation}

Obtaining the bound on the second term in the summand needs additional work. Turning to get such a bound, for any $\gG \in \gP$ and ${\bm \upbeta}=(\upbeta_1,\upbeta_2, \ldots, \upbeta_{2|\gG|}) \in ([n]\setminus \I)^{2|\gG|}$ we define $\gG({\bm \upbeta})$ to be a bipartite multigraph with partitions $\I$ and $[n] \setminus \I$ as follows: 
%Choose any arbitrary ordering of the paths contained in $\gG$. Now for each path choose its starting point in an arbitrary fashion and then order the edges in the path as they are traversed starting from the starting point. This gives an ordering of the edges in $\gG$.  
For every $k \in [|\gG|]$ replace the $k$-th edge $(i,j)$ in $\gG$ with the pair of edges $(i,\upbeta_{2k-1})$ and $(j, \upbeta_{2k})$. The resulting graph is $\gG({\bm \upbeta})$. Next for any edge $(i,\upbeta) \in \gG({\bm \upbeta})$ we further let $e(i, \upbeta)$ to be the number of times that edge appears in $\gG({\bm \upbeta})$, 
\[
e_d(i, \upbeta):= \Big| \Big\{k \in [|\gG|]: (i,i) \text{ is the } k\text{-th edge in } \gG \text{ and } \upbeta_{2k-1}=\upbeta_{2k}=\beta\Big\}\Big|, 
\]
$e_o(i, \upbeta):= e(i, \upbeta) -  2e_d(i, \upbeta)$, and $\Delta(i, \upbeta):= e(i, \upbeta)/2 -1$. Define ${\bm n}(\gG({\bm \upbeta}))$ to be the number distinct edges in $\gG({\bm \upbeta})$. Recalling the definition of $X$ and the fact that it is centred we find that
\begin{equation}\label{eq:varpi3}
\Big|\EE\Big[ \prod_{(i, \upbeta) \in \gG({\bm \upbeta})} X_{i, \upbeta}^{e_o(i, \upbeta)} \big(X_{i,\upbeta}^2 -\frac1n\big)^{e_d(i,\upbeta)}\Big]\Big| \lesssim \Big(\frac{1}{n}\Big)^{{\bm n}(\gG({\bm \upbeta}))}\cdot \Big(\frac{1}{np}\Big)^{\Delta(\gG({\bm \upbeta}))} \cdot \gI({\bm \upbeta}),
\end{equation}
where 
\begin{equation}\label{eq:gIdef}
\gI({\bm \upbeta}):= \prod_{(i,\upbeta) \in \gG({\bm \upbeta})} {\bf 1}(e_o(i,\upbeta) + e_d(i, \upbeta) \ge 2) \qquad \text{ and } \qquad \Delta(\gG({\bm \upbeta})):= \sum_{(i, \upbeta) \in \gG({\bm \upbeta)}}\Delta(i, \upbeta).
\end{equation} 
%
%Let ${\bm n}({\bm u\beta})$ to be the number of distinct and note that ${\bm n}(\gG({\bm \upbeta})) \ge {\bm n}({\bm \upbeta})$. 
On the other hand, as $\widecheck{X}_\I$ is independent of $X^{(\I)}$, upon recalling the definition of $\sE_3$ we find that
\begin{equation}\label{eq:varpi3-1}
\Big|\EE\Big[ \varpi_3( \gG) \Big| X^{(\I)}\Big] \Big|\le \sum_{{\bm \upbeta}} \varpi(\gG({\bm \upbeta})) \Big|\EE\Big[ \prod_{(i, \upbeta) \in \gG({\bm \upbeta})} X_{i, \upbeta}^{e_o(i, \upbeta)} \big(X_{i,\upbeta}^2 -\frac1n\big)^{e_d(i,\upbeta)}\Big] \Big|, %\\
%\lesssim \sum_{{\bm \upbeta}} \varpi(\gG({\bm \upbeta})) \cdot \Big(\frac{1}{n}\Big)^{{\bm n}(\gG({\bm \upbeta}))}\cdot \prod_{(i,\upbeta) \in \gG({\bm \upbeta})}\Big(\frac{1}{np}\Big)^{e_o(i,\upbeta)/2+e_d(i,\upbeta)-1} \cdot {\bf 1}(e_o(i,\upbeta) + e_d(i, \upbeta) \ge 2).  
\end{equation}
where 
\[
\varpi( \gG({\bm \upbeta})):= \prod_{k=1}^{|\gG|} |\cR^{(\I)}_{\upbeta_{2k-1}, \upbeta_{2k}}|.  
\]
To obtain the necessary bound on the RHS of \eqref{eq:varpi3-1} we need to find the same when $\varpi(\gG({\bm \upbeta}))$ is summed over various subsets of ${\bm \upbeta}$'s. We will iteratively sum over different coordinates of ${\bm \upbeta}$ that are determined by a sequence of graphs described by `partitions' induced by ${\bm \upbeta}$ which we define below. For any ${\bm \upbeta}$ as above we set $\wh\gG_{0}({\bm \upbeta})$ to be the graph with vertex set ${\bm \upbeta}$ and with edge set $\{(\upbeta_{2k-1}, \upbeta_{2k}), k \in [|\gG|]\}$. Obtain $\wh\gG_{1}({\bm \upbeta})$ from $\wh\gG_{0}({\bm \upbeta})$ by deleting all self loops, removing the possible resulting isolated vertices, and keeping one edge per each set of parallel edges. 
Having defined $\wh\gG_{\ell}({\bm \upbeta})$ for some $\ell \ge 1$ we proceed as follows: If the graph is empty or all vertices are of degree one we terminate the procedure. Otherwise, consider the vertex with the largest degree (choose any if there is more than one). Remove that vertex, all edges incident to it, and all degree one neighbours of it (if any). The resulting subgraph is set to be $\wh\gG_{\ell+1}({\bm \upbeta})$. Use ${\bm \upbeta}^{(\ell)}$ to denote the vertex set of $\wh \gG_\ell({\bm \upbeta})$. Given ${\bm s}:=(s_0, s_1, \ldots, s_L)$ we let $\gJ_{{\bm s}}({\bm \upbeta})$ to be the indicator function of the subset ${\bm \upbeta}$ for which the sequence of graphs (described above) is terminated for $\ell=L$ and $ |{\bm \upbeta}^{(\ell)}|=s_\ell$ for $\ell =0,1,\ldots, L$. For $\ell \ge 0$, by an abuse of notation, we let $\gJ_{{\bm s}}({\bm \upbeta}^{(\ell)}) = \gJ_{{\bm s}}({\bm \upbeta})$.  

We claim that, on the event $\widecheck \Omega_\I$, for any indicator function $\gI({\bm \upbeta})$ we have
\begin{equation}\label{eq:varpi-bd}
\sum_{{\bm \upbeta}} \varpi (\gG({\bm \upbeta})) \gJ_{{\bm s}}({\bm \upbeta}) \gI({\bm \upbeta})\le\sum_{{\bm \upbeta}} \varpi (\gG({\bm \upbeta})) \gJ_{{\bm s}}({\bm \upbeta}) \lesssim n^{s_0} \cdot \Big(\frac{1}{n \upeta}\Big)^{s_1/4}.
\end{equation}
Since $\varpi(\cdot)$ is nonnegative  the first inequality in \eqref{eq:varpi-bd} is obvious. To prove the second inequality define the `weights' of the graphs $\{\wh \gG_\ell({\bm \upbeta})\}_{\ell =0}^L$ as follows:
\[
\wh \varpi_\ell({\bm \upbeta}^{(\ell)}):= \wh\varpi_\ell (\wh \gG_\ell({\bm \upbeta})):= \prod_{(\upbeta, \upbeta') \in \wh\gG_\ell({\bm \upbeta})} |\cR^{(\I)}_{\upbeta, \upbeta'}|. 
\]
In words, for every edge $(\upbeta, \upbeta') \in \wh \gG_0({\bm \upbeta})$ we set its weight to be $|\cR^{(\I)}_{\upbeta, \upbeta'}|$. The weight of the graph $\wh \gG_\ell({\bm \upbeta})$, $\ell \ge 0$, is set to be the product of its edge weights (note $\wh \gG_\ell({\bm \upbeta}) \subset \wh \gG_0({\bm \upbeta})$). 

Observe that for any ${\bm \upbeta}$ such that $\gJ_{{\bm s}}({\bm \upbeta})=1$, when moving from $\wh \gG_0({\bm \upbeta})$ to $\wh \gG_1({\bm \upbeta})$, the number of vertices lost is $s_0-s_1$. These are the vertices that participate only in the self loops in $\wh \gG_0({\bm \upbeta})$, and each such vertex has at most $n$ choices. Therefore, recalling that on the event $\widecheck \Omega_\I$ one has $\max_{i,j \in [n]\setminus \I} |\cR^{(\I)}_{i,j}| \lesssim 1$ we obtain that 
\begin{equation}\label{eq:varpi-bd1}
\sum_{{\bm \upbeta}^{(0)}} \wh \varpi_0({\bm \upbeta}^{(0)}) \gJ_{{\bm s}}({\bm \upbeta}^{(0)}) \lesssim n^{s_0-s_1} \sum_{{\bm \upbeta}^{(1)}} \wh \varpi_1({\bm \upbeta}^{(1)}) \gJ_{{\bm s}}({\bm \upbeta}^{(1)}).
\end{equation}
Next we aim to argue that for any $\ell=1,2\ldots, L-1$ we have that 
\begin{equation}\label{eq:varpi-bd2}
\sum_{{\bm \upbeta}^{(\ell)}} \wh \varpi_\ell({\bm \upbeta}^{(\ell)}) \gJ_{{\bm s}}({\bm \upbeta}^{(\ell)}) \lesssim \bigg(\frac{n}{(n \upeta)^{1/4}}\bigg)^{s_\ell - s_{\ell+1}} \sum_{{\bm \upbeta}^{(\ell+1)}} \wh \varpi_\ell({\bm \upbeta}^{(\ell+1)}) \gJ_{{\bm s}}({\bm \upbeta}^{(\ell+1)}).
\end{equation}
To prove \eqref{eq:varpi-bd2} we need to consider two cases. Without loss of generality let us assume that $\upbeta_1$ is the vertex with the maximum degree in $\wh \gG_\ell({\bm \upbeta})$. By definition, the maximal degree is at least two. Further assume that all its neighbours have degree at least two, and thus $s_\ell - s_{\ell+1}=1$. Pick any two such neighbours, say $\upbeta$ and $\upbeta'$. By Ward identity (see \cite[Eqn.~(4.4)]{HKM19}) and Cauchy-Schwarz inequality, on the event $\widecheck \Omega_\I$, we find that 
\[
\sum_{\upbeta_1} \prod_{(\wh \upbeta, \upbeta_1) \in \wh \gG_\ell({\bm \upbeta})} |\cR^{(\I)}_{\wh \upbeta, \upbeta_1}| \lesssim \sum_{\upbeta_1} |\cR^{(\I)}_{\upbeta, \upbeta_1}| |\cR^{(\I)}_{\upbeta', \upbeta_1}| \le \sqrt{\frac{\Im \cR^{(\I)}_{\upbeta, \upbeta}}{\upeta}} \sqrt{\frac{\Im \cR^{(\I)}_{\upbeta', \upbeta'}}{\upeta}}\lesssim \frac1\upeta \ll \bigg(\frac{n}{(n \upeta)^{1/4}}\bigg)^{s_\ell - s_{\ell+1}}, 
\]
yielding \eqref{eq:varpi-bd2}. Next consider the case when $\upbeta_1$ has $k$ neighbours of degree one, to be denoted by $\wt \upbeta_1, \wt \upbeta_2, \ldots, \wt \upbeta_k$, for some $k \ge 1$. Using Ward identity and Cauchy-Schwarz inequality again, on the event $\widecheck \Omega_\I$, we obtain that 
\begin{multline*}
\sum_{\wt\upbeta_1, \wt \upbeta_2, \ldots, \wt \upbeta_k} \sum_{\upbeta_1} \prod_{(\wh \upbeta, \upbeta_1) \in \wh \gG_\ell({\bm \upbeta})} |\cR^{(\I)}_{\wh \upbeta, \upbeta_1}| \lesssim \sum_{\upbeta_1} \prod_{i=1}^k \bigg(\sum_{\wt \upbeta_i} |\cR^{(\I)}_{\upbeta_1, \wt \upbeta_i}|\bigg) \le \sum_{\upbeta_1}  \bigg(\frac{n \Im \cR^{(\I)}_{\upbeta_1, \upbeta_1}}{\upeta}\bigg)^{k/2} \lesssim \frac{n^{k+1}}{(n \upeta)^{k/2}} \\
\le \bigg(\frac{n}{(n \upeta)^{1/4}}\bigg)^{s_\ell - s_{\ell+1}},
\end{multline*}
where in the last step we used that $s_\ell - s _{\ell +1} =k+1$ and $k \ge 1$. Thus we have \eqref{eq:varpi-bd2} also in this case. 

Finally, recalling that $\wh \gG_L({\bm \upbeta})$ is a collection of disjoint edges, denoting $\{(\wh \upbeta_i, \wt \upbeta_i), i \in [s_L/2]\}$ to be its set of edges, and applying Ward identity and Cauchy-Schwarz inequality yet again we derive that, on the event $\widecheck \Omega_\I$, 
\begin{equation}\label{eq:varpi-bd3}
\sum_{{\bm \upbeta}^{(L)}} \wh \varpi_L({\bm \upbeta}^{(L)}) \leq  \Big(\sum_{\wh \upbeta, \wt \upbeta} |\cR^{(\I)}_{\wh \upbeta, \wt \upbeta}|\Big)^{s_L/2} \lesssim \bigg(n\sqrt{\frac{n}{\upeta}}\bigg)^{s_L/2} = \bigg(\frac{n}{(n \upeta)^{1/4}}\bigg)^{s_L}.
\end{equation}
Combining the bounds \eqref{eq:varpi-bd1}, \eqref{eq:varpi-bd2}, and \eqref{eq:varpi-bd3} we obtain \eqref{eq:varpi-bd}. 

Equipped with the necessary bounds we now proceed to combine these bounds and complete the proof of this theorem. By \eqref{eq:varpi2}, \eqref{eq:varpi3}, and \eqref{eq:varpi3-1} we have
\begin{equation}\label{eq:varpi3-2}
 \Big|\EE\Big[ \varpi_2(\wt \gG) \Big| X^{(\I)}\Big]\Big| \cdot \Big|\EE\Big[ \varpi_3(\gG) \Big| X^{(\I)}\Big]\Big| \lesssim \sum_{\upgamma_0, \upgamma_1} \sum_{{\bm s}} T(\upgamma_0, \upgamma_1, {\bm s}) \cdot \wt T,
\end{equation}
where
\begin{equation}\label{eq:varpi3-3}
T(\upgamma_0, \upgamma_1, {\bm s}):= T(\upgamma_0, \upgamma_1, {\bm s}, \gG):=\Big(\frac{1}{n}\Big)^{\upgamma_0}\cdot \Big(\frac{1}{np}\Big)^{\upgamma_1} \cdot \sum_{{\bm \upbeta}} \varpi (\gG({\bm \upbeta})) \gJ_{{\bm s}}({\bm \upbeta})  \wh \gI({\bm \upbeta}),
\end{equation}
\begin{equation}\label{eq:wtT}
\wt T:= \wt T(\wt \gG):= \Big(\frac{1}{n}\Big)^{{\bm n}(\wt \gG)} \cdot \Big(\frac{1}{np}\Big)^{|\wt \gG|/2 -  {\bm n}(\wt \gG)} \cdot \prod_{\substack{i \le j \in \I\\ (i,j) \in \wt \gG}} {\bf 1}(e(i,j) \ge 2),
\end{equation}
and $\wh \gI({\bm \upbeta}):= {\bf 1}({\bm n}(\gG({\bm \upbeta}))=\upgamma_0, \Delta(\gG({\bm \upbeta}))=\upgamma_1) \cdot \gI({\bm \upbeta})$ (recall \eqref{eq:gIdef}). We consider two cases: $k=4$ and $k=3$, separately. First we consider the case $|\I|=4$. 

Recall that the concatenation of  paths alternating in $\wt \gG$ and $\gG$ yield two paths with end points $i_1$ and $i_2$, and $i_3$ and $i_4$, hereafter to be denoted by $\cP_1$ and $\cP_2$, respectively. Since $\cP_1$ and $\cP_2$ are paths and $|\I|=4$ we observe that $\deg_\cP(i)$ is odd for all $i \in \I$, where $\cP:=\cP_1 \cup \cP_2$ and $\deg_{{\cP}}(i)$ denotes the degree of $i$ in the path ${\cP}$, meant as the number of edges in $\cP$ with one endpoint being $i$. Furthermore, as $\deg_\cP(i)=\deg_{\wt \gG}(i)+\deg_\gG(i)$ we see that either $\deg_{\wt \gG}(i)$ or $\deg_\gG(i)$ is odd for all $i \in \I$. Let $\J:=\J(\wt \gG, \gG):= \{ i \in \I: \deg_{\wt \gG}(i) \text { is odd}\}$. Hence, by the pigeonhole principle there must be at least $\lceil |\J|/2\rceil$ distinct edges in $\wt \gG$ which appear an {\em odd} number of times. 
By \eqref{eq:wtT} this means that
\begin{equation}\label{eq:wtT-bd}
\wt T(\wt \gG) \le n^{-\big\lceil \frac{|\J|}{2}\big\rceil} \cdot (np)^{-\frac12\big\lceil\frac{|\J|}{2}\big\rceil}. 
\end{equation}
On the other hand, recalling the definition of $\gG({\bm \upbeta})$ we note that $\deg_{\gG}(i) = \sum_{\upbeta \in {\bm \upbeta}} e(i, \upbeta)$. This in turn yields that for any $i \in \I\setminus \J$ there must exist some $\wt \upbeta_i \in {\bm \upbeta}$ such that $e(i, \wt \upbeta_i)$ is odd and so is $e_o(i, \wt \upbeta_i)$. For ${\bm \upbeta}$ such that $\gI({\bm \upbeta})=1$ (recall \eqref{eq:gIdef}) we therefore have that $\Delta(i, \wt {\upbeta}_i) \ge 1/2$ for any $i \in \I \setminus \J$, and thus $\Delta(\gG({\bm \upbeta})) \ge \frac12 |(\I \setminus \J)|$. Further let $\wh \J:=\{\wt \upbeta_i: i \in \I \setminus \J\}$. Observe that, by definition, for any ${\bm \upbeta}$ such that $\gJ_{{\bm s}}({\bm \upbeta})=1$ we have that $s_1 \ge |\wh \J|$. 

We also claim that for any ${\bm \upbeta}$ such that $\wh \gI({\bm \upbeta})=\gJ_{{\bm s}}({\bm \upbeta})=1$ we must have that 
\begin{equation}\label{eq:gamma-s}
\upgamma_0 -s_0 \ge |(\I\setminus \J)| - |\wh \J| \ge 0. 
\end{equation} 
%The second inequality is obvious because for every $ \upbeta\in \wh \J$ there exists some $i \in \I \setminus \J$ such that $\Delta(i, \upbeta) \ge 1/2$. 
The second inequality is obvious by the definition of $\wt{\mathbb{J}}$.
To obtain the other inequality, recalling the definition of $\upgamma_0$ and $s_0$ for any ${\bm \upbeta}$ such that $\wh \gI({\bm \upbeta})=1$ we note that
\begin{multline*}
\upgamma_0 -s _0 = \Big(\sum_{(i, \upbeta): (i, \upbeta) \in \gG({\bm \upbeta})} 1\Big) - \Big(\sum_{\upbeta \in {\bm \upbeta}} 1\Big) = \sum_{\upbeta \in {\bm \upbeta}} \Big[ \Big(\sum_{i \in \I: (i, \upbeta) \in \gG({\bm \upbeta})} 1\Big) -1 \Big] \ge \sum_{\upbeta \in \wh \J} \Big[ \Big(\sum_{i \in \I: (i, \upbeta) \in \gG({\bm \upbeta})} 1\Big) -1 \Big]\\
\ge  \sum_{\upbeta \in \wh \J} \Big(\sum_{i \in \I\setminus \J: (i, \upbeta) \in \gG({\bm \upbeta})} 1\Big) - |\wh \J| \ge |(\I \setminus \J)| - |\wh \J|,
\end{multline*}
where the third last step follows from the fact that the summands are nonnegative. 

Using the bounds \eqref{eq:varpi-bd}, $\Delta(\gG({\bm \upbeta}) ) \ge |(\I\setminus \J)|/2$, $s_1 \ge |\wh \J|$, and \eqref{eq:gamma-s}, denoting $|\J|=\upgamma_2$, and splitting the sum in the RHS of \eqref{eq:varpi3-3} according to the value of $|\wh \J| =\wh \upgamma_2$ we deduce that
\begin{equation}\label{eq:wtTTprod}
T(\upgamma_0, \upgamma_1, {\bm s}) \cdot \wt T(\wt \gG) \lesssim n^{-|\I|} (np)^{-|\I|/2} \max_{\upgamma_2, \wh \upgamma_2: \upgamma_2 +\wh \upgamma_2 \le |\I|} \exp(F(\upgamma_2, \wh \upgamma_2)),
\end{equation}
where
\[
F(\upgamma_2, \wh \upgamma_2):= \frac{\upgamma_2}{4} \log (np) - \frac{\wh \upgamma_2}{4} \log (n \upeta) + (\wh \upgamma_2 + \upgamma_2/2) \log n
\]
Since $\wh \upgamma_2 \mapsto F(\cdot, \wh \upgamma_2)$ is increasing, for $z \in \wt{\mathbb{H}}_\kappa$, using that $|\I|=4$, we obtain from above that
\begin{multline}\label{eq:wtTTprod1}
T(\upgamma_0, \upgamma_1, {\bm s}) \cdot \wt T(\wt \gG) \lesssim n^{-4} (np)^{-2} \max_{\upgamma_2 =0}^4 \exp(F(\upgamma_2, 4 - \upgamma_2)) = (np)^{-2} \cdot (n\upeta)^{-1} \max_{\upgamma_2=0}^4 \bigg(\frac{np \cdot n \upeta}{n^2}\bigg)^{\upgamma_2/4}\\
= (np)^{-2} \cdot (n\upeta)^{-1},
\end{multline}
for all large $n$, where in the last step we used that $p \upeta \ll 1$. Since the choices of $\upgamma_0, \upgamma_1, {\bm s}, \wt \gG$, and $\gG$ are $O(1)$ the bound together with \eqref{eq:cR1cR1-new1} and \eqref{eq:varpi3-2} yield \eqref{eq:cR1cR1-new} for $k=4$.  

It now remains to consider the case $k=3$. he proof splits into a few sub cases. %Without loss of generality, we assume that $i_1=i_3$ and $\{i_1, i_2, i_4\}$ are distinct. 

Since the concatenation of  paths alternating in $\wt \gG$ and $\gG$ yield two paths with end points $i_1$ and $i_2$, and $i_3$ and $i_4$, we see that it must contain two distinct edges, to be denoted by $e_1$ and $e_2$, that are not self loops. 

Therefore, if $e_1, e_2 \in \wt \gG$ then $\wt T \lesssim n^{-2}$.
%$\wt \gG$} must contain at least one edge that appears an odd number of times \corAQ{and the other edge must appear at least twice}. In that case, 
As $T(\upgamma_0, \upgamma_1, {\bm s}) \lesssim 1$ (see \eqref{eq:varpi3-3}, use \eqref{eq:varpi-bd} and the fact that $\upgamma_0 \ge s_0$) we get the desired bound.

Next we assume $e_1=(j_1, j_2) \in \gG$ and $e_2  \in \wt \gG$. By definition there must exist $\upbeta_1, \upbeta_2$ such that $(j_1, \upbeta_1), (j_2, \upbeta_2) \in \gG({\bm \upbeta})$. 
%
%If $\upbeta_1 \ne \upbeta_2$ then $s_1 \ge 2$.
%
If $\upbeta_1\ne \upbeta_2$ then $s_1\ge 2$. Observe that either there is an edge  that appears in $\wt\gG$ an odd number of times, or else repeating the proof for the case $k=4$ it follows that $\upgamma_1 \ge 1$. By \eqref{eq:wtT}, in the first sub case we have $\wt T \lesssim n^{-1} (np)^{-1/2}$, while in the second sub case $\wt T \lesssim n^{-1}$. Hence, from \eqref{eq:varpi3-3}, \eqref{eq:varpi-bd}, and \eqref{eq:gamma-s} the desired bound follows. 

If $\upbeta=\upbeta_2$, as $j_1 \ne j_2$, from the definition of $\gG({\bm \upbeta})$ and $\wh\gG_0({\bm \upbeta})$ it follows that $\upgamma_0 \ge s_0+1$ In this case, using the bound $\wt T \lesssim n^{-1}$, and the bounds \eqref{eq:varpi3-3} and \eqref{eq:varpi-bd} we obtain the required bound. %, \corAQ{and \eqref{eq:gamma-s}} we find that $T(\upgamma_0, \upgamma_1, {\bm s}) \lesssim (n \upeta)^{\corAQ{-1}}$. Since $\wt \gG \ne \emptyset$, the bound \eqref{eq:wtT} yields that $\wt T \le n^{-1}$. \corAQ{The product of the last two bounds yields the required bound}.

%If one of these two edges, say $e_1$ belong to $\wt \gG$ and appears an odd number of times then by \eqref{eq:wtT}, and as $T(\upgamma_0, \upgamma_1, {\bm s}) \lesssim 1$ (see \eqref{eq:varpi3-3}, use \eqref{eq:varpi-bd} and the fact that $\upgamma_0 \ge s_0$) we get the desired bound.

Now we assume $e_1, e_2 \in \gG$. %\neq \emptyset = \wt \gG$. In this case $e_1, e_2 \in \gG({\bm \upbeta})$. 
Similar to $e_1$, as above, we can associate some $\upbeta_3$ and $\upbeta_4$ to $e_2$. If $\upbeta_1 \ne \upbeta_2$ and $\upbeta_3 \ne \upbeta_4$ then $s_1 \ge 4$. If $\upbeta_\ell, \ell=1,\ldots, 4$, are all same then, as three distinct vertices participate in $e_1$ and $e_2$, it follows that $\upgamma_0 -s_0 \ge 2$. Moreover, for $\upbeta_1 \ne \upbeta_2$ and $\upbeta_3 = \upbeta_4$ we have $s_1 \ge 2$ and $\upgamma_0 -s_0 \ge 1$. Furthermore, $\upgamma_1 \ge 1$. Therefore, in all these cases, from \eqref{eq:varpi3-3} and \eqref{eq:varpi-bd}, we find that $T(\upgamma_0, \upgamma_1, {\bm s}) \lesssim (np)^{-1} (n \upeta)^{-1}$.

%If $e_1 \in \wt \gG$ but does not appear an odd number of times. 

%minor modifications of the above argument. Without loss of generality, we assume that $i_1=i_3$ and $\{i_1, i_2, i_4\}$ are distinct. By the same argument as above, upon setting $\I=\{i_2, i_4\}$ in this case, we note that $\deg_\cP(i)$ is odd for all $i \in \I$. With this new choice of $\I$ and upon defining $\J$ as above (define it as a subset of this new $\I$) one can verify that \eqref{eq:wtT-bd} and \eqref{eq:gamma-s} continue to hold in this case as well. Thus, arguing as in \eqref{eq:wtTTprod} and \eqref{eq:wtTTprod1} we find that 
%\begin{equation*}
%T(\upgamma_0, \upgamma_1, {\bm s}) \cdot \wt T(\wt \gG) \lesssim n^{-2} (np)^{-1} \max_{\upgamma_2 =0}^2 \exp(F(\upgamma_2, 2 - \upgamma_2)) = (np)^{-1} \cdot (n\upeta)^{-1/2}.
%\end{equation*}
%This, upon using  \eqref{eq:cR1cR1-new1} and \eqref{eq:varpi3-2}, gives \eqref{eq:cR1cR1-new} for $k=3$. 
Hence combining all sub cases we have the required bound for $k=3$ and thus the proof of the theorem for $X$ is finally complete. 

The reader can note that the above proof does not use the product structure of the entries of $X$. Therefore, repeating the exact same arguments one derives the required result for ${\rm Adj}_o/\sqrt{np}$. We omit further details.
\end{proof}

\subsection{Proof of Theorem \ref{thm:loc-law}} The required result for ${\rm Adj}_o/\sqrt{np}$ follows from \cite[Theorem 2.2]{HKM19}. The proof for $X$ follows a similar strategy.
%proof of Theorem \ref{thm:loc-law} closely follows that of \cite[Theorem 2.2]{HKM19}. The overall strategy is the same and 
The difference is that we need to use the concentration bounds proved in Lemma \ref{lem:loc-law-con-bd} instead of the concentration bounds in \cite{HKM19}. We provide a short proof for the convenience of the reader.
%\begin{proof}[Proof of Theorem \ref{thm:loc-law}]

Set $\wt M:= \max\{k \in \NN: 1- kn^{-3} \ge n^{-1+\vep}/2\}$. Fix $E_0 \in \RR$ such that $|E_0\pm 2| \ge \kappa$ and $|E_0| \le \kappa^{-1}$. Let $\upeta_k:=1-kn^{-3}$ and define $z_k:= E_0+{\rm i} \upeta_k$ for $k=0,1,\ldots, \wt M$. Since the maps $z \mapsto \cR^{(\T)}_{i,j}(z, X)$ and $z \mapsto m(z)$ are $n^{2(1-\vep)}$- Lipschitz functions on the set $\mathbb{H}_{\kappa, \vep}$ we observe that it suffices to prove that for any $M>0$ there exists $C<+\infty$ such that
\begin{equation}\label{eq:loc-law}
\PP\bigg(\max_{i, j \in [n]\setminus \T} |\cR^{(\T)}_{i,j}(z_k, X) - m(z_k) \updelta_{i,j}| \ge C \Upsilon_n \bigg) \le n^{-M},
\end{equation}
for all large $n$ and $k =0,1,\ldots, \wt M$, uniformly for all $\T \subset [n]$ such that $|\T| \lesssim 1$ and $E_0 \in \RR$ such that $|E_0\pm 2| \ge \kappa$ and $|E_0| \le \kappa^{-1}$. Since, for  $i,j, k \notin \T$ and $i,j \ne k$
\begin{equation}\label{eq:schurT}
\cR_{i,j}^{(\T \cup\{k\})}= \cR_{i,j}^{(\T)} - \frac{\cR_{i,k}^{(\T)} \cR_{k,j}^{(\T)}}{\cR_{k,k}^{(\T)}}, 
\end{equation}
(see \cite[Lemma 3.5]{BGK17}) and $|m(z)| \asymp 1$ on $\mathbb{H}_{\kappa, \vep}$ (see \cite[Lemma 3.2]{EKHY13}) we deduce that it is enough to prove \eqref{eq:loc-law} for $\T=\emptyset$. 

Turning to prove \eqref{eq:loc-law} when $\mathbb{T} = \emptyset$, using the Schur complement formula (see \cite[Eqn.~(4.1)]{BGK17}) it follows that 
\begin{multline}\label{eq:loc-law-schur}
\frac{1}{\cR_{i,i}} = z - X_{i,i} -\sum_{k, \ell \ne i} X_{i,k} \cR_{k, \ell}^{(i)} X_{\ell, i}\\
= z - m_n(z) + (m_n(z) - m_n^{(i)}(z)) - X_{i,i} - \sum_{k \ne \ell \in [n]\setminus \{i\}} X_{i,k} \cR_{k, \ell}^{(i)} X_{\ell, i} - \sum_{k \ne i} (X_{i,k}^2-n^{-1}) \cR_{k,k}^{(i)}, 
\end{multline}
where we recall that $m_n^{(i)}(z)= n^{-1} {\rm Tr} \, \cR^{(i)}(z)$. Multiplying both sides of \eqref{eq:loc-law-schur} by $\cR_{i,i}$ and taking an average over $i \in [n]$ we obtain from above 
\begin{equation}\label{eq:apprx-fxd-pt}
\mathfrak{p}_z(m_n(z)) = \frac1n\sum_{i=1}^n \wh Y_i(z) \cR_{i,i}(z),
\end{equation}
where 
\[
\wh Y_i:=\wh Y_i(z):= (m_n(z) - m_n^{(i)}(z)) - X_{i,i} - \sum_{k \ne \ell \in [n]\setminus \{i\}} X_{i,k} \cR_{k, \ell}^{(i)} X_{\ell, i} - \sum_{k \ne i} (X_{i,k}^2-n^{-1}) \cR_{k,k}^{(i)}
\]
and $\mathfrak{p}_z(w):=1 -zw +w^2$ for $w \in \CC$. For convenience further let, for $k=0,1,\ldots, \wt M$,
\[
\Upomega_k:=\Big\{|m_n(z_k) - m(z_k)| \le \frac{C}{4}\Upsilon_n\Big\}, \quad \wh \Upomega_k:=\bigcap_{i=1}^n \Upomega_{i,k}, \quad \text{ and } \quad \wh \Upomega_{i,k}:=\big\{\Upgamma_i(z_k) \le \frac32\big\}, 
\]
where $\Upgamma_i(z):= \max_{\ell, \ell' \ne i} |\cR^{(i)}_{\ell, \ell'}(z)|$. We also set $\Upgamma'(z):= \max_{i \in [n]} |\cR_{i,i}(z)|$ and $\Upomega_k':= \{\Upgamma'(z_k) \le 3/2\}$. 
Note that by Assumption \ref{hypo:subg-entry}, for any $M>0$ there exists some $C<+\infty$ such hat
\[
\max_{i \in [n]} |X_{i,i}| \le \frac{\sqrt{C}}{8} \sqrt{\frac{\log n}{np}},
\]
with probability at least $1- n^{-4M}$. Hence on the event $\wh\Upomega_k$, by Lemma \ref{lem:loc-law-con-bd}(b) and (c), the Ward identity, and the fact that $|m_n^{(i)}(z_k)-m_n(z_k)| \le (n \Im z_k)^{-1}$ we obtain that
\begin{equation}\label{eq:whYi-bd}
\max_{i \in [n]}|\wh Y_i(z_k)| \le \frac{\sqrt{C}}{4} \Upsilon_n,
\end{equation}
 for some sufficiently large constant $C$, with probability at least $1- n^{-2M}$, i.e.~the probability of the intersection of $\wh \Upomega_k$ and the event such that \eqref{eq:whYi-bd} does not hold is at most $n^{-2M}$. Thus, by \eqref{eq:apprx-fxd-pt} on the event $\wh \Upomega_k \cap \Upomega_k'$ the event $\wt \Upomega_k$ holds with probability at least $1- n^{-2M}$, where
 \[
 \wt \Upomega_k:= \big\{|\mathfrak{p}_{z_k}(m_n(z_k))| \le \frac{\sqrt{C}}{2} \Upsilon_n \big\}, \qquad k=0,1,2,\ldots, \wt M.
 \]
 Observe that $\mathfrak{p}_{z}(m_n(z)) = (w_+(z) - m_n(z))(w_-(z) - m_n(z))$, where $w_\pm(z):= \frac12(z \pm \sqrt{z^2-4})$ are the roots of quadratic polynomial $\mathfrak{p}_z$. As $|w_+(z) - w_-(z)| \ge \sqrt{\kappa}$ for any $z \in \mathbb{H}_\kappa$ we find that on the event $\wt \Upomega_k$ one has that 
 \begin{equation}\label{eq:mn-mind}
 |m_n(z_k) - w_-(z_k)| \wedge  |m_n(z_k) - w_+(z_k)| \le \frac{C}{2}\Upsilon_n,
 \end{equation}
 where we enlarge $C$ if necessary. Notice that, by the trivial bound $\|\cR(z_0)\| \le \upeta_0^{-1}=1$ we have $\PP(\widehat \Upomega_0 \cap \Upomega_0')=1$. As $\Im m_n(z_0) \le 0$, $\Im w_+(z_0) \ge \frac12 \Im z_0 \ge \frac12$, and $w_-(z)=m(z)$ we deduce from above that $\wt \Upomega_0 \subset \Upomega_0$. Thus we have shown that $\PP(\Upomega_0^c) \le n^{-2M}$. 

We next proceed to argue that on the event $\Upomega_k$ the required bounds on the individual entries of $\cR(z_k)$ (recall \eqref{eq:loc-law}) hold. To this end, by \eqref{eq:loc-law-schur} and the fact that $m(z_k)+m(z_k)^{-1}=z$ we see that
\[
1 - m(z_k)^{-1} \cR_{i,i}(z_k)= 1 - (z_k- m(z_k))\cR_{i,i}(z_k) = (m(z_k)- m_n(z_k))\cR_{i,i}(z_k) + \wh Y_i(z_k) \cR_{i,i}(z_k).
\]
Therefore, using $|m(z_k)| \le 1$ and \eqref{eq:whYi-bd} we obtain that 
\[
\max_{i \in [n]} |\cR_{i,i}(z_k) - m(z_k)| \le \frac{3}{2} \big(|m(z_k) - m_n(z_k)| + \max_{i \in [n]}  |\wh Y_i(z_k)|\big) \le C \Upsilon_n,
\]
on the event $\Upomega_k \cap \wh\Upomega_k \cap \Upomega_k'$ with probability at least $1- n^{-2M}$. To treat the off-diagonal entries we use the identity 
\[
\cR_{i,j}(z) = - \cR_{j,j}(z) \sum_{k \ne j} \cR_{i,k}^{(j)}X_{k,j}, \quad i \ne j.
\]
This identity together with Lemma \ref{lem:loc-law-con-bd}(a) and the Ward identity imply that the desired bound on the off-diagonal entries hold with probability at least $1- n^{-2M}$ on the event $\wh \Upomega_k \cap \Upomega_k'$. Now applying the union bound we derive \eqref{eq:loc-law} for $k=0$.

To complete the proof we do an induction on $k$. So, assume that the bound \eqref{eq:loc-law} holds for $k=0,1,2,\ldots, k_0$ for some $k_0 < \wt M$ and let $\Upomega_k^*$ be the complement of the event on the LHS of \eqref{eq:loc-law}. Since the maps $z \mapsto \cR_{i,j}(z)$ and $z \mapsto m(z)$ are $n^{2(1-\vep)}$- Lipschitz functions on $\mathbb{H}_{\kappa, \vep}$, using $|m(z)| \le 1$ and \eqref{eq:schurT} we find that $\Upomega_{k_0}^* \subset  \wh \Upomega_{k_0+1} \cap \Upomega_{k_0+1}'$. Hence,  we find that on $\Omega_{k_0}^*$ the event $\wt \Omega_{k_0+1}$, and thus \eqref{eq:mn-mind} for $k=k_0+1$, hold with probability at least $1 - n^{-2M}$. The Lipschitz continuity of $m_n$ and $m$ further implies  on the event $\Upomega_{k_0}^*$, that $|m_n(z_{k_0+1}) - m(z_{k_0+1})| = o(1)$, which in turn implies that $|m_n(z_{k_0+1}) - w_+(z_{k_0+1})| \gtrsim  1$. As we saw, this entails that on the same event $\Upomega_{k_0}^*$ the event $\Omega_{k_0+1}$ holds with probability at least $1-n^{-2M}$. As a consequence  the desired bounds on the entries of $\cR(z_{k_0+1})$ hold on $\Omega_{k_0}^*$ with high probability. 
This completes the induction and the proof of the theorem. 
\qed
%\end{proof}

%\end{proof}

%\corAB{need to add the proof for centred adjacency matrix}

\subsection{Isotropic local law for the centred adjacency matrix} In this short section we prove Proposition \ref{prop:conc-loc-law-o}. We need the following isotropic law. %Notice its difference with Theorem \ref{thm:loc-law-iso}.

%This proof relies on the following isotropic local law for the centred adjacency matrix ${\rm Adj}_o$.  

\begin{The}\label{thm:loc-law-iso-adj}
Assume $\log(np) \gtrsim \log n$. Let $\mathcal{U}_n$ be as in Theorem \ref{thm:loc-law-iso}. Fix $\kappa, \vep \in (0,1)$. Then %There exists a sequence $\{s_n'\}_{n \in \NN}$ with $\lim_{n \to \infty} s_n'=0$ such that
%The following isotropic local laws hold for the centred Adjacency matrix of an Erd\H{o}s-R\'enyi graph.
%\begin{enumerate}
%
%\item[(i)] For any $z \in \wt{\mathbb{H}}_\kappa$, under the assumption that $\log n \ll np \ll n$ we have
%\[
%\lim_{n \to +\infty} \sup_{u \in \cU_n} \big| \EE \langle u, \cR(z, {\rm Adj}_o/\sqrt{np}) u \rangle - m(z) \big| =0,
%\]
%where $\mathcal{U}_n$ is as in \eqref{eq:cU-def}. 

%\item[(ii)] 
%Under the additional assumption that $\log (np) \gtrsim \log n$ and $z$ \corAQ{as above}
\[
\limsup_{n \to +\infty} \sup_{z \in \wh{\mathbb{H}}_\kappa} \sup_{v \in \mathbb{S}^{n-1}} \sup_{u \in {\mathcal{U}}_{n}}\big| \EE \langle u, \cR(z,{\rm Adj}_o/\sqrt{np}) v \rangle - m(z) \langle u, v\rangle\big|=0,
\]
where $\wh{\mathbb{H}}_\kappa:= \{z \in \mathbb{H}_\kappa: s_n \vee (np)^{-1+\vep} \le \Im z \le \kappa^{-1}\}$.
%$\wh{\mathcal{U}}_n$ is as in Proposition \ref{prop:conc-loc-law-adj}(iii). 
%\end{enumerate}
\end{The}

%\begin{proof}[Proof of Proposition \ref{prop:conc-loc-law-adj}]
%Since the map $\widecheck f_{\lambda, y}(\cdot)$ is a convex Lipschitz function with Lipschitz constant $\delta^{-2}$ upon applying Talagrand's inequality (cf.~\cite[Corollary 4.7]{Ledouxmono}) part (i) follows. 
%To prove part (ii) one proceeds as in that of Proposition \ref{prop:conc-loc-law}(i) and uses Theorem \ref{thm:loc-law-iso-adj}(i) instead of Theorem \ref{thm:loc-law-iso}. Turning to prove part (iii), proceeding as in the proof of Proposition \ref{prop:conc-loc-law}(ii) we find that
%\begin{equation}\label{eq:iso-adj1}
%\lim_{n \to +\infty} |\EE\widecheck f_{\lambda, w}({\rm Adj}_o/\sqrt{np}) - \EE\langle w, \cR(z, {\rm Adj}_o/\sqrt{np}) w \rangle| =0. 
%\end{equation}
%uniformly for any $w \in \RR^n$ such that $\|w\|\le 2$, where $z= \lambda +{\rm i} \upeta$ \corAQ{with $\upeta=$}. On the other hand, for any $u, v \in \RR^n$, 
%\begin{align}
%&\, 2\langle u, \cR(z, {\rm Adj}_o/\sqrt{np}) v\rangle \label{eq:iso-adj2}\\
%= &  \, \langle u+v, \cR(z, {\rm Adj}_o/\sqrt{np}) u+v\rangle - \langle u, \cR(z, {\rm Adj}_o/\sqrt{np}) u\rangle -  \langle v, \cR(z, {\rm Adj}_o/\sqrt{np}) v\rangle. \notag
%\end{align}
%Therefore, Theorem \ref{thm:loc-law-iso-adj}(ii) together with \eqref{eq:iso-adj1}-\eqref{eq:iso-adj2} now yield part (iii). This completes the proof. 
%\end{proof}

Notice the differences between Theorems \ref{thm:loc-law-iso} and \ref{thm:loc-law-iso-adj}. Theorem \ref{thm:loc-law-iso-adj} allows us consider $\langle u, \cR(z) v \rangle$ without putting any restriction on $\|v\|_1$. However, it imposes a stronger lower bound on $p$.

\begin{proof}[Proof of Proposition \ref{prop:conc-loc-law-o}]
For $u \in\mathbb{R}^n$, let $f_{\lambda,u}$ defined on $\mathcal{H}_n^{\lambda+\delta}$ be as in the proof of Proposition \ref{prop:conc-loc-law} and $\wt{f}_{\lambda,u}$ its convex and $\delta^{-2}\|u\|^2$-Lipschitz extension to $\mathcal{H}_n$. By \cite[Theorem 8.6]{BLM} for any $t>0$,
\begin{equation}\label{eq:conc-loc-law-o-1}
\PP\Big(\big| \widetilde{f}_{\lambda, w}({\rm Adj}_o/\sqrt{np}) - \EE [\widetilde{f}_{\lambda, w}({\rm Adj}_o/\sqrt{np})]  \big| \ge  t \Big) \leq \exp \big(-{c\delta^4 t^2 np}\big),
\end{equation}
uniformly for any $w \in 2 \mathbb{B}^n$, where $c>0$ is some numerical constant. For $u,v\in \mathbb{B}^n$ and $K\in\mathcal{H}_n$, let $\wt{g}_{\lambda,u,v}(K) = \wt{f}_{\lambda,u+v}(K) - \wt{f}_{\lambda, u-v}(K)$. By the polarization identity, $\wt{g}_{\lambda, u,v}(K)= 4\langle u, (\lambda-K)^{-1}v\rangle$ for any $K\in \mathcal{H}_n^{\lambda+\delta}$.  Arguing as in the proof of \eqref{controlmean} we find that
\begin{equation}\label{eq:conc-loc-law-o-2} 
\sup_{v \in \mathbb{S}^{n-1}} \sup_{u\in \mathcal{U}_n} |\EE \wt{g}_{\lambda, u,v}({\rm Adj}_o/\sqrt{np}) - 4m(\lambda) \langle u,v\rangle|=o(1).
\end{equation}
By a union bound, \eqref{eq:conc-loc-law-o-1} entails a similar concentration inequality for $\wt{g}_{\lambda,u,v}$ instead of $\wt{f}_{\lambda,w}$. Using the same net argument as in the proof of proof of Proposition \ref{prop:conc-loc-law}, we finally get the claim.
%
%By polarization identity $f_{\lambda,u+v}(K) - f_{\lambda, u}(K) - f_{\lambda,v}(K)= 2 \langle u, (\lambda- K)^{-1} v \rangle$ for any $K \in \cH_n^\lambda$ and $u,v \in \RR^n$. Hence, using Theorem \ref{thm:loc-law-iso-adj} and arguing as in the proof of \eqref{controlmean} we find that 
%\begin{equation}\label{eq:conc-loc-law-o-2} 
%\sup_{v \in \mathbb{S}^{n-1}} \sup_{u\in \mathcal{U}_n} \big| \EE \widetilde{f}_{\lambda,u+v}({\rm Adj}_o/\sqrt{np}) -\EE \widetilde{f}_{\lambda,u}({\rm Adj}_o/\sqrt{np}) - \EE \widetilde{f}_{\lambda,v}({\rm Adj}_o/\sqrt{np}) - 2 m(\lambda) \langle u, v \rangle \big| =o(1).
%\end{equation}
%Since $\wt f_{\lambda,u}(K)=f_{\lambda,u}(K)$ for $K \in \cH_n^\lambda$ the desired result follows from \eqref{eq:conc-loc-law-o-1}-\eqref{eq:conc-loc-law-o-2}, and triangle inequality. 
%%
%proceeding as in the proof of Proposition \ref{}, the desired concentration bound in Proposition \ref{prop:conc-loc-law-o} follows. 
\end{proof}

\begin{proof}[Proof of Theorem \ref{thm:loc-law-iso-adj}]
The proof is a minor modification of that of Theorem \ref{thm:loc-law-iso}. We reuse the bounds obtained during the proof of Theorem \ref{thm:loc-law-iso}. Proceeding as in the steps leading to \eqref{eq:cR-off-diag} we find that
\[
\sup_{v \in \mathbb{S}^{n-1}} \sup_{u \in \mathcal{U}_n}  \EE\Big[\sum_{i_1 \ne i_2 \in [n]}\sum_{i_3 \ne i_4 \in [n]}  u_{i_1} v_{i_2} u_{i_3} v_{i_4} \cR_{i_1,i_2}(z) \cR_{i_3,i_4}(z){\bf 1}_{\widecheck \Omega_{\I}}\Big]\to 0, \quad \text{ as } n \to +\infty,
\]
where $\cR(z)= \cR(z, {\rm Adj}_o/\sqrt{np})$. Since $\log(np) \gtrsim \log n$, one can choose $\gf$ large enough so that \eqref{eq:cR1cR2-pre} holds with $(np)^{-4}$ replaced by $n^{-3}$. Therefore, arguing as in \eqref{eq:cR1cR2}, as $\|u\|_1 \vee \|v\|_1 \le n^{1/2}$, we find that % and $\log(np) \gtrsim \log n$, one can choose $\gf$ large so that 
\[%\label{eq:cR1cR2}
\sup_{v \in \mathbb{S}^{n-1}} \sup_{u \in \mathcal{U}_n} \bigg|\EE\Big[\sum_{i_1 \ne i_2 \in [n]}\sum_{i_3 \ne i_4 \in [n]}  u_{i_1} v_{i_2} u_{i_3} v_{i_4} (\cR_{1})_{i_1,i_2}(\cR_{2})_{i_3,i_4}{\bf 1}_{\widecheck \Omega_{\I}}\Big]\bigg| \to 0, \quad \text{ as } n \to +\infty. 
\]
Hence, similar to \eqref{eq:cR1cR1}, we deduce that it is enough to prove that for any $a_1, a_2 \in [\gf]$ and $k=3,4$:
\[
 \lim_{n \to +\infty}\EE\bigg[\sum_{i_1 \ne i_2 \in [n]}\sum_{i_3 \ne i_4 \in [n]}  u_{i_1} v_{i_2} u_{i_3} v_{i_4} \wh m(z)^{a_1+a_2+2} (\sE^{a_1})_{i_1,i_2}(\sE^{a_2})_{i_3,i_4} {\bf 1}_{\widecheck \Omega_{\I}} {\bf 1}_{|\I|=k} \bigg] =0.
\]
Since $\|u\|_1 \le s_n np$ and $\|v\|_1 \le n^{1/2}$, for $z \in \wh{\mathbb{H}}_\kappa$ the above follows from the bound \eqref{eq:cR1cR1-new}. We omit the details. This completes the proof. 
\end{proof}

\appendix

\section{Proofs of Lemmas \ref{lem:hL-prop}, \ref{enhancedconcentration}, and \ref{lem:Lam-bd}}\label{app:aux-res}

\begin{proof}[Proof of Lemma \ref{lem:hL-prop}]
It is trivial to note that $(-\infty, 0] \subset \mathcal{D}_L$. Fix any $\theta \in (0, 1/(2\beta))$ and $\wh R >0$. Let $\Gamma$ be an independent standard Gaussian random variable. Observe that, by Assumption \ref{hypo:subg-entry}, 
\begin{multline}\label{eq:DL1}
L(\theta) = \EE \big[\exp({\sqrt{2\theta} G_{1,2} \Gamma})\big] = \EE\big[ \exp\big({\Lambda_{1,2}(\sqrt{2\theta} \Gamma)}\big)\big] \\
\le \exp(\upsigma \wh R^2) + \EE \big(\Car_{|\Gamma|\geq \wh R/\sqrt{2\theta}} e^{\beta(1+o_{\wh R}(1)) \theta \Gamma^2}\big) < + \infty,
\end{multline}
where the last step follows from the fact that, as $ \theta < 1/(2\beta)$, one can choose $\wh R $ sufficiently large so that $\beta \theta (1+o_{\wh R}(1)) <1 /2$. On the other hand, fixing any $\theta > 1/(2\beta)$ (assume $\beta >0$) and arguing as in \eqref{eq:DL1} we obtain that
\begin{equation}\label{eq:DL2}
L(\theta) \ge \EE \big(\Car_{|\Gamma|\geq \wh R/\sqrt{2\theta}} e^{\beta(1-o_{\wh R}(1)) \theta \Gamma^2}\big) =\frac{2}{\sqrt{2\pi}} \int_{\wh R/\sqrt{2\theta}}^{+\infty} \exp\Big(\Big\{\beta(1-o_{\wh R}(1)) \theta -\frac12\Big\}x^2 \Big) dx= + \infty, 
\end{equation}
where the last step follows upon choosing $\wh R$ large so that $\beta(1-o_{\wh R}(1)) \theta >1/2$. Combining \eqref{eq:DL1}-\eqref{eq:DL2} we have part (a). The proof of part (b) being standard is omitted. 

Since $L$ is convex and  infinitely differentiable on ${\rm Int}(\mathcal{D}_L)$, fixing $\theta, \theta_0 < 1/(2\beta)$ we find that 
\[
L(\theta) \le L(\theta_0) + (\theta - \theta_0) L'(\theta).
\]
Now, letting $\theta \uparrow 1/(2\beta)$ above part (c) follows. Turning to prove (d), provided $x_\star < + \infty$, by parts (a) and (c) we find that
\begin{equation}\label{eq:hL1}
h_L(x) = \sup_{\theta \in(-\infty, 1/(2\beta)]} \{ \theta x - L(\theta)+1\}, \qquad x \in \RR. 
\end{equation}
By parts (b) and (c), for any $x \ge x_\star$, the map $\theta \mapsto \theta x - L(\theta)$ is strictly increasing on $(-\infty, 1/(2\beta))$.  Therefore, by \eqref{eq:hL1} and part (c) we have part (d).

It remains to prove part (e). 
Upon denoting $\partial(L-1)(\cdot)$ to be the sub differential of $L(\cdot)-1$ and using the fact that $\lim_{\theta \uparrow 1/(2\beta)}L'(\theta)=x_\star$ we note that 
\[ \partial (L-1)(\theta) = \begin{cases}
\{L'(\theta)\} & \text{ if } \theta <1/(2\beta), \\
[x_\star,+\infty) & \text{ if } \theta = 1/(2\beta), \\
\emptyset & \text{ if } \theta >1/(2\beta).
\end{cases}\]
From \cite[Corollary 23.5.1]{Rockafeller} we know that the sub differentials of $L-1$ and $h_L$ are inverse from each other in the sense of multivalued functions. 
As $L'$ is strictly increasing and strictly convex on ${\rm Int}(\mathcal{D}_L)$ we therefore have from above that the function $h_L'$ exists on $(0, x_\star)$. Moreover, on the same interval $h_L'$ is the inverse of the map $L'$ restricted to ${\rm Int}(\mathcal{D}_L)$, and therefore is strictly increasing, strictly concave, and differentiable on $(0, x_\star)$. By part (d) this implies that $h_L'$ is differentiable except possibly at $x_\star$. 

Furthermore, $h_L'(x)=\theta(x)$ for $x \in (0, x_\star)$ where $\theta(x) \in (-\infty, 1/(2\beta))$ is the unique solution of the equation $L'(\theta(x))=x$. Using that $L'$ is increasing on ${\rm Int}(\mathcal{D}_L)$ it can be argued that $\theta(x) \uparrow 1/(2\beta)$ as $x \uparrow x_\star$ which together with the fact that $h_L$ is convex imply that the left derivative of $h_L$ at $x=x_\star$ is bounded below by $1/(2\beta)$. On the other hand, by part (d) $h_L'(x)=1/(2\beta)$ for $x > x_\star$ and so the right derivative at $x=x_\star$ is bounded above by $1/(2\beta)$. Thus $h_L$ is differentiable on $(0, +\infty)$.  As $h_L'$ is strictly increasing and concave on $(0, x_\star)$, $h_L'(x)=1/(2\beta)$, and $h_L'$ is continuous on $(0, +\infty)$ it is straightforward to see that it is concave and non decreasing on entire $(0, +\infty)$.

To prove that $h_L$ is  increasing on $(1, +\infty)$ and decreasing on $(0,1)$ we recall that $h_L'(x)=\theta(x)$ for $x \in (0, x_\star)$ and $h_L'(x)=1/(2\beta)$ for $x \ge x_\star$. The requirement $L'(\theta(x))=x$ and the fact that $\EE G_{1,2}^2 =1$ implies that $\theta(x)$ is either positive or negative depending on whether $x  \in(1,x_\star)$ or $x <1$ proving the desired property for $h_L$. Finally, as $L'(0)=1$, we have $h_L'(1)=0$ which completes the proof of the lemma.
\end{proof}

\begin{proof}[Proof of Lemma \ref{enhancedconcentration}]
By \cite[Theorem 15.5]{BLM}, we know that for any $q\geq 2$, $\| (Z- \EE Z)_+\|_q \leq \sqrt{2\kappa q} \| \sqrt{V^+}\|_q$,
where $\kappa$ is some positive numerical constant.  
Using the fact that $V^+\leq \veps^2$ on $E_\veps$, and that $V^+\leq 1$ on $E_\veps^c$, we get
\[ \| \sqrt{V^+} \|_q \leq \| \Car_{E_\veps} \sqrt{V^+} \|_q+\| \Car_{E_\veps^c} \sqrt{V^+} \|_q \leq \big( \veps +\PP(E_\veps^c)^{1/q}\big).\]
For any $q\leq q_0$, $\PP(E_\veps^c)^{1/q} \leq \veps$, and therefore $\| \big(Z - \EE Z\big)_+ \|_q \leq  2\sqrt{2\kappa q} \veps$. 
By Markov inequality, we have for any $t>0$ and $2\leq q \leq q_0$,
\[  \PP\big( Z - \EE Z >t \big) \leq  (c\veps^2q )^{q/2} t^{-q},\]
where $c$ is some positive numerical constant. If $2ce\veps^2 \leq t^2 \leq c e\veps^2 q_0$ then one applies the above inequality for the optimal moment order $q = (ce\veps^2)^{-1}t^2$ which indeed gives \eqref{eq:enhancedconcentration}. 
%\[  \PP\big( Z - \EE Z >t \big) \leq \exp\Big( - \frac{t^2}{ce \veps^2}\Big).\]
This concludes the proof.
\end{proof}

\begin{proof}[Proof of Lemma \ref{lem:Lam-bd}]
The first claim can be read off \cite[Lemma 24]{NRS} using the Legendre duality. 
By definition, for any $x\in (-p,1-p)$, 
\[ \Lambda_p^*(x) =(x+p) \log \frac{x+p}{p} + (1-x-p) \log \frac{1-x-p}{1-p}.\] 
Thus, for any $x\in (-p,1-p)$, 
\[ (\Lambda^*_p)''(x) =  \frac{1}{x+p} + \frac{1}{1-x-p}, \ ( \Lambda_p^*)'''(x) = \frac{1}{(1-x-p)^2} - \frac{1}{(x+p)^2}.\]
This yields that $(\Lambda_p^*)''$ is non increasing on $[-p,1/2-p]$. 
Therefore noting that $\Lambda_p^*(0)=(\Lambda_p^*)'(0)=0$, we get  for any $-p \leq x \leq Cp$, 
\[ 
\Lambda_p^*(x) \geq (\Lambda^*_p)''(Cp) \frac{x^2}{2} \ge \frac{x^2}{2(C+1)p}.
\]
For the third claim, the concavity of the $\log$ yields for any $x\in (-p,1-p)$,
\[ \Lambda_p^*(x) \geq x \log \frac{x}{p}  -(1-x-p)\frac{x}{1-p} \geq x \log \frac{x}{p} - x \geq \Big( 1 - \frac{1}{\log C} \Big) x \log \frac{x}{p}.\]  
where the last step holds for $x\geq Cp$, with $C\geq e$. 
\end{proof}

%\section{Proof of Lemma \ref{lem:Lam-bd}}
\section{Measurable selection of sub differentials}
\begin{Lem}\label{selection} 
Let $f : \RR^{N} \mapsto \RR$ be a convex function and $F\subset \RR^N$ a closed set such that for any $x\in \RR^N$, $F\cap \partial f(x) \neq \emptyset$. Then, there exists $\zeta : \RR^N \mapsto \RR^N$ a Borel function such that for any $x\in \RR^N$, $\zeta(x)\in \partial f(x) \cap F$.
\end{Lem}

\begin{proof}
As $\partial f(x)$ is closed for any $x\in \RR^{N}$ (see \cite[Corollary 4.7]{Clarke}), it follows that $x \mapsto \partial f(x) \cap F$ is a mapping with values in the set of nonempty closed subsets of $\RR^N$.  We will show that for any open set $U$ of $\RR^N$, $\{ x\in \RR^N : \partial f(x)\cap F \cap  U \neq \emptyset\}$ is a Borel subset. Using \cite[Theorem 6.9.3]{Bogachev} this will give the claim. Since every open set of $\RR^N$ can be written as a countable union of compact subsets, it is sufficient to prove that for any compact subset $K\subset \RR^N$, $g(K) := \{ x\in \RR^N : \partial f(x)\cap K \neq \emptyset\}$ is a Borel set. Let $\phi(x,\zeta) := \inf_{v\in \mathbb{S}^{N-1}} \big( f'(x;v) - \langle \zeta,v\rangle\big)\in\RR\cup\{-\infty\}$ for any $\zeta,x\in \RR^N$, where $f'(x;v)$ denotes the directional derivative of $f$ at $x$ in the direction $v$. As $f$ is convex and finite valued on $\RR^N$ it follows from \cite[Corollary 2.35]{Clarke} that $f$ is locally Lipschitz from which we get that $v \mapsto f'(x;v)$ is continuous on $\RR^N$.
%\corAQ{the map $v \mapsto f'(x;v)$ is positively homogeneous and sub linear it is convex on $\RR^N$ (see \cite[Exercise 2.17]{Clarke}). Observe that as $f$ does note take the value $+\infty$ so does $f'$, and hence it is therefore continuous on $\RR^N$. 
Therefore so is the map
$v\mapsto f'(x;v)-\langle \zeta, v\rangle$ %is continuous 
for any $x\in \RR^N$. Thus the infimum defining $\phi(x,\zeta)$ can be taken in a dense countable subset of $\RR^N$. Since $x\mapsto f'(x;v)$ is measurable for any $v\in \RR^N$, it yields that $x\mapsto \phi(x,\zeta)$ is measurable for any $\zeta$. Moreover, note that $\zeta \mapsto \phi(x,\zeta)$ is continuous for any $x\in\RR^N$. Now, using that $f'(x;.)$ is homogenous and \cite[Proposition 4.3]{Clarke}, we can write $\partial f(x) = \{\zeta \in \RR^N : \phi(x,\zeta) \geq 0\}$ for any $x\in\RR^N$. Taking $D$ to be a countable dense subset of $K$, it follows from the stated properties of $\phi$ that 
\[ g(K) = \bigcap_{k\geq 1}\bigcup_{\zeta \in D} \{x : \phi(x,\zeta) >-1/k\},\]
which shows that $g(K)$ is a Borel set and therefore the proof ends. 
\end{proof}

\section{Upper tail large deviations of sum of independent variables}
The following proposition is well known in the literature. For the sake of completeness, we state it in the generality that we need in this article and provide a short proof.

\begin{Pro}\label{UTindep}
Fix $L \in \NN$. For any $\delta>0$ let $(X_{n,\delta}^{(k)})$, $k \in [L]$ be sequences of nonnegative random variables. Assume that there exist $I_k$, $k \in [L]$ non decreasing good rate functions on $\RR_+$ and an increasing positive sequence $(v_n)_{n\in \NN}$, diverging to infinity, such that for any $k \in [L]$ and $t>0$, 
\[ \limsup_{\delta \to +\infty} \limsup_{n\to +\infty }\frac{1}{v_n} \log \PP\big( X_{n,\delta}^{(k)} >t \big) \leq -I_k(t).\]
Then, for any deterministic sequence $(\kappa_n)_{n\in \NN}$ in a compact metric space $(\mathcal{K},d)$, any continuous functions $f_k : \mathcal{K} \to [0,+\infty)$, $k \in [L]$, and $t>0$, 
\[\limsup_{\delta \to +\infty} \limsup_{n\to +\infty} \frac{1}{n} \log \PP\big(\sum_{k=1}^L f_k(\kappa_n) X_{n,\delta}^{(k)}%+g(\kappa_n)X_{n,\delta}^{(2)}
>t\big) \leq - \Phi(t),\]
where $J$ is the lower semicontinuous function defined by
\[ \Phi(t) := 
\inf\big\{\sum_{k=1}^L I_k(x_k)%+I_2(y) 
: \sum_{k=1}^L f_k(\kappa) x_k%+g(\kappa)y
\ge t, x_k\geq 0, k \in [L], \kappa \in \mathcal{K} \big\}.\]
\end{Pro}

\begin{proof}
For convenience we only prove the case $L=2$. The general case is a straightforward extension. Let $0<t<M$. For any $\veps>0$, let $\veps=t_{1,\veps}<t_{2,\veps}<\cdots <t_{m,\veps} = M$ be a sequence of real number such that $\max_i (t_{i+1,\veps}-t_{i,\veps}) \leq \veps$. One can find such a mesh with $m\leq M/\veps+1$. As $f_1$ and $f_2$ are uniformly continuous, we can find a $\delta>0$ such that, for any  $\kappa,\kappa' \in \mathcal{K}$, 
\begin{equation} \label{unifcont} d(\kappa,\kappa') <\delta \Longrightarrow |f_1(\kappa)-f_1(\kappa')|<\veps,  \ |f_2(\kappa)-f_2(\kappa')|<\veps, \end{equation}
Let then $s_1,\ldots,s_q \in \mathcal{K}$, $q \in \NN$, such that $\mathcal{K}$ is covered by the open balls $B_d(\kappa_i,\delta)$ of center $\kappa_i$ and radius $\delta$ for the metric $d$ on $\mathcal{K}$.  Let $L>M \vee \|f_1\|_\infty \vee  \|f_2\|_\infty$. Set $Z_{n,\delta} := f_1(\kappa_n)X_{n,\delta}^{(1)}+f_2(\kappa_n)X_{n,\delta}^{(2)}$. Using \eqref{unifcont}, %we can write 
\[ \{ Z_{n,\delta} > t, X_{n,\delta}^{(1)} \vee X_{n,\delta}^{(2)} \leq M \} \subset  \bigcup_{(i,j) \in T} \{X_{n,\delta}^{(1)}>t_{i,\veps}, X_{n,\delta}^{(2)}>t_{j,\veps}\},\]
where $T :=  \bigcup_{k=1}^{q} \{(i,j) \in[m]^2 : f_1(s_k) t_{i,\veps}+f_2(s_k)t_{j,\veps} >t-3\veps L\}$. 
Using a union bound and the independence of $X_{n,R}^{(1)}$ and $X_{n,R}^{(2)}$, 
\begin{align} \label{eq:rate-fn-J}
 \limsup_{\delta \to 0} \limsup_{n\to+\infty} \frac{1}{v_n} \log \PP\big( Z_{n,\delta} > t, X_{n,\delta}^{(1)} \vee X_{n,\delta}^{(2)} \leq M \big) & \leq - \liminf_{\vep \to 0}  \min_{(i,j) \in T} \big( I_1(t_{i,\veps}) +I_2(t_{j,\veps})\big) \notag\\
& \leq - \liminf_{\vep \to 0}  \Phi(t-3\veps L) \le -\Phi(t),
\end{align}
where %the second inequality uses that $J$ is non decreasing (as $I_1$ and $I_2$ are non decreasing). %we deduce that $J$ is non-decreasing. Thus, 
%\[ \limsup_{\delta \to 0}\limsup_{n\to+\infty} \frac{1}{\corAB{v_n}} \log \PP\big( Z_{n,\delta} > t, X_{n,\delta}^{(1)} \leq M, X_{n,\delta}^{(2)} \leq M\big) \leq - J(t-3\veps M). \]
the last step follows from the fact that $\Phi$ is a good rate function which is a consequence of the facts that $I_1$ and $I_2$ are good rate functions, $f_1$ and $f_2$ are continuous, and $\mathcal{K}$ is compact. 
%Indeed, for any $\eta>0$, $\{\Phi \leq \eta\}= F(\{I\leq \eta\}\times \mathcal{K})$ where $I((t_1,t_2)) =I_1(t_1)+I_2(t_2)$ and $F(t_1,t_2,\kappa)= f_1(\kappa)t_1+f_2(\kappa)t_2$ for any $t_1,t_2\in \RR_+$, $\kappa \in \mathcal{K}$.

 %, it follows from a standard argument that $J$ is also a good rate function. By lower semicontinuity, letting $\veps\to0$ in the above inequality gives 
%\[ \limsup_{\delta \to 0}\limsup_{n\to+\infty} \frac{1}{v(n)} \log \PP\big( Z_{n,\delta} > t, X_{n,\delta}^{(1)} \leq M, X_{n,\delta}^{(2)} \leq M\big) \leq - K(t). \]
%Now, 

On the other hand as $I_1$ and $I_2$ are good rate functions we deduce that 
\[ \limsup_{M\to +\infty} \limsup_{\delta\to +\infty} \limsup_{n\to +\infty } \frac{1}{v_n} \log \PP\big( X_{n,\delta}^{(1)}\vee X_{n,\delta}^{(2)} >M\big)= -\infty.\]
This together with \eqref{eq:rate-fn-J} end the proof. 
\end{proof}

%\section{Proof of Lemma \ref{lem:hL-prop} }\label{sec:pf-lemhL-prop}

%it follows that $h_L$ is differentiable on $\RR$ and $(h_L')_{|(-\infty, x^*)}$ is the inverse of $L'$ on $(-\infty, 1/(2\beta)))$. Moreover, $h_L'(x)=1/(2\beta)$ for $x\geq x^*$. Since $L'$ is strictly convex on $\RR$, it follows that $(h_L')_{|(-\infty,x^*)}$ is strictly concave. 
%\qed

\bibliographystyle{plain}
\bibliography{main.bib}{}

\end{document}